\documentclass[a4paper,11pt,oneside]{amsart}
\usepackage[DIV=14,BCOR=2mm,headinclude=false,footinclude=false]{typearea}
\usepackage{tikz}
\usepackage{amsfonts}
\usepackage{amsmath}
\usepackage{amsthm}
\usepackage{amssymb}
\usepackage{mathtools}
\usepackage[margin=1in,footskip=0.2in]{geometry}
\usepackage{mathtools}
\usepackage{amsbsy, hyperref}
\usepackage{tcolorbox}
\usepackage{bbm}
\usepackage{mathrsfs}
\usepackage{comment}
\usepackage{enumerate}
\usepackage{dsfont}
\usepackage[numbers]{natbib}
\usepackage{graphicx,color}
\usepackage{yfonts}
\usepackage{xcolor}
\usepackage{tabu}
\usepackage{colortbl}
\usepackage[labelfont=rm,format=plain,indention=0cm,singlelinecheck=off,justification=raggedright,skip=2pt]{caption}
\usepackage[fit]{truncate}
\usepackage{float}

\usetikzlibrary{calc,matrix,backgrounds}

\pgfdeclarelayer{overlay}
\pgfsetlayers{overlay,background,main}

\tikzset{circle/.style = {rounded corners,line width=1bp,color=#1}}%

\makeatletter
\renewcommand{\tocsection}[3]{%
  \indentlabel{\@ifnotempty{#2}{\bfseries\ignorespaces#1 #2\quad}}\bfseries#3}
\renewcommand{\tocsubsection}[3]{%
  \indentlabel{\@ifnotempty{#2}{\ignorespaces#1 #2\quad}}#3}

\newcommand\@dotsep{4.5}
\def\@tocline#1#2#3#4#5#6#7{\relax
  \ifnum #1>\c@tocdepth 
  \else
    \par \addpenalty\@secpenalty\addvspace{#2}%
    \begingroup \hyphenpenalty\@M
    \@ifempty{#4}{%
      \@tempdima\csname r@tocindent\number#1\endcsname\relax
    }{%
      \@tempdima#4\relax
    }%
    \parindent\z@ \leftskip#3\relax \advance\leftskip\@tempdima\relax
    \rightskip\@pnumwidth plus1em \parfillskip-\@pnumwidth
    #5\leavevmode\hskip-\@tempdima{#6}\nobreak
    \leaders\hbox{$\m@th\mkern \@dotsep mu\hbox{.}\mkern \@dotsep mu$}\hfill
    \nobreak
    \hbox to\@pnumwidth{\@tocpagenum{\ifnum#1=1\bfseries\fi#7}}\par
    \nobreak
    \endgroup
  \fi}
\AtBeginDocument{%
\expandafter\renewcommand\csname r@tocindent0\endcsname{0pt}
}
\def\l@subsection{\@tocline{2}{0pt}{2.5pc}{5pc}{}}
\makeatother

\newcommand{\Addresses}{{
  \bigskip
  \footnotesize
  Camil Muscalu, \textsc{Department of Mathematics, Cornell University,
    Ithaca, New York 14853}\par\nopagebreak
 \textit{E-mail address}: \texttt{camil@math.cornell.edu}\\
 
  Itamar Oliveira, \textsc{Department of Mathematics, Cornell University,
    Ithaca, New York 14853}\par\nopagebreak
    \textit{Current address}: \textsc{School of Mathematics, The Watson Building, University of Birmingham, Edgbaston, Birmingham, B15 2TT, England}\par\nopagebreak
    
 \textit{E-mail address}: \texttt{oliveira.itamar.w@gmail.com}\par\nopagebreak

}}

\theoremstyle{plain}
\newtheorem*{theorem*}{Theorem}
\newtheorem{theorem}{Theorem}[section]
\newtheorem{corollary}[theorem]{Corollary}
\newtheorem{lemma}[theorem]{Lemma}
\newtheorem{claim}[theorem]{Claim}

\newtheorem{proposition}[theorem]{Proposition}

\newtheorem{conjecture}[theorem]{Conjecture}
\theoremstyle{definition}
\newtheorem{definition}[theorem]{Definition}

\newtheorem{remark}[theorem]{Remark}

\def\eqalign#1{\null\,\vcenter{\openup\jot\mathsurround\dimen12
  \ialign{\strut\hfil$\textstyle{##}$&$\textstyle{{}##}$\hfil
      \crcr#1\crcr}}\,}

\begin{document}

\begin{abstract} We propose a new approach to the Fourier restriction conjectures. It is based on a discretization of the Fourier extension operators in terms of quadratically modulated wave packets. Using this new point of view, and by combining natural scalar and mixed norm quantities from appropriate level sets, we prove that all the $L^{2}$-based $k$-linear extension conjectures are true up to the endpoint for every $1 \leq k \leq d+1$ if one of the functions involved is a full tensor. We also introduce the concept of \textit{weak transversality}, under which we show that all conjectured $L^{2}$-based multilinear extension estimates are still true up to the endpoint provided that one of the functions involved has a weaker tensor structure, and we prove that this result is sharp. Under additional tensor hypotheses, we show that one can improve the conjectured threshold of these problems in some cases. In general, the largely unknown multilinear extension theory beyond $L^{2}$ inputs remains open even in the bilinear case; with this new point of view, and still under the previous tensor hypothesis, we obtain the near-restriction target for the $k$-linear extension operator if the inputs are in a certain $L^{p}$ space for $p$ sufficiently large. The proof of this result is adapted to show that the $k$-fold product of linear extension operators (no transversality assumed) also ``maps near restriction" if one input is a tensor. Finally, we exploit the connection between the geometric features behind the results of this paper and the theory of Brascamp-Lieb inequalities, which allows us to verify a special case of a conjecture by Bennett, Bez, Flock and Lee.
\end{abstract}

\title[A new approach to the Fourier extension problem for the paraboloid]{A new approach to the Fourier extension problem for the paraboloid}
\author{Camil Muscalu and Itamar Oliveira}
\date{}

\dedicatory{Dedicated to the memory of Robert S. Strichartz.}
\maketitle

\tableofcontents

\section{Introduction}

Given a compact submanifold $S\subset\mathbb{R}^{d+1}$ and a function $f:\mathbb{R}^{d+1}\mapsto\mathbb{R}$, the \textit{Fourier restriction problem} asks for which pairs $(p,q)$ one has
$$\| \widehat{f}|_S\|_{L^{q}(S)}\lesssim \|f\|_{L^{p}(\mathbb{R}^{d+1})},$$
where $\widehat{f}|_S$ is the restriction of the Fourier transform $\widehat{f}$ to $S$. This problem arises naturally in the study of certain Fourier summability methods and is known to be connected to questions in Geometric Measure Theory and in nonlinear dispersive PDEs. The interaction between curvature and the Fourier transform has been exploited in a variety of contexts since the works of H\"ormander (\cite{Hor}), Fefferman (\cite{Fef2}) and Stein and Wainger (\cite{SW1}) in the study of oscillatory integrals. For a more detailed description of the restriction problem we refer the reader to the classical survey \cite{Tao-notes}. In this paper we work with the equivalent dual formulation of the question above (known as the \textit{Fourier extension problem}), and specialize to the case where $S$ is the compact piece of the paraboloid parametrized by $\Gamma(x)=(x,|x|^{2})\subset\mathbb{R}^{d+1}$ with $x\in[0,1]^{d}$. In this setting, the \textit{Fourier extension operator} is initially defined on $C([0,1]^{d})$ by
\begin{equation}\label{extop}
\mathcal{E}_{d}g(x_{1},\ldots,x_{d},t)=\int_{[0,1]^{d}}g(\xi_{1},\ldots,\xi_{d})e^{-2\pi i(\xi_{1}x_{1}+\ldots \xi_{d}x_{d})}e^{-2\pi i t(\xi_{1}^{2}+\ldots+\xi_{d}^{2})}\mathrm{d}\xi.
\end{equation}

E. Stein proposed the following conjecture (cf. Chapter IX of \cite{Steinbook2}):
\begin{conjecture}\label{restriction} The inequality
\begin{equation}\label{rest1}
\|\mathcal{E}_{d}g\|_{L^{q}(\mathbb{R}^{d+1})}\lesssim_{p,q,d}\|g\|_{L^{p}([0,1]^{d})}
\end{equation}
holds if and only if $q>\frac{2(d+1)}{d}$ and $q\geq \frac{(d+2)}{d}p^{\prime}$.
\end{conjecture}

Multilinear variants of Conjecture \ref{restriction} arose naturally from the works \cite{KM1},\cite{KM2} and \cite{KM3} of Klainerman and Machedon on wellposedness of certain PDEs. Given $2\leq k\leq d+1$ compact and connected domains $U_{j}\subset\mathbb{R}^{d}$, $1\leq j\leq k$, define

\begin{equation}
\mathcal{E}_{U_{j}}g(x,t):=\int_{U_{j}}g(\xi)e^{-2\pi i x\cdot\xi}e^{-2\pi it|\xi|^{2}}\mathrm{d}\xi,\quad (x,t)\in\mathbb{R}^{d}\times\mathbb{R}.
\end{equation}
Taking the product of all $k$ such operators associated to a set of \textit{transversal} $U_{j}$ leads to the following conjecture (see Appendix \ref{appendixa}):

\begin{conjecture}[\cite{Benn1}]\label{klinear} If the caps parametrized by $U_{j}$ are transversal, then
$$\left\|\prod_{j=1}^{k}\mathcal{E}_{U_{j}}g_{j}\right\|_{p} \lesssim \prod_{j=1}^{k} \|g_{j}\|_{2}$$
for all $p\geq \frac{2(d+k+1)}{k(d+k-1)}$.
\end{conjecture}

Roughly, \textit{transversality} means that any choice of one normal vector per cap is a set of linearly independent vectors, as shown below in Figure \ref{fig2}.

\begin{figure}[h]
\centering
\includegraphics[width=5cm]{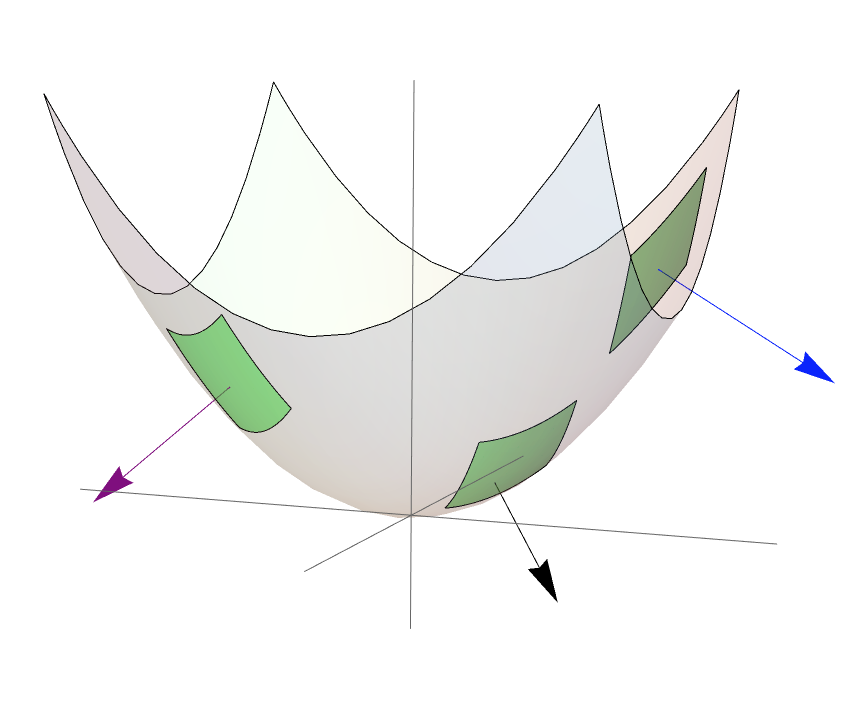}
\caption{A choice of normal vectors to the caps parametrized by $U_{j}$ via $x\mapsto |x|^{2}$.}
\label{fig2}
\end{figure}

\begin{remark} From now on, we shall refer to Conjecture \ref{restriction} as \textit{the case $k=1$}. It was settled only for $d=1$ by Fefferman and Zygmund (\cite{Fef1}, \cite{Zyg}). In higher dimensions we highlight the case $p=2$ solved by Strichartz in \cite{Stri1}, which is equivalent to the Tomas-Stein theorem (\cite{Tomas}) in the restriction setting. Progress beyond these two results was made in many works over the last decades through a diverse set of techniques: localization, bilinear estimates, wave-packet decompositions and more recently polynomial methods. We mention the papers \cite{Bourg1}, \cite{TV1}, \cite{Tao1}, \cite{AVV1}, \cite{Wang}, \cite{Guth1} and \cite{Hick-Rog}. Analogous problems for other manifolds were studied in \cite{Wolff}, \cite{Stri1} and \cite{OW}.
\end{remark}

\begin{remark} In \cite{Guth1}, Guth proved a weaker version of Conjecture \ref{klinear} for all $2\leq k\leq d+1$ and up to the endpoint, which is known as the $k$-broad restriction inequality. This estimate plays a central role in his argument in \cite{Guth1} to improve the range for which Conjecture \ref{restriction} is known.

Only three cases of Conjecture \ref{klinear} are well understood: 
\begin{enumerate}[(i)]
\item Tao settled the case $k=2$ in \cite{Tao1} up to the endpoint inspired by Wolff's work \cite{Wolff} for the cone. Lee obtained the endpoint for $k=2$ in \cite{JL}.
\item Bennett, Carbery and Tao settled the case $k=d+1$ up to the endpoint in \cite{BCT}.
\item Bejenaru settled the case $k=d$ in \cite{Bej} up to the endpoint.
\end{enumerate}

\end{remark}

The goal of this paper is to propose a new approach to these problems based on a natural discretization of the operators in terms of scalar products against quadratically modulated wave-packets. Our main theorem reads as follows:

\begin{theorem}\label{mainthmpaper} Conjectures \ref{restriction} and \ref{klinear} hold up to the endpoint if one (any) of the functions involved is a full tensor\footnote{A function $g$ in $d$ variables is a \textit{full tensor} if it can be written as
$$g(x_{1},\ldots,x_{d})=g_{1}(x_{1})\cdot\ldots\cdot g_{d}(x_{d}). $$
We refer the reader to \cite{Satoru} and \cite{Tanaka} for other results related to the restriction problem involving tensors, and we thank Terence Tao for pointing these papers out to us.}.

\end{theorem}

\begin{remark} The endpoint $(p,q)=(\frac{2(d+1)}{d},\frac{2(d+1)}{d})$ is not included in the range where \eqref{rest1} is supposed to hold, therefore our main theorem implies the case $k=1$ when $g$ is a full tensor.
\end{remark}

\begin{remark}\label{rem1may522} For $2\leq k\leq d+1$, Theorem \ref{mainthmpaper} can be proved if the caps are assumed to be \textit{weakly transversal}, which is defined in Section \ref{twt}. We will prove that transversality implies weak transversality (up to dividing the caps into finitely many pieces), the latter being what is actually exploited in this paper. Under weak transversality, Theorem \ref{mainthmpaper} holds if one (any) of the functions has a weaker tensor structure. This will be made precise in Section \ref{klineartheory}.
\end{remark}

\begin{remark}\label{rem1may722} For $2\leq k\leq d+1$, Theorem \ref{mainthmpaper} is sharp under weak transversality in the following sense: if all functions $g_{1},\ldots,g_{k}$ are generic, it does not hold if the caps are assumed to be weakly transversal. This is explained in Appendix \ref{appendixa}.

\end{remark}

\begin{remark} For $2\leq k\leq d+1$ we do not use the tensor structure explicitly. It is used in an implicit way when comparing the sizes of natural scalar and mixed norm quantities that appear in the proofs.

\end{remark}

\begin{remark} For $2\leq k\leq d$, if all functions involved are full tensors, one has more estimates than those predicted by Conjecture \ref{klinear} assuming extra \textit{degrees of transversality}, as proven in Section \ref{betteresttensors2}.

\end{remark}

It is natural to try to generalize the statement of Conjecture \ref{klinear} for $L^{p}$ inputs rather than just $L^{2}$. A motivation for that is to deeply understand the role played by transversality; as we will see, the farther our inputs are from $L^{2}$, the less impact the configuration of the caps on the paraboloid has in the best possible estimate (with a single exception to be detailed soon). The general statement of the $k$-linear extension conjecture for the paraboloid is (as in \cite{Benn1}):

\begin{conjecture}\label{generalklinearapr422} Let $k\geq 2$ and suppose that $U_{1},\ldots, U_{k}$ parametrize transversal caps of the paraboloid $x\mapsto |x|^{2}$ in $\mathbb{R}^{d+1}$. If $\frac{1}{q}<\frac{d}{2(d+1)}$, $\frac{1}{q}\leq\frac{d+k-1}{d+k+1}\frac{1}{p^{\prime}}$ and $\frac{1}{q}\leq\frac{d-k+1}{d+k+1}\frac{1}{p^{\prime}}+\frac{k-1}{k+d+1}$, then

$$\left\|\prod_{j=1}^{k}\mathcal{E}_{U_{j}}g_{j}\right\|_{L^{\frac{q}{k}}(\mathbb{R}^{d+1})} \lesssim_{p,q} \prod_{j=1}^{k} \|g_{j}\|_{L^{p}(U_{j})}.$$
\end{conjecture}

For $2\leq k<d+1$, to recover the whole range, it is enough\footnote{The full range of estimates follows by interpolation between these two cases and the trivial bound $(p,q)=(1,\infty)$.} to prove Conjecture \ref{klinear} and
\begin{equation}\label{ineq1-100123}
\left\|\prod_{j=1}^{k}\mathcal{E}_{U_{j}}g_{j}\right\|_{L^{\frac{2(d+1)}{kd}+\varepsilon}(\mathbb{R}^{d+1})} \lesssim_{\varepsilon} \prod_{j=1}^{k} \|g_{j}\|_{L^{\frac{2(d+1)}{d}}(U_{j})},
\end{equation}
for all $\varepsilon>0$.

\begin{remark} Observe that \eqref{ineq1-100123} covers the case $(p,q)=\left(\frac{2(d+1)}{d},\frac{2(d+1)}{d}+\varepsilon\right)$ of Conjecture \eqref{generalklinearapr422}. Notice also that this case would follow from the case $(p,q)=\left(\frac{2(d+1)}{d},\frac{2(d+1)}{d}+\varepsilon\right)$ of the \textit{linear} extension conjecture \ref{restriction} and H\"older's inequality. This means that the closer we get to the endpoint extension exponent, the less improvements transversality yield in the multilinear theory. The exception to this is the $k=d+1$ case, for which $L^{2}$ functions give the best possible output for the corresponding multilinear operator (rather than $L^{\frac{2(d+1)}{d}}$). Indeed, when one function is a tensor, the best result in this case are obtained in Section \ref{dlineartheory}.
\end{remark}

By adapting the argument that shows the case $2\leq k\leq d+1$ of Theorem \ref{mainthmpaper}, we are able to prove the following weaker version of \eqref{ineq1-100123}:

\begin{theorem}\label{thm1-240123} Let $2\leq k<d+1$. If $g_{1}$ is a tensor in addition to the hypotheses of Conjecture \ref{generalklinearapr422}, the following estimate holds:
\begin{equation}
\left\|\prod_{j=1}^{k}\mathcal{E}_{U_{j}}g_{j}\right\|_{L^{\frac{2(d+1)}{kd}+\varepsilon}(\mathbb{R}^{d+1})} \lesssim_{\varepsilon} \prod_{j=1}^{k} \|g_{j}\|_{L^{p(k,d)}(U_{j})}
\end{equation}
for all $\varepsilon>0$, where
$$
p(k,d)=
\begin{cases}
\frac{4(d+1)}{d+k+1},\quad\textnormal{if }2\leq k<\frac{d}{2},\\
\frac{4(d+1)}{2d-k+1},\quad\textnormal{if }\frac{d}{2}\leq k<d+1.
\end{cases}
$$
\end{theorem}

\begin{remark} Notice that $\frac{2(d+1)}{d}<p(k,d)$, so Theorem \ref{thm1-240123} is not optimal on the space of the input functions. On the other hand, the output $L^{\frac{2(d+1)}{kd}+\varepsilon}$ (for all $\varepsilon>0$) is the best to which one can hope to map the multilinear operator on the left-hand side. The case $k=d+1$ of the theorem above coincides with the case $k=d+1$ of the $L^{2}$-based theory, which is covered in Section \ref{dlineartheory}.
\end{remark}

\begin{remark} Bounds such as the one from Theorem \ref{thm1-240123}, i.e. in which one needs $p$ big enough (and not sharp) to map $L^{p}$ inputs to a fixed $L^{q}$, are common in the linear extension theory. For example, in \cite{Wang} Wang shows that $\mathcal{E}_{2}$ maps $L^{\infty}([-1,1]^{2})$ to $L^{q}(\mathbb{R}^{3})$ for $q>3+\frac{3}{13}$. As mentioned in \cite{Wang}, this implies the (seemingly stronger) bound
$$\|\mathcal{E}_{2}g\|_{L^{q}(\mathbb{R}^{3})}\lesssim_{q} \|g\|_{L^{q}([-1,1]^{2})} $$
for $q>3+\frac{3}{13}$ via the factorization theory of Nikishin and Pisier (see Bourgain's paper \cite{Bourg1}).
\end{remark}

\begin{remark} The multilinear extension theory for inputs near $L^{\frac{2(d+1)}{d}}$ remains largely unknown in general (except for the almost optimal result in the $k=d+1$ case in \cite{BCT}). In fact, it is not fully settled even in the $k=2$, $d>1$ case (whose $L^{2}$-based analogue is known). We refer the reader to the recent paper \cite{COh} for partial results in this direction.
\end{remark}

\begin{remark} As the reader may expect, any function can be taken to be the tensor in the statement of Theorem \ref{thm1-240123}.
\end{remark}

The linear and multilinear theories studied in this paper meet very naturally once more in the context of the techniques we use: the simplest multilinear variant of a linear operator $T$ is given by the product of a certain number of identical copies of it:

$$T_{(k)}(g_{1},\ldots,g_{k}):= \prod_{j=1}^{k}Tg_{j}. $$

Proving that $T$ maps $L^{p}(U)$ to $L^{q}(V)$ is equivalent to proving that $T_{(k)}$ maps $L^{p}(U)$ to $L^{\frac{q}{k}}(V)$, as one can easily check with H\"older's inequality. Multilinearizing $\mathcal{E}_{d}$ without any regard to transversality yields the operator

\begin{equation}\label{kprod-apr2923}
\mathcal{E}_{d,(k)}(g_{1},\ldots,g_{k}):= \prod_{j=1}^{k}\mathcal{E}_{d,(k)}g_{j}.
\end{equation}

Combining the previous observation with the factorization theory of Nikishin and Pisier, Conjecture \ref{restriction} follows from the bound

\begin{equation}\label{boundnotransv-apr2823}
\left\|\prod_{j=1}^{k}\mathcal{E}_{d,(k)}g_{j}\right\|_{L^{\frac{2(d+1)}{kd}+\varepsilon}}\lesssim_{\varepsilon} \prod_{j=1}^{k}\|g_{j}\|_{L^{\infty}([0,1]^{d})}.
\end{equation}

The proof of Theorem \ref{thm1-240123} can be adapted to show the following:

\begin{theorem}\label{thm1-apr2823} Let $2\leq k\leq d+1$. If $g_{1}$ is a tensor, the inequality
\begin{equation}
\left\|\prod_{j=1}^{k}\mathcal{E}_{d,(k)}g_{j}\right\|_{L^{\frac{2(d+1)}{kd}+\varepsilon}(\mathbb{R}^{d+1})} \lesssim_{\varepsilon} \prod_{j=1}^{k} \|g_{j}\|_{L^{4}([0,1]^{d})}
\end{equation}
holds for all $\varepsilon>0$.
\end{theorem}

\begin{remark} Since the inputs $g_{j}$ are compactly supported, Theorem \ref{thm1-apr2823} implies \eqref{boundnotransv-apr2823}.
\end{remark}

\begin{remark} Given that the proof of Theorem \ref{thm1-apr2823} has the $L^{4}-L^{4+\varepsilon}$ bound for $\mathcal{E}_{1}$ as its main building block, it is not surprising that we have a product of $L^{4}$ norms in the right-hand side of the statement above.
\end{remark}

We finish this introduction by highlighting the close connection between our results and the  theory of linear and non-linear Brascamp-Lieb inequalities. The concept of \textit{weak transversality} that we introduce can be characterized in terms of certain Brascamp-Lieb data, and by exploiting the geometric features arising from this fact we are able to verify a special case of a conjecture by Bennett, Bez, Flock and Lee. 

The paper is organized as follows: in Section \ref{discret} we present the linear and multilinear models that we will work with in the proof of Theorem \ref{mainthmpaper}. We also highlight the main differences between the linearized models that are used in most recent approaches and ours. In Section \ref{twt} we define the concepts of transversality and weak transversality, and state in what sense the former implies the latter. Section \ref{descriptionmethod} presents what we refer to as the \textit{building blocks} of our approach. Sections \ref{TSreproven}, \ref{dim1restriction} and \ref{bilinearparabola} establish these building blocks: in Section \ref{TSreproven} we revisit the case $k=1$ and $p=2$ for our model, in Section \ref{dim1restriction} we revisit Zygmund's argument and recover the case $k=1$ for $d=1$, and in Section \ref{bilinearparabola} we deal with the case $k=2$ and $d=1$. In Section \ref{lineartheory} we settle the case $k=1$ of Theorem \ref{mainthmpaper}, and in Section \ref{klineartheory} we show the cases $2\leq k\leq d+1$. Section \ref{dlineartheory} covers the endpoint estimate of the case $k=d+1$. In Section \ref{betteresttensors2} we discuss how one can improve the bounds of Conjecture \ref{klinear} under extra transversality and tensor hypotheses. Theorem \ref{thm1-240123} (our partial result beyond the $L^{2}$-based $k$-linear theory) is presented in Section \ref{beyondL2} along with its ``non-transversal" counterpart Theorem \ref{thm1-apr2823}. In Section \ref{WTBLA} we establish a connection between the classical theory of Brascamp-Lieb inequalities and our results, and give an application of this link to a conjecture made in \cite{BBFL1}. In Section \ref{finalremarks} we make a few additional remarks. Appendix \ref{appendixa} contains examples that show that the range of $p$ in Conjecture \ref{klinear} is sharp, and also that one can not obtain this range in general under a condition that is strictly weaker than transversality. Appendix \ref{appendixb} contains technical results used throughout the paper.

We thank David Beltr\'an, Jonathan Bennett, Emanuel Carneiro, Andr\'es Fernandez, Jonathan Hickman, Victor Lie, Diogo Oliveira e Silva, Keith Rogers, Mateus Sousa, Terence Tao, Joshua Zahl and the anonymous referees for many important remarks, corrections and for pointing out references in the literature.

\section{Discrete models}\label{discret}

A common first step of the earlier works is to \textit{linearize} the contribution of the quadratic phase $x\mapsto |x|^{2}$. One starts by studying $\mathcal{E}_{d}g$ on a ball of radius $R$ (hence $|(x,t)|\leq R$) and splits the domain of $g$ into balls $\theta_{k}$ of radius $R^{-\frac{1}{2}}$. Let us consider $d=1$ here for simplicity. If
$$g_{\theta_{k}}:= g\cdot\varphi_{\theta_{k}}, $$
where $\varphi_{\theta_{k}}$ is a bump adapted to $[kR^{-\frac{1}{2}},(k+1)R^{-\frac{1}{2}}]$, the quadratic exponential 

\begin{equation}
e_{x,t}(\xi) = e^{2\pi ix\xi}e^{2\pi it\xi^{2}}
\end{equation}
behaves in a similar way to a linear exponential $e^{i\# \xi}$ when restricted to this interval. Indeed, the phase-space portrait of $e_{x,t}$ is the (oblique if $t\neq 0$) line
$$u\mapsto x+2tu, $$
as it is explained in more detail in Chapter 1 of \cite{MS2}. When we evaluate this line at the endpoints of the support of $g_{\theta_{k}}$ (taking into account that $|t|\leq R$), we see that the phase-space portrait of

 $$\varphi_{\theta_{k}}\cdot e_{x,t}$$
is a parallelogram that essentially coincides with the rectangle
\begin{equation}\label{intervalfreq}
I\times J = [kR^{-\frac{1}{2}},(k+1)R^{-\frac{1}{2}}]\times [x+2tkR^{-\frac{1}{2}},x+2tkR^{-\frac{1}{2}}+R^{\frac{1}{2}}].
\end{equation}
Observe that $I\times J$ has area $1$. On the other hand, the phase-space portrait of $\varphi_{\theta_{k}}$ is a Heisenberg box of sizes $R^{-\frac{1}{2}}$ and $R^{\frac{1}{2}}$, and the linear modulation
\begin{equation}\label{linearoct242021}
e^{2\pi i\xi(x+2tkR^{-\frac{1}{2}})}
\end{equation}
shifts it in frequency to $J$. The conclusion is that the phase-space portrait of 

$$\varphi_{\theta_{k}}\cdot e^{2\pi i\xi(x+2tkR^{-\frac{1}{2}})} $$
is the Heisenberg box \eqref{intervalfreq}, hence the effect of the quadratic modulation $e_{x,t}$ in this setting is essentially the same as the linear one in \eqref{linearoct242021}.

Using bumps such as $\varphi_{\theta}$ to decompose the domain of $g$ and expanding each $g_{\theta}$ into Fourier series allows us to write

$$g(x)=\sum_{\theta\in R^{-\frac{1}{2}}\mathbb{Z}^{d}\cap [0,1]^{d}}\overbrace{g(x)\varphi_{\theta}(x)}^{g_{\theta}(x)}\widetilde{\varphi}_{\theta}(x)=\sum_{\theta\in R^{-\frac{1}{2}}\mathbb{Z}^{d}\cap [0,1]^{d}}\sum_{\nu\in R^{\frac{1}{2}}\mathbb{Z}^{d}}\overbrace{c_{\nu,\theta}e^{2\pi ix\cdot\nu}\widetilde{\varphi}_{\theta}(x)}^{g_{\theta,\nu}(x)},$$
where $\widetilde{\varphi}_{\theta}$ is $\equiv 1$ on the support of $\varphi_{\theta}$ and decays very fast away from it. Applying $\mathcal{E}_{d}$ and using the previous intuition gives rise to the \textit{wave packet decomposition}
$$\mathcal{E}_{d}g=\sum_{(\theta,\nu)\in R^{-\frac{1}{2}}\mathbb{Z}^{d}\cap [0,1]^{d}\times R^{\frac{1}{2}}\mathbb{Z}^{d}}\mathcal{E}_{d}(g_{\theta,\nu}), $$
where $\mathcal{E}_{d}(g_{\theta,\nu})$ is essentially supported on a tube in $\mathbb{R}^{d+1}$ of size $R^{\frac{1}{2}}\times\ldots\times R^{\frac{1}{2}}\times R$ whose direction is determined by $\theta$ and that is translated by a parameter depending on $\nu$. With this linearized model at hand, one can study the interference between these tubes pointing in different directions (both in the linear and multilinear settings) and take advantage of orthogonality both in space and in frequency. This leads to local estimates of type
$$\|\mathcal{E}_{d}g\|_{L^{q}(B(0,R))}\lesssim_{\varepsilon}R^{\varepsilon}\|f\|_{p},\quad\forall\varepsilon>0 $$
and multilinear analogues of it that are later used to obtain global estimates via $\varepsilon$-removal arguments (as in \cite{Tao2}). The reader is referred to \cite{Guth2} for the details of the decomposition above. This approach has given the current best $L^{p}$ bounds for $\mathcal{E}_{d}$.

In our case, we do not linearize the contribution of the quadratic phase. Instead, we consider a discrete model that keeps the quadratic nature of $\mathcal{E}_{d}$ intact.

\subsection{The linear model ($k=1$)}\label{discretizationlinear}

We consider $d=1$ for simplicity, but the discretization process is analogous for all $d>1$. Recall that the extension operator for the parabola defined for functions supported on $[0,1]$ is given by

\begin{equation}\label{def1-apr2823}
\mathcal{E}_{1}g(x,t)=\int_{0}^{1}g(\xi)e^{-2\pi ix\xi}e^{-2\pi it\xi^{2}}\mathrm{d}\xi.
\end{equation}

We can insert a bump $\varphi$ in the integrand that is equal to $1$ on $[0,1]$ and supported in a small neighborhood of this interval. Tiling $\mathbb{R}^{2}$ with unit squares with vertices in $\mathbb{Z}^{2}$ and rewriting $\mathcal{E}_{1}$,

$$\mathcal{E}_{1}g(x,t)=\sum_{n,m\in\mathbb{Z}}\left[\int g(u)\varphi(u)e^{-2\pi ixu}e^{-2\pi itu^{2}}\mathrm{d}u\right]\chi_{n}(x)\chi_{m}(t), $$
where $\chi_{n}:=\chi_{[n,n+1)}$. For a fixed $(x,t)$, one can write
\begin{equation*}
    \eqalign{
    e^{-2\pi ix\xi} e^{-2\pi it\xi^{2}}\varphi(\xi)&\displaystyle = e^{-2\pi in\xi} e^{-2\pi im\xi^{2}}\cdot e^{-2\pi i(x-n)\xi}e^{-2\pi i(t-m)\xi^{2}}\varphi(\xi) \cr
    &\displaystyle= e^{-2\pi in\xi} e^{-2\pi im\xi^{2}}\cdot\sum_{u\in\mathbb{Z}}\langle e^{-2\pi i(x-n)(\cdot)}e^{-2\pi i(t-m)(\cdot)^{2}},\varphi_{[0,1]}^{u}\rangle\cdot\varphi_{[0,1]}^{u}(\xi) \cr
    &= \displaystyle e^{-2\pi in\xi} e^{-2\pi im\xi^{2}}\cdot\sum_{u\in\mathbb{Z}}C^{n,m,x,t}_{u}\cdot\varphi_{[0,1]}^{u}(\xi) \cr
    }
\end{equation*}
where we expanded $e^{-2\pi i(x-n)\xi} e^{-2\pi i(t-m)\xi^{2}}$ as a Fourier series at scale $1$,
$$C^{n,m,x,t}_{u}:=\langle e^{-2\pi i(x-n)(\cdot)}e^{-2\pi i(t-m)(\cdot)^{2}},\varphi_{[0,1]}^{u}\rangle\,$$
$$\varphi_{[0,1]}^{u}(\xi) := \varphi_{[0,1]}(\xi)\cdot e^{-2\pi i u\cdot\xi}$$
and $\varphi_{[0,1]}$ is a bump adapted to $[0,1]$ (and compactly supported) just like\footnote{We will not distinguish between $\varphi_{[0,1]}$ and $\varphi$ from now on.} $\varphi$. Plugging this in \eqref{def1-apr2823},

\begin{equation*}
\eqalign{
\displaystyle\mathcal{E}_{1}g(x,t)&\displaystyle =\sum_{n,m\in\mathbb{Z}}\left[\int g(\xi)\varphi(\xi)e^{-2\pi ix\xi}e^{-2\pi it\xi^{2}}\mathrm{d}\xi\right]\chi_{n}(x)\chi_{m}(t)\cr
&\displaystyle=\sum_{n,m\in\mathbb{Z}}\left[\int g(\xi)\left(e^{-2\pi in\xi}e^{-2\pi im\xi^{2}}\cdot\sum_{u\in\mathbb{Z}}C^{n,m,x,t}_{u}\cdot\varphi^{u}(\xi)\right)\mathrm{d}\xi\right]\chi_{n}(x)\chi_{m}(t)\cr
&\displaystyle=\sum_{u\in\mathbb{Z}}\sum_{n,m\in\mathbb{Z}}C^{n,m,x,t}_{u}\cdot\left[\int g(\xi)e^{-2\pi in\xi} e^{-2\pi im\xi^{2}}\cdot\varphi^{u}(\xi)\mathrm{d}\xi\right]\chi_{n}(x)\chi_{m}(t).\cr
}
\end{equation*}

For the expression defining $\mathcal{E}_{1}$ to be nonzero, $(n,m)$ must satisfy $|x-n|\leq 1$ and $|t-m|\leq 1$, hence the Fourier coefficients $C_{u}^{n,m,x,t}$ decay like $O(|u|^{-100})$. In addition, the extra factor $\varphi^{u}$ in the integral simply shifts the integrand in frequency, and this does not affect in any way the arguments that follow. In order to obtain the final form of our linear model, let us introduce the following notation: if $\varphi$ is a compactly supported bump (say, in a very small open neighborhood of $[0,1]^{d}$) with $\varphi\equiv 1$ on $[0,1]^{d}$ we set

\begin{equation}\label{2wp}
\varphi_{\overrightarrow{\textit{\textbf{n}}},m}(x):=\varphi(x)e^{2\pi ix\cdot\overrightarrow{\textit{\textbf{n}}}}e^{2\pi i|x|^{2}m}.
\end{equation}

Due to the fast decay of $C_{u}^{n,x}$ and $C_{v}^{m,t}$, it is then enough to bound the $u=v=0$ piece of the sum above, which leads to the discretized model\footnote{There is a slight abuse of notation here: observe that $\widetilde{\chi}_{n}(x)\widetilde{\chi}_{m}(t):=C_{0}^{n,m,x,t}\cdot\chi_{n}(x)\chi_{m}(t)$ is a smooth function supported in $[n,n+1)\times [m,m+1)$, which is all that is needed in the proof. We will continue to call it $\chi_{n}(x)\chi_{m}(t)$ to lighten the notation.}:

$$E_{1}(g)=\sum_{(n,m)\in\mathbb{Z}^{2}}\langle g,\varphi_{n,m}\rangle(\chi_{n}\otimes\chi_{m}).$$

With the appropriate adaptations, one proceeds in the exact same way in dimension $d$ to reduce matters to the study of the following model operator:

\begin{definition}\label{linearmodelfeb21} Let $E_{d}$ be defined on $C([0,1]^{d})$ given by

$$E_{d}(g)=\sum_{\substack{\overrightarrow{\textit{\textbf{n}}}\in\mathbb{Z}^{d} \\ m\in\mathbb{Z}}}\langle g,\varphi_{\overrightarrow{\textit{\textbf{n}}},m}\rangle(\chi_{\overrightarrow{\textit{\textbf{n}}}}\otimes\chi_{m}), $$
where $\chi_{\overrightarrow{\textit{\textbf{n}}}}$ and $\chi_{m}$ are the characteristic functions of the boxes $[n_{1},n_{1}+1)\times\ldots\times[n_{d},n_{d}+1)$ and $[m,m+1)$, respectively.\footnote{Morally speaking, the discrete model and the original operator are ``comparable", but we were not able to prove that rigorously. For that reason we included the proof of known extension estimates for $E_{d}$.}
\end{definition}

The wave packets \eqref{2wp} have a natural phase-space portrait that consist of parallelograms in the phase plane.

\begin{figure}[H]
  \centering
\captionsetup{font=normalsize,skip=1pt,singlelinecheck=on}
  \includegraphics[scale=0.65]{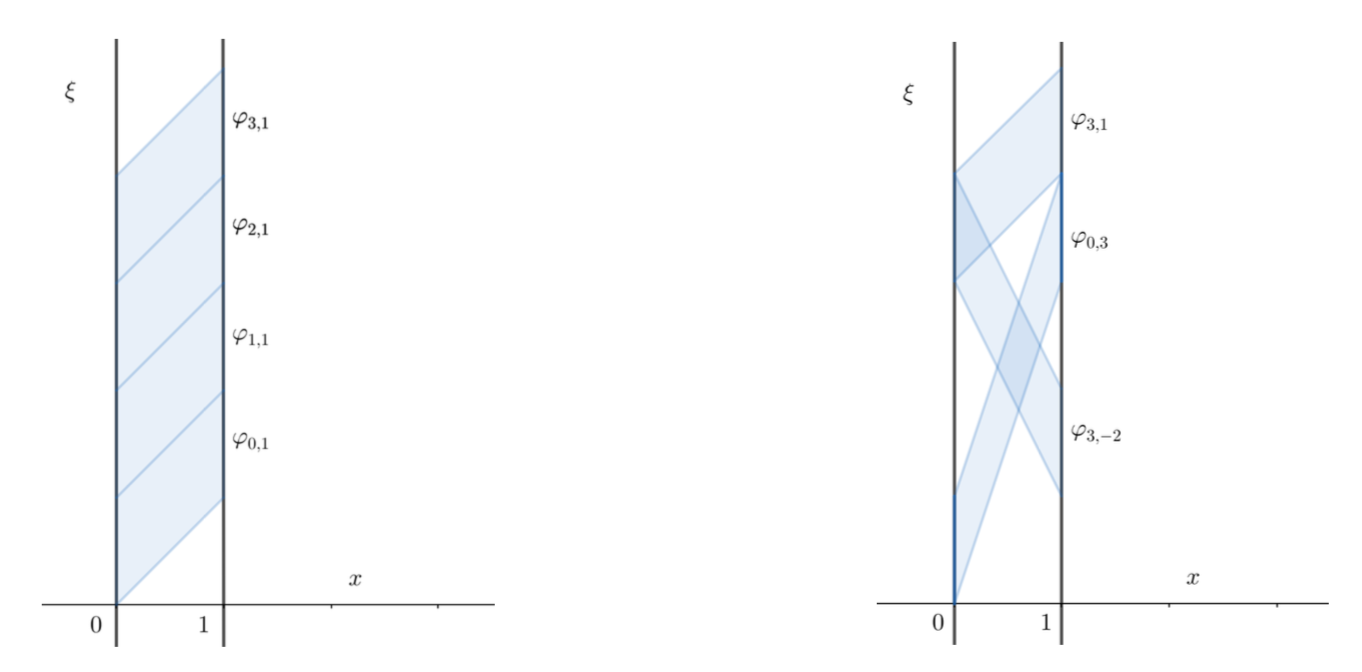}
  \caption{The phase-space portrait of $\varphi_{n,m}$}
\end{figure}

By keeping the quadratic nature of $E_{d}$ intact we take advantage of orthogonality in different ways. For example, for a fixed $m$ the wave packets $\varphi_{n,m}$ are almost orthogonal, as suggested by the fact that the corresponding parallelograms are (almost) disjoint.

\subsection{The multilinear model ($2\leq k\leq d+1$)} We recall the definition of the $k$-linear extension operator:

\begin{definition}\label{defdlinear} For $\mathcal{Q}=\{Q_{1},\ldots,Q_{k}\}$ a transversal set of cubes, the $k$-linear extension operator is given by
\begin{equation}\label{multop-apr2823}
\mathcal{M}\mathcal{E}_{k,d}(g_{1},\ldots,g_{k}):=\prod_{j=1}^{k}\mathcal{E}_{Q_{j}}g_{j},
\end{equation}
where
$$\mathcal{E}_{Q_{j}}g_{j}(x,t)=\int_{Q_{j}}g_{j}(\xi)e^{-2\pi ix\cdot\xi}e^{-2\pi it|\xi|^{2}}\mathrm{d}\xi, \quad (x,t)\in\mathbb{R}^{d}\times\mathbb{R}. $$
\end{definition}

By an analogous argument to the one we showed in Subsection \ref{discretizationlinear}, it is enough to prove the corresponding bounds for the following model operator:

\begin{definition}\label{defklinearmodel} Let $ME_{k,d}$ be defined on $C(Q_{1})\times\ldots\times C(Q_{k})$ by
$$ME_{k,d}(g_{1},\ldots,g_{k}):=\sum_{(\overrightarrow{\textit{\textbf{n}}},m)\in\mathbb{Z}^{d+1}}\prod_{j=1}^{k}\langle g_{j},\varphi^{j}_{\overrightarrow{\textit{\textbf{n}}},m}\rangle(\chi_{\overrightarrow{\textit{\textbf{n}}}}\otimes\chi_{m}).$$

where 
$$\varphi^{j}_{\overrightarrow{\textit{\textbf{n}}},m}=\bigotimes_{l=1}^{d}\varphi^{l,j}_{n_{l},m},\quad \varphi^{l,j}_{n_{l},m}(x_{l})=\varphi^{l,j}(x_{l})e^{2\pi in_{l}x_{l}}e^{2\pi imx_{l}^{2}} $$
and $\varphi^{l,j}(x)$ is $\equiv 1$ on the $l$-coordinate projection of the domain of $g_{j}$ defined above and decays fast away from it.

\end{definition}

\begin{remark} It is clear that the discretization process does not depend on whether the collection $\mathcal{Q}$ is made of transversal cubes or not. In particular, it will be of interest in Subsection \ref{nontransv-apr2823} to study the operator given by the right-hand side of \eqref{multop-apr2823}, but \textit{without} the assumption that the cubes $Q_{j}$ are transversal. The model for such operator is also given by $ME_{k,d}$, but without that hypothesis.
\end{remark}

\section{Transversality versus weak transversality}\label{twt}

We recall the following definition from \cite{Benn1}:

\begin{definition}\label{deftransversality}
Let $2\leq k\leq d+1$ and $c>0$. A $k$-tuple $S_{1},\ldots,S_{k}$ of smooth codimension-one submanifolds of $\mathbb{R}^{d+1}$ is \textit{$c$-transversal} if
$$|v_{1}\wedge\ldots\wedge v_{k}|\geq c $$
for all choices $v_{1},\ldots,v_{k}$ of unit normal vectors to $S_{1},\ldots,S_{k}$, respectively. We say that $S_{1},\ldots,S_{k}$ are \textit{transversal} if they are $c$-transversal for some $c>0$.
\end{definition}

In other words, if the $k$-dimensional volume of the parallelepiped generated by $v_{1},\ldots,v_{k}$ is bounded below by some absolute constant for any choice of normal vectors $v_{j}$, the submanifolds are transversal. From now on, we will say that a collection of $k$ cubes in $\mathbb{R}^{d}$ is \textit{transversal} if the associated caps defined by them on the paraboloid are transversal in the sense of Definition \ref{deftransversality}.

One can assume without loss of generality that the $U_{j}$'s in the statements of Conjecture \ref{klinear} are cubes that parametrize transversal caps on $\mathbb{P}^{d}$ via the map $x\mapsto |x|^{2}$. Even though these conjectures are known to fail in general if one does not assume transversality between the caps (see Appendix \ref{appendixa2}), the theorem that we will prove holds under a weaker condition, since one of the functions is a tensor.

\begin{definition}\label{weaktransv} Let $\mathcal{Q}=\{Q_{1},\ldots, Q_{k}\}$ be a collection of $k$ (open or closed) cubes\footnote{The word \textit{cube} will be used throughout the paper to refer to any rectangular box in $\mathbb{R}^{d}$, regardless of the sizes of its edges, and they always refer to the supports of the input functions of our linear and multilinear operators. In this paper, it will not be relevant whether the sides of a box have the same length or not, therefore this slight abuse of terminology is harmless.} in $\mathbb{R}^{d}$. $\mathcal{Q}$ is said to be \textit{weakly transversal with pivot $Q_{j}$} if for all $1\leq j\leq k$ there is a set of $(k-1)$ distinct directions $\mathscr{E}_{j}=\{e_{i_{1}},\ldots,e_{i_{k-1}}\}$ (depending on $j$) of the canonical basis such that

\begin{equation}\label{eqwtmay722}
    \begin{dcases}
        \overline{\pi_{i_{1}}(Q_{j})} \cap \overline{\pi_{i_{1}}(Q_{1})} =\emptyset, \\
       \qquad \vdots \\
        \overline{\pi_{i_{j-1}}(Q_{j})} \cap \overline{\pi_{i_{j-1}}(Q_{j-1})} =\emptyset, \\
        \overline{\pi_{i_{j}}(Q_{j})} \cap \overline{\pi_{i_{j}}(Q_{j+1})} =\emptyset, \\
        \qquad \vdots \\
        \overline{\pi_{i_{k-1}}(Q_{j})} \cap \overline{\pi_{i_{k-1}}(Q_{k})} =\emptyset, \\
    \end{dcases}
\end{equation}
where $\pi_{l}$ is the projection onto $e_{l}$. We say that $\mathcal{Q}$ is \textit{weakly transversal} if it is weakly transversal with pivot $Q_{j}$ for all $1\leq j\leq k$.\footnote{The estimates that we will prove depend on the separation of the projections in Definition \ref{weaktransv}, just as they depend on the behavior of $c$ from Definition \ref{deftransversality} in the general case for transversal caps.}

\end{definition}

\begin{remark} For each $1\leq j\leq k$, from now on we will refer to a set\footnote{The typeface $\mathscr{E}_{j}$ is being used to distinguish this concept from the previously defined operators $\mathcal{E}_{d}$ and $E_{d}$.} $\mathscr{E}_{j}$ above as \textit{a set of directions associated to $Q_{j}$}. Notice that there could be many of such sets for a single $j$. Also, if $j_{1}\neq j_{2}$, it could be the case that no set of directions associated to $Q_{j_{1}}$ is associated to $Q_{j_{2}}$.  
\end{remark}

Let us give a few examples to distinguish between definitions \ref{deftransversality} and \ref{weaktransv}. Consider the case $d=2$, $k=3$, $Q_{1}=[0,1]^{2}$, $Q_{2}=[2,3]^{2}$, and $Q_{3}=[4,5]^{2}$. The line $y=x$ intersects $Q_{1}$, $Q_{2}$ and $Q_{3}$, then it follows from Definition \ref{deftransversality} that they are not transversal. However, observe that

\[
    \begin{dcases}
        \pi_{1}(Q_{1}) \cap \pi_{1}(Q_{2}) =\emptyset, \\
        \pi_{2}(Q_{1}) \cap \pi_{2}(Q_{3}) =\emptyset, \\
    \end{dcases}
\]
so $\{e_{1},e_{2}\}$ is a set associated to $Q_{1}$ (and similarly one can verify that it is also associated to $Q_{2}$ and $Q_{3}$). This shows that the collection defined by $Q_{1}$, $Q_{2}$ and $Q_{3}$ is weakly transversal.

Consider now the cubes $K_{1}=[0,1]^{2}$, $K_{2}=[4,5]\times [0,1]$ and $K_{3}=[2,3]^{2}$. Not only are they transversal in the sense of Definition \ref{deftransversality}, but also weakly transversal.

\begin{figure}[H]
  \centering
\captionsetup{font=normalsize,skip=1pt,singlelinecheck=on}
  \includegraphics[scale=.55]{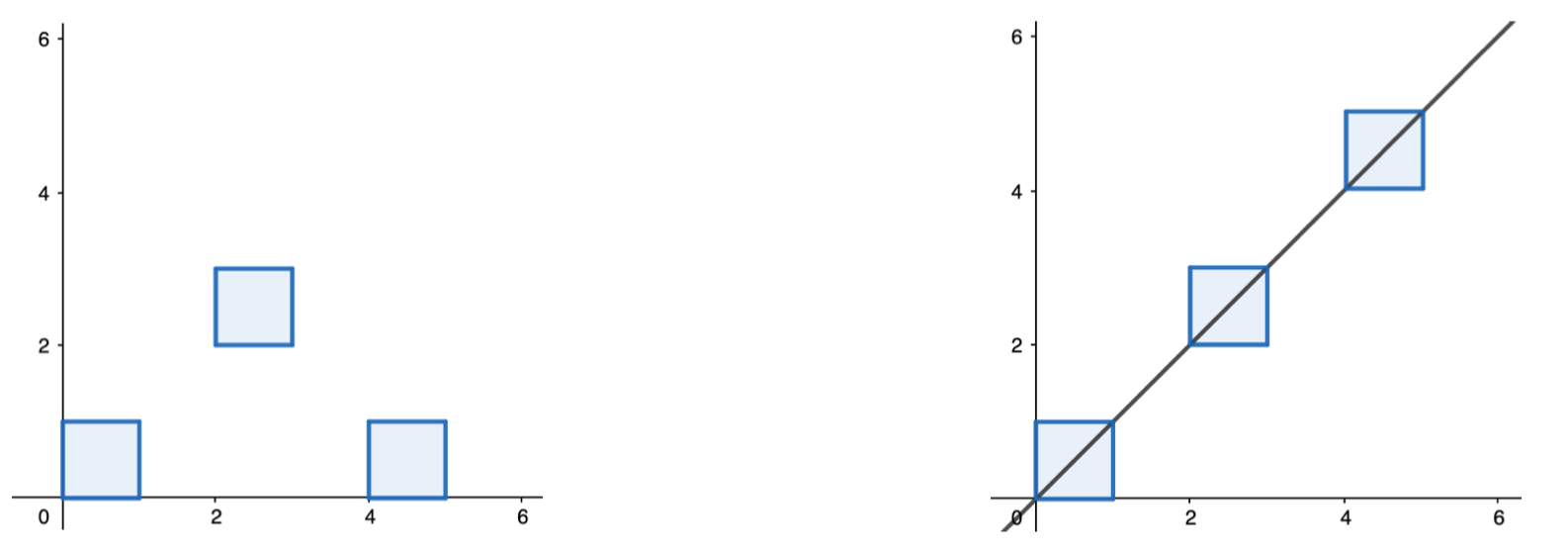}
 \caption{Transversality versus weak transversality}
\end{figure}
This is not by chance: a given transversal collection of $k$ cubes can be ``decomposed" into finitely many collections of $k$ cubes that are \textbf{also} weakly transversal. 

\begin{claim}\label{claim2-081221} Given a collection $\mathcal{Q}=\{Q_{1},\ldots,Q_{k}\}$ of transversal cubes, each $Q_{l}\in\mathcal{Q}$ can be partitioned into $O(1)$ many sub-cubes
$$Q_{l}=\bigcup_{i}Q_{l,i} $$
so that all collections $\widetilde{\mathcal{Q}}$ made of picking one sub-cube $Q_{l,i}$ per $Q_{l}$
$$\widetilde{\mathcal{Q}}=\{\widetilde{Q}_{1},\ldots,\widetilde{Q}_{k}\},\quad \widetilde{Q}_{l}\in\{Q_{l,i}\}_{i},$$
are weakly transversal.

\end{claim}

\begin{proof} See Claim \ref{claima2-081221} in the appendix.

\end{proof}
As a consequence of Claim \ref{claim2-081221}, to prove the case $2\leq k\leq d+1$ of Theorem \ref{mainthmpaper} it suffices to show it for weakly transversal collections. To simplify the exposition, we will present our results for the cubes

\begin{equation*}
\eqalign{
Q_{1}&=[0,1]^{d}, \cr
Q_{j}&= [2,3]^{j-2}\times [4,5]\times [0,1]^{d-j+1},\quad 2\leq j\leq k. \cr
}
\end{equation*}
The associated directions to $Q_{1}$ are $\{e_{1},\ldots,e_{k-1}\}$, and we will use it as the pivot. Any other weakly transversal collection of cubes can be dealt with in the same way.

\section{Our approach and its building blocks}\label{descriptionmethod}

Notice that the operators $\mathcal{E}_{d}$ and $\mathcal{ME}_{k,d}$ are pointwise bounded by $E_{d}$ and $ME_{k,d}$, respectively, therefore we can not directly conclude any result about the models from the fact that they hold for the original operators. Some of these results will be reproven for the models in this paper, and they will act as \textit{building blocks} in the proof of Theorem \ref{mainthmpaper}, which is presented in Sections \ref{lineartheory} and \ref{klineartheory}. More precisely, Theorem \ref{mainthmpaper} relies on the following:
\vspace{.5cm}

\begin{enumerate}
\item \underline{Mixed norm Strichartz/Tomas-Stein ($k=1$, $p=2$).} In Section \ref{TSreproven} we show the following: 

\begin{proposition}\label{prop1apr2922} For all $p > \frac{2(d+2)}{d}$,
$$\|E_{d}g\|_{p}\lesssim_{p}\|g\|_{2}. $$
\end{proposition}
As a consequence, we have:

\begin{corollary}\label{cor1apr2922} For all $\varepsilon>0$,

\begin{equation}\label{mixedfeb2122}
\|E_{d}(g)\|_{L^{\frac{2(d-l+2)}{d-l}+\varepsilon}_{x_{l+1},\ldots,x_{d},t}L^{2}_{x_{1},\ldots,x_{l}}}\lesssim_{\varepsilon} \|g\|_{2}.
\end{equation}
\end{corollary}
\begin{proof} Apply Minkowski's inequality and Proposition \ref{prop1apr2922} in dimension $d-l$\footnote{Notice that, after taking $L^{2}$ norm in the first $l$ variables, we can use Bessel to bound the left-hand side of \eqref{mixedfeb2122} by
$$\left[\sum_{n_{l+1},\ldots,n_{d},m}\left(\sum_{n_{1},\ldots,n_{l}}|\langle\langle g,\varphi_{n_{l+1},\ldots,n_{d},m}\rangle,\varphi_{n_{1},\ldots,n_{l},m}\rangle|^{2}\right)^{\frac{p_{0}}{2}}\right]^{\frac{1}{p_{0}}}\lesssim\left[\sum_{n_{l+1},\ldots,n_{d},m}\left\|\langle g,\varphi_{n_{l+1},\ldots,n_{d},m}\rangle\right\|_{2}^{p_{0}}\right]^{\frac{1}{p_{0}}}, $$
where $p_{0}=\frac{2(d-l+2)}{d-l}+\varepsilon$. This is how we will use Corollary \ref{cor1apr2922} in \eqref{bound2-3122020}.
}.
\end{proof}

We will use Corollary \ref{cor1apr2922} in Section \ref{klinear} to prove Theorem \ref{mainthmpaper} for $2\leq k\leq d+1$. It will not be needed when $k=d+1$.

\vspace{.5cm}

\item \underline{Extension conjecture for the parabola ($k=1$, $d=1$, $p=4$).} In Section \ref{dim1restriction} we prove the following: 

\begin{proposition}\label{prop2apr2922} For all $\varepsilon>0$,
\begin{equation}\label{restfeb2122}
\|E_{1}g\|_{4+\varepsilon}\lesssim_{\varepsilon}\|g\|_{4}.
\end{equation}
\end{proposition}

One can show by interpolation that Proposition \ref{prop2apr2922} implies Conjecture \ref{restriction} for $d=1$. We will use it in Section \ref{lineartheory} to settle the case $k=1$ of Theorem \ref{mainthmpaper}.

\vspace{.5cm}

\item \underline{Bilinear extension conjecture for the parabola ($k=2$, $d=1$).} In Section \ref{bilinearparabola} we show that the model $ME_{2,1}$ in Definition \ref{defklinearmodel} maps $L^{2}([0,1])\times L^{2}([4,5])$ to $L^{2}(\mathbb{R}^{2})$. 

\begin{proposition}\label{prop3apr2922} The following estimate holds:
\begin{equation}\label{estfeb2122}
\|ME_{2,1}(f,g)\|_{2}\lesssim \|f\|_{2}\cdot\|g\|_{2}.
\end{equation}
\end{proposition}

Transversality will be captured in Section \ref{klineartheory} through \eqref{estfeb2122}.
\vspace{.5cm}
\end{enumerate}

By combining scalar and mixed norm stopping times\footnote{This is not meant in a literal probabilistic sense; strictly speaking, the argument combines the level sets of various scalar and mixed norm quantities that appear naturally in our analysis.} performed simultaneously, we are able to put together the key estimates \eqref{mixedfeb2122}, \eqref{restfeb2122} and \eqref{estfeb2122}. In the $2\leq k\leq d+1$ case, the tensor structure is used in an implicit way to allow us to better relate these scalar and mixed norm stopping times.

\begin{remark}\label{rmk1-070123} The tensor structure $g=g_{1}\otimes\ldots\otimes g_{d}$ in the $k=1$ case allows us to write

\begin{equation}\label{tensorstructure}
\langle g,\varphi_{\overrightarrow{\textit{\textbf{n}}},m}\rangle=\prod_{j=1}^{d}\langle g_{j},\varphi_{n_{j},m}\rangle.
\end{equation}

We then obtain the following multilinear form by dualization:

\begin{equation}\label{linearmodel2}
\Lambda_{d}(g_{1},\ldots,g_{d},h)=\langle E_{d}(g),h\rangle=\sum_{\substack{\overrightarrow{\textit{\textbf{n}}}\in\mathbb{Z}^{d} \\ m\in\mathbb{Z}}}\prod_{j=1}^{d}\langle g_{j},\varphi_{n_{j},m}\rangle\cdot\langle h,\chi_{\overrightarrow{\textit{\textbf{n}}}}\otimes\chi_{m}\rangle,
\end{equation}

The goal in the $k=1$ case is to show that
$$|\Lambda_{d}(g_{1},\ldots,g_{d},h)|\lesssim \|h\|_{q}\cdot\prod_{j=1}^{d}\|g_{j}\|_{p_{j}}$$
for appropriate exponents $p_{j}$ and $q$. Interpolation theory shows that it suffices to obtain
\begin{equation}\label{restweaktypebdd}
|\Lambda_{d}(g_{1},\ldots,g_{d},h)|\lesssim_{\varepsilon} |F|^{\gamma_{d+1}}\cdot\prod_{j=1}^{d}|E_{j}|^{\gamma_{j}} 
\end{equation}
for all $\varepsilon>0$, $|g_{j}|\leq \chi_{E_{j}}$, $|h|\leq\chi_{F}$, \footnote{There is an overlap of classical notation here that we hope will not compromise the comprehension of the paper: we chose the typeface $E_{d}$ to represent the discrete model of the official extension operator $\mathcal{E}$. On the other hand, the classical theory of restricted weak-type multilinear interpolation usually labels the measurable sets involved in the problems by $E_{j}$ or $F_{j}$. The context will make it clear which object we are referring to.}$E_{j}\subset [0,1]$ and $F\subset\mathbb{R}^{3}$ measurable sets such that $\gamma_{j}$ $(1\leq j\leq d)$ and $\gamma_{d+1}$ are in a small neighborhood of $\frac{d}{2(d+1)}$ and $\frac{d+2}{2(d+1)}+\varepsilon$, respectively\footnote{Rigorously, this only verifies the case $k=1$ near the endpoint $(\frac{2(d+1)}{d},\frac{2(d+1)}{d})$, but this is known to imply the desired estimates in the full range. For details, see Theorem 19.8 of Mattila's book \cite{Mat1}.}. We refer the reader to Chapter 3 of \cite{Thiele1} for a detailed account of multilinear interpolation theory. To keep the notation simple, all restricted weak-type estimates we will prove in this paper will be for the centers of such neighborhoods. For example, we will show that

\begin{equation}\label{ineq1-070123}
|\Lambda_{d}(g_{1},\ldots,g_{d},h)|\lesssim_{\varepsilon} |F|^{\frac{d+2}{2(d+1)}+\varepsilon}\cdot\prod_{j=1}^{d}|E_{j}|^{\frac{d}{2(d+1)}},
\end{equation}
for all $\varepsilon>0$, but it will be clear from the arguments that as long as we give this $\varepsilon>0$ away, a slightly different choice of interpolation parameters yields \eqref{restweaktypebdd}. The restricted weak-type estimates that we will prove in the $2\leq k\leq d+1$ case will also be for the centers of the corresponding neighborhoods.

\end{remark}

\section{Proof of Proposition \ref{prop1apr2922} - Strichartz/Tomas-Stein for $E_{d}$ ($k=1$, $p=2$)}\label{TSreproven}

Our proof is inspired by the classical $TT^{\ast}$ argument. It is possible to prove the endpoint estimate directly for the model $E_{d}$ by repeating the steps of this argument (see for example Section 11.2.2 in \cite{MS}), but we chose the following approach because of its similarity with the one we will use to prove Theorem \ref{mainthmpaper}. By interpolation with the trivial bound for $q=\infty$, it is enough to prove the bound

\begin{equation*}
\|E_{d}g\|_{\frac{2(d+2)}{d}+\varepsilon}\lesssim_{\varepsilon}\|g\|_{2}
\end{equation*}
for all $\varepsilon>0$.

We start by dualizing $E_{d}$ to obtain a bilinear form $\Lambda_{d}$:

\begin{equation*}
\Lambda_{d}(g,h)=\langle E_{d}(g),h\rangle=\sum_{\substack{\overrightarrow{\textit{\textbf{n}}}\in\mathbb{Z}^{d} \\ m\in\mathbb{Z}}}\langle g,\varphi_{\overrightarrow{\textit{\textbf{n}}},m}\rangle\cdot\langle h,\chi_{\overrightarrow{\textit{\textbf{n}}}}\otimes\chi_{m}\rangle.
\end{equation*}

Let $E_{1}\subset\mathbb{R}^{d}$ and $E_{2}\subset\mathbb{R}^{d+1}$ be measurable sets of finite measure with $|g|\leq\chi_{E_{1}}$ and $|h|\leq\chi_{E_{2}}$. Split $\mathbb{Z}^{d+1}$ in two ways:

$$\mathbb{Z}^{d+1}=\bigcup_{l_{1}\in\mathbb{Z}}\mathbb{A}^{l_{1}},\textnormal{ where }(\overrightarrow{\textit{\textbf{n}}},m)\in\mathbb{A}^{l_{1}}\iff |\langle g,\varphi_{\overrightarrow{\textit{\textbf{n}}},m}\rangle|\approx 2^{-l_{1}}.$$
$$\mathbb{Z}^{d+1}=\bigcup_{l_{2}\in\mathbb{Z}}\mathbb{B}^{l_{2}},\textnormal{ where }(\overrightarrow{\textit{\textbf{n}}},m)\in\mathbb{B}^{l_{2}}\iff |\langle h,\chi_{\overrightarrow{\textit{\textbf{n}}}}\otimes\chi_{m}\rangle|\approx 2^{-l_{2}}.$$
Define $\mathbb{X}^{l_{1},l_{2}}:=\mathbb{A}^{l_{1}}\cap\mathbb{B}^{l_{2}}$ and observe that

$$
|\Lambda_{d}(g,h)|\lesssim\displaystyle\sum_{l_{1},l_{2}\in\mathbb{Z}}2^{-l_{1}}2^{-l_{2}}\#\mathbb{X}^{l_{1},l_{2}}.
$$
Notice that, for all $(\overrightarrow{\textit{\textbf{n}}},m)\in\mathbb{X}^{l_{1},l_{2}}$,

$$2^{-l_{1}}\lesssim\int_{\mathbb{R}^{d}}|g(x)||\varphi_{\overrightarrow{\textit{\textbf{n}}},m}(x)|\mathrm{d}x\leq \min{\{|E_{1}|,1\}},$$
$$2^{-l_{2}}\lesssim\int_{\mathbb{R}^{d}}|h(x)||\chi_{\overrightarrow{\textit{\textbf{n}}}}\otimes\chi_{m}(x)|\mathrm{d}x\leq \min{\{|E_{2}|,1\}}, $$
In particular, $l_{1},l_{2}\geq 0$ in the sum above. Now we bound $\#\mathbb{X}^{l_{1},l_{2}}$ in two different ways and interpolate between them:

\begin{enumerate}[(a)]
\item \underline{\textbf{$L^{1}$-type bound:}} Exploit $h$.

\begin{equation}\label{B1D}
\#\mathbb{X}^{l_{1},l_{2}}\leq \#\mathbb{B}^{l_{2}} \lesssim 2^{l_{2}}\sum_{(\overrightarrow{\textit{\textbf{n}}},m)\in\mathbb{B}^{l_{2}}}|\langle h,\chi_{\overrightarrow{\textit{\textbf{n}}}}\otimes\chi_{m}\rangle|\lesssim 2^{l_{2}}\sum_{(\overrightarrow{\textit{\textbf{n}}},m)\in\mathbb{Z}^{d+1}}\int_{Q_{\overrightarrow{\textit{\textbf{n}}},m}}|h|= 2^{l_{2}}\|h\|_{1}\leq 2^{l_{2}}|E_{2}|,
\end{equation}
where $Q_{\overrightarrow{\textit{\textbf{n}}},m}:= \Pi_{i=1}^{d}[n_{i},n_{i}+1]\times[m,m+1]$, $\overrightarrow{\textit{\textbf{n}}}=(n_{1},\ldots,n_{d})$.

\item \underline{\textbf{$L^{2}$-type bound:}} Exploit $g$.

\begin{equation}\label{L2typefinal}
\eqalign{
\#\mathbb{X}^{l_{1},l_{2}}&\displaystyle\lesssim 2^{2l_{1}}\sum_{(\overrightarrow{\textit{\textbf{n}}},m)\in\mathbb{X}^{l_{1},l_{2}}}|\langle g,\varphi_{\overrightarrow{\textit{\textbf{n}}},m}\rangle|^{2}\cr
&\displaystyle=2^{2l_{1}}\left|\left\langle\sum_{(\overrightarrow{\textit{\textbf{n}}},m)\in\mathbb{X}^{l_{1},l_{2}}}\langle g,\varphi_{\overrightarrow{\textit{\textbf{n}}},m}\rangle\varphi_{\overrightarrow{\textit{\textbf{n}}},m},g\right\rangle\right| \cr
&\displaystyle \leq 2^{2l_{1}}|E_{1}|^{\frac{1}{2}}\underbrace{\left\|\sum_{(\overrightarrow{\textit{\textbf{n}}},m)\in\mathbb{X}^{l_{1},l_{2}}}\langle g,\varphi_{\overrightarrow{\textit{\textbf{n}}},m}\rangle\varphi_{\overrightarrow{\textit{\textbf{n}}},m}\right\|_{2}}_{(\ast)}. \cr
}
\end{equation}

For each set $\mathbb{X}^{l_{1},l_{2}}$ define $\pi_{m}:=\{\overrightarrow{\textit{\textbf{n}}}\in\mathbb{Z}^{d};(\overrightarrow{\textit{\textbf{n}}},m)\in\mathbb{X}^{l_{1},l_{2}}\}$. Observe that:

\begin{equation*}
(\ast)^{2}=\sum_{\substack{m; \\ \pi_{m}\neq\emptyset}}
\sum_{\substack{\tilde{m}; \\ \pi_{\tilde{m}}\neq\emptyset}}\underbrace{\sum_{\overrightarrow{\textit{\textbf{n}}}\in\pi_{m}}\sum_{\overrightarrow{\textit{\textbf{k}}}\in\pi_{\tilde{m}}}\langle g,\varphi_{\overrightarrow{\textit{\textbf{n}}},m}\rangle\overline{\langle g,\varphi_{\overrightarrow{\textit{\textbf{k}}},\tilde{m}}\rangle}\langle \varphi_{\overrightarrow{\textit{\textbf{n}}},m},\varphi_{\overrightarrow{\textit{\textbf{k}}},\tilde{m}}\rangle}_{U\left((\langle g,\varphi_{\overrightarrow{\textit{\textbf{n}}},m}\rangle)_{\overrightarrow{\textit{\textbf{n}}}\in\pi_{m}},(\langle g,\varphi_{\overrightarrow{\textit{\textbf{k}}},\tilde{m}}\rangle)_{\overrightarrow{\textit{\textbf{k}}}\in\pi_{\tilde{m}}}\right)}
\end{equation*}

We will estimate $U$ in two ways. Let $a_{\overrightarrow{\textit{\textbf{n}}},m}:=\langle g,\varphi_{\overrightarrow{\textit{\textbf{n}}},m}\rangle$. First, by the triangle inequality and the stationary phase Theorem \ref{stat-22aug2021}:

\begin{equation*}
\eqalign{
\left|U\left((a_{\overrightarrow{\textit{\textbf{n}}},m})_{\overrightarrow{\textit{\textbf{n}}}\in\pi_{m}},(a_{\overrightarrow{\textit{\textbf{k}}},\tilde{m}})_{\overrightarrow{\textit{\textbf{k}}}\in\pi_{\tilde{m}}}\right)\right|&\leq\displaystyle\sum_{\overrightarrow{\textit{\textbf{n}}}\in\pi_{m}}\sum_{\overrightarrow{\textit{\textbf{k}}}\in\pi_{\tilde{m}}}|\langle g,\varphi_{\overrightarrow{\textit{\textbf{n}}},m}\rangle|\cdot|\langle g,\varphi_{\overrightarrow{\textit{\textbf{k}}},\tilde{m}}\rangle|\frac{1}{\langle m-\tilde{m}\rangle^{\frac{d}{2}}} \cr
&=\displaystyle\frac{\|\langle g,\varphi_{\cdot,m}\rangle\|_{\ell^{1}(\pi_{m})}\cdot\|\langle g,\varphi_{\cdot,\tilde{m}}\rangle\|_{\ell^{1}(\pi_{\tilde{m}})}}{\langle m-\tilde{m}\rangle^{\frac{d}{2}}}. \cr
}
\end{equation*}

Another possibility is:

\begin{equation*}
\eqalign{
&\displaystyle \left|U\left((a_{\overrightarrow{\textit{\textbf{n}}},m})_{\overrightarrow{\textit{\textbf{n}}}\in\pi_{m}},(a_{\overrightarrow{\textit{\textbf{k}}},\tilde{m}})_{\overrightarrow{\textit{\textbf{k}}}\in\pi_{\tilde{m}}}\right)\right| \cr
&\leq\displaystyle\left|\int_{\mathbb{R}^{d}}\left(\sum_{\overrightarrow{\textit{\textbf{n}}}\in\pi_{m}}\langle g,\varphi_{\overrightarrow{\textit{\textbf{n}}},m}\rangle e^{2\pi i\overrightarrow{\textit{\textbf{n}}}\cdot x}\right)\left(\sum_{\overrightarrow{\textit{\textbf{k}}}\in\pi_{\tilde{m}}}\langle g,\varphi_{\overrightarrow{\textit{\textbf{k}}},m}\rangle e^{2\pi i\overrightarrow{\textit{\textbf{k}}}\cdot x}\right)\varphi(x)\varphi(x)e^{2\pi i(m-\tilde{m})|x|^{2}}\mathrm{d}x\right| \cr
&\lesssim\displaystyle\|\langle g,\varphi_{\cdot,m}\rangle\|_{\ell^{2}(\pi_{m})}\cdot\|\langle g,\varphi_{\cdot,\tilde{m}}\rangle\|_{\ell^{2}(\pi_{\tilde{m}})} \cr
}
\end{equation*}
by Cauchy-Schwarz and orthogonality on the sets $\pi_{m}$ and $\pi_{\widetilde{m}}$ (recall that $m$ and $\widetilde{m}$ are fixed). Interpolating between these bounds for $1\leq p\leq 2$:

$$\left|U\left((a_{\overrightarrow{\textit{\textbf{n}}},m})_{\overrightarrow{\textit{\textbf{n}}}\in\pi_{m}},(a_{\overrightarrow{\textit{\textbf{k}}},\tilde{m}})_{\overrightarrow{\textit{\textbf{k}}}\in\pi_{\tilde{m}}}\right)\right|\lesssim\frac{\|\langle g,\varphi_{\cdot,m}\rangle\|_{\ell^{p}(\pi_{m})}\cdot\|\langle g,\varphi_{\cdot,\tilde{m}}\rangle\|_{\ell^{p}(\pi_{\tilde{m}})}}{\langle m-\tilde{m}\rangle^{\frac{d}{2}\left(\frac{1}{p}-\frac{1}{p^{\prime}}\right)}}. $$

Back to $(\ast)$:

\begin{equation*}
\eqalign{
(\ast)^{2}&\lesssim\displaystyle\sum_{\substack{m; \\ \pi_{m}\neq\emptyset}}
\sum_{\substack{\tilde{m}; \\ \pi_{\tilde{m}}\neq\emptyset}}\frac{\|\langle g,\varphi_{\cdot,m}\rangle\|_{\ell^{p}(\pi_{m})}\cdot\|\langle g,\varphi_{\cdot,\tilde{m}}\rangle\|_{\ell^{p}(\pi_{\tilde{m}})}}{\langle m-\tilde{m}\rangle^{\frac{d}{2}\left(\frac{1}{p}-\frac{1}{p^{\prime}}\right)}} \cr
&=\displaystyle\sum_{\substack{m; \\ \pi_{m}\neq\emptyset}}\|\langle g,\varphi_{\cdot,m}\rangle\|_{\ell^{p}(\pi_{m})}
\sum_{\substack{\tilde{m}; \\ \pi_{\tilde{m}}\neq\emptyset}}\frac{\|\langle g,\varphi_{\cdot,\tilde{m}}\rangle\|_{\ell^{p}(\pi_{\tilde{m}})}}{\langle m-\tilde{m}\rangle^{\frac{d}{2}\left(\frac{1}{p}-\frac{1}{p^{\prime}}\right)}} \cr
&\leq\displaystyle \left\|\|\langle g,\varphi_{\cdot,m}\rangle\|_{\ell^{p}(\pi_{m})}
\right\|_{\ell^{p}(\mathbb{Z})}\left\|\sum_{\substack{\tilde{m}; \\ \pi_{\tilde{m}}\neq\emptyset}}\frac{\|\langle g,\varphi_{\cdot,\tilde{m}}\rangle\|_{\ell^{p}(\pi_{\tilde{m}})}}{\langle m-\tilde{m}\rangle^{\frac{d}{2}\left(\frac{1}{p}-\frac{1}{p^{\prime}}\right)}}\right\|_{\ell^{p^{\prime}}(\mathbb{Z})} \cr
&\leq\displaystyle  \left\|\|\langle g,\varphi_{\cdot,m}\rangle\|_{\ell^{p}(\pi_{m})}
\right\|_{\ell^{p}(\mathbb{Z})}\cdot \left\|\|\langle g,\varphi_{\cdot,\tilde{m}}\rangle\|_{\ell^{p}(\pi_{\tilde{m}})}
\right\|_{\ell^{p}(\mathbb{Z})} \cr
&\displaystyle=\left\|\|\langle g,\varphi_{\cdot,m}\rangle\|_{\ell^{p}(\pi_{m})}
\right\|_{\ell^{p}(\mathbb{Z})}^{2}, \cr
}
\end{equation*}
as long as
$$\frac{1}{p}-\frac{1}{p^{\prime}}=1-\frac{d}{2}\left(\frac{1}{p}-\frac{1}{p^{\prime}}\right)\Leftrightarrow \frac{1}{p}-\frac{1}{p^{\prime}}=\frac{2}{d+2}\Leftrightarrow\frac{2}{p^{\prime}}=\frac{d}{d+2}\Leftrightarrow p^{\prime}=\frac{2d+4}{d},$$
by discrete fractional integration. Plugging this back in \eqref{L2typefinal}:

\begin{equation*}
\eqalign{
\#\mathbb{X}^{l_{1},l_{2}}&\lesssim\displaystyle 2^{2l_{1}}|E_{1}|^{\frac{1}{2}}\left\|\|\langle g,\varphi_{\cdot,m}\rangle\|_{\ell^{p}(\pi_{m})}
\right\|_{\ell^{p}(\mathbb{Z})} \cr
&\displaystyle= 2^{2l_{1}}|E_{1}|^{\frac{1}{2}}\left(\sum_{(\overrightarrow{\textit{\textbf{n}}},m)\in\mathbb{X}^{l_{1},l_{2}}}|\langle g,\varphi_{\overrightarrow{\textit{\textbf{n}}},m}\rangle|^{p}\right)^{\frac{1}{p}} \cr
&\lesssim\displaystyle 2^{2l_{1}}|E_{1}|^{\frac{1}{2}}(2^{-pl_{1}}\#\mathbb{X}^{l_{1},l_{2}})^{\frac{1}{p}},\cr
}
\end{equation*}
which implies

\begin{equation}\label{finalestX}
\#\mathbb{X}^{l_{1},l_{2}}\lesssim 2^{\left(2+\frac{4}{d}\right)l_{1}}|E_{1}|^{1+\frac{2}{d}}
\end{equation}
\end{enumerate}

Interpolating between \eqref{B1D} and \eqref{finalestX}:

\begin{equation}\label{finalestTS}
\eqalign{
|\Lambda_{d}(g,h)| &\lesssim \displaystyle\sum_{l_{1},l_{2}\geq 0}2^{-l_{1}}2^{-l_{2}}\left(2^{\left(2+\frac{4}{d}\right)l_{1}}|E_{1}|^{1+\frac{2}{d}}\right)^{\theta_{1}}(2^{l_{2}}|E_{2}|)^{\theta_{2}}\cr
&= \displaystyle\left(\sum_{l_{1}\geq 0}2^{-l_{1}\left(1-\left(2+\frac{4}{d}\right)\theta_{1}\right)}\right)\left(\sum_{l_{2}\geq 0}2^{-l_{2}\left(1-\theta_{2}\right)}\right)|E_{1}|^{\left(1+\frac{2}{d}\right)\theta_{1}}|E_{2}|^{\theta_{2}} \cr
&\lesssim 2^{-\tilde{l}_{1}\left(1-\left(2+\frac{4}{d}\right)\theta_{1}\right)}2^{-\tilde{l}_{2}\left(1-\theta_{2}\right)}|E_{1}|^{\left(1+\frac{2}{d}\right)\theta_{1}}|E_{2}|^{\theta_{2}} \cr
&\lesssim \min{\{|E_{1}|^{\left(1-\left(2+\frac{4}{d}\right)\theta_{1}\right)},1\}}\min{\{|E_{2}|^{1-\theta_{2}},1\}}|E_{1}|^{\left(1+\frac{2}{d}\right)\theta_{1}}|E_{2}|^{\theta_{2}}  \cr
&\lesssim |E_{1}|^{\alpha_{1}\left(1-\left(2+\frac{4}{d}\right)\theta_{1}\right)+\left(1+\frac{2}{d}\right)\theta_{1}}|E_{2}|^{\alpha_{2}(1-\theta_{2})+\theta_{2}},
}
\end{equation}

for all $0\leq\alpha_{1},\alpha_{2}\leq 1$,  $\theta_{1}+\theta_{2}=1$, with $0\leq \left(2+\frac{4}{d}\right)\theta_{1}<1$, $0\leq \theta_{2}<1$, where $\tilde{l}_{1}$ is the smallest possible value of $l_{1}$ for which $\mathbb{A}^{l_{1}}\neq\emptyset$ and $\tilde{l}_{2}$ is defined analogously. Picking $\alpha_{1}=\frac{1}{2}$, $\alpha_{2}=0$, $\theta_{1}=\frac{d}{2(d+2)}-\varepsilon$ and $\theta_{2}=\frac{d+4}{2(d+2)}+\varepsilon$ gives 

$$|\Lambda_{d}(g,h)|\lesssim_{\varepsilon}|E_{1}|^{\frac{1}{2}}\cdot |E_{2}|^{\frac{d+4}{2(d+2)}+\varepsilon}$$
for all $\varepsilon>0$, which proves the proposition by restricted weak-type interpolation.

\section{Proof of Proposition \ref{prop2apr2922} - Conjecture \ref{restriction} for $E_{1}$ ($k=1$, $d=1$, $p=4$)}\label{dim1restriction}

The following argument is inspired by Zygmund's original proof of this case. Define

$$\Phi_{n,m}(s,t):=|t-s|^{\frac{1}{2}}\varphi(s)\varphi(t)e^{2\pi i(s-t)n}e^{2\pi i(s^{2}-t^{2})m}.$$

\begin{claim}
$$\langle\Phi_{n,m},\Phi_{\tilde{n},\tilde{m}}\rangle=O_{N}\left(\frac{1}{|(n-\tilde{n})(m-\tilde{m})|^{N}}\right)$$
for any natural $N$ if $n\neq\tilde{n}$ and $m\neq\tilde{m}$.
\end{claim}
\begin{proof}
\begin{equation*}
\eqalign{
\langle\Phi_{n,m},\Phi_{\tilde{n},\tilde{m}}\rangle &= \displaystyle\iint_{[0,1]^{2}}|t-s||\varphi(s)|^{2}|\varphi(t)|^{2}e^{2\pi i(s-t)(n-\tilde{n})}e^{2\pi i(s^{2}-t^{2})(m-\tilde{m})}\mathrm{d}s\mathrm{d}t \cr
&=\displaystyle\iint_{R}\frac{|u|}{|u|}\psi(u,v)e^{2\pi iu(n-\tilde{n})}e^{2\pi iv(m-\tilde{m})}\mathrm{d}u\mathrm{d}v, \cr
}
\end{equation*}
where $R$ is the region that we obtain after making the change of variables $s-t=u$, $s^{2}-t^{2}=v$, and $\psi(u,v)=\varphi\otimes\varphi(\frac{v+u^{2}}{u},\frac{v-u^{2}}{u})$. The claim follows by the non-stationary phase Theorem \ref{nonstat-22aug2021}.
\end{proof}

We now prove the following:

\begin{lemma}\label{ao192021} For $G$ smooth supported on $[0,1]\times[0,1]$,
$$\left\|\sum_{n,m\in\mathbb{Z}}\left\langle G,\varphi_{n,m}\otimes\overline{\varphi_{n,m}}\right\rangle(\chi_{n}\otimes\chi_{m})\right\|_{2}\lesssim\left(\iint_{[0,1]^{2}}\frac{|G(s,t)|^{2}}{|s-t|}\mathrm{d}s\mathrm{d}t\right)^{\frac{1}{2}} $$
\end{lemma}

\begin{proof}
Define $\tilde{G}(s,t)=\frac{G(s,t)}{|s-t|^{\frac{1}{2}}}$ on $[0,1]^{2}\backslash\{(x,x);0\leq x\leq 1\}$. Observe that
\begin{equation*}
\eqalign{
\displaystyle\left\|\sum_{n,m\in\mathbb{Z}}\left\langle G,\varphi_{n,m}\otimes\overline{\varphi_{n,m}}\right\rangle(\chi_{n}\otimes\chi_{m})\right\|_{2}^{2} &= \displaystyle\sum_{n,m\in\mathbb{Z}}|\langle G,\varphi_{n,m}\otimes\overline{\varphi_{n,m}}\rangle|^{2} \cr
&=\displaystyle\sum_{n,m\in\mathbb{Z}}|\langle \tilde{G},\Phi_{n,m}\rangle|^{2} \cr
&\lesssim \|\tilde{G}\|_{2}^{2}, \cr
 }
\end{equation*}
by the almost orthogonality of the $\Phi_{n,m}$ proved in the previous claim.
\end{proof}

\begin{remark}\label{interpolationaug192021} By the triangle inequality,
$$\left\|\sum_{n,m\in\mathbb{Z}}\left\langle G,\varphi_{n,m}\otimes\overline{\varphi_{n,m}}\right\rangle(\chi_{n}\otimes\chi_{m})\right\|_{\infty}\lesssim\iint_{[0,1]^{2}}|G(s,t)|\mathrm{d}s\mathrm{d}t.$$
Hence by interpolation we obtain
\begin{equation}
\left\|\sum_{n,m\in\mathbb{Z}}\left\langle G,\varphi_{n,m}\otimes\overline{\varphi_{n,m}}\right\rangle(\chi_{n}\otimes\chi_{m})\right\|_{p}\lesssim\left(\iint_{[0,1]^{2}}\frac{|G(s,t)|^{p^{\prime}}}{|s-t|^{p^{\prime}-1}}\mathrm{d}s\mathrm{d}t\right)^{\frac{1}{p^{\prime}}}
\end{equation}
for $2\leq p\leq\infty$.
\end{remark}

Let $E\subset\mathbb{R}^{d}$ be a measurable set of finite measure with $|g|\leq\chi_{E}$. Using Remark \ref{interpolationaug192021} and Lemma \ref{ao192021} for $G=g\otimes\overline{g}$, we have
\begin{equation*}
\eqalign{
\displaystyle\left[\sum_{n,m\in\mathbb{Z}}|\left\langle g,\varphi_{n,m}\right\rangle|^{4+\varepsilon}\right]^{\frac{2}{4+\varepsilon}} &=\displaystyle\left[\int_{\mathbb{R}^{2}}\left(\sum_{n,m\in\mathbb{Z}}|\left\langle g,\varphi_{n,m}\right\rangle|^{4+\varepsilon}(\chi_{n}\otimes\chi_{m})\right)\right]^{\frac{2}{4+\varepsilon}} \cr
&\leq \displaystyle\left[\int_{\mathbb{R}^{2}}\left(\sum_{n,m\in\mathbb{Z}}|\left\langle g,\varphi_{n,m}\right\rangle|^{2}(\chi_{n}\otimes\chi_{m})\right)^{\frac{4+\varepsilon}{2}}\right]^{\frac{2}{4+\varepsilon}} \cr
&=\displaystyle\left\|\sum_{n,m\in\mathbb{Z}}|\left\langle g,\varphi_{n,m}\right\rangle|^{2}(\chi_{n}\otimes\chi_{m})\right\|_{2+\frac{\varepsilon}{2}} \cr
&\displaystyle\lesssim \left(\iint_{[0,1]^{2}}\frac{|g(s)|^{p^{\prime}}|g(t)|^{p^{\prime}}}{|s-t|^{p^{\prime}-1}}\mathrm{d}s\mathrm{d}t\right)^{\frac{1}{p^{\prime}}}, \cr
}
\end{equation*}
where $p^{\prime}=\frac{4+\varepsilon}{2+\varepsilon}$. To bound this last integral, we proceed as follows:

\begin{equation*}
\eqalign{
\displaystyle\int_{0}^{1}\int_{0}^{1}\frac{|\rho(s)|\cdot|\rho(t)|}{|s-t|^{\gamma}}\mathrm{d}s\mathrm{d}t &= \displaystyle\int_{0}^{1}|\rho(t)|\int_{0}^{1}\frac{|\rho(s)|}{|s-t|^{\gamma}}\mathrm{d}s\mathrm{d}t \cr
&=\displaystyle\int_{0}^{1}|\rho(t)|\cdot\left(|\rho|\ast\frac{1}{|s|^{\gamma}}\right)(t)\mathrm{d}t \cr
&=\left\| |\rho|\left(|\rho|\ast\frac{1}{|s|^{\gamma}}\right)\right\|_{L^{1}(\mathrm{d}t)} \cr
&\leq \|\rho\|_{L^{q}(\mathrm{d}t)}\left\||\rho|\ast\frac{1}{|s|^{\gamma}}\right\|_{L^{p^{\prime}}(\mathrm{d}t)} \cr
&\lesssim_{\varepsilon} \|\rho\|_{p}^{2}, \cr
}
\end{equation*}
if $\frac{1}{p^{\prime}}=\frac{1}{p}-(1-\gamma)$, by Theorem \ref{fracint}. In our case, $\rho=|g|^{p^{\prime}}$, $\gamma=p^{\prime}-1$ and $pp^{\prime}=\frac{(4+\varepsilon)^{2}}{2(2+\varepsilon)}>4$, then

\begin{equation*}
\eqalign{
\left(\displaystyle\int_{0}^{1}\int_{0}^{1}\frac{|g(s)|^{p^{\prime}}\cdot|g(t)|^{p^{\prime}}}{|s-t|^{p^{\prime}-1}}\mathrm{d}s\mathrm{d}t\right)^{\frac{1}{p^{\prime}}} &\displaystyle\lesssim\left(\int_{0}^{1}|g(t)|^{pp^{\prime}}\mathrm{d}t\right)^{\frac{2}{pp^{\prime}}} \cr
&=\displaystyle \left(\int_{0}^{1}|g(t)|^{4+\left(\frac{(4+\varepsilon)^{2}}{2(2+\varepsilon)}-4\right)}\mathrm{d}t\right)^{\frac{4(2+\varepsilon)}{(4+\varepsilon)^{2}}} \cr
&\lesssim\displaystyle \left(\int_{0}^{1}|g(t)|^{4}\mathrm{d}t\right)^{\frac{4(2+\varepsilon)}{(4+\varepsilon)^{2}}} \cr
&= |E|^{\frac{4(2+\varepsilon)}{(4+\varepsilon)^{2}}}.
}
\end{equation*}

Observed that in the second line of the chain of inequalities above we used the fact that $|g|\leq 1$. Finally,

$$\|E_{1}g\|_{4+\varepsilon}=\left[\sum_{n,m\in\mathbb{Z}}|\left\langle g,\varphi_{n,m}\right\rangle|^{4+\varepsilon}\right]^{\frac{1}{4+\varepsilon}}\lesssim |E|^{\frac{2(2+\varepsilon)}{(4+\varepsilon)^{2}}}\leq |E|^{\frac{1}{4}}. $$

This shows that $E_{1}$ maps $L^{4}([0,1])$ to $L^{q}(\mathbb{R}^{2})$ for any $q>4$ by restricted weak-type interpolation. 

\section{Proof of Proposition \ref{prop3apr2922} - Conjecture \ref{klinear} for $ME_{2,1}$ ($k=2$, $d=1$)}\label{bilinearparabola}

The model to be treated is

$$ME_{2,1}(f,g):=\sum_{(n,m)\in\mathbb{Z}^{2}}\langle f,\varphi^{1}_{n,m}\rangle\cdot \langle g,\varphi^{2}_{n,m}\rangle(\chi_{n}\otimes\chi_{m}).$$

Since $d=1$, we do not have to deal with the multivariable quantity 
$$\varphi^{j}_{\overrightarrow{\textit{\textbf{n}}},m}=\bigotimes_{l=1}^{d}\varphi^{l,j}_{n_{l},m} $$
from Definition \ref{defklinearmodel}, so we will simplify the notation by taking $\varphi_{n,m}^{1}:=\varphi^{1,1}_{n,m}$ and $\varphi_{n,m}^{2}:=\varphi^{1,2}_{n,m}$. We also replaced $(g_{1},g_{2})$ by $(f,g)$ here to reduce the number of indices carried through the section.

We provide a simple argument involving Bessel's inequality. After a change of variables to move the domain of $\varphi^{2}$ to be the same as the one of $\varphi^{1}$, we have:

\begin{equation*}
\eqalign{
|ME_{2,1}(f,g)|&\displaystyle\lesssim \sum_{(n,m)\in\mathbb{Z}^{2}}|\langle f,\varphi^{1}_{n,m}\rangle||\overline{\langle (g)_{-4},\varphi^{1}_{n+8m,m}\rangle}|(\chi_{n}\otimes\chi_{m}) \cr
&\displaystyle=\sum_{(n,m)\in\mathbb{Z}^{2}}|\langle f\otimes (g)_{-4},\varphi^{1}_{n,m}\otimes\overline{\varphi^{1}_{n+8m,m}}\rangle|(\chi_{n}\otimes\chi_{m}), \cr
}
\end{equation*}
where\footnote{This was done to bring the support of $\varphi^{2}_{n,m}$ to the one of $\varphi^{1}_{n+8m,m}$. The price to pay is the $+4m$ shift in the linear modulation index of the bump.} $(g)_{-4}(y)=g(y+4)$. Observe that

\begin{equation*}
\eqalign{
\langle f\otimes (g)_{-4},\varphi^{1}_{n,m}\otimes\overline{\varphi^{1}_{n+8m,m}}\rangle&\displaystyle=\iint f(x)g(y+4)\varphi^{1}(x)\varphi^{1}(y)e^{-2\pi inx}e^{-2\pi imx^{2}}e^{2\pi i(n+8m)y}e^{2\pi imy^{2}}\mathrm{d}x\mathrm{d}y \cr
&=\displaystyle\iint f(x)g(y+4)e^{2\pi in(y-x)}e^{2\pi im(y-x)(y+x)}e^{16\pi imy}\mathrm{d}x\mathrm{d}y \cr
&\lesssim\displaystyle\int\underbrace{\left[\int f\left(\frac{v-u}{2}\right) g\left(\frac{v+u}{2}+4\right)e^{2\pi imuv}e^{8\pi im(u+v)}\mathrm{d}v\right]}_{H_{m}(u)}e^{2\pi in u}\mathrm{d}u\cr
&=\displaystyle \widehat{H_{m}}(-n)
}
\end{equation*}

Hence

\begin{equation*}
\|ME_{2,1}(f,g)\|_{2}^{2}\lesssim\sum_{m\in\mathbb{Z}}\sum_{n\in\mathbb{Z}}|\widehat{H_{m}}(-n)|^{2}=\sum_{m\in\mathbb{Z}}\|H_{m}\|_{2}^{2},
\end{equation*}
by Bessel. On the other hand,

\begin{equation*}
\eqalign{
\|H_{m}\|_{2}^{2}&\displaystyle=\int\left|\int f\left(\frac{v-u}{2}\right) g\left(\frac{v+u}{2}+4\right)e^{2\pi imuv}e^{8\pi im(u+v)}\mathrm{d}v\right|^{2}\mathrm{d}u \cr
&\displaystyle=\int\left|\int \underbrace{f\left(\frac{v-u}{2}\right) g\left(\frac{v+u}{2}+4\right)}_{\tilde{H}_{u}(v)}e^{2\pi imv(u+4)}\mathrm{d}v\right|^{2}\mathrm{d}u \cr
&\displaystyle=\int |\widehat{\tilde{H}_{u}}(m(u+4))|^{2}du.
}
\end{equation*}

Transversality enters the picture here through the factor $(u+4)$ above: the $+4$ shift in $u$ comes from the fact that the supports of $\varphi^{1}$ and $\varphi^{2}$ are disjoint and far enough from each other, hence $u+4\geq c>0$. This way,

\begin{equation*}
\eqalign{
\|ME_{2,1}(f,g)\|_{2}^{2}&\displaystyle\lesssim\int\left(\sum_{m\in\mathbb{Z}}|\widehat{\tilde{H}_{u}}(m(u+4))|^{2}\right)\mathrm{d}u \cr
&\lesssim\displaystyle\int\int|\tilde{H}_{u}(v)|^{2}\mathrm{d}v\mathrm{d}u \cr
&\lesssim \|f\|_{2}^{2}\|g\|_{2}^{2},
}
\end{equation*}
by Bessel again.

\section{Case $k=1$ of Theorem \ref{mainthmpaper}}\label{lineartheory}

In this section we start the proof of Theorem \ref{mainthmpaper}. There are two main ingredients in the argument for the case $k=1$: Proposition \ref{prop2apr2922} and the fact that the wave packets 

\begin{equation*}
\varphi_{\overrightarrow{\textit{\textbf{n}}},m}(x):=\varphi(x_{1})\cdot\ldots\cdot\varphi(x_{d})e^{2\pi ix\cdot\overrightarrow{\textit{\textbf{n}}}}e^{2\pi i|x|^{2}m}.
\end{equation*}
are almost orthogonal for a fixed $m$ and $\overrightarrow{\textit{\textbf{n}}}$ varying in $\mathbb{Z}^{d}$. The latter fact will be exploited through Bessel's inequality whenever possible. Recall from Remark \ref{rmk1-070123} that, since $g=g_{1}\otimes\ldots\otimes g_{d}$, it suffices to study the multilinear form

\begin{equation*}
\Lambda_{d}(g_{1},\ldots,g_{d},h)=\sum_{\substack{\overrightarrow{\textit{\textbf{n}}}\in\mathbb{Z}^{d} \\ m\in\mathbb{Z}}}\prod_{j=1}^{d}\langle g_{j},\varphi_{n_{j},m}\rangle\cdot\langle h,\chi_{\overrightarrow{\textit{\textbf{n}}}}\otimes\chi_{m}\rangle,
\end{equation*}

Now we focus on obtaining \eqref{ineq1-070123}. Let $E_{j}\subset [0,1], 1\leq j\leq d$, and $F\subset\mathbb{R}^{d+1}$ be measurable sets for which $|g_{j}|\leq\chi_{E_{j}}$ and $|h|\leq\chi_{F}$. Define the sets

$$\mathbb{A}_{j}^{l_{j}}:=\{(n_{j},m)\in\mathbb{Z}^{2};\quad |\langle g_{j},\varphi_{n_{j},m}\rangle|\approx 2^{-l_{j}}\},\quad 1\leq j\leq d. $$
$$\mathbb{B}^{l_{d+1}}:=\{(\overrightarrow{\textit{\textbf{n}}},m)\in\mathbb{Z}^{d+1};\quad |\langle h,\chi_{\overrightarrow{\textit{\textbf{n}}}}\otimes\chi_{m}\rangle|\approx 2^{-l_{d+1}}\}. $$
$$\mathbb{X}^{l_{1},\ldots,l_{d+1}}:=\{(\overrightarrow{\textit{\textbf{n}}},m)\in\mathbb{Z}^{d+1};\quad (n_{j},m)\in\mathbb{A}^{l_{j}}_{j},1\leq j\leq d\}\cap\mathbb{B}^{l_{d+1}}. $$

Hence,

$$|\Lambda_{d}(g_{1},\ldots,g_{d},h)|\lesssim \sum_{l_{1},\ldots,l_{d+1}\in\mathbb{Z}}2^{-l_{1}}\cdot\ldots\cdot 2^{-l_{d+1}}\#\mathbb{X}^{l_{1},\ldots,l_{d+1}}. $$

As in Section \ref{TSreproven}, we know that $l_{1},\ldots,l_{d+1}\geq 0$. We can estimate $\#\mathbb{X}^{l_{1},\ldots,l_{d+1}}$ using the function $h$:

\begin{equation}\label{bound1}
\#\mathbb{X}^{l_{1},\ldots,l_{d+1}}\lesssim 2^{l_{d+1}}\sum_{(\overrightarrow{\textit{\textbf{n}}},m)\in\mathbb{Z}^{d+1}}|\langle h,\chi_{\overrightarrow{\textit{\textbf{n}}}}\otimes\chi_{m}\rangle|\lesssim 2^{l_{d+1}}|F|.
\end{equation}

Alternatively, many bounds for $\#\mathbb{X}^{l_{1},\ldots,l_{d+1}}$ can be obtained using the input functions $g_{1},\ldots,g_{d}$:

\begin{equation}
\eqalign{
\#\mathbb{X}^{l_{1},\ldots,l_{d+1}}&\displaystyle\lesssim\sum_{(\overrightarrow{\textit{\textbf{n}}},m)\in\mathbb{Z}^{d+1}}\mathbbm{1}_{\mathbb{A}^{l_{1}}_{1}}(n_{1},m)\cdot\ldots\cdot\mathbbm{1}_{\mathbb{A}^{l_{d}}_{d}}(n_{d},m) \cr
&\displaystyle=\sum_{m\in\mathbb{Z}}\sum_{n_{1}\in\mathbb{Z}}\cdots\sum_{n_{d-1}\in\mathbb{Z}}\mathbbm{1}_{\mathbb{A}^{l_{1}}_{1}}(n_{1},m)\cdot\ldots\cdot\mathbbm{1}_{\mathbb{A}^{l_{d-1}}_{d-1}}(n_{d-1},m)\underbrace{\sum_{n_{d}\in\mathbb{Z}}\mathbbm{1}_{\mathbb{A}^{l_{d}}_{d}}(n_{d},m)}_{\alpha_{d,m}} \cr
}
\end{equation}

Observe that $\alpha_{d,m}=\#\{n;(n,m)\in \mathbb{A}^{l_{d}}_{d}\}$ and $(n,m)\in\mathbb{A}^{l_{d}}_{d}\Rightarrow 1\lesssim 2^{2l_{d}}|\langle g_{d},\varphi_{n,m}\rangle|^{2}$. Adding up in $n$,

$$\alpha_{d,m}\lesssim 2^{2l_{d}}\sum_{n;(n,m)\in \mathbb{A}^{l_{d}}_{d}}|\langle g_{d},\varphi_{n,m}\rangle|^{2}\lesssim 2^{2l_{d}}|E_{d}| $$
by orthogonality. Notice that this quantity does not depend on $m$, therefore we can iterate this argument for $d-2$ of the remaining $d-1$ characteristic functions:

\begin{equation}
\eqalign{
\#\mathbb{X}^{l_{1},\ldots,l_{d+1}}
&\displaystyle\lesssim2^{2l_{d}}|E_{d}|\sum_{m\in\mathbb{Z}}\sum_{n_{1}\in\mathbb{Z}}\mathbbm{1}_{\mathbb{A}^{l_{1}}_{1}}(n_{1},m)\cdot\ldots\cdot\mathbbm{1}_{\mathbb{A}^{l_{d-2}}_{d-2}}(n_{d-2},m)\underbrace{\sum_{n_{d-1}\in\mathbb{Z}}\mathbbm{1}_{\mathbb{A}^{l_{d-1}}_{d-1}}(n_{d-1},m)}_{\alpha_{d-1,m}}\cr
&\displaystyle\lesssim 2^{2l_{d}}|E_{d}|2^{2l_{d-1}}|E_{d-1}|\sum_{m\in\mathbb{Z}}\sum_{n_{1}\in\mathbb{Z}}\mathbbm{1}_{\mathbb{A}^{l_{1}}_{1}}(n_{1},m)\cdot\ldots\cdot\mathbbm{1}_{\mathbb{A}^{l_{d-3}}_{d-3}}(n_{d-3},m)\sum_{n_{d-2}\in\mathbb{Z}}\mathbbm{1}_{\mathbb{A}^{l_{d-2}}_{d-2}}(n_{d-1},m)\cr
&\displaystyle\lesssim 2^{2l_{d}}2^{2l_{d-1}}\ldots2^{2l_{2}}|E_{d}|\ldots|E_{2}|\underbrace{\sum_{m\in\mathbb{Z}}\sum_{n_{1}\in\mathbb{Z}}\mathbbm{1}_{\mathbb{A}^{l_{1}}_{1}}(n_{1},m)}_{\#\mathbb{A}^{l_{1}}_{1}}. \cr
}
\end{equation}

To bound $\#\mathbb{A}^{l_{1}}_{1}$ we can use Proposition \ref{prop2apr2922}. For $\varepsilon>0$ we have:

$$(n,m)\in\mathbb{A}^{l_{1}}_{1}\Rightarrow 1\lesssim 2^{(4+\varepsilon)l_{1}}|\langle g_{1},\varphi_{n,m}\rangle|^{4+\varepsilon}\Rightarrow \#\mathbb{A}^{l_{1}}_{1}\lesssim 2^{(4+\varepsilon)l_{1}}\sum_{(n,m)\in\mathbb{A}^{l_{1}}_{1}}|\langle g_{1},\varphi_{n,m}\rangle|^{4+\varepsilon}\lesssim_{\varepsilon}2^{(4+\varepsilon)l_{1}}|E_{1}|.$$

Using this above,

\begin{equation}
\#\mathbb{X}^{l_{1},\ldots,l_{d+1}}\lesssim_{\varepsilon}2^{2l_{d}}2^{2l_{d-1}}\ldots2^{2l_{2}}2^{(4+\varepsilon)l_{1}}|E_{d}|\ldots|E_{2}||E_{1}|
\end{equation}

We could have used the $L^{4}-L^{4+\varepsilon}$ bound for any $g_{j}$ and a Bessel bound for the remaining ones. More precisely, if $\sigma\in S_{d}$ is a permutation, we have

\begin{equation}
\#\mathbb{X}^{l_{1},\ldots,l_{d+1}}\lesssim_{\varepsilon}2^{2l_{\sigma(d)}}2^{2l_{\sigma(d-1)}}\ldots2^{2l_{\sigma(2)}}2^{(4+\varepsilon)l_{\sigma(1)}}|E_{\sigma(d)}|\ldots|E_{\sigma(2)}||E_{\sigma(1)}|
\end{equation}

This amounts to exactly $d$ different estimates. Interpolating between all of them with equal weight $\frac{1}{d}$, we obtain:

\begin{equation}\label{bound2}
\eqalign{
\#\mathbb{X}^{l_{1},\ldots,l_{d+1}}&\displaystyle\lesssim_{\varepsilon}2^{\frac{2(d-1)+4+\varepsilon}{d}l_{1}}\ldots2^{\frac{2(d-1)+4+\varepsilon}{d}l_{d}}|E_{1}|\ldots|E_{d}| \cr
&\displaystyle= 2^{\left(2+\frac{2}{d}+\frac{\varepsilon}{d}\right)l_{1}}\ldots2^{\left(2+\frac{2}{d}+\frac{\varepsilon}{d}\right)l_{d}}|E_{1}|\ldots|E_{d}|.\cr
}
\end{equation}

Finally, we interpolating between bounds \eqref{bound1} and \eqref{bound2}:

\begin{equation*}
\eqalign{
\displaystyle|\Lambda_{d}(g_{1},\ldots & , g_{d},h)| \cr
&\lesssim\displaystyle \sum_{l_{1},\ldots,l_{d+1}\in\mathbb{Z}_{+}}2^{-l_{1}}\cdot\ldots\cdot 2^{-l_{d+1}}\#\mathbb{X}^{l_{1},\ldots,l_{d+1}} \cr
&\lesssim\displaystyle\sum_{l_{1},\ldots,l_{d+1}\in\mathbb{Z}_{+}}2^{-l_{1}}\cdot\ldots\cdot 2^{-l_{d+1}}\left(2^{\left(2+\frac{2}{d}+\frac{\varepsilon}{d}\right)l_{1}}\ldots2^{\left(2+\frac{2}{d}+\frac{\varepsilon}{d}\right)l_{d}}|E_{1}|\ldots|E_{d}|\right)^{\theta_{1}}\left(2^{l_{d+1}}|F|\right)^{\theta_{2}} \cr
&\lesssim\displaystyle\left(\sum_{l_{d+1}\geq 0}2^{-(1-\theta_{2})l_{d+1}}|F|^{\theta_{2}}\right)\prod_{j=1}^{d}\sum_{l_{j}\geq 0}2^{-\left(1-\left(2+\frac{2}{d}+\frac{\varepsilon}{d}\right)\theta_{1}\right)l_{j}}|E_{j}|^{\theta_{1}} \cr
&\lesssim\displaystyle |E_{1}|^{\alpha\left(1-\left(2+\frac{2}{d}+\frac{\varepsilon}{d}\right)\theta_{1}\right)+\theta_{1}}\cdots|E_{d}|^{\alpha\left(1-\left(2+\frac{2}{d}+\frac{\varepsilon}{d}\right)\theta_{1}\right)+\theta_{1}}|F|^{\theta_{2}},
}
\end{equation*}

for any $0\leq\alpha\leq 1$. On the other hand, for several of the series above to converge we need $\left(2+\frac{2}{d}+\frac{\varepsilon}{d}\right)\theta_{1}>1$. By choosing the appropriate $\alpha$ and $\theta_{1}$ close to $(2+\frac{2}{d})^{-1}$, one concludes this case.

\section{Case $2\leq k\leq d+1$ of Theorem \ref{mainthmpaper}}\label{klineartheory}

Recall that we fixed a set of weakly transversal cubes $\mathcal{Q}=\{Q_{1},\ldots,Q_{k}\}$ in Section \ref{twt} and let $g_{j}$ be supported on $Q_{j}$. The averaged $k$-linear extension operator\footnote{We consider this averaged version of $ME_{k,d}$ for technical reasons. The conjectured bounds for it have a Banach space as target, as opposed to the quasi-Banach space (for most $k$ and $d$) $L^{\frac{2(d+k+1)}{k(d+k-1)}}$ that is the target of Conjecture \ref{klinear}. The fact that $L^{p}$ for $p\geq \frac{2(d+k+1)}{(d+k-1)}$ is Banach lets us use \eqref{bound1-dec12020} effectively in the interpolation argument, since it forces the final power $\gamma$ on $|F|^{\gamma}$ to be positive.

When $k=d=2$, Conjecture \ref{klinear} has $L^{\frac{5}{3}}$ as target. We will discuss this case first to help digest the main ideas of the general argument, and since this space is Banach, we can work directly with $ME_{2,2}$ instead of considering the averaged operator $ME_{2,2}^{\frac{1}{2}}$.
} in $\mathbb{R}^{d}$ is given by 
$$ME_{k,d}^{\frac{1}{k}}(g_{1},\ldots,g_{k})=\sum_{(\overrightarrow{\textit{\textbf{n}}},m)\in\mathbb{Z}^{d+1}}\left(\prod_{i=1}^{k}|\langle g_{j},\varphi^{j}_{\overrightarrow{\textit{\textbf{n}}},m}\rangle|\right)^{\frac{1}{k}}(\chi_{\overrightarrow{\textit{\textbf{n}}}}\otimes\chi_{m}).$$
The conjectured bounds for it are 
\begin{equation}\label{klinearextconj}
\|ME_{k,d}^{\frac{1}{k}}(g_{1},\ldots,g_{k})\|_{L^{p}(\mathbb{R}^{d+1})}\lesssim\prod_{j=1}^{k}\|g_{j}\|_{L^{2}(Q_{j})}^{\frac{1}{k}}, 
\end{equation}
for all $p\geq \frac{2(d+k+1)}{(d+k-1)}$.

As done in the case $k=1$, it's enough to prove certain restricted weak-type bounds for its associated form
\begin{equation}\label{form1-jan1223}
\tilde{\tilde{\Lambda}}_{k,d}(g,h):=\sum_{(\overrightarrow{\textit{\textbf{n}}},m)\in\mathbb{Z}^{d+1}}\left(\prod_{i=1}^{k}|\langle g_{j},\varphi^{j}_{\overrightarrow{\textit{\textbf{n}}},m}\rangle|\right)^{\frac{1}{k}}\langle h,\chi_{\overrightarrow{\textit{\textbf{n}}}}\otimes\chi_{m}\rangle,
\end{equation}
where $g:=(g_{1},\ldots,g_{k})$ by a slight abuse of notation. 

\begin{remark} We will prove \eqref{klinearextconj} up to the endpoint assuming that $g_{1}$ is a full tensor, but the argument can be repeated if any other $g_{j}$ is assumed to be of this type. As the reader will notice, the proof depends on the fact that we can find $k-1$ canonical directions associated to $Q_{j}$, which is the defining property of a weakly transversal collection of cubes with pivot $Q_{j}$. In what follows, we are taking $\{e_{1},\ldots,e_{k-1}\}$ to be the set of directions associated to $Q_{1}$.

\end{remark}

\begin{remark} As we mentioned in Remark \ref{rem1may522}, under weak transversality alone we do not need $g_{1}$ to be a full tensor to prove the case $2\leq k\leq d$ of Theorem \ref{mainthmpaper}. In fact, the following structure is enough in this section:

$$g_{1}(x_{1},\ldots,x_{d})=g_{1,1}(x_{1})\cdot g_{1,2}(x_{2})\cdot\ldots\cdot g_{1,k-1}(x_{k-1})\cdot g_{1,k}(x_{k},\ldots,x_{d}).$$
Notice that we have $k-1$ single variable functions and one function in $d-k+1$ variables. The single variable ones are defined along $k-1$ canonical directions $\{e_{1},\ldots,e_{k-1}\}$ associated to $Q_{1}$, and $g_{1,k}$ is a function in the remaining variables.

In general, if we are given a weakly transversal collection $\widetilde{\mathcal{Q}}$, for a fixed $1\leq j\leq k-1$ we have a set of associated directions $\mathscr{E}_{j}=\{e_{i_{1}},\ldots,e_{i_{k-1}}\}$ (see Definition \ref{weaktransv}). Denote by $x_{\mathscr{E}_{j}^{c}}$ the vector of $d-k+1$ entries obtained after removing $x_{i_{1}},\ldots,x_{i_{k-1}}$ from $(x_{1},\ldots,x_{d})$. Assuming that the functions $g_{l}$ for $l\neq j$ are generic and that $g_{j}$ has the weaker tensor structure
\begin{equation}\label{structure1may622}
g_{j}(x_{1},\ldots,x_{d})=g_{j,1}(x_{i_{1}})\cdot\ldots\cdot g_{j,k-1}(x_{i_{k-1}})\cdot g_{\mathscr{E}_{j}^{c}}(x_{\mathscr{E}_{j}^{c}})
\end{equation}
will suffice to conclude Theorem \ref{mainthmpaper} for $\widetilde{\mathcal{Q}}$ through the argument that we will present in this section. 
\end{remark}

\begin{remark} As a consequence of Claim \ref{claim2-081221}, a collection $\mathcal{Q}=\{Q_{1},\ldots,Q_{k}\}$ of transversal cubes generates finitely many sub-collections $\widetilde{\mathcal{Q}}$ of weakly transversal ones (after partitioning each $Q_{l}$ into small enough cubes and defining new collections with them). However, for a fixed $1\leq j\leq k$, the associated $k-1$ directions in $\mathscr{E}_{j}$ can potentially change from one such weakly transversal sub-collection to another, and this is why we assume $g_{j}$ to be a full tensor under the transversality assumption.
\end{remark}

In this section we will use the following conventions:

\begin{itemize}
\item The variables of $g_{j}$ are $x_{1},x_{2},\ldots,x_{d}$, but we will split them in two groups: $k-1$ blocks of one variable represented by $x_{i}$, $1\leq i\leq k-1$, and one block of $d-k+1$ variables $\overrightarrow{x_{k}}=(x_{k},x_{k+1},\ldots,x_{d-1},x_{d})$.
\item The index $x_{i}$ in $\langle \cdot,\cdot\rangle_{x_{i}}$ indicates that the inner product is an integral in the variable $x_{i}$ only. For instance,
\begin{equation}\label{ex1-211221}
\langle g_{j},\varphi\rangle_{x_{1}} := \int_{\mathbb{R}}g_{j}(x_{1},\ldots,x_{d})\cdot\overline{\varphi}(x_{1},\ldots,x_{d})\mathrm{d}x_{1}
\end{equation}
is now a function of the variables $x_{2},\ldots, x_{d}$. The vector index $\overrightarrow{x_{k}}$ in $\langle \cdot,\cdot\rangle_{\overrightarrow{x_{k}}}$ is understood analogously:
\begin{equation}
\langle g_{j},\varphi\rangle_{\overrightarrow{x_{k}}} := \int_{\mathbb{R}^{d-k+1}}g_{j}(x_{1},\ldots,x_{d})\cdot\overline{\varphi}(x_{1},\ldots,x_{d})\mathrm{d}\overrightarrow{x_{k}}
\end{equation}

\item The expression $\left\|\langle g_{j},\cdot\rangle_{x_{i}}\right\|_{2} $ is the $L^{2}$ norm of a function in the variables $x_{l}$, $1\leq l\leq k-1$, $l\neq i$. To illustrate using \eqref{ex1-211221},
$$\|\langle g_{j},\varphi\rangle_{x_{1}}\|_{2} = \left[\int_{\mathbb{R}^{d-1}}\left|\int_{\mathbb{R}}g_{j}(x_{1},\ldots,x_{d})\cdot\overline{\varphi}(x_{1},\ldots,x_{d})\mathrm{d}x_{1}\right|^{2}\mathrm{d}x_{2}\ldots\mathrm{d}x_{d}\right]^{\frac{1}{2}}.
 $$
 The quantity $\left\|\langle g_{j},\cdot\rangle_{\overrightarrow{x_{k}}}\right\|_{2} $ is defined analogously as
 $$\|\langle g_{j},\varphi\rangle_{\overrightarrow{x_{k}}}\|_{2} = \left[\int_{\mathbb{R}^{k-1}}\left|\int_{\mathbb{R}^{d-k+1}}g_{j}(x_{1},\ldots,x_{d})\cdot\overline{\varphi}(x_{1},\ldots,x_{d})\mathrm{d}\overrightarrow{x_{k}}\right|^{2}\mathrm{d}x_{1}\ldots\mathrm{d}x_{k-1}\right]^{\frac{1}{2}}.
 $$

\item For $\overrightarrow{\textit{\textbf{n}}}=(n_{1},\ldots,n_{d})$, define the vector

$$\widehat{n_{i}}:=(n_{1},\ldots,n_{i-1},n_{i+1},\ldots,n_{d}). $$
In other words, the hat on $\widehat{n_{i}}$ indicates that $n_{i}$ was removed from the vector $\overrightarrow{\textit{\textbf{n}}}$. For $f:\mathbb{Z}^{d}\rightarrow\mathbb{C}$, define
$$\left\|f(\overrightarrow{\textit{\textbf{n}}})\right\|_{\ell^{1}_{\widehat{n_{i}}}}:=\sum_{\widehat{n_{i}}\in\mathbb{Z}^{d-1}}|f(\overrightarrow{\textit{\textbf{n}}})|. $$
That is, $\left\|f(\overrightarrow{\textit{\textbf{n}}})\right\|_{\ell^{1}_{\widehat{n_{i}}}}$ is the $\ell^{1}$ norm of $f$ over all $n_{1},\ldots,n_{d}$, except for $n_{i}$. Hence $\left\|f(\overrightarrow{\textit{\textbf{n}}})\right\|_{\ell^{1}_{\widehat{n_{i}}}}$ is a function of the remaining variable $n_{i}$. The quantity $\left\|f(\overrightarrow{\textit{\textbf{n}}})\right\|_{\ell^{1}_{\widehat{\overrightarrow{n_{k}}}}}$ is defined analogously as

$$\left\|f(\overrightarrow{\textit{\textbf{n}}})\right\|_{\ell^{1}_{\widehat{\overrightarrow{n_{k}}}}}:=\sum_{(n_{1},\ldots,n_{k-1})\in\mathbb{Z}^{k-1}}|f(\overrightarrow{\textit{\textbf{n}}})|. $$
Finally, the integral $\int g\mathrm{d}\widehat{x_{i}}$ means the following:

$$\int g(x_{1},\ldots,x_{d})\mathrm{d}\widehat{x_{i}}:= \int g(x_{1},\ldots,x_{d})\mathrm{d}x_{1}\ldots\mathrm{d}x_{i-1}\mathrm{d}x_{i+1}\ldots\mathrm{d}x_{d}.$$

\end{itemize}
In what follows, let $E_{1,1},\ldots,E_{1,k-1}\subset [0,1]$, $E_{1,k}\subset [0,1]^{d-k+1}$, $E_{j}\subset Q_{j}$ ($2\leq j\leq k$) and $F\subset\mathbb{R}^{d+1}$ be measurable sets such that $|g_{1,l}|\leq\chi_{E_{1,l}}$ for $1\leq l\leq k-1$, $|g_{1,k}|\leq\chi_{E_{1,k}}$, $|g_{j}|\leq\chi_{E_{j}}$ for $2\leq j\leq k$ and $|h|\leq\chi_{F}$. Furthermore, $E_{1}:= E_{1,1}\times\ldots\times E_{1,k-1}\times E_{1,k}$.

A rough description of the argument in one sentence is: \textit{the proof is a combination of Strichartz in some variables and bilinear extension in many pairs of the other variables}. In order to illustrate that, we will first present the simplest case in an informal way, which means that we will avoid the purely technical aspects in this preliminary part. Once this is understood, it will be clear how to rigorously extend the argument in general.

\subsection{Understanding the core ideas in the $k=d=2$ case}

Consider the model
$$ME_{2,2}(g_{1},g_{2})=\sum_{(\overrightarrow{\textit{\textbf{n}}},m)\in\mathbb{Z}^{3}}\langle g_{1},\varphi^{1}_{\overrightarrow{\textit{\textbf{n}}},m}\rangle\langle g_{2},\varphi^{2}_{\overrightarrow{\textit{\textbf{n}}},m}\rangle(\chi_{\overrightarrow{\textit{\textbf{n}}}}\otimes\chi_{m})$$
and its associated trilinear form\footnote{There is a slight abuse of notation here: we are using $\tilde{\tilde{\Lambda}}_{2,2}$ for the form associated to $ME_{2,2}$ and not for its averaged version $ME_{2,2}$, as established in the beginning of this section.}
$$\tilde{\tilde{\Lambda}}_{2,2}(g_{1},g_{2},h)=\sum_{(\overrightarrow{\textit{\textbf{n}}},m)\in\mathbb{Z}^{3}}\langle g_{1},\varphi^{1}_{\overrightarrow{\textit{\textbf{n}}},m}\rangle\langle g_{2},\varphi^{2}_{\overrightarrow{\textit{\textbf{n}}},m}\rangle(\chi_{\overrightarrow{\textit{\textbf{n}}}}\otimes\chi_{m}).$$
Assuming that $g_{1}=g_{1,1}\otimes g_{1,2}$, we want to prove that
\begin{equation*}
|\tilde{\tilde{\Lambda}}_{2,2}(g_{1},g_{2})|\lesssim_{\varepsilon} |E_{1}|^{\frac{1}{2}}\cdot |E_{2}|^{\frac{1}{2}}\cdot |F|^{\frac{2}{5}+\varepsilon},
\end{equation*}
for all $\varepsilon>0$. The $L^{2}\times L^{2}\mapsto L^{\frac{5}{3}+\varepsilon}$ bound will then follow by multilinear interpolation and Remark \ref{rmk1-070123}. Given the expository character of this subsection, we adopt the informal convention

$$
\begin{cases}
x^{+} \textnormal{ means $x+\delta$, where $\delta>0$ is arbitrarily small,}\\
x^{-} \textnormal{ means $x-\delta$, where $\delta>0$ is arbitrarily small.}\\
\end{cases}
$$
We will always be able to control how small the $\delta$ above is, so we do not worry about making it precise for now.

The first step is to define the level sets of the scalar products appearing in $ME_{2,2}$:

$$\mathbb{A}^{l_{1}}_{1}=\left\{(\overrightarrow{\textit{\textbf{n}}},m): |\langle g_{1},\varphi^{1}_{\overrightarrow{\textit{\textbf{n}}},m}\rangle| \approx 2^{-l_{1}}\right\}, $$
$$\mathbb{A}^{l_{2}}_{2}=\left\{(\overrightarrow{\textit{\textbf{n}}},m): |\langle g_{2},\varphi^{2}_{\overrightarrow{\textit{\textbf{n}}},m}\rangle| \approx 2^{-l_{2}}\right\}. $$

Transversality will be captured by exploiting the sizes of ``lower-dimensional" information: in fact, we want to make the operator $ME_{2,1}$ appear, and this will be possible thanks to the interaction between the quantities associated to the following level sets:

$$\mathbb{B}^{r_{1}}_{1}=\left\{(n_{1},m):\left\|\left\langle g_{1},\varphi_{n_{1},m}^{1,1}\right\rangle_{x_{1}}\right\|_{2}\approx 2^{-r_{1}}\right\}, $$
$$\mathbb{C}^{s_{1}}_{1}=\left\{(n_{1},m):\left\|\left\langle g_{2},\varphi_{n_{1},m}^{1,2}\right\rangle_{x_{1}}\right\|_{2}\approx 2^{-s_{1}}\right\}. $$

Since there is only one direction along which one can exploit transversality, we will use the $L^{2}$ theory for $E_{1}$ (i.e. Strichartz) along the remaining one. In order to do that, the following level sets will be used:

$$\mathbb{B}^{r_{2}}_{2}=\left\{(n_{2},m):\left\|\left\langle g_{1},\varphi_{n_{2},m}^{2,1}\right\rangle_{x_{2}}\right\|_{2}\approx 2^{-r_{2}}\right\}, $$
$$\mathbb{C}^{s_{2}}_{2}=\left\{(n_{2},m):\left\|\left\langle g_{2},\varphi_{n_{2},m}^{2,2}\right\rangle_{x_{2}}\right\|_{2}\approx 2^{-s_{2}}\right\}. $$

The size of the scalar product involving $h$ will be captured by the following set:

$$\mathbb{D}^{k}=\{(\overrightarrow{\textit{\textbf{n}}},m): |\langle H,\chi_{\overrightarrow{\textit{\textbf{n}}}}\otimes\chi_{m}\rangle| \approx 2^{-k}\}. $$

We will also need to organize all the information above in appropriate ``slices" and in a major set that takes everything into account. The sets that do that are
$$\mathbb{X}^{l_{2},s_{1}}:=\mathbb{A}^{l_{2}}_{2}\cap\{(\overrightarrow{\textit{\textbf{n}}},m); (n_{1},m)\in\mathbb{C}^{s_{1}}_{1}\}, $$
$$\mathbb{X}^{l_{2},s_{2}}:=\mathbb{A}^{l_{2}}_{2}\cap\{(\overrightarrow{\textit{\textbf{n}}},m); (n_{2},m)\in\mathbb{C}^{s_{2}}_{2}\}, $$
$$\mathbb{X}^{\overrightarrow{l},\overrightarrow{r},\overrightarrow{s},k}=\mathbb{A}^{l_{1}}_{1}\cap\mathbb{A}^{l_{2}}_{2}\cap\{(\overrightarrow{\textit{\textbf{n}}},m): (n_{1},m)\in \mathbb{B}^{r_{1}}_{1}\cap\mathbb{C}^{s_{1}}_{1}, (n_{2},m)\in\mathbb{B}^{r_{2}}_{2}\cap\mathbb{C}^{s_{2}}_{2}\}\cap \mathbb{D}^{k}, $$
where we are using the abbreviations $\overrightarrow{l}=(l_{1},l_{2})$, $\overrightarrow{r}=(r_{1},r_{2})$ and $\overrightarrow{s}=(s_{1},s_{2})$. This gives us
\begin{equation*}
\eqalign{
|\tilde{\tilde{\Lambda}}_{2,2}(g_{1},g_{2},h)| &\lesssim\displaystyle\sum_{\overrightarrow{l},\overrightarrow{r},\overrightarrow{s},k}2^{-l_{1}}2^{-l_{2}}2^{-k}\#\mathbb{X}^{\overrightarrow{l},\overrightarrow{r},\overrightarrow{s},k}. \cr
}
\end{equation*}
For the sake of simplicity, let us assume that $g_{1}=\mathbbm{1}_{E_{1,1}}\otimes\mathbbm{1}_{E_{1,2}}$, $g_{2}=\mathbbm{1}_{E_{2}}$ and $h=\mathbbm{1}_{F}$\footnote{These indicator functions actually \textit{bound} $g_{1}$ and $g_{2}$, but this does not affect the core of the argument.}. We will need efficient ways of relating the scalar and mixed-norm quantities above. A direct computation (using the definition of $\mathbb{X}^{\overrightarrow{l},\overrightarrow{r},\overrightarrow{s},k}$) shows that

\begin{equation}\label{mot3-jan1223}
2^{-l_{1}}= \frac{2^{-r_{1}}\cdot 2^{-r_{2}}}{|E_{1}|^{\frac{1}{2}}}.
\end{equation}
Using Bessel along a direction, for $(n_{1},n_{2},m)\in \mathbb{X}^{l_{2},s_{1}}$ we have:
\begin{equation}\label{mot1-jan1223}
\eqalign{
\displaystyle1\approx 2^{2l_{2}}|\langle g_{2},\varphi^{2}_{\overrightarrow{\textit{\textbf{n}}},m}\rangle|^{2}&\displaystyle\Longrightarrow \#\mathbb{X}^{l_{2},s_{1}}_{(n_{1},m)} \approx 2^{2l_{2}}\sum_{n_{2}\in\mathbb{X}^{l_{2},s_{1}}_{(n_{1},m)}}|\langle g_{2},\varphi^{2}_{\overrightarrow{\textit{\textbf{n}}},m}\rangle|^{2} \cr
&\displaystyle\Longrightarrow  \#\mathbb{X}^{l_{2},s_{1}}_{(n_{1},m)} \lesssim 2^{2l_{2}}\left\|\langle g_{2},\varphi^{1,2}_{n_{1},m}\rangle\right\|_{2}^{2}  \cr
&\displaystyle\Longrightarrow 2^{-l_{2}}\lesssim \frac{2^{-s_{1}}}{\left(\#\mathbb{X}^{l_{2},s_{1}}_{(n_{1},m)}\right)^{\frac{1}{2}}} \cr
&\displaystyle\Longrightarrow 2^{-l_{2}}\lesssim \frac{2^{-s_{1}}}{\left\|\mathbbm{1}_{\mathbb{X}^{l_{2},s_{1}}}\right\|_{\ell^{\infty}_{n_{1},m}\ell^{1}_{n_{2}}}^{\frac{1}{2}}},
}
\end{equation}
by taking the supremum in $(n_{1},m)$. Analogously,
\begin{equation}\label{mot2-jan1223}
2^{-l_{2}}\lesssim  \frac{2^{-s_{2}}}{\left\|\mathbbm{1}_{\mathbb{X}^{l_{2},s_{2}}}\right\|_{\ell^{\infty}_{n_{2},m}\ell^{1}_{n_{1}}}^{\frac{1}{2}}}.
\end{equation}
Relations \eqref{mot3-jan1223}, \eqref{mot1-jan1223} and \eqref{mot2-jan1223} play a major role in the proof. The last major piece is a way of bounding $\#\mathbb{X}^{\overrightarrow{l},\overrightarrow{r},\overrightarrow{s},k}$ that allows us to exploit transversality and Strichartz along the right directions, as well as the dual function $h$. We start with the simplest one of them:

\begin{equation}\label{boundX3-jan1223}
\#\mathbb{X}^{\overrightarrow{l},\overrightarrow{r},\overrightarrow{s},k}\lesssim 2^{k}\sum_{(\overrightarrow{\textit{\textbf{n}}},m)\in\mathbb{Z}^{3}}|\langle h,\chi_{\overrightarrow{\textit{\textbf{n}}}}\otimes\chi_{m}\rangle|= 2^{k}|F|.
\end{equation}
By dropping most of the indicator functions in the definition of $\mathbb{X}^{\overrightarrow{l},\overrightarrow{r},\overrightarrow{s},k}$ and using H\"older, we obtain

$$\#\mathbb{X}^{\overrightarrow{l},\overrightarrow{r},\overrightarrow{s},k} \leq \sum_{(\overrightarrow{\textit{\textbf{n}}},m)\in\mathbb{Z}^{3}}\mathbbm{1}_{\mathbb{X}^{l_{2},s_{1}}}(\overrightarrow{\textit{\textbf{n}}},m)\cdot \mathbbm{1}_{\mathbb{B}^{r_{1}}_{1}\cap\mathbb{C}^{s_{1}}_{1}}(n_{1},m)  \leq\|\mathbbm{1}_{\mathbb{X}^{l_{2},s_{1}}}\|_{\ell^{\infty}_{n_{1},m}\ell^{1}_{n_{2}}} \cdot \|\mathbbm{1}_{\mathbb{B}^{r_{1}}_{1}\cap\mathbb{C}^{s_{1}}_{1}}\|_{\ell^{1}_{n_{1},m}}. $$
The second factor of the inequality above will be bounded by the one-dimensional bilinear theory:

\begin{equation*}
\eqalign{
\#\mathbb{B}^{r_{1}}_{1}\cap\mathbb{C}^{s_{1}}_{1}&\displaystyle\lesssim 2^{2r_{1}+2s_{1}}\sum_{n_{1},m\in\mathbb{Z}} \left\|\left\langle g_{1},\varphi_{n_{1},m}^{1,1}\right\rangle_{x_{1}}\right\|_{2}^{2}\cdot \left\|\left\langle g_{2},\varphi_{n_{1},m}^{1,2}\right\rangle_{x_{1}}\right\|_{2}^{2}\cr
&=2^{2r_{1}+2s_{1}}\displaystyle\iint\left(\sum_{n_{1},m\in\mathbb{Z}}\left|\left\langle g_{1},\varphi_{n_{1},m}^{1,1}\right\rangle_{x_{1}}\right|^{2}\cdot\left|\left\langle g_{2},\varphi_{n_{1},m}^{1,2}\right\rangle_{x_{1}}\right|^{2}\right)\mathrm{d}x_{2}\mathrm{d}\widetilde{x}_{2} \cr
&=2^{2r_{1}+2s_{1}}\displaystyle\iint \|g_{1}\|_{L^{2}_{x_{1}}}^{2}\cdot \|g_{2}\|_{L^{2}_{x_{1}}}^{2}\mathrm{d}x_{2}\mathrm{d}\widetilde{x}_{2} \cr
&\leq\displaystyle 2^{2r_{1}+2s_{1}}\|g_{1}\|_{2}^{2}\cdot \|g_{2}\|_{2}^{2},
}
\end{equation*}
by Proposition \ref{prop3apr2922} since the supports of $\varphi^{1,1}$ and $\varphi^{1,2}$ are disjoint (this is equivalent to transversality in dimension one). This gives us
\begin{equation}\label{boundX1-jan1223}
\#\mathbb{X}^{\overrightarrow{l},\overrightarrow{r},\overrightarrow{s},k} \leq \|\mathbbm{1}_{\mathbb{X}^{l_{2},s_{1}}}\|_{\ell^{\infty}_{n_{1},m}\ell^{1}_{n_{2}}} \cdot 2^{2r_{1}+2s_{1}}\cdot |E_{1}|\cdot |E_{2}|.
\end{equation}
Alternatively,

\begin{equation*}
\eqalign{
\displaystyle \#\mathbb{X}^{\overrightarrow{l},\overrightarrow{r},\overrightarrow{s},k} &\displaystyle\leq \sum_{(n_{2},m)\in\mathbb{Z}^{2}} \mathbbm{1}_{\mathbb{B}^{r_{2}}_{2}\cap\mathbb{C}^{s_{2}}_{2}}(n_{2},m)\sum_{n_{1}\in\mathbb{Z}}\mathbbm{1}_{\mathbb{X}^{l_{2},s_{2}}}(\overrightarrow{\textit{\textbf{n}}},m)\cdot \mathbbm{1}_{\mathbb{B}^{r_{1}}_{1}}(n_{1},m) \cr
&\displaystyle\leq \|\mathbbm{1}_{\mathbb{X}^{l_{2},s_{2}}}\|^{\frac{1}{2}}_{\ell^{\infty}_{n_{2},m}\ell^{1}_{n_{1}}}\cdot \|\mathbbm{1}_{\mathbb{B}^{r_{1}}_{1}}\|_{\ell^{\infty}_{m}\ell^{1}_{n_{1}}}^{\frac{1}{2}} \cdot \|\mathbbm{1}_{\mathbb{B}^{r_{2}}_{2}\cap\mathbb{C}^{s_{2}}_{2}}\|_{\ell^{1}_{n_{2},m}}. \cr
}
\end{equation*}
We can treat the last two factors appearing in the right-hand side above as follows: for a fixed $m\in\mathbb{Z}$,

\begin{equation*}
\displaystyle\sum_{n_{1}\in\mathbb{Z}}\mathbbm{1}_{\mathbb{B}^{r_{1}}}(n_{1},m)\displaystyle\lesssim 2^{2r_{1}}\sum_{n_{1}\in\mathbb{Z}} \left\|\left\langle g_{1},\varphi_{n_{1},m}^{1,1}\right\rangle_{x_{1}}\right\|_{2}^{2}
\leq 2^{2r_{1}}\cdot \|g_{1}\|_{2}^{2}
\end{equation*}
by Bessel (recall that the modulated bumps $\varphi_{n_{1},m}^{1,1}$ are almost-orthogonal if $n_{1}$ varies and $m$ is fixed), and then we take the supremum in $m$. As for the other factor, observe that\footnote{Here we are also ignoring the fact that we do not prove the endpoint $L^{2}-L^{6}$ estimate for the model $E_{1}$. It will not compromise this preliminary exposition.}

\begin{equation*}
\eqalign{
\#\mathbb{B}^{r_{2}}_{2}\cap\mathbb{C}^{s_{2}}_{2}&\displaystyle\lesssim 2^{5r_{2}+s_{2}}\sum_{n_{2},m\in\mathbb{Z}}\left\|\left\langle g_{1},\varphi_{n_{2},m}^{2,1}\right\rangle_{x_{2}}\right\|_{2}^{5}\cdot\left\|\left\langle g_{2},\varphi_{n_{2},m}^{2,2}\right\rangle_{x_{2}}\right\|_{2} \cr
&\displaystyle\lesssim 2^{5r_{2}+s_{2}}\left(\sum_{n_{2},m\in\mathbb{Z}}\left\|\left\langle g_{1},\varphi_{n_{2},m}^{2,1}\right\rangle_{x_{2}}\right\|_{2}^{6}\right)^{\frac{5}{6}}\left(\sum_{n_{2},m\in\mathbb{Z}}\left\|\left\langle g_{2},\varphi_{n_{2},m}^{2,2}\right\rangle_{x_{2}}\right\|_{2}^{6}\right)^{\frac{1}{6}} \cr
&\leq\displaystyle 2^{5r_{2}+s_{2}}\|g_{1}\|_{2}^{5}\cdot\|g_{2}\|_{2}\cr
}
\end{equation*}
by Corollary \ref{cor1apr2922}. These last to estimates give the following bound on $ \#\mathbb{X}^{\overrightarrow{l},\overrightarrow{r},\overrightarrow{s},k}$:

\begin{equation}\label{boundX2-jan1223}
 \#\mathbb{X}^{\overrightarrow{l},\overrightarrow{r},\overrightarrow{s},k} \lesssim \|\mathbbm{1}_{\mathbb{X}^{l_{2},s_{2}}}\|^{\frac{1}{2}}_{\ell^{\infty}_{n_{2},m}\ell^{1}_{n_{1}}} \cdot 2^{r_{1}}\cdot |E_{1}|^{\frac{1}{2}} \cdot 2^{5r_{2}+s_{2}}\cdot |E_{1}|^{\frac{5}{2}}\cdot |E_{2}|^{\frac{1}{2}}.
\end{equation}
In what follows, we interpolate between \label{boundX1-jan1223}, \label{boundX2-jan1223} and \label{boundX3-jan1223} with weights $\frac{2}{5}^{-}$, $\frac{1}{5}^{-}$ and $\frac{2}{5}^{+}$, respectively. We also take an appropriate of combination between \eqref{mot1-jan1223} and \eqref{mot2-jan1223}, and use \eqref{mot3-jan1223}:

\begin{equation*}
\eqalign{
|\tilde{\tilde{\Lambda}}_{2,2}(g_{1},g_{2},h)| &``\lesssim"\displaystyle\sum_{\overrightarrow{r},\overrightarrow{s},k}\displaystyle \frac{2^{-r_{1}}\cdot 2^{-r_{2}}}{|E_{1}|^{\frac{1}{2}}} \cdot \frac{2^{-\frac{4}{5}s_{1}}}{\|\mathbbm{1}_{\mathbb{X}^{l_{2},s_{1}}}\|^{\frac{2}{5}}_{\ell_{n_{1},m}^{\infty}\ell^{1}_{n_{2}}}}\cdot \frac{2^{-\frac{1}{5}s_{2}}}{\|\mathbbm{1}_{\mathbb{X}^{l_{s},s_{2}}}\|^{\frac{1}{10}}_{\ell_{n_{2},m}^{\infty}\ell^{1}_{n_{1}}}}\cdot 2^{-k} \cr
&\displaystyle\qquad\qquad\cdot \left(\|\mathbbm{1}_{\mathbb{X}^{l_{2},s_{1}}}\|_{\ell^{\infty}_{n_{1},m}\ell^{1}_{n_{2}}} \cdot 2^{2r_{1}+2s_{1}}\cdot |E_{1}|\cdot |E_{2}|\right)^{\frac{2}{5}^{-}} \cr
&\displaystyle \qquad\qquad\cdot \left(\|\mathbbm{1}_{\mathbb{X}^{l_{2},s_{2}}}\|^{\frac{1}{2}}_{\ell^{\infty}_{n_{2},m}\ell^{1}_{n_{1}}} \cdot 2^{r_{1}}\cdot |E_{1}|^{\frac{1}{2}} \cdot 2^{5r_{2}+s_{2}}\cdot |E_{1}|^{\frac{5}{2}}\cdot |E_{2}|^{\frac{1}{2}}\right)^{\frac{1}{5}-} \cr
&\displaystyle\qquad\qquad \cdot\left(2^{k}|F|\right)^{\frac{2}{5}^{-}} \cr
&\lesssim\displaystyle |E_{1}|^{\frac{1}{2}}\cdot |E_{2}|^{\frac{1}{2}}\cdot |F|^{\frac{2}{5}^{+}},\cr
}
\end{equation*}
which is the estimate that we were looking for\footnote{This bound on $\tilde{\tilde{\Lambda}}_{2,2}$ is of course informal, which is why we wrote $``\lesssim"$. Observe that we also removed the sum in $\overrightarrow{l}$; it contributes with a term that depends on $\varepsilon$ in the formal argument. Later in the text we will see why we can assume $\overrightarrow{r},\overrightarrow{s},k\geq 0$ in the sum above.}.

\subsection{The general argument} Roughly, this is a one-paragraph outline of the proof: we split the sum in \eqref{form1-jan1223} into certain level sets, find good upper bounds for how many points $(\overrightarrow{\textit{\textbf{n}}},m)$ are in each level set using the weak transversality and Strichartz information, and then average all this data appropriately.

First we will prove the bound

\begin{equation}\label{mtl1apr29}
\|ME_{k,d}^{\frac{1}{k}}(g)\|_{L^{\frac{2(d+k+1)}{(d+k-1)}+\varepsilon}(\mathbb{R}^{d+1})}\lesssim_{\varepsilon}\prod_{l=1}^{k}|E_{1,l}|^{\frac{1}{2k}}\cdot \prod_{j=2}^{k}|E_{j}|^{\frac{1}{2k}},
\end{equation}
for every $\varepsilon>0$. As we remarked at the end of Section \ref{descriptionmethod}, this is the restricted weak-type bound that will be proved directly; all the other ones that are necessary for multilinear interpolation can be proved in a similar way, as the reader will notice.

We will define several level sets that encode the sizes of many quantities that will play a role in the proof. We start with the ones involving the scalar products in the multilinear form above.

$$\mathbb{A}_{j}^{l_{j}}:=\{(\overrightarrow{\textit{\textbf{n}}},m)\in\mathbb{Z}^{d+1};\quad |\langle g_{j},\varphi_{\overrightarrow{\textit{\textbf{n}}},m}^{j}\rangle|\approx 2^{-l_{j}}\},\quad 1\leq j\leq k. $$

The sizes of the $\langle g_{j},\varphi_{\overrightarrow{\textit{\textbf{n}}},m}\rangle$ are not the only information that we will need to control. As in the previous subsection, some mixed-norm quantities appear naturally after using Bessel's inequality along certain directions, and we will need to capture these as well:
$$\mathbb{B}_{i,1}^{r_{i,1}}:=\left\{(n_{i},m)\in\mathbb{Z}^{2};\quad \left\|\langle g_{1},\varphi_{n_{i},m}^{i,1}\rangle_{x_{i}}\right\|_{2}\approx 2^{-r_{i,1}}\right\},\quad 1\leq i\leq k-1,$$
$$\mathbb{B}_{i,i+1}^{r_{i,i+1}}:=\left\{(n_{i},m)\in\mathbb{Z}^{2};\quad \left\|\langle g_{i+1},\varphi_{n_{i},m}^{i,i+1}\rangle_{x_{i}}\right\|_{2}\approx 2^{-r_{i,i+1}}\right\}, \quad 1\leq i\leq k-1, $$
$$\mathbb{B}_{k,j}^{r_{k,j}}:=\left\{(\overrightarrow{\textit{\textbf{n}}_{k}},m)\in\mathbb{Z}^{d-k+2};\quad \left\|\langle g_{j},\varphi_{\overrightarrow{\textit{\textbf{n}}_{k}},m}^{k,j}\rangle_{\overrightarrow{x_{k}}}\right\|_{2}\approx 2^{-r_{k,j}}\right\},\quad 1\leq j\leq k. $$

Set $\mathbb{B}_{i,j}^{r_{i,j}}:=\emptyset$ for any other pair $(i,j)$ not included in the above definitions. Observe that $g_{1}$ (the function that has a tensor structure) has $k$ sets $\mathbb{B}$ associated to it: $k-1$ sets $\mathbb{B}^{r_{i,1}}_{i,1}$ and $1$ set $\mathbb{B}^{r_{k,1}}_{k,1}$. The other functions $g_{j}$, $j\neq 1$, have only two: $1$ set $\mathbb{B}^{r_{j-1,j}}_{j-1,j}$ and $1$ set $\mathbb{B}^{r_{k,j}}_{k,j}$ for each $1\leq j \leq k$. The idea behind the sets $\mathbb{B}_{i,1}^{r_{i,1}}$ and $\mathbb{B}_{i,i+1}^{r_{i,i+1}}$ is to isolate the ``piece" of each function that encodes the weak transversality information from the part that captures the Strichartz/Tomas-Stein behavior, which is in the set $\mathbb{B}_{k,j}^{r_{k,j}}$. For each $1\leq i\leq k-1$, we will pair the information of the sets $\mathbb{B}_{i,1}^{r_{i,1}}$ and $\mathbb{B}_{i,i+1}^{r_{i,i+1}}$ and use Proposition \ref{prop3apr2922} to extract the gain yielded by weak transversality. The information contained in the sets $\mathbb{B}_{k,j}^{r_{k,j}}$ will be exploited via Corollary \ref{cor1apr2922}.

The last quantity we have to control is the one arising from the dualizing function $h$:
$$\mathbb{C}^{t}:=\{(\overrightarrow{\textit{\textbf{n}}},m)\in\mathbb{Z}^{d+1};\quad |\langle h,\chi_{\overrightarrow{\textit{\textbf{n}}}}\otimes\chi_{m}\rangle|\approx 2^{-t}\}. $$

In order to prove some crucial bounds, at some point we will have to isolate the previous information for only one of the functions $g_{j}$. This will be done in terms of the following set\footnote{Many of these sets are empty since we set $\mathbb{B}_{i,j}^{r_{i,j}}=\emptyset$ for most $(i,j)$, but only the nonempty ones will appear in the argument.}: 
$$\mathbb{X}^{l_{j};r_{i,j}} = \mathbb{A}_{j}^{l_{j}}\cap \left\{(\overrightarrow{\textit{\textbf{n}}},m)\in\mathbb{Z}^{d+1};\quad (n_{i},m)\in \mathbb{B}_{i,j}^{r_{i,j}}\right\}.$$

In other words, $\mathbb{X}^{l_{j};r_{i,j}}$ contains all the $(n_{1},\ldots,n_{d},m)$ whose corresponding scalar product $\langle g_{j},\varphi_{\overrightarrow{\textit{\textbf{n}}},m}\rangle$ has size about $2^{-l_{j}}$ and with $(n_{i},m)$ being such that $\left\|\langle g_{j},\varphi_{n_{i},m}^{i,j}\rangle_{x_{i}}\right\|_{2}$ has size about $2^{-r_{i,j}}$.

Finally, it will also be important to encode all the previous information into one single set. This will be done with

$$\mathbb{X}^{\overrightarrow{l},R,t}:=\bigcap_{1\leq j\leq k}\mathbb{A}_{j}^{l_{j}}\cap \left\{(\overrightarrow{\textit{\textbf{n}}},m)\in\mathbb{Z}^{d+1};\quad (n_{i},m)\in\bigcap_{j}\mathbb{B}_{i,j}^{r_{i,j}},\quad 1\leq i\leq d\right\}\cap\mathbb{C}^{t}, $$
where we are using the abbreviations $\overrightarrow{l}=(l_{1},\ldots,l_{k})$ and $R:=(r_{i,j})_{i,j}$. Hence we can bound the form $\widetilde{\widetilde{\Lambda}}_{k,d}$ as follows:

\begin{equation}\label{form1-231221}
|\tilde{\tilde{\Lambda}}_{k,d}(g,h)|\lesssim \sum_{\overrightarrow{l},R,t\geq 0}2^{-t}\prod_{j=1}^{k}2^{-\frac{l_{j}}{k}}\#\mathbb{X}^{\overrightarrow{l},R,t}.
\end{equation}
Observe that we are assuming without loss of generality that $l_{j},r_{i,j},t\geq 0$. Indeed,

$$2^{-l_{j}}\lesssim |\langle g_{j},\varphi_{\overrightarrow{\textit{\textbf{n}}},m}^{j}\rangle|\leq \|g_{j}\|_{\infty}\cdot \|\varphi\|_{1}\lesssim 1,$$
so $l_{j}$ is at least as big as a universal integer. The argument for the remaining indices is the same.

The following two lemmas play a crucial role in the argument by relating the scalar and mixed-norm quantities involved in the stopping-time above. Lemma \ref{lemma1-3122020} allows us to do that for the quantities associated to $g_{1}$, the function that has a tensor structure. We remark that this is the only place in the proof where the tensor structure is used.

\begin{lemma}\label{lemma1-3122020} If $\mathbb{X}^{\overrightarrow{l},R,t}\neq\emptyset$, then:
\item $$2^{-l_{1}}\approx\frac{2^{-r_{1,1}}\cdot\ldots\cdot 2^{-r_{k,1}}}{\|g_{1}\|^{k-1}_{2}}, $$
\end{lemma}

\begin{proof} Observe that

\begin{equation*}
\eqalign{
 2^{-r_{1,1}}\cdot\ldots\cdot 2^{-r_{k,1}} &\approx\displaystyle \prod_{i=1}^{k}\left\|\langle g_{1},\varphi_{n_{i},m}^{i,1}\rangle_{x_{i}}\right\|_{2} \cr
 &=\displaystyle  \prod_{i=1}^{k}\left\|\langle g_{1,1}\otimes\ldots\otimes g_{1,k},\varphi_{n_{i},m}^{i,1}\rangle_{x_{i}}\right\|_{2} \cr
 &=\displaystyle \prod_{i=1}^{k}|\langle g_{1,i},\varphi_{n_{i},m}^{i,1}\rangle_{x_{i}}|\cdot\|g_{1,1}\otimes\ldots\otimes \widehat{g_{1,i}}\otimes\ldots\otimes g_{1,k}\|_{2} \cr
 &=\displaystyle |\langle g_{1},\varphi_{\overrightarrow{\textit{\textbf{n}}},m}^{1}\rangle|\cdot \|g_{1}\|_{2}^{k-1} \cr
 &\approx\displaystyle 2^{-l_{1}}\cdot \|g_{1}\|_{2}^{k-1},
}
\end{equation*}
and this proves the lemma. 
\end{proof}

Lemma \ref{lemma1-231221} gives us an alternative way of relating the quantities previously defined for the generic functions $g_{2},\ldots,g_{k}$.

\begin{lemma}\label{lemma1-231221} If $\mathbb{X}^{\overrightarrow{l},R,t}\neq\emptyset$, then:

\begin{equation}\label{rel1apr2922}
 2^{-l_{i+1}}\lesssim \frac{2^{-r_{i,i+1}}}{\left\|\mathbbm{1}_{\mathbb{X}^{l_{i+1};r_{i,i+1}}}\right\|_{\ell^{\infty}_{n_{i},m}\ell^{1}_{\widehat{n_{i}}}}^{\frac{1}{2}}},
 \end{equation}
 \begin{equation}\label{rel2apr2922}
 2^{-l_{i+1}}\lesssim \frac{2^{-r_{k,i+1}}}{\left\|\mathbbm{1}_{\mathbb{X}^{l_{i+1};r_{k,i+1}}}\right\|_{\ell^{\infty}_{\overrightarrow{n_{k}},m}\ell^{1}_{\widehat{\overrightarrow{n_{k}}}}}^{\frac{1}{2}}},
 \end{equation}
for all $1\leq i\leq k-1$.
\end{lemma}

\begin{proof} \eqref{rel1apr2922} is a consequence of orthogonality: for a fixed $(n_{i},m)$, denote 

$$\mathbb{X}^{l_{i+1};r_{i,i+1}}_{(n_{i},m)} :=\{ \widehat{n_{i}};\quad (\overrightarrow{\textit{\textbf{n}}},m)\in \mathbb{X}^{l_{i+1},r_{i,i+1}}\}. $$
This way,
\begin{equation*}
\eqalign{
\#\mathbb{X}^{l_{i+1};r_{i,i+1}}_{(n_{i},m)} &\approx\displaystyle 2^{2l_{i+1}}\sum_{\widehat{n_{i}}\in \mathbb{X}^{l_{i+1};r_{i,i+1}}_{(n_{i},m)}} |\langle g_{i+1},\varphi_{\overrightarrow{\textit{\textbf{n}}},m}^{i+1}\rangle|^{2}\cr
&\displaystyle\displaystyle\leq 2^{2l_{i+1}}\sum_{\widehat{n_{i}}}\left|\int\langle g_{i+1},\varphi_{n_{i},m}^{i,i+1}\rangle_{x_{i}}\cdot e^{-2\pi im(\sum_{j\neq i}x_{j}^{2})}\cdot \prod_{j\neq i}e^{-2\pi in_{j}x_{j}}\cdot\mathrm{d}\widehat{x_{i}}\right|^2 \cr
&\leq\displaystyle 2^{2l_{i+1}}\int \left|\langle g_{i+1},\varphi_{n_{i},m}^{i,i+1}\rangle_{x_{i}}\right|^{2}\mathrm{d}\widehat{x_{i}} \cr
&\approx\displaystyle 2^{2l_{i+1}}\cdot 2^{-2r_{i,i+1}},
}
\end{equation*}
where we used Bessel's inequality from the second to the third line. The lemma follows by taking the supremum in $(n_{i},m)$. \eqref{rel2apr2922} is proven analogously.

\end{proof}

The following corollary gives a convex combination of the relations in Lemma \ref{lemma1-231221} that will be used in the proof.

\begin{corollary}\label{cor1-3122020} For $1\leq i\leq k-1$ we have
$$2^{-l_{i+1}}\lesssim \frac{2^{-\frac{2k}{d+k+1}\cdot r_{i,i+1}}}{\left\|\mathbbm{1}_{\mathbb{X}^{l_{i+1};r_{i,i+1}}}\right\|_{\ell^{\infty}_{n_{i},m}\ell^{1}_{\widehat{n_{i}}}}^{\frac{k}{d+k+1}}}\cdot\frac{2^{-\frac{(d-k+1)}{(d+k+1)}\cdot r_{k,i+1}}}{\left\|\mathbbm{1}_{\mathbb{X}^{l_{i+1};r_{k,i+1}}}\right\|_{\ell^{\infty}_{\overrightarrow{\textit{\textbf{n}}_{k}},m}\ell^{1}_{\widehat{\overrightarrow{n_{k}}}}}^{\frac{(d-k+1)}{2(d+k+1)}}}.$$
\end{corollary}
\begin{proof} Interpolate between the bounds of Lemma \ref{lemma1-231221} with weights $\frac{2k}{d+k+1}$ and $\frac{d-k+1}{d+k+1}$, respectively.
\end{proof}

We now concentrate on estimating the right-hand side of \eqref{form1-231221} by finding good bounds for $\#\mathbb{X}^{\overrightarrow{l},R,t}$. The following bound follows immediately from the disjointness of the supports of $\chi_{\overrightarrow{\textit{\textbf{n}}}}\otimes\chi_{m}$:

\begin{equation}\label{bound1-dec12020}
\#\mathbb{X}^{\overrightarrow{l},R,t}\lesssim\sum_{(\overrightarrow{\textit{\textbf{n}}},m)\in\mathbb{Z}^{d+1}}|\langle h,\chi_{\overrightarrow{\textit{\textbf{n}}}}\otimes\chi_{m}\rangle|\lesssim 2^{t}|F|.
\end{equation}
By definition of the set $\mathbb{X}^{\overrightarrow{l},R,t}$:

\begin{equation}\label{ineqX-231221}
\eqalign{
\#\mathbb{X}^{\overrightarrow{l},R,t}&\displaystyle\leq\sum_{(\overrightarrow{\textit{\textbf{n}}},m)\in\mathbb{Z}^{d+1}}\prod_{j=1}^{k}\mathbbm{1}_{\mathbb{A}^{l_{j}}_{j}}(\overrightarrow{\textit{\textbf{n}}},m)\cdot \prod_{\substack{i,j \\ \mathbb{B}_{i,j}^{r_{i,j}}\neq\emptyset}}\mathbbm{1}_{\mathbb{B}_{i,j}^{r_{i,j}}}(n_{i},m).\cr
}
\end{equation}
We will manipulate \eqref{ineqX-231221} in $k$ different ways: $k-1$ of them will exploit orthogonality (through the one-dimensional bilinear theory after combining the sets $\mathbb{B}_{i,1}^{r_{i,1}}$ and $\mathbb{B}_{i,i+1}^{r_{i,i+1}}$, $1\leq i\leq k-1$) and the last one will reflect Strichartz/Tomas-Stein in an appropriate dimension. The following lemma gives us estimates for the cardinality of $\mathbb{X}^{\overrightarrow{l},R,t}$ based on the sizes of some of its slices along canonical directions\footnote{The reader may associate this idea to certain discrete Loomis-Whitney or Brascamp-Lieb inequalities. While reducing matters to lower dimensional theory is at the core of our paper, we do not yet have a genuine ``Brascamp-Lieb way" of bounding $\#\mathbb{X}^{\overrightarrow{l},R,t}$ for which our methods work. For instance, no ``slice" of $\mathbb{X}^{\overrightarrow{l},R,t}$ given by fixing a few (or all) $n_{j}$ and summing over $m$ appears in our estimates, which breaks the Loomis-Whitney symmetry.}.

\begin{lemma}\label{lemma1-21aug2021} The bounds above imply
\begin{enumerate}[(a)]
\item The \textit{orthogonality-type bounds}\footnote{Weak transversality enters the picture here.}:
\begin{equation}\label{bound3-3122020}
\#\mathbb{X}^{\overrightarrow{l},R,t}\leq \left\|\mathbbm{1}_{\mathbb{X}^{l_{i+1};r_{i,i+1}}}\right\|_{\ell^{\infty}_{n_{i},m}\ell^{1}_{\widehat{n_{i}}}}\cdot 2^{2r_{i,1}+2r_{i,i+1}}\cdot \|g_{1}\|_{2}^{2}\cdot \|g_{i+1}\|_{2}^{2},\qquad 1\leq i\leq k-1.
\end{equation}
\item The \textit{Strichartz-type bound}:
\begin{equation}\label{bound4-3122020}
\#\mathbb{X}^{\overrightarrow{l},R,t}\leq \prod_{j=2}^{k}\left\|\mathbbm{1}_{\mathbb{X}^{l_{j};r_{k,j}}}\right\|_{\ell^{\infty}_{\overrightarrow{\textit{\textbf{n}}_{k}},m}\ell^{1}_{\widehat{\overrightarrow{n_{k}}}}}^{\frac{1}{k}}\cdot 2^{\frac{2}{k}\sum_{i=1}^{k-1}r_{i,1}}\cdot \|g_{1}\|_{2}^{\frac{2(k-1)}{k}}\cdot 2^{\alpha\cdot r_{k,1}+\sum_{l=2}^{k}\beta\cdot r_{k,l}}\cdot \|g_{1}\|_{2}^{\alpha}\cdot \prod_{l=2}^{k}\|g_{l}\|^{\beta}_{2},
\end{equation}
where
$$\alpha := \frac{2(d+k+1)}{k(d-k+1)}+\delta\cdot\frac{(d+k+1)}{k(d-k+3)}, \quad \beta := \frac{2}{k}+\tilde{\delta}\cdot\frac{(d-k+1)}{k(d-k+3)}, $$
with $\delta,\tilde{\delta}> 0$ being arbitrarily small parameters to be chosen later\footnote{One should think of $\delta$ and $\tilde{\delta}$ as being ``morally zero". They will be chosen as a function of the initially given $\varepsilon>0$, and the only reason we introduce them is to make the appropriate up to the endpoint Strichartz exponent appear in \eqref{bound2-3122020}. The main terms of $\alpha$ and $\beta$ are also chosen with that in mind.}.
\end{enumerate}

\end{lemma}
\begin{proof} For each $1\leq i\leq k-1$ we bound most of the indicator functions in \eqref{ineqX-231221} by $1$ and obtain

\begin{equation}\label{bound1-dec22020}
\eqalign{
\#\mathbb{X}^{\overrightarrow{l},R,t}&\displaystyle\leq\sum_{(\overrightarrow{\textit{\textbf{n}}},m)\in\mathbb{Z}^{d+1}}\mathbbm{1}_{\mathbb{A}^{l_{i+1}}_{i+1}}(\overrightarrow{\textit{\textbf{n}}},m)\cdot \mathbbm{1}_{\mathbb{B}_{i,1}^{r_{i,1}}}(n_{i},m) \cdot \mathbbm{1}_{\mathbb{B}_{i,i+1}^{r_{i,i+1}}}(n_{i},m)\cr
&\displaystyle= \sum_{(\overrightarrow{\textit{\textbf{n}}},m)\in\mathbb{Z}^{d+1}}\mathbbm{1}_{\mathbb{X}^{l_{i+1};r_{i,i+1}}}(\overrightarrow{\textit{\textbf{n}}},m)\cdot \mathbbm{1}_{\mathbb{B}_{i,1}^{r_{i,1}}\cap \mathbb{B}_{i,i+1}^{r_{i,i+1}}}(n_{i},m) \cr
&\displaystyle=\sum_{n_{i},m}\mathbbm{1}_{\mathbb{B}_{i,1}^{r_{i,1}}\cap \mathbb{B}_{i,i+1}^{r_{i,i+1}}}(n_{i},m)\sum_{\widehat{n_{i}}}\mathbbm{1}_{\mathbb{X}^{l_{i+1};r_{i,i+1}}}(\overrightarrow{\textit{\textbf{n}}},m) \cr
&\displaystyle\leq \left\|\mathbbm{1}_{\mathbb{X}^{l_{i+1};r_{i,i+1}}}\right\|_{\ell^{\infty}_{n_{i},m}\ell^{1}_{\widehat{n_{i}}}}\cdot \left\|\mathbbm{1}_{\mathbb{B}_{i,1}^{r_{i,1}}\cap \mathbb{B}_{i,i+1}^{r_{i,i+1}}}\right\|_{\ell^{1}_{n_{i},m}}. \cr
}
\end{equation}
\textbf{Transversality is exploited now:} the cube $Q_{1}$ with $\{e_{1},\ldots,e_{k-1}\}$ as associated set of directions satisfies \eqref{eqwtmay722}, which allows us to apply Proposition \ref{prop3apr2922} for each $1\leq i\leq k-1$ since weak transversality is equivalent to transversality in dimension $d=1$. By definition of the sets $\mathbb{B}_{i,1}^{r_{i,1}}$ and $\mathbb{B}_{i,i+1}^{r_{i,i+1}}$, Fubini and Proposition \ref{prop3apr2922} we have:

\begin{equation*}
\eqalign{
\left\|\mathbbm{1}_{\mathbb{B}_{i,1}^{r_{i,1}}\cap \mathbb{B}_{i,i+1}^{r_{i,i+1}}}\right\|_{\ell^{1}_{n_{i},m}} &\displaystyle\lesssim 2^{2r_{i,1}+2r_{i,i+1}}\sum_{(n_{i},m)\in\mathbb{B}_{i,1}^{r_{i,1}}\cap \mathbb{B}_{i,i+1}^{r_{i,i+1}}} \left\|\langle g_{1},\varphi_{n_{i},m}^{i,1}\rangle_{x_{i}}\right\|_{2}^{2} \cdot \left\|\langle g_{i+1},\varphi_{n_{i},m}^{i,i+1}\rangle_{y_{i}}\right\|_{2}^{2} \cr
&\displaystyle\leq 2^{2r_{i,1}+2r_{i,i+1}}\int\int \left(\sum_{(n_{i},m)\in\mathbb{Z}^{2}}\left|\langle g_{1},\varphi_{n_{i},m}^{i,1}\rangle_{x_{i}}\right|^{2}\cdot \left|\langle g_{i+1},\varphi_{n_{i},m}^{i,i+1}\rangle_{y_{i}}\right|^{2}\right)\mathrm{d}\widehat{x_{i}}\mathrm{d}\widehat{y_{i}} \cr
&\displaystyle \leq 2^{2r_{i,1}+2r_{i,i+1}}\int\int \|g_{1}\|_{L^{2}_{x_{i}}}^{2}\cdot \|g_{i+1}\|_{L^{2}_{y_{i}}}^{2}\mathrm{d}\widehat{x_{i}}\mathrm{d}\widehat{y_{i}} \cr
&\displaystyle= 2^{2r_{i,1}+2r_{i,i+1}}\cdot \|g_{1}\|_{2}^{2}\cdot \|g_{i+1}\|_{2}^{2}
}
\end{equation*}
Using this in \eqref{bound1-dec22020} gives \textit{(a)}. As for \textit{(b)}, bound $\#\mathbb{X}^{\overrightarrow{l},R,t}$ as follows:
\begin{equation}\label{bound2-dec22020}
\eqalign{
\#\mathbb{X}^{\overrightarrow{l},R,t}&\displaystyle = \sum_{(\overrightarrow{\textit{\textbf{n}}},m)\in\mathbb{Z}^{d+1}}\mathbbm{1}_{\mathbb{X}^{\overrightarrow{l},R,t}}(\overrightarrow{\textit{\textbf{n}}},m)\cr
&\displaystyle\leq\sum_{(\overrightarrow{\textit{\textbf{n}}},m)\in\mathbb{Z}^{d+1}}\prod_{j=2}^{k}\mathbbm{1}_{\mathbb{X}^{l_{j};r_{k,j}}}(\overrightarrow{\textit{\textbf{n}}},m)\prod_{i=1}^{k-1}\mathbbm{1}_{\mathbb{B}_{i,1}^{r_{i,1}}}(n_{i},m)\cdot\prod_{l=1}^{k}\mathbbm{1}_{\mathbb{B}_{k,l}^{r_{k,l}}}(\overrightarrow{\textit{\textbf{n}}_{k}},m) \cr
&\displaystyle=\sum_{\overrightarrow{\textit{\textbf{n}}_{k}},m}\prod_{l=1}^{k}\mathbbm{1}_{\mathbb{B}_{k,l}^{r_{k,l}}}(\overrightarrow{\textit{\textbf{n}}_{k}},m)\sum_{n_{1},\ldots,n_{k-1}}\prod_{j=2}^{k}\mathbbm{1}_{\mathbb{X}^{l_{j};r_{k,j}}}(\overrightarrow{\textit{\textbf{n}}},m)\prod_{i=1}^{k-1}\mathbbm{1}_{\mathbb{B}_{i,1}^{r_{i,1}}}(n_{i},m) \cr
&\displaystyle\leq \sum_{\overrightarrow{\textit{\textbf{n}}_{k}},m}\prod_{l=1}^{k}\mathbbm{1}_{\mathbb{B}_{k,l}^{r_{k,l}}}(\overrightarrow{\textit{\textbf{n}}_{k}},m)\prod_{j=2}^{k}\left\|\mathbbm{1}_{\mathbb{X}^{l_{j};r_{k,j}}}(\overrightarrow{\textit{\textbf{n}}},m)\right\|_{\ell^{1}_{\widehat{\overrightarrow{n_{k}}}}}^{\frac{1}{k}}\cdot \left\|\prod_{i=1}^{k-1}\mathbbm{1}_{\mathbb{B}_{i,1}^{r_{i,1}}}(n_{i},m)\right\|_{\ell^{1}_{\widehat{\overrightarrow{n_{k}}}}}^{\frac{1}{k}} \cr
&\displaystyle\leq \prod_{j=2}^{k}\left\|\mathbbm{1}_{\mathbb{X}^{l_{j};r_{k,j}}}\right\|_{\ell^{\infty}_{\overrightarrow{\textit{\textbf{n}}_{k}},m}\ell^{1}_{\widehat{\overrightarrow{n_{k}}}}}^{\frac{1}{k}}\cdot \prod_{i=1}^{k-1}\left\|\mathbbm{1}_{\mathbb{B}_{i,1}^{r_{i,1}}}\right\|_{\ell^{\infty}_{m}\ell^{1}_{n_{i}}}^{\frac{1}{k}} \cdot\left\|\prod_{l=1}^{k}\mathbbm{1}_{\mathbb{B}_{k,l}^{r_{k,l}}}\right\|_{\ell^{1}_{\overrightarrow{\textit{\textbf{n}}_{k}},m}},\cr
}
\end{equation}
where we used H\"older's inequality from the third to fourth line. Next, notice that

\begin{equation}\label{bound1-3122020}
\eqalign{
\left\|\mathbbm{1}_{\mathbb{B}_{i,1}^{r_{i,1}}}\right\|_{\ell^{\infty}_{m}\ell^{1}_{n_{i}}} &\lesssim\displaystyle\sup_{m} 2^{2r_{i,1}}\sum_{n_{i}}\left\|\langle g_{1},\varphi_{n_{i},m}^{i,1}\rangle_{x_{i}}\right\|_{2}^{2}\cr
&\displaystyle=\sup_{m}2^{2r_{i,1}}\int\sum_{n_{i}}\left|\langle g_{1},\varphi_{n_{i},m}^{i,1}\rangle_{x_{i}}\right|^{2}\mathrm{d}\widehat{x_{i}} \cr
&\displaystyle\lesssim 2^{2r_{i,1}}\cdot \|g_{1}\|_{2}^{2} \cr
}
\end{equation}
by orthogonality. Now let

$$p_{k,1}:=\frac{k(d-k+3)}{(d+k+1)}, $$
$$p_{k,l}:=\frac{k(d-k+3)}{(d-k+1)},\quad\forall\quad 2\leq l\leq k $$
and notice that
$$\sum_{l=1}^{k}\frac{1}{p_{k,l}}=1. $$
This way, by definition of $\mathbb{B}_{k,l}^{r_{k,l}}$ and by H\"older's inequality with these $p_{k,l}$ we have

\begin{equation}\label{bound2-3122020}
\eqalign{
&\displaystyle\left\|\prod_{l=1}^{k}\mathbbm{1}_{\mathbb{B}_{k,l}^{r_{k,l}}}\right\|_{\ell^{1}_{\overrightarrow{\textit{\textbf{n}}_{k}},m}} \cr
&\quad\quad\lesssim\displaystyle 2^{\alpha\cdot r_{k,1}+\sum_{l=2}^{k}\beta\cdot r_{k,l}}\sum_{(\overrightarrow{\textit{\textbf{n}}_{k}},m)}\left\|\langle g_{1},\varphi_{\overrightarrow{\textit{\textbf{n}}_{k}},m}^{k,1}\rangle_{\overrightarrow{x_{k}}}\right\|_{2}^{\alpha}\cdot\prod_{l=2}^{k}\left\|\langle g_{l},\varphi_{\overrightarrow{\textit{\textbf{n}}_{k}},m}^{k,l}\rangle_{\overrightarrow{x_{k}}}\right\|_{2}^{\beta} \cr
&\quad\quad\leq\displaystyle 2^{\alpha\cdot r_{k,1}+\sum_{l=2}^{k}\beta\cdot r_{k,l}}\left(\sum_{(\overrightarrow{\textit{\textbf{n}}_{k}},m)}\left\|\langle g_{1},\varphi_{\overrightarrow{\textit{\textbf{n}}_{k}},m}^{k,1}\rangle_{\overrightarrow{x_{k}}}\right\|_{2}^{\alpha\cdot p_{k,1}}\right)^{\frac{1}{p_{k,1}}}\cdot\prod_{l=2}^{k}\left(\sum_{(\overrightarrow{\textit{\textbf{n}}_{k}},m)}\left\|\langle g_{l},\varphi_{\overrightarrow{\textit{\textbf{n}}_{k}},m}^{k,l}\rangle_{\overrightarrow{x_{k}}}\right\|_{2}^{\beta\cdot p_{k,l}} \right)^{\frac{1}{p_{k,l}}} \cr
&\quad\quad=\displaystyle 2^{\alpha\cdot r_{k,1}+\sum_{l=2}^{k}\beta\cdot r_{k,l}}\left(\sum_{(\overrightarrow{\textit{\textbf{n}}_{k}},m)}\left\|\langle g_{1},\varphi_{\overrightarrow{\textit{\textbf{n}}_{k}},m}^{k,1}\rangle_{\overrightarrow{x_{k}}}\right\|_{2}^{\frac{2(d-k+3)}{(d-k+1)}+\delta}\right)^{\frac{1}{p_{k,1}}}\cdot\prod_{l=2}^{k}\left(\sum_{(\overrightarrow{\textit{\textbf{n}}_{k}},m)}\left\|\langle g_{l},\varphi_{\overrightarrow{\textit{\textbf{n}}_{k}},m}^{k,l}\rangle_{\overrightarrow{x_{k}}}\right\|_{2}^{\frac{2(d-k+3)}{(d-k+1)}+\tilde{\delta}} \right)^{\frac{1}{p_{k,l}}} \cr
&\quad\quad\leq\displaystyle 2^{\alpha\cdot r_{k,1}+\sum_{l=2}^{k}\beta\cdot r_{k,l}}\cdot \|g_{1}\|_{2}^{\alpha}\cdot \prod_{l=2}^{k}\|g_{l}\|^{\beta}_{2},
}
\end{equation}
by the up to the endpoint mixed-norm Strichartz bound in Corollary \ref{cor1apr2922}\footnote{See the footnote related to Corollary \ref{cor1apr2922}.}. Using \eqref{bound1-3122020} and \eqref{bound2-3122020} in \eqref{bound2-dec22020} yields \textit{(b)}.
\end{proof}

Given $\varepsilon>0$ small\footnote{Perhaps it is helpful for the reader to think of $\varepsilon$, $\delta$ and $\tilde{\delta}$ as equal to zero to focus on the important parts of the proof. The presence of these parameters here is a mere technicality, except of course for the fact that $\varepsilon>0$ makes us lose the endpoint in this case.}, we interpolate between $k+1$ bounds for $\#\mathbb{X}^{\overrightarrow{l},R,t}$ with the following weights\footnote{Observe that $\sum_{l=1}^{k+1}\theta_{l}=1$. These weights are chosen so that the correct powers of the measures $|E_{j}|$ and $|F|$ appear in \eqref{eq1may622}.}:

$$
\begin{cases}
\theta_{l}=\frac{1}{d+k+1}-\frac{\varepsilon}{k},\quad 1\leq l\leq k-1\quad\textnormal{ for \eqref{bound3-3122020}},\\
\theta_{k}=\frac{(d-k+1)}{2(d+k+1)}-\frac{\varepsilon}{k}\quad\textnormal{ for \eqref{bound4-3122020}},\\
\theta_{k+1}= \left[1-\frac{(d+k-1)}{2(d+k+1)}\right]+\varepsilon\quad\textnormal{ for \eqref{bound1-dec12020}},
\end{cases}
$$
which leads to

\begin{equation*}
\eqalign{
|&\displaystyle\tilde{\tilde{\Lambda}}_{k,d}(g,h)| \cr
&\lesssim\displaystyle \sum_{\overrightarrow{l},R,t\geq 0} 2^{-t}\displaystyle\times\prod_{j=1}^{k}2^{-\frac{l_{j}}{k}}\cr
&\displaystyle\times\prod_{l=1}^{k-1}\left(\left\|\mathbbm{1}_{\mathbb{X}^{l_{l+1};r_{l,l+1}}}\right\|_{\ell^{\infty}_{n_{l},m}\ell^{1}_{\widehat{n_{l}}}}\cdot 2^{2r_{l,1}+2r_{l,l+1}}\cdot\|g_{1}\|_{2}^{2}\cdot \|g_{l+1}\|_{2}^{2}\right)^{\frac{1}{d+k+1}-\frac{\varepsilon}{k}} \cr
&\displaystyle\times\left(\prod_{j=2}^{k}\left\|\mathbbm{1}_{\mathbb{X}^{l_{j};r_{k,j}}}\right\|_{\ell^{\infty}_{\overrightarrow{\textit{\textbf{n}}_{k}},m}\ell^{1}_{\widehat{\overrightarrow{n_{k}}}}}^{\frac{1}{k}}\cdot 2^{\frac{2}{k}\sum_{i=1}^{k-1}r_{i,1}}\cdot \|g_{1}\|_{2}^{\frac{2(k-1)}{k}}\cdot 2^{\alpha\cdot r_{k,1}+\sum_{l=2}^{k}\beta\cdot r_{k,l}}\cdot \|g_{1}\|_{2}^{\alpha}\cdot \prod_{l=2}^{k}\|g_{l}\|^{\beta}_{2}\right)^{\frac{(d-k+1)}{2(d+k+1)}-\frac{\varepsilon}{k}}\cr
&\displaystyle\times\left(2^{t}|F|\right)^{\left[1-\frac{(d+k-1)}{2(d+k+1)}\right]+\varepsilon}, \cr
}
\end{equation*}

Using Lemma \ref{lemma1-3122020} and Corollary \ref{cor1-3122020} to bound the $2^{-l_{j}}$ in the form $\tilde{\tilde{\Lambda}}_{k,d}$ yields:

\begin{equation*}
\eqalign{
|&\displaystyle \tilde{\tilde{\Lambda}}_{k,d}(g,h)| \cr
&\displaystyle\lesssim\displaystyle \sum_{\overrightarrow{l},R,t\geq 0} 2^{-t}\displaystyle\times 2^{-\frac{\varepsilon}{k^{2}}l_{1}}\times\left(\frac{1}{\|g_{1}\|_{2}^{k-1}}\prod_{j=1}^{k}2^{-r_{j,1}}\right)^{\frac{1}{k}-\frac{\varepsilon}{k^{2}}} \cr
&\displaystyle\times \prod_{i=1}^{k-1}2^{-\frac{\varepsilon}{k^{2}}l_{i+1}}\times\prod_{i=1}^{k-1}\left[\frac{2^{-\frac{2}{d+k+1}\cdot r_{i,i+1}}}{\left\|\mathbbm{1}_{\mathbb{X}^{l_{i+1};r_{i,i+1}}}\right\|_{\ell^{\infty}_{n_{i},m}\ell^{1}_{\widehat{n_{i}}}}^{\frac{1}{d+k+1}}}\cdot\frac{2^{-\frac{(d-k+1)}{k(d+k+1)}\cdot r_{k,i+1}}}{\left\|\mathbbm{1}_{\mathbb{X}^{l_{i+1};r_{k,i+1}}}\right\|_{\ell^{\infty}_{\overrightarrow{\textit{\textbf{n}}_{k}},m}\ell^{1}_{\widehat{\overrightarrow{n_{k}}}}}^{\frac{(d-k+1)}{2k(d+k+1)}}}\right]^{1-\frac{\varepsilon}{k}} \cr
&\displaystyle\times\prod_{l=1}^{k-1}\left(\left\|\mathbbm{1}_{\mathbb{X}^{l_{l+1};r_{l,l+1}}}\right\|_{\ell^{\infty}_{n_{l},m}\ell^{1}_{\widehat{n_{l}}}}\cdot 2^{2r_{l,1}+2r_{l,l+1}}\cdot\|g_{1}\|_{2}^{2}\cdot \|g_{l+1}\|_{2}^{2}\right)^{\frac{1}{d+k+1}-\frac{\varepsilon}{k}} \cr
&\displaystyle\times\left(\prod_{j=2}^{k}\left\|\mathbbm{1}_{\mathbb{X}^{l_{j};r_{k,j}}}\right\|_{\ell^{\infty}_{\overrightarrow{\textit{\textbf{n}}_{k}},m}\ell^{1}_{\widehat{\overrightarrow{n_{k}}}}}^{\frac{1}{k}}\cdot 2^{\frac{2}{k}\sum_{i=1}^{k-1}r_{i,1}}\cdot \|g_{1}\|_{2}^{\frac{2(k-1)}{k}}\cdot 2^{\alpha\cdot r_{k,1}+\sum_{l=2}^{k}\beta\cdot r_{k,l}}\cdot \|g_{1}\|_{2}^{\alpha}\cdot \prod_{l=2}^{k}\|g_{l}\|^{\beta}_{2}\right)^{\frac{(d-k+1)}{2(d+k+1)}-\frac{\varepsilon}{k}}\cr
&\displaystyle\times\left(2^{t}|F|\right)^{\left[1-\frac{(d+k-1)}{2(d+k+1)}\right]+\varepsilon}, \cr
}
\end{equation*}

Developing the expression above,

\begin{equation*}
\eqalign{
|\tilde{\tilde{\Lambda}}_{k,d}(g,h)|\lesssim\displaystyle &\displaystyle\sum_{\overrightarrow{l},R,t\geq 0} 2^{-t}\displaystyle\times 2^{-\frac{\varepsilon}{k^{2}}l_{1}}\times\left(\prod_{j=1}^{k}2^{-r_{j,1}}\right)^{\frac{1}{k}-\frac{\varepsilon}{k^{2}}}\cdot \|g_{1}\|_{2}^{\frac{(k-1)}{k^{2}}\varepsilon-\frac{(k-1)}{k}}  \cr
&\displaystyle\times \prod_{i=1}^{k-1}2^{-\frac{\varepsilon}{k^{2}}l_{i+1}}\times\prod_{i=1}^{k-1}\left[2^{-\frac{2}{d+k+1}\cdot r_{i,i+1}}\cdot 2^{-\frac{(d-k+1)}{k(d+k+1)}\cdot r_{k,i+1}}\right]^{1-\frac{\varepsilon}{k}} \cr
&\displaystyle \times \prod_{i=1}^{k-1}\left[\left\|\mathbbm{1}_{\mathbb{X}^{l_{i+1};r_{i,i+1}}}\right\|_{\ell^{\infty}_{n_{i},m}\ell^{1}_{\widehat{n_{i}}}}^{\frac{1}{d+k+1}\cdot \left(\frac{\varepsilon}{k}-1\right)}\cdot \left\|\mathbbm{1}_{\mathbb{X}^{l_{i+1};r_{k,i+1}}}\right\|_{\ell^{\infty}_{\overrightarrow{\textit{\textbf{n}}_{k}},m}\ell^{1}_{\widehat{\overrightarrow{n_{k}}}}}^{\frac{(d-k+1)}{2k(d+k+1)}\cdot\left(\frac{\varepsilon}{k}-1\right)}\right]\cr
&\displaystyle\times\left[\prod_{l=1}^{k-1}\left\|\mathbbm{1}_{\mathbb{X}^{l_{l+1};r_{l,l+1}}}\right\|_{\ell^{\infty}_{n_{l},m}\ell^{1}_{\widehat{n_{l}}}}^{\frac{1}{d+k+1}-\frac{\varepsilon}{k}} \right]\times\left[\prod_{l=1}^{k-1}\left(2^{r_{l,1}+r_{l,l+1}}\right)^{\frac{2}{d+k+1}-\frac{2\varepsilon}{k}}\right] \cr
&\times\displaystyle \|g_{1}\|_{2}^{\frac{2(k-1)}{d+k+1}-\frac{2(k-1)\varepsilon}{k}}\cdot \prod_{l=1}^{k-1}\|g_{l+1}\|_{2}^{\frac{2}{d+k+1}-\frac{2\varepsilon}{k}} \cr
&\displaystyle\times\prod_{j=2}^{k}\left\|\mathbbm{1}_{\mathbb{X}^{l_{j};r_{k,j}}}\right\|_{\ell^{\infty}_{\overrightarrow{\textit{\textbf{n}}_{k}},m}\ell^{1}_{\widehat{\overrightarrow{n_{k}}}}}^{\frac{1}{k}\cdot \left(\frac{(d-k+1)}{2(d+k+1)}-\frac{\varepsilon}{k}\right)}\cdot\left(2^{\frac{2}{k}\sum_{i=1}^{k-1}r_{i,1}}\cdot  2^{\alpha\cdot r_{k,1}+\sum_{l=2}^{k}\beta\cdot r_{k,l}}\right)^{\frac{(d-k+1)}{2(d+k+1)}-\frac{\varepsilon}{k}}\cr
&\displaystyle\times \|g_{1}\|_{2}^{\left(\frac{2(k-1)}{k}+\alpha\right)\cdot \left(\frac{(d-k+1)}{2(d+k+1)}-\frac{\varepsilon}{k}\right)}\cdot \prod_{l=2}^{k}\|g_{l}\|^{\beta\cdot\left(\frac{(d-k+1)}{2(d+k+1)}-\frac{\varepsilon}{k}\right)}_{2}\cr
&\displaystyle\times\left(2^{t}|F|\right)^{\left[1-\frac{(d+k-1)}{2(d+k+1)}\right]+\varepsilon}. \cr
}
\end{equation*}
At this point we set the values of $\delta$ and $\tilde{\delta}$ (as functions of $\varepsilon$) to be such that\footnote{We emphasize that these particular choices are just for computational convenience, and we have not developed the expressions because this is exactly how we use them to simplify the previous calculations.}
\begin{equation*}
\eqalign{
\displaystyle\delta\cdot\left[\frac{(d-k+1)}{k(d-k+3)}-\frac{(d+k+1)\varepsilon}{k^{2}(d-k+3)}\right]&\displaystyle=\frac{1}{2}\left[-\frac{\varepsilon}{k^{2}}+\frac{2(d+k+1)\varepsilon}{k^{2}(d-k+1)}\right] \cr
\displaystyle\tilde{\delta}\cdot\left[\frac{(d-k+1)^{2}}{2k(d+k+1)(d-k+3)}-\frac{(d-k+1)\varepsilon}{k^{2}(d-k+3)}\right]&\displaystyle = \frac{1}{2}\left[\frac{2\varepsilon}{k^{2}}-\frac{(d-k+1)\varepsilon}{k^{2}(d+k+1)}\right].\cr
}
\end{equation*}
Simplifying (and using the expressions that define $\alpha$ and $\beta$ in Lemma \ref{lemma1-21aug2021}),

\begin{equation*}
\eqalign{
|\tilde{\tilde{\Lambda}}_{k,d}(g,h)|\lesssim\displaystyle &\displaystyle\displaystyle\left[\sum_{l_{1}\geq 0} 2^{-\frac{\varepsilon}{k^{2}}l_{1}}\right]\times\left[\prod_{j=1}^{k-1}\left(\sum_{r_{j,1}\geq 0}2^{-\left(\frac{2\varepsilon}{k}+\frac{\varepsilon}{k^{2}}\right)r_{j,1}}\right)\right]\cdot \left[\sum_{r_{k,1}\geq 0}2^{-r_{k,1}\left(-\frac{\varepsilon}{2k^{2}}+\frac{(d+k+1)}{(d-k+1)}\frac{\varepsilon}{k^{2}}\right)}\right]  \cr
&\displaystyle\times \left[\prod_{i=1}^{k-1}\left(\sum_{l_{i+1}\geq 0}2^{-\frac{\varepsilon}{k^{2}}l_{i+1}}\right)\right]\times \left[\sum_{t\geq 0}2^{-t\left(\frac{(d+k-1)}{2(d+k+1)}-\varepsilon\right)}\right]\cr
&\displaystyle \times  \left[\prod_{i=1}^{k-1}\left(\sum_{r_{i,i+1}\geq 0}2^{-\frac{2\varepsilon}{k}\left(1-\frac{1}{d+k+1}\right) r_{i,i+1}}\right)\right]\times\left[\prod_{i=1}^{k-1}\left(\sum_{r_{k,i+1}\geq 0}2^{-\frac{\varepsilon}{k^{2}}\left(1-\frac{(d-k+1)}{2(d+k+1)}\right)r_{k,i+1}}\right)\right]  \cr
&\displaystyle \times \prod_{i=1}^{k-1}\left[\sup_{l_{i+1},r_{i,i+1}}\left\|\mathbbm{1}_{\mathbb{X}^{l_{i+1};r_{i,i+1}}}\right\|_{\ell^{\infty}_{n_{i},m}\ell^{1}_{\widehat{n_{i}}}}^{-\frac{\varepsilon}{k}\left(1-\frac{1}{d+k+1}\right)}\cdot \sup_{l_{i+1},r_{k,i+1}}\left\|\mathbbm{1}_{\mathbb{X}^{l_{i+1};r_{k,i+1}}}\right\|_{\ell^{\infty}_{\overrightarrow{\textit{\textbf{n}}_{k}},m}\ell^{1}_{\widehat{\overrightarrow{n_{k}}}}}^{-\frac{\varepsilon}{k^{2}}\left(1-\frac{(d-k+1)}{2(d+k+1)}\right)}\right]\cr
&\displaystyle\times \|g_{1}\|_{2}^{\frac{1}{k}-\frac{4\varepsilon}{k(d-k+1)}+\frac{\varepsilon}{k}-\frac{\varepsilon}{k^{2}}-2\varepsilon+\frac{1}{2}\left(-\frac{\varepsilon}{k^{2}}+\frac{2(d+k+1)}{(d-k+1)}\frac{\varepsilon}{k^{2}}\right)}\cr
&\displaystyle\times \prod_{l=2}^{k}\|g_{l}\|^{\frac{1}{k}-\frac{2\varepsilon}{k}+\frac{\varepsilon}{k^{2}}\left(\frac{(d-k+1)}{2(d+k+1)}-1\right)}_{2} \cr
&\displaystyle\times |F|^{\left[1-\frac{(d+k-1)}{2(d+k+1)}\right]+\varepsilon}. \cr
}
\end{equation*}
Observe that

$$\sum_{l_{1}\geq 0}2^{-\frac{\varepsilon}{k^{2}}l_{1}}\lesssim_{\varepsilon} 2^{-\frac{\varepsilon}{k^{2}}\widetilde{l_{1}}}, $$
where $\widetilde{l_{1}}$ is the smallest index $l_{1}$ such that $\mathbb{X}^{\overrightarrow{l},R,t}\neq\emptyset$. Hence there exists some $(\overrightarrow{\textit{\textbf{k}}},\widetilde{m})$ such that

$$2^{-\widetilde{l_{1}}}\approx |\langle g_{1},\varphi_{\overrightarrow{\textit{\textbf{k}}},\widetilde{m}}^{1}\rangle|\leq |E_{1}|,$$
therefore

$$\sum_{l_{1}\geq 0}2^{-\frac{\varepsilon}{k^{2}}l_{1}}\lesssim_{\varepsilon} |E_{1}|^{\frac{\varepsilon}{k^{2}}}.$$
Notice also that

$$\sum_{r_{j,1}\geq 0}2^{-\left(\frac{2\varepsilon}{k}+\frac{\varepsilon}{k^{2}}\right)r_{j,1}}\lesssim_{\varepsilon} 2^{-\left(\frac{2\varepsilon}{k}+\frac{\varepsilon}{k^{2}}\right)\widetilde{r_{j,1}}}, $$
where $\widetilde{r_{j,1}}$ is defined analogously. We can then find $(n_{j},m)$ such that

$$2^{-r_{j,1}}\lesssim \left\|\langle g_{1},\varphi_{n_{j},m}^{j,1}\rangle_{x_{j}}\right\|_{2}\leq |E_{1}|^{\frac{1}{2}},$$
therefore
$$\sum_{r_{j,1}\geq 0}2^{-\left(\frac{2\varepsilon}{k}+\frac{\varepsilon}{k^{2}}\right)r_{j,1}}\lesssim_{\varepsilon} |E_{1}|^{\frac{\varepsilon}{k}+\frac{\varepsilon}{2k^{2}}}. $$
We can estimate all other sums in the bound above analogously. Observe that since the cardinalities appearing in

\begin{equation}\label{intmar822}
\prod_{i=1}^{k-1}\left[\sup_{l_{i+1},r_{i,i+1}}\left\|\mathbbm{1}_{\mathbb{X}^{l_{i+1};r_{i,i+1}}}\right\|_{\ell^{\infty}_{n_{i},m}\ell^{1}_{\widehat{n_{i}}}}^{-\frac{\varepsilon}{k}\left(1-\frac{1}{d+k+1}\right)}\cdot \sup_{l_{i+1},r_{k,i+1}}\left\|\mathbbm{1}_{\mathbb{X}^{l_{i+1};r_{k,i+1}}}\right\|_{\ell^{\infty}_{\overrightarrow{\textit{\textbf{n}}_{k}},m}\ell^{1}_{\widehat{\overrightarrow{n_{k}}}}}^{-\frac{\varepsilon}{k^{2}}\left(1-\frac{(d-k+1)}{2(d+k+1)}\right)}\right]
\end{equation}
are integers, the whole expression \eqref{intmar822} is $O(1)$. Using these observations and the fact that $|E_{j}|<1$ gives us
\begin{equation}\label{eq1may622}
|\tilde{\tilde{\Lambda}}_{k,d}(g,h)|\lesssim_{\varepsilon}|F|^{1-\frac{(d+k-1)}{2(d+k+1)}+\varepsilon}\cdot\prod_{j=1}^{k}|E_{j}|^{\frac{1}{2k}}.
\end{equation}
To simplify our notation, set $g:=(g_{1,1},g_{1,2},\ldots,g_{1,k-1},g_{1,k},g_{2},\ldots,g_{k})$. To rigorously use multilinear interpolation theory, one can run the argument above for the following averaged multilinearized version of $ME_{k,d}$:

$$\widetilde{ME}^{\frac{1}{k}}_{k,d}(g):=\sum_{(\overrightarrow{\textit{\textbf{n}}},m)\in\mathbb{Z}^{d+1}}\left(\prod_{l=1}^{k-1}|\langle g_{1,l},\varphi^{l,1}_{n_{l},m}\rangle|\right)^{\frac{1}{k}}\cdot|\langle g_{1,k},\varphi^{k,1}_{\overrightarrow{\textit{\textbf{n}}_{k}},m}\rangle|^{\frac{1}{k}}\cdot\left(\prod_{j=2}^{k}|\langle g_{j},\varphi^{j}_{\overrightarrow{\textit{\textbf{n}}},m}\rangle|\right)^{\frac{1}{k}}(\chi_{\overrightarrow{\textit{\textbf{n}}}}\otimes\chi_{m}),$$
with associated dual form\footnote{There is a slight difference between the forms $\tilde{\tilde{\Lambda}}_{k,d}$ and $\tilde{\Lambda}_{k,d}$: the latter is $2(k-1)$-linear, whereas the former is $k$-linear. We can not apply multilinear interpolation theory with inequality \eqref{eq1may622} directly, because all we proved is that it holds when $g_{1}$ is a tensor. In order to correctly place our estimates in the context of multilinear interpolation, we need to consider a form that has the appropriate level of multilinearity, which is $\tilde{\Lambda}_{k,d}$.}

$$\tilde{\Lambda}_{k,d}(g,h):=\sum_{(\overrightarrow{\textit{\textbf{n}}},m)\in\mathbb{Z}^{d+1}}\left(\prod_{l=1}^{k-1}|\langle g_{1,l},\varphi^{l,1}_{n_{l},m}\rangle|\right)^{\frac{1}{k}}\cdot|\langle g_{1,k},\varphi^{k,1}_{\overrightarrow{\textit{\textbf{n}}_{k}},m}\rangle|^{\frac{1}{k}}\left(\prod_{j=2}^{k}|\langle g_{j},\varphi^{j}_{\overrightarrow{\textit{\textbf{n}}},m}\rangle|\right)^{\frac{1}{k}}\langle h,\chi_{\overrightarrow{\textit{\textbf{n}}}}\otimes\chi_{m}\rangle.$$
Hence \eqref{eq1may622} gives us 
\begin{equation}\label{mtl1may622}
\|\widetilde{ME}_{k,d}^{\frac{1}{k}}(g)\|_{L^{\frac{2(d+k+1)}{(d+k-1)}+\varepsilon}(\mathbb{R}^{d+1})}\lesssim_{\varepsilon}\prod_{l=1}^{k}|E_{1,l}|^{\frac{1}{2k}}\cdot \prod_{j=2}^{k}|E_{j}|^{\frac{1}{2k}},
\end{equation}
which is \eqref{mtl1apr29} for $\widetilde{ME}_{k,d}$. Finally, observe that

\begin{equation}\label{conta1apr2922}
\eqalign{
\displaystyle\|\widetilde{ME}_{k,d}(g)\|_{L^{\frac{2(d+k+1)}{k(d+k-1)}+\varepsilon}(\mathbb{R}^{d+1})}&\displaystyle\leq\overbrace{\|\widetilde{ME}_{k,d}(g)^{\frac{1}{k}}\|_{L^{\frac{2(d+k+1)}{(d+k-1)}+k\varepsilon}(\mathbb{R}^{d+1})}\cdot\ldots\cdot \|\widetilde{ME}_{k,d}(g)^{\frac{1}{k}}\|_{L^{\frac{2(d+k+1)}{(d+k-1)}+k\varepsilon}(\mathbb{R}^{d+1})}}^{\textnormal{$k$ times}} \cr
&\displaystyle\lesssim \left[\prod_{l=1}^{k}|E_{1,l}|^{\frac{1}{2k}}\cdot \prod_{j=2}^{k}|E_{j}|^{\frac{1}{2k}}\right]^{k} \cr
&=\displaystyle \prod_{l=1}^{k}|E_{1,l}|^{\frac{1}{2}}\cdot \prod_{j=2}^{k}|E_{j}|^{\frac{1}{2}}, \cr
}
\end{equation}
which finishes the proof of the case $2\leq k\leq d+1$ by restricted weak-type interpolation.

\section{The endpoint estimate of the case $k=d+1$ of Theorem \ref{mainthmpaper}}\label{dlineartheory}

Let $g_{1}:Q_{1}\rightarrow\mathbb{R}$, $g_{j}:Q_{j}\rightarrow\mathbb{R}$ for $2\leq j\leq d+1$ be continuous functions. Recall that the multilinear model for $k=d+1$ is given in Section \ref{discret} by:

$$ME_{d+1,d}(g_{1},\ldots,g_{d+1}):=\sum_{(\overrightarrow{\textit{\textbf{n}}},m)\in\mathbb{Z}^{d+1}}\prod_{j=1}^{d+1}\langle g_{j},\varphi^{j}_{\overrightarrow{\textit{\textbf{n}}},m}\rangle(\chi_{\overrightarrow{\textit{\textbf{n}}}}\otimes\chi_{m}),$$
where

$$\varphi^{j}_{\overrightarrow{\textit{\textbf{n}}},m}=\bigotimes_{l=1}^{d}\varphi^{l,j}_{n_{l},m},\quad \varphi^{l,j}_{n_{l},m}(x_{l})=\varphi^{l,j}(x_{l})e^{2\pi in_{l}x_{l}}e^{2\pi imx^{2}_{l}} $$
and $\varphi^{l,j}(x)$ was defined in Section \ref{discret}. From now on, we will assume without loss of generality that $g_{1}$ is the full tensor. To simplify our notation, set $g:=(g_{1,1},\ldots,g_{1,d}, g_{2},\ldots,g_{d+1})$. Define

$$\widetilde{ME}_{d+1,d}(g):=\sum_{(\overrightarrow{\textit{\textbf{n}}},m)\in\mathbb{Z}^{d+1}}\prod_{l=1}^{d}\langle g_{1,l},\varphi^{l,1}_{n_{l},m}\rangle\prod_{j=2}^{d+1}\langle g_{j},\varphi^{j}_{\overrightarrow{\textit{\textbf{n}}},m}\rangle(\chi_{\overrightarrow{\textit{\textbf{n}}}}\otimes\chi_{m}).$$

We will show that $\widetilde{ME}_{d+1,d}$ maps 

$$\underbrace{L^{2}([0,1])\times\ldots\times L^{2}([0,1])\times L^{2}(Q_{2})\times\ldots\times L^{2}(Q_{d+1})}_{\textnormal{$2d$ times}}$$
to $L^{\frac{2}{d}}$, which implies the endpoint estimate of the case $k=d+1$ in Theorem \ref{mainthmpaper}.

\begin{proof}[Endpoint estimate of the case $k=d+1$] Notice that we have $d$ factors in the first product and $d$ factors in the second. We will pair them in the following way:

$$\widetilde{ME}_{d+1,d}(g):=\sum_{(\overrightarrow{\textit{\textbf{n}}},m)\in\mathbb{Z}^{d+1}}\prod_{j=2}^{d+1}\langle g_{j},\varphi^{j}_{\overrightarrow{\textit{\textbf{n}}},m}\rangle\cdot \langle g_{1,j-1},\varphi^{1,j-1}_{n_{j-1},m}\rangle(\chi_{\overrightarrow{\textit{\textbf{n}}}}\otimes\chi_{m}) $$

Now observe that
\begin{equation}\label{bound1-jan1523}
\eqalign{
\|\widetilde{ME}_{d+1,d}(g)\|_{\frac{2}{d}}^{\frac{2}{d}}&=\displaystyle\sum_{(\overrightarrow{\textit{\textbf{n}}},m)\in\mathbb{Z}^{d+1}}\prod_{j=2}^{d+1}|\langle g_{j}\otimes \overline{g_{1,j-1}},\varphi^{j}_{\overrightarrow{\textit{\textbf{n}}},m}\otimes\overline{\varphi^{j-1,1}_{n_{j-1},m}}\rangle|^{\frac{2}{d}} \cr
&\leq\displaystyle \prod_{j=2}^{d+1}\left(\sum_{(\overrightarrow{\textit{\textbf{n}}},m)\in\mathbb{Z}^{d+1}}|\langle g_{j}\otimes \overline{g_{1,j-1}},\varphi^{j}_{\overrightarrow{\textit{\textbf{n}}},m}\otimes\overline{\varphi^{j-1,1}_{n_{j-1},m}}\rangle|^{2}\right)^{\frac{1}{d}} \cr
}
\end{equation}
Let us analyze the $j=2$ scalar product inside the parentheses (the others are dealt with in a similar way):

\begin{equation*}
\eqalign{
\displaystyle \langle g_{j}\otimes \overline{g_{1,1}}&\displaystyle ,\varphi^{2}_{\overrightarrow{\textit{\textbf{n}}},m}\otimes\overline{\varphi^{1,1}_{n_{1},m}}\rangle \cr
&\displaystyle= \int_{\mathbb{R}^{d-1}}\langle g_{2,1}\otimes\overline{g_{1,1}},\varphi_{n_{1},m}^{1,2}\otimes\overline{\varphi^{1,1}_{n_{1},m}} \rangle\left(\prod_{u\geq 2}\varphi^{u,2}(x_{u})\right)e^{-2\pi im\left(\sum_{l\geq 2}x_{l}^{2}\right)}e^{-2\pi i\left(\sum_{l\geq 2}n_{l}x_{l}\right)}\widehat{\mathrm{d}x_{1}}\cr
&\displaystyle= \widehat{H_{n_{1},m}}(n_{2},\ldots,n_{d}), \cr
}
\end{equation*}
where
$$H_{n_{1},m}(x_{2},\ldots,x_{d}):= \langle g_{2,1}\otimes\overline{g_{1,1}},\varphi_{n_{1},m}^{1,2}\otimes\overline{\varphi^{1,1}_{n_{1},m}} \rangle\left(\prod_{u\geq 2}\varphi^{u,2}(x_{u})\right)e^{-2\pi im\left(\sum_{l\geq 2}x_{l}^{2}\right)}. $$
We can then use Plancherel if we sum over $n_{2},\ldots,n_{d}$ first:

\begin{equation*}
\eqalign{
\displaystyle\sum_{(\overrightarrow{\textit{\textbf{n}}},m)\in\mathbb{Z}^{d+1}}|\langle g_{j}\otimes \overline{g_{1,j-1}},&\displaystyle\varphi^{j}_{\overrightarrow{\textit{\textbf{n}}},m}\otimes\overline{\varphi^{j-1,1}_{n_{j-1},m}}\rangle|^{2} \cr
&\displaystyle = \sum_{n_{1},m}\sum_{n_{2},\ldots,n_{d}}\left|\widehat{H_{n_{1},m}}(n_{2},\ldots,n_{d})\right|^{2}\cr
&\displaystyle = \sum_{n_{1},m}\|H_{n_{1},m}\|_{2}^{2}\cr
&\displaystyle = \int_{\mathbb{R}^{d-1}}\left(\prod_{u\geq 2}\varphi^{u,2}(x_{u})\right)\left(\sum_{n_{1},m}|\langle g_{2}\otimes\overline{g_{1,1}},\varphi_{n_{1},m}^{1,2}\otimes\overline{\varphi^{1,1}_{n_{1},m}} \rangle|^{2}\right)\widehat{\mathrm{d}x_{1}} \cr
}
\end{equation*}
By our initial choice of cubes, $\textnormal{supp}(\varphi_{n_{1},m}^{1,1})\cap\textnormal{supp}(\varphi_{n_{1},m}^{1,2})=\emptyset$, so the sum in $(n_{1},m)$  is actually $M_{2,1}$ (we are freezing $d-1$ variables of $g_{2}$ in this sum). Our results from Section \ref{bilinearparabola} imply
\begin{equation*}
\sum_{(\overrightarrow{\textit{\textbf{n}}},m)\in\mathbb{Z}^{d+1}}|\langle g_{j}\otimes \overline{g_{1,j-1}},\varphi^{j}_{\overrightarrow{\textit{\textbf{n}}},m}\otimes\overline{\varphi^{j-1,1}_{n_{j-1},m}}\rangle|^{2} = \|g_{2}\otimes\overline{g_{1,1}}\|_{2}^{2}.
\end{equation*}
Arguing like that for all $2\leq j\leq d+1$, \eqref{bound1-jan1523} gives us
\begin{equation*}
\eqalign{
\|\widetilde{ME}_{d+1,d}(g)\|_{\frac{2}{d}}^{\frac{2}{d}}&\displaystyle\leq\prod_{j=2}^{d+1}\|g_{2}\otimes\overline{g_{1,j-1}}\|_{2}^{\frac{2}{d}} \cr
&\displaystyle =\prod_{j=1}^{d+1}\|g_{j}\|_{2}^{\frac{2}{d}}
}
\end{equation*}
and the result follows.
\end{proof}

\section{Improved $k$-linear bounds for tensors}\label{betteresttensors2}

In this section we investigate the following question: \textit{can one obtain better bounds than those of Conjecture \ref{klinear} if one is restricted to the class of tensors?}\footnote{Extension estimates beyond the conjectured range have been verified by Oliveira e Silva and Mandel in \cite{OM} for a certain class of functions when the underlying submanifold is $\mathbb{S}^{d-1}$. \cite{Shao} also contains results of this kind for the paraboloid.} The answer depends on the concept of \textit{degree of transversality}. The extra information that the input functions are supported on cubes that have disjoint projections along many directions leads to new transversality conditions, and we can take advantage of it in the full tensor case. This is the content of Theorem \ref{improvedklinearthm2}.

Let $\{e_{j}\}_{1\leq j\leq d}$ be the canonical basis of $\mathbb{R}^{d}$. If $Q\subset\mathbb{R}^{d}$ is a cube, $\pi_{j}(Q)$ represents the projection of $Q$ along the $e_{j}$ direction.

\begin{definition}\label{def1-3jan2022} Let $\{Q_{1},\ldots,Q_{k}\}$ be a collection of $k$ closed unit cubes in $\mathbb{R}^{d}$ with vertices in $\mathbb{Z}^{d}$. We associate to this collection its \textit{transversality vector}
$$\tau=(\tau_{1},\ldots,\tau_{d}), $$
where
$\tau_{j}=1$ if there are at least two distinct intervals among the projections $\pi_{j}(Q_{l}), 1\leq l\leq k$, and $\tau_{j}=0$ otherwise. The \textit{total degree of transversality} of the collection $\{Q_{1},\ldots,Q_{k}\}$ is
$$|\tau|:=\sum_{1\leq l \leq d}\tau_{l}. $$
\end{definition}

The $k$-linear extension model for a set of cubes $\{Q_{l}\}_{1\leq l\leq k}$ as in Definition \ref{def1-3jan2022} is initially given on $C(Q_{1})\times\ldots\times C(Q_{k})$ by

\begin{equation}\label{betterestMKD2}
ME_{k,d}^{Q_{1},\ldots,Q_{k}}(g_{1},\ldots,g_{k}):=\sum_{(\overrightarrow{\textit{\textbf{n}}},m)\in\mathbb{Z}^{d+1}}\prod_{j=1}^{k}\langle g_{j},\varphi^{j}_{\overrightarrow{\textit{\textbf{n}}},m}\rangle(\chi_{\overrightarrow{\textit{\textbf{n}}}}\otimes\chi_{m}).
\end{equation}
where the bumps $\varphi^{j}_{\overrightarrow{\textit{\textbf{n}}},m}$ are analogous to the ones in Section \ref{klineartheory}, but now adapted to the cubes $Q_{k}$.

From now on we will assume that $g_{j}$ is a full tensor $g_{j}^{1}\otimes\ldots\otimes g_{j}^{d}$ for $1\leq j\leq k$ and that the transversality vector of the collection $\{Q_{1},\ldots,Q_{k}\}$ is $\tau$. To simplify the notation, we will replace the superscripts $Q_{j}$ in \eqref{betterestMKD2} with $\tau$ and denote

$$g:=(g_{1}^{1},\ldots,g_{1}^{d},\ldots,g_{j}^{1},\ldots,g_{j}^{d},\ldots,g_{k}^{1},\ldots,g_{k}^{d}). $$
We are then led to consider
\begin{equation}\label{def1-31dec2021}
ME_{k,d}^{\tau}(g):=\sum_{(\overrightarrow{\textit{\textbf{n}}},m)\in\mathbb{Z}^{d+1}}\prod_{j=1}^{k}\prod_{l=1}^{d}\langle g_{j}^{l},\varphi^{l,j}_{n_{l},m}\rangle(\chi_{\overrightarrow{\textit{\textbf{n}}}}\otimes\chi_{m}),
\end{equation}
where
$$\varphi^{l,j}_{n_{l},m}(x)=\varphi^{l,j}(x)e^{2\pi in_{l}x}e^{2\pi imx^{2}}, \quad\textnormal{supp}(\varphi^{l,j})\subset\pi_{l}(Q_{j}).$$

As it was the case in Section \ref{klineartheory}, we will deal first with an averaged version of $ME_{k,d}^{\tau}$ for technical reasons. Define

\begin{equation}\label{definition1-1jan2022}
\widetilde{ME}_{k,d}^{\tau}(g):=\sum_{(\overrightarrow{\textit{\textbf{n}}},m)\in\mathbb{Z}^{d+1}}\prod_{j=1}^{k}\prod_{l=1}^{d}|\langle g_{j}^{l},\varphi^{l,j}_{n_{l},m}\rangle|^{\frac{1}{k}}(\chi_{\overrightarrow{\textit{\textbf{n}}}}\otimes\chi_{m}),
\end{equation}
and consider its dual form

$$\widetilde{\Lambda}^{\tau}_{k,d}(g,h):=\sum_{(\overrightarrow{\textit{\textbf{n}}},m)\in\mathbb{Z}^{d+1}}\prod_{j=1}^{k}\prod_{l=1}^{d}|\langle g_{j}^{l},\varphi^{l,j}_{n_{l},m}\rangle|^{\frac{1}{k}} \cdot\langle h,\chi_{\overrightarrow{\textit{\textbf{n}}}}\otimes\chi_{m}\rangle.$$

Let $E_{j,l}$, $1\leq j\leq k$ and $1\leq l\leq d$, be measurable sets such that $|g_{j}^{l}|\leq\chi_{E_{j,l}}$. Let $F\subset\mathbb{R}^{d+1}$ be a measurable set such that $|h|\leq \chi_{F}$. Under these conditions we have the following result:

\begin{theorem}\label{improvedklinearthm2} $ME^{\tau}_{k,d}$ satisfies
$$\|ME^{\tau}_{k,d}(g)\|_{L^{p}(\mathbb{R}^{d+1})}\lesssim_{p}\prod_{j=1}^{k}\prod_{l=1}^{d}\|g_{j}^{l}\|_{2} $$
for all $p>p_{\tau}:=\frac{2(d+|\tau|+2)}{k(d+|\tau|)}$.
\end{theorem}

\begin{proof} It is enough to prove that

$$\|\widetilde{ME}^{\tau}_{k,d}(g)\|_{L^{p}(\mathbb{R}^{d+1})}\lesssim_{p} \prod_{j=1}^{k}\prod_{l=1}^{d}|E_{j,l}|^{\frac{1}{2k}},$$
holds for every
$$p>\frac{2(d+|\tau|+2)}{(d+|\tau|)}. $$

Define the level sets

$$\mathbb{A}_{j,l}^{r_{j,l}}:=\{(n_{l},m)\in\mathbb{Z}^{2};\quad |\langle g_{j}^{l},\varphi^{l,j}_{n_{l},m}\rangle|\approx 2^{-r_{j,l}}\}, $$
$$\mathbb{B}^{t}:=\{(\overrightarrow{\textit{\textbf{n}}},m)\in\mathbb{Z}^{d+1};\quad |\langle h,\chi_{\overrightarrow{\textit{\textbf{n}}}}\otimes\chi_{m}\rangle|\approx 2^{-t}\}. $$

Set $R:=(r_{i,j})_{i,j}$ and

$$\mathbb{X}^{R,t}:=\left\{(\overrightarrow{\textit{\textbf{n}}},m)\in\mathbb{Z}^{d+1};\quad (n_{l},m)\in\bigcap_{j=1}^{k}\mathbb{A}_{j,l}^{r_{j,l}},\quad 1\leq l\leq d\right\}\cap\mathbb{B}^{t}. $$

We then have

$$|\widetilde{\Lambda}^{\tau}_{k,d}(g,h)|\lesssim \sum_{R,t\geq 0}2^{-t}\cdot\prod_{j=1}^{k}\prod_{l=1}^{d}2^{-\frac{r_{j,l}}{k}}\cdot\#\mathbb{X}^{R,t}. $$

As in the previous section, we can assume without loss of generality that $r_{j,l},t\geq 0$. We can estimate $\#\mathbb{X}^{R,t}$ using the function $h$:

\begin{equation}\label{boundH-30dec2021}
\#\mathbb{X}^{R,t}\lesssim 2^{t}\sum_{(\overrightarrow{\textit{\textbf{n}}},m)\in\mathbb{Z}^{d+1}}|\langle h,\chi_{\overrightarrow{\textit{\textbf{n}}}}\otimes\chi_{m}\rangle|\lesssim 2^{t}|F|.
\end{equation}

Alternatively, by the definition of $\mathbb{X}^{R,t}$:

\begin{equation}\label{chain1-31dec2021}
\#\mathbb{X}^{R,t}\displaystyle\leq\sum_{(\overrightarrow{\textit{\textbf{n}}},m)\in\mathbb{Z}^{d+1}}\prod_{j=1}^{k}\prod_{l=1}^{d}\mathbbm{1}_{\mathbb{A}^{r_{j,l}}_{j,l}}(n_{l},m)
\end{equation}

There are many ways to estimate the right-hand side above. We will obtain $d$ different bounds for it, each one arising from summing in a different order. Fix $1\leq l\leq d$ and leave the sum over $(n_{l},m)$ for last:

\begin{equation}\label{chain2-31dec2021}
\eqalign{
\#\mathbb{X}^{R,t}&\displaystyle=\sum_{(n_{l},m)\in\mathbb{Z}^{2}}\left[\prod_{j=1}^{k}\mathbbm{1}_{\mathbb{A}^{r_{j,l}}_{j,l}}(n_{l},m)\right]\cdot\prod_{\substack{\widetilde{l}=1 \\ \widetilde{l}\neq l}}^{d}\left[\sum_{n_{\widetilde{l}}}\prod_{\widetilde{j}=1}^{k}\mathbbm{1}_{\mathbb{A}^{r_{\widetilde{j},\widetilde{l}}}_{\widetilde{j},\widetilde{l}}}(n_{\widetilde{l}},m)\right]\cr
&\leq\displaystyle\sum_{(n_{l},m)\in\mathbb{Z}^{2}}\left[\prod_{j=1}^{k}\mathbbm{1}_{\mathbb{A}^{r_{j,l}}_{j,l}}(n_{l},m)\right]\cdot\prod_{\substack{\widetilde{l}=1 \\ \widetilde{l}\neq l}}^{d}\prod_{\widetilde{j}=1}^{k}\left[\sum_{n_{\widetilde{l}}}\mathbbm{1}_{\mathbb{A}^{r_{\widetilde{j},\widetilde{l}}}_{\widetilde{j},\widetilde{l}}}(n_{\widetilde{l}},m)\right]^{\gamma_{l,\widetilde{j},\widetilde{l}}},\cr
}
\end{equation}
where we used H\"older's inequality in the last line and $\gamma_{l,\widetilde{j},\widetilde{l}}$ are generic parameters such that
\begin{equation}\label{chain3-31dec2021}
\sum_{\widetilde{j}=1}^{k}\gamma_{l,\widetilde{j},\widetilde{l}}=1
\end{equation}
for all $1\leq l,\widetilde{l}\leq d$ with $l\neq \widetilde{l}$ fixed. Let us briefly explain the labels in these parameters that we just introduced:

\[ \gamma_{l,\tilde{j},\tilde{l}}\longrightarrow\begin{cases} 
      \textnormal{ $l$ indicates that the last variables to be summed are $(n_{l},m)$, } &  \\
     \textnormal{$\widetilde{j}$ corresponds to the $\widetilde{j}$-th function $g_{\widetilde{j}}$,} \\
      \textnormal{$\widetilde{l}\neq l$ corresponds to the $\widetilde{l}$-th variable $n_{\widetilde{l}}$.} & 
   \end{cases}
\]

We will not make any specific choice for the $\gamma_{l,\widetilde{j},\widetilde{l}}$ since condition \eqref{chain3-31dec2021} will suffice. Now observe that for a fixed $m\in\mathbb{Z}$ we have:

\begin{equation}\label{chain4-31dec2021}
\displaystyle \sum_{n_{\widetilde{l}}}\mathbbm{1}_{\mathbb{A}^{r_{\widetilde{j},\widetilde{l}}}_{\widetilde{j},\widetilde{l}}}(n_{\widetilde{l}},m) \leq 2^{2r_{\widetilde{j},\widetilde{l}}}\sum_{n_{\widetilde{l}}}|\langle g_{\widetilde{j}}^{\widetilde{l}},\varphi^{\widetilde{l},\widetilde{j}}_{n_{\widetilde{l}},m}\rangle|^{2} \leq 2^{2r_{\widetilde{j},\widetilde{l}}}\cdot |E_{\widetilde{j},\widetilde{l}}|
\end{equation}
by Bessel's inequality. Using \eqref{chain4-31dec2021} back in \eqref{chain2-31dec2021}:

\begin{equation}\label{chain5-31dec2021}
\eqalign{
\#\mathbb{X}^{R,t}&\displaystyle\leq\prod_{\substack{\widetilde{l}=1 \\ \widetilde{l}\neq l}}^{d}\prod_{\widetilde{j}=1}^{k}2^{2\gamma_{l,\widetilde{j},\widetilde{l}}r_{\widetilde{j},\widetilde{l}}}\cdot |E_{\widetilde{j},\widetilde{l}}|^{\gamma_{l,\widetilde{j},\widetilde{l}}}\cdot\displaystyle\sum_{(n_{l},m)\in\mathbb{Z}^{2}}\left[\prod_{j=1}^{k}\mathbbm{1}_{\mathbb{A}^{r_{j,l}}_{j,l}}(n_{l},m)\right],\cr
&\displaystyle= \prod_{\substack{\widetilde{l}=1 \\ \widetilde{l}\neq l}}^{d}\prod_{\widetilde{j}=1}^{k}2^{2\gamma_{l,\widetilde{j},\widetilde{l}}r_{\widetilde{j},\widetilde{l}}}\cdot |E_{\widetilde{j},\widetilde{l}}|^{\gamma_{l,\widetilde{j},\widetilde{l}}}\cdot\sum_{(n_{l},m)\in\mathbb{Z}^{2}}\left[\prod_{\substack{(j_{1},j_{2}) \\ j_{1}\neq j_{2}}}\mathbbm{1}_{\mathbb{A}^{r_{j_{1},l}}_{j_{1},l}}(n_{l},m)\cdot \mathbbm{1}_{\mathbb{A}^{r_{j_{2},l}}_{j_{2},l}}(n_{l},m)\right].\cr
}
\end{equation}

We simply used the fact that $\mathbbm{1}^{2}=\mathbbm{1}$ in the last line above. Our goal is to pair the scalar products in \eqref{def1-31dec2021} corresponding to the functions $g_{j_{1}}^{l}$ and $g_{j_{2}}^{l}$. There are two kinds of such pairs:

\begin{enumerate}[(a)]
\item A pair $(j_{1},j_{2})$ with $j_{1}\neq j_{2}$ is \textit{$l$-transversal} if $\textnormal{supp}(\varphi^{l,j_{1}})\cap\textnormal{supp}(\varphi^{l,j_{2}})=\emptyset$.
\item A pair $(j_{1},j_{2})$ with $j_{1}\neq j_{2}$ is \textit{non-$l$-transversal} along the direction $e_{l}$ if $\textnormal{supp}(\varphi^{l,j_{1}})\cap\textnormal{supp}(\varphi^{l,j_{2}})\neq\emptyset$.
\end{enumerate}

Thus we have by H\"older's inequality for generic parameters $\alpha_{l,j_{1},j_{2}}$ and $\beta_{l,j_{1},j_{2}}$:
\begin{equation}\label{chain6-31dec2021}
\eqalign{
\#\mathbb{X}^{R,t}&\leq\displaystyle \prod_{\substack{\widetilde{l}=1 \\ \widetilde{l}\neq l}}^{d}\prod_{\widetilde{j}=1}^{k}2^{2\gamma_{l,\widetilde{j},\widetilde{l}}r_{\widetilde{j},\widetilde{l}}}\cdot |E_{\widetilde{j},\widetilde{l}}|^{\gamma_{l,\widetilde{j},\widetilde{l}}} \cdot \prod_{\substack{(j_{1},j_{2}) \\ \textnormal{$l$-transversal, } j_{1}\neq j_{2}}}\left(\sum_{(n_{l},m)\in\mathbb{Z}^{2}}\mathbbm{1}_{\mathbb{A}^{r_{j_{1},l}}_{j_{1},l}}(n_{l},m)\cdot \mathbbm{1}_{\mathbb{A}^{r_{j_{2},l}}_{j_{2},l}}(n_{l},m)\right)^{\alpha_{l,j_{1},j_{2}}} \cr
&\qquad\qquad\qquad\qquad\quad\qquad\displaystyle \times \prod_{\substack{(j_{1},j_{2}) \\ \textnormal{non-$l$-transversal, } j_{1}\neq j_{2}}}\left(\sum_{(n_{l},m)\in\mathbb{Z}^{2}}\mathbbm{1}_{\mathbb{A}^{r_{j_{1},l}}_{j_{1},l}}(n_{l},m)\cdot \mathbbm{1}_{\mathbb{A}^{r_{j_{2},l}}_{j_{2},l}}(n_{l},m)\right)^{\beta_{l,j_{1},j_{2}}}
}
\end{equation}

Define
\begin{equation*}
\eqalign{
\alpha_{l,j_{1},j_{2}} &= 0,\quad\textnormal{if $(j_{1},j_{2})$ is non-$l$-transversal,}\cr
\beta_{l,j_{1},j_{2}} &= 0,\quad\textnormal{if $(j_{1},j_{2})$ is $l$-transversal.}\cr
}
\end{equation*}

Hence H\"older's condition is
\begin{equation}\label{chain9-31dec2021}
\sum_{\substack{(j_{1},j_{2}) \\ 1\leq j_{1},j_{2}\leq k \\ j_{1}\neq j_{2}}}\alpha_{l,j_{1},j_{2}}+\beta_{l,j_{1},j_{2}}=2,
\end{equation}
since we are counting each $\alpha_{l,j_{1},j_{2}}$ and $\beta_{l,j_{1},j_{2}}$ twice, for all $1\leq l\leq d$. The labels in the parameters $\alpha$ and $\beta$ track the following information:

\[ \textnormal{$\alpha_{l,j_{1},j_{2}}$ and $\beta_{l,j_{1},j_{2}}$}\longrightarrow\begin{cases} 
      \textnormal{ $l$ indicates that we are summing over $(n_{l},m)$, } &  \\
     \textnormal{$j_{1}$ and $j_{2}$ correspond to two distinct functions $g_{j_{1}}$ and $g_{j_{2}}$.} \\
   \end{cases}
\]

We can then use Proposition \ref{prop3apr2922} for the transversal pairs and a combination of one-dimensional Strichartz/Tomas-Stein with H\"older for the non-transversal ones: 

\begin{equation}\label{chain7-31dec2021}
\eqalign{
\#\mathbb{X}^{R,t}&\leq\displaystyle \prod_{\substack{\widetilde{l}=1 \\ \widetilde{l}\neq l}}^{d}\prod_{\widetilde{j}=1}^{k}2^{2\gamma_{l,\widetilde{j},\widetilde{l}}r_{\widetilde{j},\widetilde{l}}}\cdot |E_{\widetilde{j},\widetilde{l}}|^{\gamma_{l,\widetilde{j},\widetilde{l}}} \cdot \prod_{\substack{(j_{1},j_{2}) \\ \textnormal{$l$-transversal, } j_{1}\neq j_{2}}}\left(2^{2\alpha_{l,j_{1},j_{2}}(r_{j_{1},l}+r_{j_{2},l})}\cdot |E_{j_{1},l}|^{\alpha_{l,j_{1},j_{2}}}\cdot|E_{j_{2},l}|^{\alpha_{l,j_{1},j_{2}}}\right) \cr
&\qquad\qquad\qquad\qquad\quad\qquad\displaystyle \times \prod_{\substack{(j_{1},j_{2}) \\ \textnormal{non-$l$-transversal, } j_{1}\neq j_{2}}}\left(2^{3\beta_{l,j_{1},j_{2}}(r_{j_{1},l}+r_{j_{2},l})}\cdot |E_{j_{1},l}|^{\frac{3}{2}\beta_{l,j_{1},j_{2}}}\cdot|E_{j_{2},l}|^{\frac{3}{2}\beta_{l,j_{1},j_{2}}}\right).
}
\end{equation}

As mentioned earlier in this section, we have $d$ estimates like \eqref{chain7-31dec2021}. We will interpolate between them with weights $\theta_{l}$:

$$\#\mathbb{X}^{R,t}=\prod_{l=1}^{d}(\#\mathbb{X}^{R,t})^{\theta_{l}} $$
with
\begin{equation}\label{chain8-31dec2021}
\sum_{l=1}^{d}\theta_{l}=1.
\end{equation}

This yields
\begin{equation}
\#\mathbb{X}^{R,t}\lesssim \prod_{j=1}^{k}\prod_{l=1}^{d}2^{\#_{j,l}\cdot r_{j,l}}\cdot |E_{j,l}|^{\frac{\#_{j,l}}{2}},
\end{equation}
where
\begin{equation*}
\#_{j,l}=\left[\sum_{j_{1}\neq j}(2\alpha_{l,j,j_{1}}+3\beta_{l,j,j_{1}})\right]\cdot\theta_{l} + \sum_{\widetilde{l}\neq l}2\gamma_{\widetilde{l},j,l}\cdot\theta_{\widetilde{l}}.
\end{equation*}

In order to prove an estimate like $L^{2}\times\ldots\times L^{2}\mapsto L^{p}$, we will need all these coefficients $\#_{j,l}$ to be equal. Let us call them all $X$ for now and sum over $j$:

\begin{equation*}
\sum_{j=1}^{k}X=\left[\sum_{j=1}^{k}\sum_{j_{1}\neq j}(2\alpha_{l,j,j_{1}}+3\beta_{l,j,j_{1}})\right]\cdot\theta_{l}+\sum_{\widetilde{l}\neq l}2\left[\sum_{j=1}^{k}\gamma_{\widetilde{l},j,l}\right]\cdot\theta_{\widetilde{l}}
\end{equation*}

By \eqref{chain3-31dec2021} and \eqref{chain9-31dec2021}:

\begin{equation}\label{chain1-1jan2022}
X=\frac{1}{k}\left[6-\sum_{j=1}^{k}\sum_{j_{1}\neq j}\alpha_{l,j,j_{1}}\right]\cdot\theta_{l}+\sum_{\widetilde{l}\neq l}\frac{2}{k}\cdot\theta_{\widetilde{l}}
\end{equation}
for all $1\leq l\leq d$. Together with \eqref{chain8-31dec2021}, \eqref{chain1-1jan2022} gives us a linear system of $d$ equations in the $d$ variables $\theta_{1},\ldots,\theta_{d}$. The solution is

\begin{equation}\label{chain2-1jan2022}
\theta_{l}=\displaystyle\left[\sum_{\widetilde{l}=1}^{d}\frac{4-\sum_{j=1}^{k}\sum_{j_{1}\neq j}\alpha_{l,j,j_{1}}}{4-\sum_{j=1}^{k}\sum_{j_{1}\neq j}\alpha_{\widetilde{l},j,j_{1}}}\right]^{-1}.
\end{equation}

Plugging \eqref{chain2-1jan2022} back in \eqref{chain1-1jan2022} gives us

\begin{equation}\label{chain3-1jan2022}
X=\frac{2}{k}\left[1+\left(\sum_{\widetilde{l}=1}\displaystyle\frac{1}{\left[4-\sum_{j=1}^{k}\sum_{j_{1}\neq j}\alpha_{\widetilde{l},j,j_{1}}\right]}\right)^{-1}\right].
\end{equation}

To minimize $X$ we must maximize

$$\sum_{j=1}^{k}\sum_{j_{1}\neq j}\alpha_{\widetilde{l},j,j_{1}}. $$

This is achieved by choosing $\beta_{l,j_{1},j_{2}}=0$ for all $(j_{1},j_{2})$ if there is at least one $l$-transversal pair $(j_{1},j_{2})$. In other words, choose

$$\beta_{l,j_{1},j_{2}}=0\qquad\textnormal{for all $(j_{1},j_{2})$ if $\tau_{l}=1$.} $$

Hence by \eqref{chain9-31dec2021}:

\[\sum_{j=1}^{k}\sum_{j_{1}\neq j}\alpha_{\widetilde{l},j,j_{1}}= \begin{cases} 
     2 &\textnormal{ if $\tau_{\widetilde{l}}=1$,}  \\
     0& \textnormal{ if $\tau_{\widetilde{l}}=0$.} \\
   \end{cases}
\]

This choice of parameters gives us

$$X=\frac{2(d+|\tau|+2)}{k(d+|\tau|)},$$
which implies the following estimate for $\#\mathbb{X}^{R,t}$:

\begin{equation}\label{chain4-1jan2022}
\#\mathbb{X}^{R,t}\lesssim \prod_{j=1}^{k}\prod_{l=1}^{d}2^{X\cdot r_{j,l}}\cdot |E_{j,l}|^{\frac{X}{2}},
\end{equation}

Finally, we interpolate between \eqref{chain4-1jan2022} with weight $\frac{1}{k\cdot X}-\varepsilon$ and \eqref{boundH-30dec2021} with weight $(1-\frac{1}{k\cdot X})+\varepsilon$ to bound the form $\Lambda^{\tau}_{k,d}$:

\begin{equation*}
\eqalign{
|\Lambda^{\tau}_{k,d}(g,h)|&\displaystyle \lesssim \sum_{R,t\geq 0}2^{-t}\cdot\prod_{j=1}^{k}\prod_{l=1}^{d}2^{-\frac{r_{j,l}}{k}} \cr
&\times\displaystyle  \left[\prod_{j=1}^{k}\prod_{l=1}^{d}2^{X\cdot r_{j,l}}\cdot |E_{j,l}|^{\frac{X}{2}}\right]^{\frac{1}{k\cdot X}-\varepsilon}\cdot \left[2^{t}|F|\right]^{\left(1-\frac{1}{k\cdot X}\right)+\varepsilon} \cr
}
\end{equation*}
Developing the right-hand side:

\begin{equation*}
\eqalign{
|\Lambda^{\tau}_{k,d}(g,h)|&\displaystyle \lesssim \left(\sum_{t\geq 0}2^{-\left(\frac{1}{k\cdot X}-\varepsilon\right)t}\right)\prod_{j=1}^{k}\prod_{l=1}^{d}\left(\sum_{r_{j,l}\geq 0}2^{-\varepsilon X\cdot r_{j,l}}\right) \cr
&\times\displaystyle  \left[\prod_{j=1}^{k}\prod_{l=1}^{d} |E_{j,l}|^{\frac{1}{2k}-\frac{\varepsilon X}{2}}\right]\cdot |F|^{\left(1-\frac{1}{k\cdot X}\right)+\varepsilon} \cr
}
\end{equation*}
As in the previous section, these series are summable. We have

$$\sum_{r_{j,l}\geq 0}2^{-\varepsilon X\cdot r_{j,l}}\lesssim_{\varepsilon} |E_{j,l}|^{\varepsilon X}. $$
For the series in $t$ we can just bound it by an absolute constant depending on $\varepsilon$. This leads to

\begin{equation*}
\eqalign{
|\Lambda^{\tau}_{k,d}(g,h)|&\displaystyle\lesssim_{\varepsilon} \left[\prod_{j=1}^{k}\prod_{l=1}^{d} |E_{j,l}|^{\frac{1}{2k}+\frac{\varepsilon X}{2}}\right]\cdot |F|^{\left(1-\frac{1}{k\cdot X}\right)+\varepsilon} \cr
&\displaystyle\lesssim \left[\prod_{j=1}^{k}\prod_{l=1}^{d} |E_{j,l}|^{\frac{1}{2k}}\right]\cdot |F|^{\left(1-\frac{1}{k\cdot X}\right)+\varepsilon}, \cr
}
\end{equation*}
since $|E_{j,l}|\leq 1$, which finishes the proof by multilinear interpolation.

\end{proof}

\begin{remark} If $\tau_{l}=0$ for $1\leq l\leq d$, then

$$p_{\tau}=\frac{2(d+2)}{kd}, $$
which could have been proven in general with H\"older and Strichartz/Tomas-Stein. This is because there is no transversality to exploit, therefore the best bounds we can hope for in the multilinear setting come from the linear one.
\end{remark}

\begin{remark} If there are exactly $k-1$ indices $l$ such that $\tau_{l}=1$, then

$$p_{\tau}=\frac{2(d+k+1)}{k(d+k-1)}, $$
which is consistent with Theorem \ref{mainthmpaper}.
\end{remark}

\begin{remark} Finally, if one has more than $k-1$ indices $l$ such that $\tau_{l}=1$, then
$$p_{\tau}<\frac{2(d+k+1)}{k(d+k-1)}, $$
which clearly illustrates the point of this section. The extreme case is when $\tau_{l}=1$ for $1\leq l\leq d$, which gives
$$p_{\tau}=\frac{2(d+1)}{kd}.$$
This can be seen as an improvement upon the linear extension conjecture itself in the following sense: if we take the product of $k$ extensions $E_{U_{j}}g_{j}$, $1\leq j\leq k$, and combine the linear extension conjecture with H\"older's inequality, we obtain an operator that maps $L^{\frac{2(d+1)}{d}}\times\ldots\times L^{\frac{2(d+1)}{d}}$ to $L^{\frac{2(d+1)}{kd}+\varepsilon}$. On the other hand, if we are in a situation in which we have as much transversality as possible and all $g_{j}$ are full tensors, we obtain $L^{2}\times\ldots\times L^{2}$ to $L^{\frac{2(d+1)}{kd}+\varepsilon}$.

\end{remark}

\section{Beyond the $L^{2}$-based $k$-linear theory with and without transversality}\label{beyondL2}

Given a collection $\mathcal{Q}=\{Q_{1},\ldots,Q_{k}\}$ of cubes, the purpose of this section is to investigate \textit{near-restriction} $k$-linear estimates associated to $\mathcal{Q}$. In other words, we study bounds of the form

\begin{equation}
\left\|\prod_{j=1}^{k}\mathcal{E}_{Q_{j}}g_{j}\right\|_{L^{\frac{2(d+1)}{kd}+\varepsilon}(\mathbb{R}^{d+1})} \lesssim_{\varepsilon} \prod_{j=1}^{k} \|g_{j}\|_{L^{p}(Q_{j})}
\end{equation}
for all $\varepsilon>0$ and for some $p>1$. There are two cases of interest here:

\begin{itemize}
\item $\mathcal{Q}$ is a collection of transversal cubes.
\item All cubes in $Q$ are the same.
\end{itemize}

It will be clear that all cases in between these two can be studied in the same framework that we now present.

\subsection{Near-restriction estimates with transversality} We start by restating \eqref{ineq1-100123}. For $2\leq k<d+1$, to recover the whole range of the generalized $k$-linear extension conjecture, it is enough to prove Conjecture \ref{klinear} and
\begin{equation}\label{bklinear-jan2423}
\left\|\prod_{j=1}^{k}\mathcal{E}_{U_{j}}g_{j}\right\|_{L^{\frac{2(d+1)}{kd}+\varepsilon}(\mathbb{R}^{d+1})} \lesssim_{\varepsilon} \prod_{j=1}^{k} \|g_{j}\|_{L^{\frac{2(d+1)}{d}}(U_{j})}
\end{equation}
for all $\varepsilon>0$.

Let $\mathcal{Q}=\{Q_{1},\ldots,Q_{k}\}$ be our initially fixed set of cubes\footnote{See Section \ref{twt}.}. In what follows, we recast the statement of Theorem \ref{thm1-240123} in terms of this set:

\begin{theorem}\label{thm1-100123} If $\mathcal{Q}$ is a collection of transversal cubes and $g_{1}$ is a tensor, the operator $\mathcal{ME}_{k,d}(g_{1},\ldots,g_{k})$ satisfies
\begin{equation}\label{ineq0-jan1423}
\|\mathcal{ME}_{k,d}(g_{1},\ldots,g_{k})\|_{L^{\frac{2(d+1)}{kd}+\varepsilon}(\mathbb{R}^{d+1})}\lesssim_{\varepsilon}\prod_{j=1}^{k}\|g_{j}\|_{L^{p(k,d)}(Q_{j})}, 
\end{equation}
where
$$
p(k,d)=
\begin{cases}
\frac{4(d+1)}{d+k+1},\quad\textnormal{if }2\leq k<\frac{d}{2},\\
\frac{4(d+1)}{2d-k+1},\quad\textnormal{if }\frac{d}{2}\leq k<d+1.
\end{cases}
$$
\end{theorem}

As anticipated in the introduction, we prove it by adapting the argument from Section \ref{klineartheory}.

\begin{remark} As in Section \ref{klineartheory}, the theorem above holds under the assumption that the given set of cubes is weakly transversal and any other $g_{j}$, $j\neq 1$, can be assumed to be the tensor.
\end{remark}

\begin{remark} Roughly speaking, the difference between the proof of Theorem \ref{thm1-100123} and the one done in Section \ref{klineartheory} is in the building blocks we use: instead of Strichartz/Tomas-Stein (in the form of Corollary \ref{cor1apr2922}), we will use the best extension bound for the parabola (in the form of Proposition \ref{prop2apr2922}). One can think of the argument in this section as a rigorous way of replacing the former piece by the latter in our machinery.

\end{remark}

\begin{proof}[Proof of Theorem \ref{thm1-100123}] We work in the same setting as in Section \ref{klineartheory}. Even though there are some slight differences between the level sets from that section and the ones that we will define here, the approach is very similar.

It is convenient to recall a few important points from Section \ref{klineartheory}:

\begin{itemize}
\item The form of interest here is (in its averaged form):
\begin{equation}\label{form1-100123}
\widetilde{\widetilde{\Lambda}}_{k,d}(g,h):=\sum_{(\overrightarrow{\textit{\textbf{n}}},m)\in\mathbb{Z}^{d+1}}\left(\prod_{i=1}^{k}|\langle g_{j},\varphi^{j}_{\overrightarrow{\textit{\textbf{n}}},m}\rangle|\right)^{\frac{1}{k}}\langle h,\chi_{\overrightarrow{\textit{\textbf{n}}}}\otimes\chi_{m}\rangle.
\end{equation}
\item The tensor $g_{1}$ has the structure $g_{1}=g_{1,1}\otimes\ldots\otimes g_{1,d}$.
\item $E_{1,1},\ldots,E_{1,d}\subset [0,1]$, $E_{j}\subset Q_{j}$ ($2\leq j\leq k$) and $F\subset\mathbb{R}^{d+1}$ are measurable sets such that $|g_{1,l}|\leq\chi_{E_{1,l}}$ for $1\leq l\leq d$, $|g_{j}|\leq\chi_{E_{j}}$ for $2\leq j\leq k$ and $|h|\leq\chi_{F}$. Furthermore, $E_{1}:= E_{1,1}\times\ldots\times E_{1,d}$.
\end{itemize}

We start by encoding the sizes of the scalar products appearing in \eqref{form1-100123}:

$$\mathbb{A}_{j}^{l_{j}}:=\{(\overrightarrow{\textit{\textbf{n}}},m)\in\mathbb{Z}^{d+1};\quad |\langle g_{j},\varphi_{\overrightarrow{\textit{\textbf{n}}},m}\rangle|\approx 2^{-l_{j}}\},\quad 1\leq j\leq k. $$

Now we see the first difference between the argument in this section and the one in Section \ref{klineartheory}: the mixed-norm quantities here are all of the same kind, in the sense that the inner products inside the $L^{2}$ norms are all one-dimensional:
$$\mathbb{B}_{i,j}^{r_{i,j}}:=\left\{(n_{i},m)\in\mathbb{Z}^{2};\quad \left\|\langle g_{j},\varphi_{n_{i},m}^{i,j}\rangle_{x_{i}}\right\|_{2}\approx 2^{-r_{i,j}}\right\},\quad 1\leq i\leq d, \quad 1\leq j\leq k, $$

The remaining sets are defined just as in Section \ref{klineartheory}, and with the exact same purpose:
$$\mathbb{C}^{t}:=\{(\overrightarrow{\textit{\textbf{n}}},m)\in\mathbb{Z}^{d+1};\quad |\langle h,\chi_{\overrightarrow{\textit{\textbf{n}}}}\otimes\chi_{m}\rangle|\approx 2^{-t}\}, $$
$$\mathbb{X}^{l_{j};r_{i,j}} = \mathbb{A}_{j}^{l_{j}}\cap \left\{(\overrightarrow{\textit{\textbf{n}}},m)\in\mathbb{Z}^{d+1};\quad (n_{i},m)\in \mathbb{B}_{i,j}^{r_{i,j}}\right\},$$
$$\mathbb{X}^{\overrightarrow{l},R,t}:=\bigcap_{1\leq j\leq k}\mathbb{A}_{j}^{l_{j}}\cap \left\{(\overrightarrow{\textit{\textbf{n}}},m)\in\mathbb{Z}^{d+1};\quad (n_{i},m)\in\bigcap_{1\leq j\leq k}\mathbb{B}_{i,j}^{r_{i,j}},\quad 1\leq i\leq d\right\}\cap\mathbb{C}^{t}, $$
where we are using the abbreviations $\overrightarrow{l}=(l_{1},\ldots,l_{k})$ and $R:=(r_{i,j})_{i,j}$. Hence,

$$|\tilde{\tilde{\Lambda}}_{k,d}(g,h)|\lesssim \sum_{\overrightarrow{l},R,t}2^{-t}\prod_{j=1}^{k}2^{-\frac{l_{j}}{k}}\#\mathbb{X}^{\overrightarrow{l},R,t}. $$

The analogue of Lemma \ref{lemma1-3122020} is the bound

\begin{equation}\label{analog1-jan1123}
2^{-l_{1}}\approx \frac{2^{-r_{1,1}}\cdot\ldots\cdots 2^{-r_{d,1}}}{\|g_{1}\|_{2}^{d-1}},
\end{equation}
which is proven in the same way. Also, by an argument entirely analogous to the one of Lemma \ref{lemma1-231221}, we can show that

\begin{equation}\label{relation1-100123}
2^{-l_{j}}\lesssim \frac{2^{-r_{i,j}}}{\left\|\mathbbm{1}_{\mathbb{X}^{l_{j};r_{i,j}}}\right\|_{\ell^{\infty}_{n_{i},m}\ell^{1}_{\widehat{n_{i}}}}^{\frac{1}{2}}},\quad\forall \quad 1\leq i\leq d,\quad 2\leq j\leq k.
\end{equation}

The following corollary of the estimates above will give us the appropriate convex combination of such relations\footnote{Notice that instead of using just two mixed quantities for each scalar one (as in Corollary \ref{cor1-3122020}), we are using $d-k+2$ many of them here.}:

\begin{corollary}\label{cor1-100123} For $1\leq i\leq k-1$ we have
$$2^{-l_{i+1}}\lesssim \frac{2^{-\frac{k}{d+1}\cdot r_{i,i+1}}}{\left\|\mathbbm{1}_{\mathbb{X}^{l_{i+1};r_{i,i+1}}}\right\|_{\ell^{\infty}_{n_{i},m}\ell^{1}_{\widehat{n_{i}}}}^{\frac{1}{2(d+1)}}}\cdot\prod_{u=k}^{d}\frac{2^{-\frac{1}{d+1}\cdot r_{u,i+1}}}{\left\|\mathbbm{1}_{\mathbb{X}^{l_{i+1};r_{u,i+1}}}\right\|_{\ell^{\infty}_{n_{u},m}\ell^{1}_{\widehat{n_{u}}}}^{\frac{1}{2k(d+1)}}} .$$
\end{corollary}
\begin{proof} Interpolate between the bounds in \eqref{relation1-100123} with one weight equal to $\frac{k}{d+1}$ for $(i,j):=(i,i+1)$ and $(d-k+1)$ weights $\frac{1}{d+1}$ for $(i,j):=(u,i+1)$, $k\leq u\leq d$.
\end{proof}

We can estimate $\#\mathbb{X}^{\overrightarrow{l},R,t}$ using the function $h$:

\begin{equation}\label{ineq2-jan1023}
\#\mathbb{X}^{\overrightarrow{l},R,t}\lesssim\sum_{(\overrightarrow{\textit{\textbf{n}}},m)\in\mathbb{Z}^{d+1}}|\langle h,\chi_{\overrightarrow{\textit{\textbf{n}}}}\otimes\chi_{m}\rangle|\lesssim 2^{t}|F|.
\end{equation}

Alternatively,

\begin{equation}\label{ineq3-jan1023}
\eqalign{
\#\mathbb{X}^{\overrightarrow{l},R,t}&\displaystyle\leq\sum_{(\overrightarrow{\textit{\textbf{n}}},m)\in\mathbb{Z}^{d+1}}\prod_{j=1}^{k}\mathbbm{1}_{\mathbb{A}^{l_{j}}_{j}}(\overrightarrow{\textit{\textbf{n}}},m)\prod_{i=1}^{d}\prod_{j=1}^{k}\mathbbm{1}_{\mathbb{B}_{i,j}^{r_{i,j}}}(n_{i},m).\cr
}
\end{equation}

Similarly to what was done in Section \ref{klineartheory}, we will manipulate the inequality above in $d$ ways: $k-1$ of them will exploit orthogonality (from the combination of the sets $\mathbb{B}_{i,1}^{r_{i,1}}$ and $\mathbb{B}_{i,i+1}^{r_{i,i+1}}$, $1\leq i\leq k-1$), but now the other $d-k+1$ ones will reflect the linear extension problem in dimension $1$. The following lemma is the appropriate analogue of Lemma \ref{lemma1-21aug2021} in this section:

\begin{lemma}\label{lemma1-jan1123} The bounds above imply
\begin{enumerate}[(a)]
\item The \textit{orthogonality-type bounds}: for all $1\leq i\leq k-1$,
\begin{equation}\label{ineq4-jan1023}
\#\mathbb{X}^{\overrightarrow{l},R,t}\leq \left\|\mathbbm{1}_{\mathbb{X}^{l_{i+1};r_{i,i+1}}}\right\|_{\ell^{\infty}_{n_{i},m}\ell^{1}_{\widehat{n_{i}}}}\cdot 2^{2r_{i,1}+2r_{i,i+1}}\cdot \|g_{1}\|_{2}^{2}\cdot \|g_{i+1}\|_{2}^{2}.
\end{equation}
\item The \textit{extension-type bounds}: for all $k\leq u\leq d$,
\begin{equation}\label{ineq5-jan1023}
\eqalign{
\displaystyle\#\mathbb{X}^{\overrightarrow{l},R,t} &\displaystyle \leq \prod_{j=2}^{k}\left\|\mathbbm{1}_{\mathbb{X}^{l_{j};r_{u,j}}}\right\|_{\ell^{\infty}_{n_{u},m}\ell^{1}_{\widehat{n_{u}}}}^{\frac{1}{k}}\cdot 2^{\frac{2}{k}\sum_{i\neq u}r_{i,1}}\cdot \|g_{1}\|_{2}^{\frac{2(d-1)}{k}} \cr
&\displaystyle \times 2^{\alpha\cdot r_{u,1}+\sum_{l=2}^{k}\beta\cdot r_{u,l}}\cdot \left(\prod_{j\neq u}\|g_{1,j}\|_{2}\right)^{\alpha} \cdot\|g_{1,1}\|_{4}^{\alpha}\cdot \prod_{l=2}^{k}\|g_{l}\|_{4}^{\beta},
}
\end{equation}
where
$$\alpha := \frac{2(k+1)}{k}+\delta\cdot\frac{(k+1)}{2k}, \quad \beta := \frac{2}{k}+\tilde{\delta}\cdot\frac{1}{2k}, $$
with $\delta,\tilde{\delta}>0$ being arbitrarily small parameters to be chosen later.
\end{enumerate}

\end{lemma}
\begin{proof} Part \textit{(a)} is the same as in Lemma \ref{lemma1-21aug2021} \textit{(a)}. As for \textit{(b)}, fix $k\leq u\leq d$ and bound $\#\mathbb{X}^{\overrightarrow{l},R,t}$ as follows:
\begin{equation}\label{bound1-jan1123}
\eqalign{
\#\mathbb{X}^{\overrightarrow{l},R,t}&\displaystyle=\sum_{(\overrightarrow{\textit{\textbf{n}}},m)\in\mathbb{Z}^{d+1}}\mathbbm{1}_{\mathbb{X}^{\overrightarrow{l},R,t}}(\overrightarrow{\textit{\textbf{n}}},m)\cr
&\displaystyle\leq\sum_{(\overrightarrow{\textit{\textbf{n}}},m)\in\mathbb{Z}^{d+1}}\prod_{j=2}^{k}\mathbbm{1}_{\mathbb{X}^{l_{j};r_{u,j}}}(\overrightarrow{\textit{\textbf{n}}},m)\prod_{i\neq u}\mathbbm{1}_{\mathbb{B}_{i,1}^{r_{i,1}}}(n_{i},m)\cdot\prod_{l=1}^{k}\mathbbm{1}_{\mathbb{B}_{u,l}^{r_{u,l}}}(n_{u},m) \cr
&\displaystyle=\sum_{n_{u},m}\prod_{l=1}^{k}\mathbbm{1}_{\mathbb{B}_{u,l}^{r_{u,l}}}(n_{u},m)\sum_{\widehat{n_{u}}}\prod_{j=2}^{k}\mathbbm{1}_{\mathbb{X}^{l_{j};r_{u,j}}}(\overrightarrow{\textit{\textbf{n}}},m)\prod_{i\neq u}\mathbbm{1}_{\mathbb{B}_{i,1}^{r_{i,1}}}(n_{i},m) \cr
&\displaystyle\leq \sum_{n_{u},m}\prod_{l=1}^{k}\mathbbm{1}_{\mathbb{B}_{u,l}^{r_{u,l}}}(n_{u},m)\prod_{j=2}^{k}\left\|\mathbbm{1}_{\mathbb{X}^{l_{j};r_{u,j}}}(\overrightarrow{\textit{\textbf{n}}},m)\right\|_{\ell^{1}_{\widehat{n_{u}}}}^{\frac{1}{k}}\cdot \left\|\prod_{i\neq u}\mathbbm{1}_{\mathbb{B}_{i,1}^{r_{i,1}}}(n_{i},m)\right\|_{\ell^{1}_{\widehat{n_{u}}}}^{\frac{1}{k}} \cr
&\displaystyle\leq \prod_{j=2}^{k}\left\|\mathbbm{1}_{\mathbb{X}^{l_{j};r_{u,j}}}\right\|_{\ell^{\infty}_{n_{u},m}\ell^{1}_{\widehat{n_{u}}}}^{\frac{1}{k}}\cdot \prod_{i\neq u}\left\|\mathbbm{1}_{\mathbb{B}_{i,1}^{r_{i,1}}}\right\|_{\ell^{\infty}_{m}\ell^{1}_{n_{i}}}^{\frac{1}{k}} \cdot\left\|\prod_{l=1}^{k}\mathbbm{1}_{\mathbb{B}_{u,l}^{r_{u,l}}}\right\|_{\ell^{1}_{n_{u},m}},\cr
}
\end{equation}
where we used H\"older's inequality from the third to fourth line. Next, notice that

\begin{equation}\label{bound2-jan1123}
\eqalign{
\left\|\mathbbm{1}_{\mathbb{B}_{i,1}^{r_{i,1}}}\right\|_{\ell^{\infty}_{m}\ell^{1}_{n_{i}}} &\lesssim\displaystyle\sup_{m} 2^{2r_{i,1}}\sum_{n_{i}}\left\|\langle g_{1},\varphi_{n_{i},m}^{i,1}\rangle_{x_{i}}\right\|_{2}^{2}\cr
&\displaystyle=\sup_{m}2^{2r_{i,1}}\int\sum_{n_{i}}\left|\langle g_{1},\varphi_{n_{i},m}^{i,1}\rangle_{x_{i}}\right|^{2}\mathrm{d}\widehat{x_{i}} \cr
&\displaystyle\lesssim 2^{2r_{i,1}}\cdot \|g_{1}\|_{2}^{2} \cr
}
\end{equation}
by orthogonality. Now let

$$p_{u,1}:=\frac{2k}{(k+1)}, $$
$$p_{u,l}:=2k,\quad\forall\quad 2\leq l\leq k $$
and notice that
$$\sum_{l=1}^{k}\frac{1}{p_{u,l}}=1. $$

This way, by definition of $\mathbb{B}_{u,l}^{r_{u,l}}$ and by H\"older's inequality with these $p_{u,l}$ we have

\begin{equation}\label{bound3-jan1123}
\eqalign{
&\displaystyle\left\|\prod_{l=1}^{k}\mathbbm{1}_{\mathbb{B}_{u,l}^{r_{u,l}}}\right\|_{\ell^{1}_{n_{u},m}} \cr
&\lesssim\displaystyle 2^{\alpha\cdot r_{u,1}+\sum_{l=2}^{k}\beta\cdot r_{u,l}}\sum_{(n_{u},m)}\left\|\langle g_{1},\varphi_{n_{u},m}^{u,1}\rangle_{x_{u}}\right\|_{2}^{\alpha}\cdot\prod_{l=2}^{k}\left\|\langle g_{l},\varphi_{n_{u},m}^{u,l}\rangle_{x_{u}}\right\|_{2}^{\beta} \cr
&\leq\displaystyle 2^{\alpha\cdot r_{u,1}+\sum_{l=2}^{k}\beta\cdot r_{u,l}}\left(\sum_{(n_{u},m)}\left\|\langle g_{1},\varphi_{n_{u},m}^{u,1}\rangle_{x_{u}}\right\|_{2}^{\alpha\cdot p_{u,1}}\right)^{\frac{1}{p_{u,1}}}\cdot\prod_{l=2}^{k}\left(\sum_{(n_{u},m)}\left\|\langle g_{l},\varphi_{n_{u},m}^{u,l}\rangle_{x_{u}}\right\|_{2}^{\beta\cdot p_{u,l}} \right)^{\frac{1}{p_{u,l}}} \cr
&=\displaystyle 2^{\alpha\cdot r_{u,1}+\sum_{l=2}^{k}\beta\cdot r_{u,l}}\left(\sum_{(n_{u},m)}\left\|\langle g_{1},\varphi_{n_{u},m}^{u,1}\rangle_{x_{u}}\right\|_{2}^{4+\delta}\right)^{\frac{1}{p_{u,1}}}\cdot\prod_{l=2}^{k}\left(\sum_{(n_{u},m)}\left\|\langle g_{l},\varphi_{n_{u},m}^{u,l}\rangle_{x_{u}}\right\|_{2}^{4+\tilde{\delta}} \right)^{\frac{1}{p_{u,l}}} \cr
}
\end{equation}

At this point we see another difference between this proof and the argument in Section \ref{klineartheory}: we do not obtain a pure $L^{p}$ norm when using the near-$L^{4}$ extension analogue of Corollary \ref{cor1apr2922} for $l=d-1$. Alternatively, we use H\"older in the term involving $g_{1}$ once more:
\begin{equation*}
\eqalign{
\displaystyle\left\|\langle g_{1},\varphi_{n_{u},m}^{u,1}\rangle_{x_{u}}\right\|_{2}^{4+\delta}&\displaystyle =\left[\int \left(\prod_{j\neq u}|g_{1,j}|^{2}(x_{j})\right)\cdot |\langle g_{1,u},\varphi_{n_{u},m}^{u,1}\rangle_{x_{u}}|^{2}\widehat{\mathrm{d}x_{u}}\right]^{\frac{4+\delta}{2}} \cr
&\displaystyle\leq \left(\prod_{j\neq u}\|g_{1,j}\|_{2}\right)^{4+\delta}\cdot |\langle g_{1,u},\varphi_{n_{u},m}^{u,1}\rangle_{x_{u}}|^{4+\delta}. \cr
}
\end{equation*}

For the remaining $g_{l}$ we simply use H\"older and the fact that they are compactly supported\footnote{We use this crude estimate for the remaining $g_{l}$ because they do not have the same structure that allows ``pulling out" the one-dimensional functions $g_{1,j}$, like $g_{1}$ does. There is a clear loss here and it is reflected in the fact that $p(k,d)$ is not the best exponent for which \eqref{ineq0-jan1423} holds.}:
$$\left\|\langle g_{l},\varphi_{n_{u},m}^{u,l}\rangle_{x_{u}}\right\|_{2}^{4+\tilde{\delta}}\lesssim \left\|\langle g_{l},\varphi_{n_{u},m}^{u,l}\rangle_{x_{u}}\right\|_{4}^{4+\tilde{\delta}}. $$

These observations imply

\begin{equation}\label{bound4-jan1123}
\eqalign{
\displaystyle\left\|\prod_{l=1}^{k}\mathbbm{1}_{\mathbb{B}_{u,l}^{r_{u,l}}}\right\|_{\ell^{1}_{n_{u},m}} &\leq\displaystyle 2^{\alpha\cdot r_{u,1}+\sum_{l=2}^{k}\beta\cdot r_{u,l}}\cdot  \left(\prod_{j\neq u}\|g_{1,j}\|_{2}\right)^{\frac{4+\delta}{p_{u,1}}}\cdot\left(\sum_{(n_{u},m)}|\langle g_{1,u},\varphi_{n_{u},m}^{u,1}\rangle_{x_{u}}|^{4+\delta}\right)^{\frac{1}{p_{u,1}}} \cr
&\displaystyle\qquad\qquad\qquad\qquad\qquad\qquad\qquad\quad\qquad\cdot\prod_{l=2}^{k}\left(\sum_{(n_{u},m)}\left\|\langle g_{l},\varphi_{n_{u},m}^{u,l}\rangle_{x_{u}}\right\|_{4}^{4+\tilde{\delta}}\right)^{\frac{1}{p_{u,l}}} \cr
&\leq\displaystyle 2^{\alpha\cdot r_{u,1}+\sum_{l=2}^{k}\beta\cdot r_{u,l}}\cdot \left(\prod_{j\neq u}\|g_{1,j}\|_{2}\right)^{\alpha}\cdot \|g_{1,u}\|_{4}^{\alpha}\cdot \prod_{l=2}^{k}\|g_{l}\|_{4}^{\beta},
}
\end{equation}
where we used Minkowski for norms and the $L^{4}-L^{4+\tilde{\delta}}$ one-dimensional extension estimate from the second to third line above. Part \textit{(b)} follows from applying \eqref{bound2-jan1123} and \eqref{bound4-jan1123} to \eqref{bound1-jan1123}.
\end{proof}

Given $\varepsilon>0$, we bound the multilinear form $\tilde{\tilde{\Lambda}}_{k,d}$ using the estimates from \eqref{analog1-jan1123} and Corollary \ref{cor1-100123} (with the appropriate $\varepsilon$-losses for later convenience), and the ones from Lemma \ref{lemma1-jan1123} with the following weights:

$$
\begin{cases}
\theta_{l}=\frac{1}{2(d+1)}-\frac{\varepsilon}{d},\quad 1\leq l\leq d,\quad\textnormal{ for the $d$ estimates in \eqref{ineq4-jan1023} and \eqref{ineq5-jan1023}},\\
\theta_{d+1}= 1-\frac{d}{2(d+1)}+\varepsilon\quad\textnormal{ for \eqref{ineq2-jan1023}}.
\end{cases}
$$

Hence,

\begin{equation*}
\eqalign{
&\displaystyle|\tilde{\tilde{\Lambda}}_{k,d}(g,h)|\lesssim \cr
&\displaystyle \sum_{\overrightarrow{l},R,t\geq 0} 2^{-t}\displaystyle\times 2^{-\frac{(d+1)\varepsilon}{2kd}l_{1}}\times\left(\frac{1}{\|g_{1}\|_{2}^{d-1}}\prod_{j=1}^{d}2^{-r_{j,1}}\right)^{\frac{1}{k}-\frac{(d+1)\varepsilon}{2kd}} \cr
&\displaystyle\times\prod_{i=1}^{k-1}2^{-\frac{(d+1)\varepsilon}{2kd}l_{i+1}}\times\prod_{i=1}^{k-1}\left[\frac{2^{-\frac{1}{d+1}\cdot r_{i,i+1}}}{\left\|\mathbbm{1}_{\mathbb{X}^{l_{i+1};r_{i,i+1}}}\right\|_{\ell^{\infty}_{n_{i},m}\ell^{1}_{\widehat{n_{i}}}}^{\frac{1}{2(d+1)}}}\cdot\prod_{u=k}^{d}\frac{2^{-\frac{1}{k(d+1)}\cdot r_{u,i+1}}}{\left\|\mathbbm{1}_{\mathbb{X}^{l_{i+1};r_{u,i+1}}}\right\|_{\ell^{\infty}_{n_{u},m}\ell^{1}_{\widehat{n_{u}}}}^{\frac{1}{2k(d+1)}}}\right]^{1-\frac{(d+1)\varepsilon}{2d}} \cr
&\displaystyle\times\prod_{l=1}^{k-1}\left(\left\|\mathbbm{1}_{\mathbb{X}^{l_{l+1};r_{l,l+1}}}\right\|_{\ell^{\infty}_{n_{l},m}\ell^{1}_{\widehat{n_{l}}}}\cdot 2^{2r_{l,1}+2r_{l,l+1}}\cdot \|g_{1}\|_{2}^{2}\cdot \|g_{l+1}\|_{2}^{2}\right)^{\frac{1}{2(d+1)}-\frac{\varepsilon}{d}} \cr
&\displaystyle\times\prod_{k\leq u\leq d}\left(\prod_{j=2}^{k}\left\|\mathbbm{1}_{\mathbb{X}^{l_{j};r_{u,j}}}\right\|_{\ell^{\infty}_{n_{u},m}\ell^{1}_{\widehat{n_{u}}}}^{\frac{1}{k}}\cdot 2^{\frac{2}{k}\sum_{i\neq u}r_{i,1}}\cdot \|g_{1}\|_{2}^{\frac{2(d-1)}{k}}\right)^{\frac{1}{2(d+1)}-\frac{\varepsilon}{d}}\cr
&\displaystyle \times\prod_{k\leq u\leq d}\left(2^{\alpha\cdot r_{u,1}+\sum_{l=2}^{k}\beta\cdot r_{u,l}}\cdot \left(\prod_{j\neq u}\|g_{1,j}\|_{2}\right)^{\alpha} \cdot\|g_{1,u}\|_{4}^{\alpha}\cdot \prod_{l=2}^{k}\|g_{l}\|_{4}^{\beta}\right)^{\frac{1}{2(d+1)}-\frac{\varepsilon}{d}} \cr
&\displaystyle\times\left(2^{t}|F|\right)^{1-\frac{d}{2(d+1)}+\varepsilon}. \cr
}
\end{equation*}

Developing the expression above,

\begin{equation}\label{ineq1-apr3023}
\eqalign{
&\displaystyle|\tilde{\tilde{\Lambda}}_{k,d}(g,h)|\lesssim \cr
&\displaystyle \sum_{\overrightarrow{l},R,t\geq 0} 2^{-t}\times 2^{-\frac{(d+1)\varepsilon}{2kd}l_{1}}\times\left(\prod_{j=1}^{d}2^{-r_{j,1}}\right)^{\frac{1}{k}-\frac{(d+1)\varepsilon}{2kd}}\times \|g_{1}\|_{2}^{\frac{(d+1)(d-1)\varepsilon}{2kd}-\frac{(d-1)}{k}}   \cr
&\displaystyle\times\prod_{i=1}^{k-1}2^{-\frac{(d+1)\varepsilon}{2kd}l_{i+1}}\times\prod_{i=1}^{k-1}\left[2^{-\frac{1}{d+1}\cdot r_{i,i+1}}\cdot\prod_{u=k}^{d} 2^{-\frac{1}{k(d+1)}\cdot r_{u,i+1}}\right]^{1-\frac{(d+1)\varepsilon}{2d}} \cr
&\displaystyle\times\prod_{i=1}^{k-1}\left[\left\|\mathbbm{1}_{\mathbb{X}^{l_{i+1};r_{i,i+1}}}\right\|_{\ell^{\infty}_{n_{i},m}\ell^{1}_{\widehat{n_{i}}}}^{\frac{1}{2(d+1)}\cdot\left(\frac{(d+1)\varepsilon}{2d}-1\right)}\cdot\prod_{u=k}^{d}\left\|\mathbbm{1}_{\mathbb{X}^{l_{i+1};r_{u,i+1}}}\right\|_{\ell^{\infty}_{n_{u},m}\ell^{1}_{\widehat{n_{u}}}}^{\frac{1}{2k(d+1)}\left(\frac{(d+1)\varepsilon}{2d}-1\right)}\right]\cr
&\displaystyle\times\left[\prod_{l=1}^{k-1}\left\|\mathbbm{1}_{\mathbb{X}^{l_{l+1};r_{l,l+1}}}\right\|_{\ell^{\infty}_{n_{l},m}\ell^{1}_{\widehat{n_{l}}}}^{\frac{1}{2(d+1)}-\frac{\varepsilon}{d}}\right]\times\textcolor{red}{\left[\prod_{l=1}^{k-1}\left(2^{r_{l,1}+r_{l,l+1}}\right)^{\frac{1}{d+1}-\frac{2\varepsilon}{d}}\right]} \cr
&\displaystyle\times \textcolor{red}{\|g_{1}\|_{2}^{\frac{(k-1)}{d+1}-\frac{2(k-1)\varepsilon}{d}}\cdot\prod_{l=1}^{k-1}\|g_{l+1}\|_{2}^{\frac{1}{d+1}-\frac{2\varepsilon}{d}}}\cr
&\displaystyle\times \prod_{u=k}^{d}\left[\left(\prod_{j=2}^{k}\left\|\mathbbm{1}_{\mathbb{X}^{l_{j};r_{u,j}}}\right\|_{\ell^{\infty}_{n_{u},m}\ell^{1}_{\widehat{n_{u}}}}^{\frac{1}{k}\cdot\left(\frac{1}{2(d+1)}-\frac{\varepsilon}{d}\right)}\right)\cdot \left(2^{\frac{2}{k}\sum_{i\neq u}r_{i,1}}\cdot 2^{\alpha\cdot r_{u,1}+\sum_{l=2}^{k}\beta\cdot r_{u,l}}\right)^{\frac{1}{2(d+1)}-\frac{\varepsilon}{d}}\right] \cr
&\displaystyle\times \|g_{1}\|_{2}^{\frac{2(d-k+1)(d-1)}{k}\left(\frac{1}{2(d+1)}-\frac{\varepsilon}{d}\right)}\cdot \prod_{k\leq u\leq d}\left[ \textcolor{blue}{\left(\|g_{1,u}\|_{4}\cdot\prod_{j\neq u}\|g_{1,j}\|_{2}\right)}^{\alpha\left(\frac{1}{2(d+1)}-\frac{\varepsilon}{d}\right)}\right] \cdot \prod_{l=2}^{k}\|g_{l}\|_{4}^{\beta(d-k+1)\left(\frac{1}{2(d+1)}-\frac{\varepsilon}{d}\right)} \cr
&\displaystyle\times\left(2^{t}|F|\right)^{\left[1-\frac{d}{2(d+1)}\right]+\varepsilon}. \cr
}
\end{equation}

Observe that the product of the blue factors above (for $k\leq u\leq d$) is\footnote{Recall that $|g_{1}|=|g_{1,1}\otimes\ldots\otimes g_{1,d}|\leq \mathbbm{1}_{E_{1,1}}\otimes\ldots\otimes\mathbbm{1}_{E_{1,d}}\leq\mathbbm{1}_{E_{1}}$.}

\begin{equation*}
\eqalign{
\displaystyle\prod_{k\leq u\leq d}\left(\|g_{1,u}\|_{4}\cdot\prod_{j\neq u}\|g_{1,j}\|_{2}\right)&\displaystyle = \left[\prod_{l=1}^{k-1}\|g_{1,l}\|_{2}^{d-k+1}\right]\cdot\prod_{u=k}^{d}\left[\|g_{1,u}\|_{2}^{d-k}\cdot \|g_{1,u}\|_{4}\right] \cr
&\displaystyle = \left[\prod_{j=1}^{d}\|g_{1,j}\|_{2}^{d-k}\right]\cdot\left[\prod_{l=1}^{k-1}\|g_{1,l}\|_{2}\right]\cdot\left[\prod_{u=k}^{d}\|g_{1,u}\|_{4}\right] \cr
&\displaystyle\leq \|g_{1}\|_{2}^{d-k}\cdot |E_{1}|^{\frac{1}{4}}.\cr
}
\end{equation*}

Notice that the previous step was lossy, which also reflects in the suboptimal final exponent $p(k,d)$. Now we set the values of $\delta$ and $\tilde{\delta}$ (as functions of $\varepsilon$) to be such that

\begin{equation*}
\eqalign{
\displaystyle\delta\cdot\frac{(k+1)}{2k}\left(\frac{1}{2(d+1)}-\frac{\varepsilon}{d}\right)&\displaystyle =\frac{(d+1)\varepsilon}{kd}, \cr
\displaystyle\tilde{\delta}\cdot\frac{1}{2k}\left(\frac{1}{2(d+1)}-\frac{\varepsilon}{d}\right)&=\displaystyle\frac{\varepsilon}{kd}. \cr
}
\end{equation*}

Simplifying the expression above with this choice of $\delta$ and $\tilde{\delta}$,

\begin{equation}\label{finalbound1-apr3023}
\eqalign{
\displaystyle|\tilde{\tilde{\Lambda}}_{k,d}(g,h)|&\displaystyle\lesssim \left[\sum_{l_{1}\geq 0}2^{-\frac{(d+1)\varepsilon}{2kd}\cdot l_{1}}\right]\times \left[\prod_{j=1}^{k-1}\left(\sum_{r_{j,1}\geq 0}2^{-\frac{3(d+1)\varepsilon}{2kd}\cdot r_{j,1}}\right)\right] \times  \left[\prod_{u=k}^{d}\left(\sum_{r_{u,1}\geq 0}2^{-\frac{(d+1)\varepsilon}{2kd}\cdot r_{u,1}}\right)\right] \cr
&\displaystyle\times \left[\prod_{i=1}^{k-1}\left(\sum_{l_{i+1}\geq 0}2^{-\frac{(d+1)\varepsilon}{2kd}\cdot l_{i+1}}\right)\right]\times \left[\sum_{t\geq 0}2^{-t\left(\frac{d}{2(d+1)}-\varepsilon\right)}\right] \cr
&\displaystyle\times \prod_{i=1}^{k-1}\left[\left(\sum_{r_{i,i+1}\geq 0}2^{-\frac{3\varepsilon}{2d}\cdot r_{i,i+1}}\right)\cdot\prod_{u=k}^{d}\left(\sum_{r_{u,i+1}\geq 0}2^{-\frac{\varepsilon}{2kd}\cdot r_{u,i+1}}\right)\right]\cr
&\displaystyle\times\prod_{i=1}^{k-1}\left[\sup_{l_{i+1},r_{i,i+1}}\left\|\mathbbm{1}_{\mathbb{X}^{l_{i+1};r_{i,i+1}}}\right\|_{\ell^{\infty}_{n_{i},m}\ell^{1}_{\widehat{n_{i}}}}^{-\frac{3\varepsilon}{4d}}\cdot\prod_{u=k}^{d}\sup_{l_{i+1},r_{u,i+1}}\left\|\mathbbm{1}_{\mathbb{X}^{l_{i+1};r_{u,i+1}}}\right\|_{\ell^{\infty}_{n_{u},m}\ell^{1}_{\widehat{n_{u}}}}^{-\frac{3\varepsilon}{4kd}}\right]\cr
&\displaystyle\times \|g_{1}\|_{2}^{\frac{(d-k)}{k(d+1)} - \frac{2(d-k)(k+1)\varepsilon}{kd}+\frac{(d-k)(d+1)\varepsilon}{kd}+\frac{(d+1)(d-1)\varepsilon}{2kd}-\frac{2(k-1)\varepsilon}{d}-\frac{2(d-k+1)(d-1)\varepsilon}{kd}}\cdot |E_{1}|^{\frac{(k+1)}{4k(d+1)}+\frac{(d+1)\varepsilon}{4kd}}\cr
&\displaystyle\times\prod_{l=1}^{k-1}|E_{l+1}|^{\frac{(d+k+1)}{4k(d+1)}-\frac{\varepsilon}{d}-\frac{(d-k+1)\varepsilon}{2kd}+\frac{(d-k+1)\varepsilon}{4kd}} \cr
&\displaystyle\times |F|^{\left[1-\frac{d}{2(d+1)}\right]+\varepsilon}.\cr
}
\end{equation}

By considerations identical to the ones in the end of Section \ref{klineartheory}, this implies

\begin{equation}\label{finalbound1-jan1423}
|\tilde{\tilde{\Lambda}}_{k,d}(g,h)|\lesssim_{\varepsilon} |F|^{1-\frac{d}{2(d+1)}+\varepsilon}\cdot |E_{1}|^{\frac{2d-k+1}{4k(d+1)}}\prod_{l=1}^{k-1}|E_{l+1}|^{\frac{d+k+1}{4k(d+1)}}.
\end{equation}

To make all exponents of $|E_{j}|$ ($1\leq j\leq k$) the same, we have to take
$$\frac{1}{\widetilde{p}(k,d)}=\min\left\{\frac{2d-k+1}{4k(d+1)},\frac{d+k+1}{4k(d+1)}\right\}. $$

Again by the same considerations from Section \ref{klineartheory}, \eqref{finalbound1-jan1423} implies\footnote{Notice that we obtain something slightly better than Theorem \ref{thm1-100123} if one is looking for \textit{asymmetric estimates}: \eqref{finalbound1-jan1423} implies a bound of type $L^{p_{1}}\times L^{p_{2}}\times L^{p_{2}}\times\ldots\times L^{p_{2}}\rightarrow L^{\frac{2(d+1)}{kd}+\varepsilon}$, $p_{1}\neq p_{2}$ and $p_{1},p_{2}\leq p(k,d)$, if $g_{1}$ is a tensor.} Theorem \ref{thm1-100123}.
\end{proof}

\subsection{Near-restriction estimates without transversality}\label{nontransv-apr2823} To make the notation ligther, let us omit the index $Q$ and set $\mathcal{E}_{d}$ be the extension operator associated to a fixed cube $Q\subset\mathbb{R}^{d}$. Recall the $k$-product operator obtained from $\mathcal{E}_{d}$ defined in \eqref{kprod-apr2923}

$$\mathcal{E}_{d,(k)}(g_{1},\ldots,g_{k})= \prod_{j=1}^{k}\mathcal{E}_{d,(k)}g_{j}.$$

In this subsection we prove Theorem \ref{thm1-apr2823}, which we restate here for the convenience of the reader.

\begin{theorem}\label{thm1-apr2923} Let $2\leq k\leq d+1$. If $g_{1}$ is a tensor, the inequality
\begin{equation}
\left\|\prod_{j=1}^{k}\mathcal{E}_{d,(k)}g_{j}\right\|_{L^{\frac{2(d+1)}{kd}+\varepsilon}(\mathbb{R}^{d+1})} \lesssim_{Q,\varepsilon} \prod_{j=1}^{k} \|g_{j}\|_{L^{4}(Q)}
\end{equation}
holds for all $\varepsilon>0$.
\end{theorem}

\begin{remark} As in the previous subsection, the difference between the proof of Theorem \ref{thm1-apr2923} and the one done in Section \ref{klineartheory} is in the building blocks used: since there is no transversality to be exploited, we only use the best extension bound for the parabola (in the form of Proposition \ref{prop2apr2922}).
\end{remark}

\begin{proof}[Proof of Theorem \ref{thm1-apr2923}] The framework is the exact same as in the proof of Theorem \ref{thm1-100123}. We have to bound $\#\mathbb{X}^{\overrightarrow{l},R,t}$ to effectively estimate\footnote{Rigorously, we are dealing with a different operator here, but we will keep the notation unchanged for simplicity.}

$$|\tilde{\tilde{\Lambda}}_{k,d}(g,h)|\lesssim \sum_{\overrightarrow{l},R,t}2^{-t}\prod_{j=1}^{k}2^{-\frac{l_{j}}{k}}\#\mathbb{X}^{\overrightarrow{l},R,t}$$
in terms of the measures of the sets $E_{1,\ell}$, $1\leq\ell\leq d$, $E_{j}$, $2\leq j\leq k$, and $F$. This will be done by the following analogue of Lemma \ref{lemma1-jan1123}:

\begin{lemma}\label{analog1-apr2823} The two following extension-type bounds for the cardinality $\#\mathbb{X}^{\overrightarrow{l},R,t}$ hold.
\begin{enumerate}[(a)]
\item For all $1\leq i\leq k-1$ and all\footnote{The parameter $\lambda$ will be chosen later. It should be regarded as morally zero, and we only introduce it to be able to use Proposition \ref{prop2apr2922} since it does not hold at the endpoint.} $\lambda>0$,
\begin{equation}\label{ineq4-apr2923}
\#\mathbb{X}^{\overrightarrow{l},R,t}\leq \left\|\mathbbm{1}_{\mathbb{X}^{l_{i+1};r_{i,i+1}}}\right\|_{\ell^{\infty}_{n_{i},m}\ell^{1}_{\widehat{n_{i}}}}\cdot 2^{(2+\lambda)(r_{i,1}+r_{i,i+1})}\cdot \|g_{1,i}\|_{4}^{2+\lambda}\cdot \left(\prod_{\ell\neq i}\|g_{1,\ell}\|_{2+\lambda}^{2+\lambda}\right)\cdot \|g_{i+1}\|_{4}^{2+\lambda}.
\end{equation}
\item If $k<d+1$, for all $k\leq u\leq d$,
\begin{equation}\label{ineq5-apr2923}
\eqalign{
\displaystyle\#\mathbb{X}^{\overrightarrow{l},R,t} &\displaystyle \leq \prod_{j=2}^{k}\left\|\mathbbm{1}_{\mathbb{X}^{l_{j};r_{u,j}}}\right\|_{\ell^{\infty}_{n_{u},m}\ell^{1}_{\widehat{n_{u}}}}^{\frac{1}{k}}\cdot 2^{\frac{2}{k}\sum_{i\neq u}r_{i,1}}\cdot \|g_{1}\|_{2}^{\frac{2(d-1)}{k}} \cr
&\displaystyle \times 2^{\alpha\cdot r_{u,1}+\sum_{l=2}^{k}\beta\cdot r_{u,l}}\cdot \left(\prod_{j\neq u}\|g_{1,j}\|_{2}\right)^{\alpha} \cdot\|g_{1,1}\|_{4}^{\alpha}\cdot \prod_{l=2}^{k}\|g_{l}\|_{4}^{\beta},
}
\end{equation}
where
$$\alpha := \frac{2(k+1)}{k}+\delta\cdot\frac{(k+1)}{2k}, \quad \beta := \frac{2}{k}+\tilde{\delta}\cdot\frac{1}{2k}, $$
with $\delta,\tilde{\delta}>0$ being arbitrarily small parameters to be chosen later.
\end{enumerate}
\end{lemma}

\begin{remark} We highlight that \eqref{ineq5-apr2923} is only going to be used if $k<d+1$. The argument that follows will make it clear what changes in the case $k=d+1$ if we only use \eqref{ineq4-apr2923}.
\end{remark}

\begin{proof} We only prove \eqref{ineq4-apr2923}, since \eqref{ineq5-apr2923} is identical to \eqref{ineq5-jan1023}. From \eqref{bound1-dec22020},

\begin{equation*}
\#\mathbb{X}^{\overrightarrow{l},R,t}\leq \left\|\mathbbm{1}_{\mathbb{X}^{l_{i+1};r_{i,i+1}}}\right\|_{\ell^{\infty}_{n_{i},m}\ell^{1}_{\widehat{n_{i}}}}\cdot \left\|\mathbbm{1}_{\mathbb{B}_{i,1}^{r_{i,1}}\cap \mathbb{B}_{i,i+1}^{r_{i,i+1}}}\right\|_{\ell^{1}_{n_{i},m}}.
\end{equation*}

We bound the second factor in the right-hand side above as follows:

\begin{equation*}
\eqalign{
&\displaystyle\left\|\mathbbm{1}_{\mathbb{B}_{i,1}^{r_{i,1}}\cap \mathbb{B}_{i,i+1}^{r_{i,i+1}}}\right\|_{\ell^{1}_{n_{i},m}} \cr
&\quad\displaystyle\lesssim 2^{(2+\lambda)(r_{i,1}+r_{i,i+1})}\sum_{(n_{i},m)\in\mathbb{B}_{i,1}^{r_{i,1}}\cap \mathbb{B}_{i,i+1}^{r_{i,i+1}}} \left\|\langle g_{1},\varphi_{n_{i},m}^{i,1}\rangle_{x_{i}}\right\|_{2}^{2+\lambda} \cdot \left\|\langle g_{i+1},\varphi_{n_{i},m}^{i,i+1}\rangle_{y_{i}}\right\|_{2}^{2+\lambda} \cr
&\quad\displaystyle\leq 2^{(2+\lambda)(r_{i,1}+r_{i,i+1})}\sum_{(n_{i},m)\in\mathbb{B}_{i,1}^{r_{i,1}}\cap \mathbb{B}_{i,i+1}^{r_{i,i+1}}} \left\|\langle g_{1},\varphi_{n_{i},m}^{i,1}\rangle_{x_{i}}\right\|_{2+\lambda}^{2+\lambda} \cdot \left\|\langle g_{i+1},\varphi_{n_{i},m}^{i,i+1}\rangle_{y_{i}}\right\|_{2+\lambda}^{2+\lambda} \cr
&\quad\displaystyle\leq 2^{(2+\lambda)(r_{i,1}+r_{i,i+1})}\iint \left(\sum_{(n_{i},m)\in\mathbb{Z}^{2}}\left|\langle g_{1},\varphi_{n_{i},m}^{i,1}\rangle_{x_{i}}\right|^{2+\lambda}\cdot \left|\langle g_{i+1},\varphi_{n_{i},m}^{i,i+1}\rangle_{y_{i}}\right|^{2+\lambda}\right)\mathrm{d}\widehat{x_{i}}\mathrm{d}\widehat{y_{i}} \cr
&\quad\displaystyle\leq 2^{(2+\lambda)(r_{i,1}+r_{i,i+1})}\cr
&\qquad\quad\times\displaystyle\iint \left(\sum_{(n_{i},m)\in\mathbb{Z}^{2}}\left|\langle g_{1},\varphi_{n_{i},m}^{i,1}\rangle_{x_{i}}\right|^{4+2\lambda}\right)^{\frac{1}{2}}\cdot \left(\sum_{(n_{i},m)\in\mathbb{Z}^{2}}\left|\langle g_{i+1},\varphi_{n_{i},m}^{i,i+1}\rangle_{y_{i}}\right|^{4+2\lambda}\right)^{\frac{1}{2}}\mathrm{d}\widehat{x_{i}}\mathrm{d}\widehat{y_{i}} \cr
&\quad\displaystyle \lesssim_{\lambda} 2^{(2+\lambda)(r_{i,1}+r_{i,i+1})}\iint \|g_{1}\|_{L^{4}_{x_{i}}}^{2+\lambda}\cdot \|g_{i+1}\|_{L^{4}_{y_{i}}}^{2+\lambda}\mathrm{d}\widehat{x_{i}}\mathrm{d}\widehat{y_{i}} \cr
&\quad\displaystyle\lesssim 2^{(2+\lambda)(r_{i,1}+r_{i,i+1})}\cdot \|g_{1,i}\|_{4}^{2+\lambda}\cdot \left(\prod_{\ell\neq i}\|g_{1,\ell}\|_{2+\lambda}^{2+\lambda}\right)\cdot \|g_{i+1}\|_{4}^{2+\lambda},
}
\end{equation*}
where we used H\"older's inequality from the second to third lines, Fubini from the third to fourth, H\"older again twice, Proposition \ref{prop2apr2922} and the fact that $g_{1}$ is a tensor. This finishes the proof of the lemma.
\end{proof}

As in the previous subsection, given $\varepsilon>0$, we bound $\tilde{\tilde{\Lambda}}_{k,d}$ using the estimates from \eqref{analog1-jan1123} and Corollary \ref{cor1-100123}, and the ones from Lemma \ref{analog1-apr2823} with the exact same weights\footnote{If $k=d+1$, we give weight $\frac{1}{2(d+1)}-\frac{\varepsilon}{d}$ to each one of the $d$ estimates in \eqref{ineq4-apr2923} only.} we used in the proof of Theorem \ref{thm1-100123}:

$$
\begin{cases}
\theta_{l}=\frac{1}{2(d+1)}-\frac{\varepsilon}{d},\quad 1\leq l\leq d,\quad\textnormal{ for the $d$ estimates in \eqref{ineq4-apr2923} and \eqref{ineq5-apr2923}},\\
\theta_{d+1}= 1-\frac{d}{2(d+1)}+\varepsilon\quad\textnormal{ for \eqref{ineq2-jan1023}}.
\end{cases}
$$

Hence,

\begin{equation*}
\eqalign{
|\tilde{\tilde{\Lambda}}_{k,d}(g,h)|&\displaystyle\lesssim \sum_{\overrightarrow{l},R,t\geq 0} 2^{-t}\displaystyle\times 2^{-\frac{(d+1)\varepsilon}{2kd}l_{1}}\times\left(\frac{1}{\|g_{1}\|_{2}^{d-1}}\prod_{j=1}^{d}2^{-r_{j,1}}\right)^{\frac{1}{k}-\frac{(d+1)\varepsilon}{2kd}} \cr
&\displaystyle\times\prod_{i=1}^{k-1}2^{-\frac{(d+1)\varepsilon}{2kd}l_{i+1}}\times\prod_{i=1}^{k-1}\left[\frac{2^{-\frac{1}{d+1}\cdot r_{i,i+1}}}{\left\|\mathbbm{1}_{\mathbb{X}^{l_{i+1};r_{i,i+1}}}\right\|_{\ell^{\infty}_{n_{i},m}\ell^{1}_{\widehat{n_{i}}}}^{\frac{1}{2(d+1)}}}\cdot\prod_{u=k}^{d}\frac{2^{-\frac{1}{k(d+1)}\cdot r_{u,i+1}}}{\left\|\mathbbm{1}_{\mathbb{X}^{l_{i+1};r_{u,i+1}}}\right\|_{\ell^{\infty}_{n_{u},m}\ell^{1}_{\widehat{n_{u}}}}^{\frac{1}{2k(d+1)}}}\right]^{1-\frac{(d+1)\varepsilon}{2d}} \cr
&\displaystyle\times\prod_{l=1}^{k-1}\left(2^{(2+\lambda)(r_{l,1}+r_{l,l+1})}\cdot \|g_{1,l}\|_{4}^{2+\lambda}\cdot \left(\prod_{\ell\neq l}\|g_{1,\ell}\|_{2+\lambda}^{2+\lambda}\right)\cdot \|g_{l+1}\|_{4}^{2+\lambda}\right)^{\frac{1}{2(d+1)}-\frac{\varepsilon}{d}} \cr
&\displaystyle\times\prod_{k\leq u\leq d}\left(\prod_{j=2}^{k}\left\|\mathbbm{1}_{\mathbb{X}^{l_{j};r_{u,j}}}\right\|_{\ell^{\infty}_{n_{u},m}\ell^{1}_{\widehat{n_{u}}}}^{\frac{1}{k}}\cdot 2^{\frac{2}{k}\sum_{i\neq u}r_{i,1}}\cdot \|g_{1}\|_{2}^{\frac{2(d-1)}{k}}\right)^{\frac{1}{2(d+1)}-\frac{\varepsilon}{d}}\cr
&\displaystyle \times\prod_{k\leq u\leq d}\left(2^{\alpha\cdot r_{u,1}+\sum_{l=2}^{k}\beta\cdot r_{u,l}}\cdot \left(\prod_{j\neq u}\|g_{1,j}\|_{2}\right)^{\alpha} \cdot\|g_{1,u}\|_{4}^{\alpha}\cdot \prod_{l=2}^{k}\|g_{l}\|_{4}^{\beta}\right)^{\frac{1}{2(d+1)}-\frac{\varepsilon}{d}} \cr
&\displaystyle\times\left(2^{t}|F|\right)^{1-\frac{d}{2(d+1)}+\varepsilon}. \cr
}
\end{equation*}

Developing the expression above\footnote{The products in the fourth and fifth lines above are void if $k=d+1$. We can think of them as being $1$.},

\begin{equation}\label{ineq2-apr3023}
\eqalign{
&\displaystyle|\tilde{\tilde{\Lambda}}_{k,d}(g,h)|\lesssim \cr
&\displaystyle \sum_{\overrightarrow{l},R,t\geq 0} 2^{-t}\times 2^{-\frac{(d+1)\varepsilon}{2kd}l_{1}}\times\left(\prod_{j=1}^{d}2^{-r_{j,1}}\right)^{\frac{1}{k}-\frac{(d+1)\varepsilon}{2kd}}\times \|g_{1}\|_{2}^{\frac{(d+1)(d-1)\varepsilon}{2kd}-\frac{(d-1)}{k}}   \cr
&\displaystyle\times\prod_{i=1}^{k-1}2^{-\frac{(d+1)\varepsilon}{2kd}l_{i+1}}\times\prod_{i=1}^{k-1}\left[2^{-\frac{1}{d+1}\cdot r_{i,i+1}}\cdot\prod_{u=k}^{d} 2^{-\frac{1}{k(d+1)}\cdot r_{u,i+1}}\right]^{1-\frac{(d+1)\varepsilon}{2d}} \cr
&\displaystyle\times\prod_{i=1}^{k-1}\left[\left\|\mathbbm{1}_{\mathbb{X}^{l_{i+1};r_{i,i+1}}}\right\|_{\ell^{\infty}_{n_{i},m}\ell^{1}_{\widehat{n_{i}}}}^{\frac{1}{2(d+1)}\cdot\left(\frac{(d+1)\varepsilon}{2d}-1\right)}\cdot\prod_{u=k}^{d}\left\|\mathbbm{1}_{\mathbb{X}^{l_{i+1};r_{u,i+1}}}\right\|_{\ell^{\infty}_{n_{u},m}\ell^{1}_{\widehat{n_{u}}}}^{\frac{1}{2k(d+1)}\left(\frac{(d+1)\varepsilon}{2d}-1\right)}\right]\cr
&\displaystyle\times\left[\prod_{l=1}^{k-1}\left\|\mathbbm{1}_{\mathbb{X}^{l_{l+1};r_{l,l+1}}}\right\|_{\ell^{\infty}_{n_{l},m}\ell^{1}_{\widehat{n_{l}}}}^{\frac{1}{2(d+1)}-\frac{\varepsilon}{d}}\right]\times\textcolor{red}{\left[\prod_{l=1}^{k-1}\left(2^{r_{l,1}+r_{l,l+1}}\right)^{\left(2+\lambda\right)\cdot\left(\frac{1}{2(d+1)}-\frac{\varepsilon}{d}\right)}\right]} \cr
&\displaystyle\times \textcolor{red}{\left[\prod_{l=1}^{k-1}|E_{1,l}|^{\left(\frac{2+\lambda}{4}+(k-2)\right)\cdot\left(\frac{1}{2(d+1)}-\frac{\varepsilon}{d}\right)}\right]\cdot\left[\prod_{u=k}^{d}|E_{1,u}|^{\left(\frac{1}{2(d+1)}-\frac{\varepsilon}{d}\right)\cdot(k-1)}\right]\cdot\left[\prod_{l=1}^{k-1}|E_{l+1}|^{\frac{(2+\lambda)}{4}\cdot\left(\frac{1}{2(d+1)}-\frac{\varepsilon}{d}\right)}\right]}  \cr
&\displaystyle\times \prod_{u=k}^{d}\left[\left(\prod_{j=2}^{k}\left\|\mathbbm{1}_{\mathbb{X}^{l_{j};r_{u,j}}}\right\|_{\ell^{\infty}_{n_{u},m}\ell^{1}_{\widehat{n_{u}}}}^{\frac{1}{k}\cdot\left(\frac{1}{2(d+1)}-\frac{\varepsilon}{d}\right)}\right)\cdot \left(2^{\frac{2}{k}\sum_{i\neq u}r_{i,1}}\cdot 2^{\alpha\cdot r_{u,1}+\sum_{l=2}^{k}\beta\cdot r_{u,l}}\right)^{\frac{1}{2(d+1)}-\frac{\varepsilon}{d}}\right] \cr
&\displaystyle\times \|g_{1}\|_{2}^{\frac{2(d-k+1)(d-1)}{k}\left(\frac{1}{2(d+1)}-\frac{\varepsilon}{d}\right)}\cdot \prod_{k\leq u\leq d}\left[ \textcolor{blue}{\left(\|g_{1,u}\|_{4}\cdot\prod_{j\neq u}\|g_{1,j}\|_{2}\right)}^{\alpha\left(\frac{1}{2(d+1)}-\frac{\varepsilon}{d}\right)}\right] \cdot \prod_{l=2}^{k}\|g_{l}\|_{4}^{\beta(d-k+1)\left(\frac{1}{2(d+1)}-\frac{\varepsilon}{d}\right)} \cr
&\displaystyle\times\left(2^{t}|F|\right)^{\left[1-\frac{d}{2(d+1)}\right]+\varepsilon}. \cr
}
\end{equation}

Observe that we highlighted a few factors in red in \eqref{ineq2-apr3023}; this is just to compare them to the red terms in \eqref{ineq1-apr3023}: the red terms are the only ones that differ in the right-hand sides of \eqref{ineq1-apr3023} and \eqref{ineq2-apr3023}. On the other hand, we will bound the product of the blue factors\footnote{The seventh and eighth lines are void if $k=d+1$, hence the blue factors do not contribute at all in this case.} in \eqref{ineq2-apr3023} in a slightly better way than we did in the proof of Theorem \ref{thm1-100123}:

\begin{equation}\label{extrablue-apr3023}
\eqalign{
\displaystyle\prod_{k\leq u\leq d}\left(\|g_{1,u}\|_{4}\cdot\prod_{j\neq u}\|g_{1,j}\|_{2}\right)
&\displaystyle = \left[\prod_{j=1}^{d}\|g_{1,j}\|_{2}^{d-k}\right]\cdot\left[\prod_{l=1}^{k-1}\|g_{1,l}\|_{2}\right]\cdot\left[\prod_{u=k}^{d}\|g_{1,u}\|_{4}\right] \cr
&\displaystyle\leq\left[\prod_{l=1}^{k-1}|E_{1,l}|^{\frac{d-k+1}{2}}\right]\cdot\left[\prod_{u=k}^{d}|E_{1,u}|^{\frac{d-k}{2}+\frac{1}{4}}\right].\cr
}
\end{equation}

Setting $\delta$ and $\widetilde{\delta}$ exactly as in the previous subsection and using the observations we just made, we conclude that the final bound for $|\tilde{\tilde{\Lambda}}_{k,d}(g,h)|$ compares to \eqref{finalbound1-jan1423} exactly as follows:

\begin{itemize}
\item The coefficients of the ``$r_{j,1}$ power" is now
$$2^{\left[-\frac{3(d+1)\varepsilon}{2kd}+\lambda\left(\frac{1}{2(d+1)}-\frac{\varepsilon}{d}\right)\right]r_{j,1}}, $$
whereas in \eqref{finalbound1-jan1423} it was
$$2^{\left(-\frac{3(d+1)\varepsilon}{2kd}\right)r_{j,1}}.$$
\item For $1\leq l\leq k-1$, \eqref{extrablue-apr3023} gives $|E_{1,l}|$ an extra power of\footnote{Here we are using the explicit choice of $\delta$.}
$$\left(\frac{1}{2}+\frac{1}{2k}\right)\cdot\left(\frac{1}{2(d+1)}-\frac{\varepsilon}{d}\right)+\frac{(d+1)\varepsilon}{4kd}.$$
On the other hand, still for $1\leq l\leq k-1$, the red factors in \eqref{ineq2-apr3023} produce a power of $|E_{1,l}|$ that is exactly
\begin{equation}\label{loss1-apr3023}
\frac{(2-\lambda)}{4}\cdot\left(\frac{1}{2(d+1)}-\frac{\varepsilon}{d}\right)
\end{equation}
less than the one produced by the corresponding red factors in \eqref{ineq1-apr3023}. If $k<d+1$, these provide a \textit{net gain} of
$$\left(\frac{1}{2k}-\frac{\lambda}{4}\right)\cdot\left(\frac{1}{2(d+1)}-\frac{\varepsilon}{d}\right)+\frac{(d+1)\varepsilon}{4kd} $$
in the final power of $|E_{1,l}|$. If $k=d+1$, we just lose (compared to \eqref{finalbound1-jan1423}) \eqref{loss1-apr3023} in the final power of $|E_{1,l}|$.
\item For $k\leq u\leq d$, the powers of the measures $|E_{1,u}|$ are exactly the same in both \eqref{ineq1-apr3023} and in \eqref{ineq2-apr3023}.
\item For $2\leq l\leq k$, the red factors in \eqref{ineq2-apr3023} produce a power of $|E_{l}|$ that is exactly
$$\frac{(2-\lambda)}{4}\left(\frac{1}{2(d+1)}-\frac{\varepsilon}{d}\right) $$
less than the one produced by the corresponding red factors in \eqref{ineq1-apr3023}.
\item All other factors are precisely the same.
\end{itemize}

By choosing $\lambda$ small enough compared to $\varepsilon$ and by the same considerations made in the end of Section \ref{klineartheory}, this implies

\begin{equation*}
|\tilde{\tilde{\Lambda}}_{k,d}(g,h)|\lesssim_{\varepsilon} |F|^{1-\frac{d}{2(d+1)}+\varepsilon}\cdot |E_{1}|^{\frac{2d-k+2}{4k(d+1)}}\prod_{l=1}^{k-1}|E_{l+1}|^{\frac{1}{4k}}
\end{equation*}
for $k<d+1$ and
\begin{equation*}
|\tilde{\tilde{\Lambda}}_{k,d}(g,h)|\lesssim_{\varepsilon} |F|^{1-\frac{d}{2(d+1)}+\varepsilon}\cdot \prod_{l=1}^{k}|E_{l}|^{\frac{1}{4k}}
\end{equation*}
for $k=d+1$. Again by the same considerations from Section \ref{klineartheory}, these imply Theorem \ref{thm1-apr2923}.

\end{proof}

\section{Weak transversality, Brascamp-Lieb and an application}\label{WTBLA}

We were recently asked by Jonathan Bennett if there was a link between our results and the theory of Brascamp-Lieb inequalities. The motivation for that comes from the fact that, assuming $g_{1}=g_{1,1}\otimes\ldots\otimes g_{1,d}$, one can see the operator $\mathcal{ME}_{d+1,d}$ as the $2d$-linear object

\begin{equation*}
T(g_{1,1},\ldots,g_{1,d},g_{2},\ldots,g_{d+1}):=\mathcal{ME}_{d+1,d}(g_{1,1}\otimes\ldots\otimes g_{1,d},g_{2},\ldots,g_{d+1}),
\end{equation*}
and given that such a link exists in the theory of $\mathcal{ME}_{d+1,d}$ (see \cite{Benn1}), it is natural to wonder if boundedness for $T$ is related somehow to the finiteness condition of certain Brascamp-Lieb constants $\textnormal{BL}(\textnormal{\textbf{L}},\textnormal{\textbf{p}})$.

The purposes of this section are to make this connection clear and to give a modest application of our results to the theory of \textit{Restriction-Brascamp-Lieb inequalities}.

\subsection{A link between weak transversality and Brascamp-Lieb inequalities}

We start with some classical background. Let $L_{j}:\mathbb{R}^{n}\rightarrow\mathbb{R}^{n_{j}}$ be linear maps and $p_{j}\geq 0$, $1\leq j\leq m$. Inequalities of the form

\begin{equation}\label{BL-010123}
\int_{\mathbb{R}^{n}}\prod_{j=1}^{m}(f_{j}\circ L_{j})^{p_{j}}(v)\mathrm{d}v \leq C\prod_{j=1}^{m}\left(\int_{\mathbb{R}^{n_{j}}}f_{j}(y_{j})\mathrm{d}y_{j}\right)^{p_{j}}
\end{equation}
are called \textit{Brascamp-Lieb inequalities}. In \cite{BCCT1}, Bennett, Carbery, Christ and Tao established for which \textit{Brascamp-Lieb data} $(\textnormal{\textbf{L}},\textnormal{\textbf{p}})$ the inequality above holds, where $\textnormal{\textbf{L}}=(L_{1},\ldots,L_{m})$ and $\textnormal{\textbf{p}}=(p_{1},\ldots,p_{m})$. The best constant for which \eqref{BL-010123} holds for all nonnegative input functions $f_{j}\in L^{1}(\mathbb{R}^{n_{j}})$ is denoted by $\textnormal{BL}(\textnormal{\textbf{L}},\textnormal{\textbf{p}})$.

\begin{theorem}[\cite{BCCT1}]\label{BCCT-030123} The constant $\textnormal{BL}(\textnormal{\textbf{L}},\textnormal{\textbf{p}})$ in \eqref{BL-010123} is finite if and only if for all subspaces $V\subset\mathbb{R}^{n}$
\begin{equation}\label{cond1BL-010123}
\textnormal{dim}(V)\leq\sum_{j=1}^{m}p_{j}\textnormal{dim}(L_{j}V)
\end{equation}
and
\begin{equation}\label{cond2BL-010123}
\sum_{j=1}^{m}p_{j}n_{j}=n.
\end{equation}
\end{theorem}
\begin{remark} By taking $V=\mathbb{R}^{n}$ in \eqref{cond1BL-010123} it follows that each $L_{j}$ must be surjective for \eqref{cond2BL-010123} to hold as well.
\end{remark}

We will work with explicit maps $L_{j}$ and use Theorem \ref{BCCT-030123} to establish a link between the concept of weak transversality and inequalities such as \eqref{BL-010123}\footnote{From now on, we will replace $n$ by $d+1$ when referring to the dimension of the euclidean space.}. These maps will be associated to the submanifolds relevant to the problem at hand: the $d$-dimensional paraboloid $\mathbb{P}^{d}$ in $\mathbb{R}^{d+1}$ and some ``canonical" two-dimensional parabolas.

In order to define $L_{j}$, we fix standard parametrizations for the submanifolds mentioned above. Let

\begin{align}\label{param1-jan1623}
\Gamma\colon \mathbb{R}^{d} & \longrightarrow\mathbb{R}^{d+1}\\
(x_{1},\ldots,x_{d})&\longmapsto \left(x_{1},\ldots,x_{d},\sum_{i=1}^{d}x_{i}^{2}\right),
\end{align}
parametrize $\mathbb{P}^{d}$ and
\begin{align}\label{param2-jan1623}
\gamma_{j}\colon \mathbb{R} & \longrightarrow\mathbb{R}^{d+1}\\
x&\longmapsto (x\cdot\delta_{1j},\ldots,x\cdot\delta_{dj},x^{2})
\end{align}
parametrize a parabola in the two-dimensional canonical subspace generated by $e_{j}$ and $e_{d+1}$ ($\delta_{ij}$ is the Kronecker delta). Their differentials are given by

\begin{align*}
\mathrm{d}\Gamma\colon \mathbb{R}^{d} & \longrightarrow M_{(d+1)\times d}\\
(x_{1},\ldots,x_{d})&\longmapsto\begin{bmatrix}
    1 & 0 &  \dots  & 0 \\
    0 & 1  & \dots  & 0 \\
    \vdots  & \vdots & \ddots & \vdots \\
    0  & 0 & \dots & 1 \\
    2x_{1} & 2x_{2}  & \dots  & 2x_{d}
\end{bmatrix},
\end{align*}
and
\begin{align*}
\mathrm{d}\gamma_{j}\colon \mathbb{R} & \longrightarrow M_{(d+1)\times 1}\\
x&\longmapsto \begin{bmatrix}
    \delta_{1j}  \\
    \delta_{2j}  \\
    \vdots \\
    \delta_{dj} \\
    2x
\end{bmatrix}.
\end{align*}

For $d+1$ points $x^{j}=(x_{1}^{j},\ldots,x_{d}^{j})\in \mathbb{R}^{d}$, $1\leq j\leq d+1$, define the linear maps\footnote{We highlight that the \textit{superscript} $j$ in $x^{j}_{i}$ denotes the \textit{point}, whereas the \textit{subscript} $i$ denotes the \textit{$i$-coordinate} of the corresponding point. Notice also that we are identifying the adjoint operator $T^{\ast}$ with the transpose of the matrix that represents $T$ in the canonical basis.}
\begin{equation}\label{maps1-jan1323}
\eqalign{
\displaystyle L_{\ell}^{x_{\ell}^{1}}&\displaystyle :=\left(\mathrm{d}\gamma_{\ell}(x_{\ell}^{1})\right)^{\ast},\quad\forall 1\leq\ell\leq d, \cr
\displaystyle L_{d+\ell}^{x^{\ell+1}}&\displaystyle :=\left(\mathrm{d}\Gamma(x_{1}^{\ell+1},\ldots,x_{d}^{\ell+1})\right)^{\ast},\quad\forall 1\leq\ell\leq d. \cr
}
\end{equation}

It is important to emphasize that $L_{d+\ell}$ \textit{depends on $x^{\ell+1}$} (and similarly, $L_{\ell}$ depends on $x^{1}_{\ell}$). The main result of this subsection is:

\begin{theorem}\label{thm1-010123} Let $\mathcal{Q}=\{Q_{1},\ldots,Q_{d+1}\}$ be a collection of closed cubes in $\mathbb{R}^{d}$. If $\mathcal{Q}$ is weakly transversal with pivot $Q_{1}$, then for any choice of points $x^{j}=(x_{1}^{j},\ldots,x_{d}^{j})\in Q_{j}$, the linear maps in \eqref{maps1-jan1323} satisfy 
\begin{equation}\label{cond1-020123}
\textnormal{BL(\textbf{L}$(x)$,\textbf{p})}<\infty \textnormal{ for } \textnormal{
\textbf{L}$(x)$}=(L_{1}^{x_{1}^{1}},\ldots,L_{2d}^{x^{d+1}})\textnormal{ and } \textnormal{\textbf{p}}=\left(\frac{1}{d},\ldots,\frac{1}{d}\right).
\end{equation} 
Conversely, if \eqref{cond1-020123} is satisfied by the linear maps in \eqref{maps1-jan1323} for any choice of points $x^{j}=(x_{1}^{j},\ldots,x_{d}^{j})\in Q_{j}$, then $\mathcal{Q}$ can be decomposed into $O(1)$ weakly transversal collections $\mathcal{Q}^{\prime}$ of $d+1$ cubes, each one having a cube $Q_{1}^{\prime}\subset Q_{1}$ as pivot.
\end{theorem}

\begin{remark} If $\mathcal{Q}$ can be decomposed into $O(1)$ weakly transversal collections $\mathcal{Q}^{\prime}$ of $d+1$ cubes (in the sense of Claim \ref{claim2-081221}), each one having a cube $Q_{1}^{\prime}\subset Q_{1}$ as pivot, then the conclusion of the first part of the theorem above also holds for $\mathcal{Q}$. Some important examples to keep in mind are the ones of transversal configurations that are \textit{not} weakly transversal by themselves, but that are decomposable into such: for instance, $\{Q_{1},Q_{2},Q_{3}\}$ where $Q_{1}=[1,4]\times [2,3]$, $Q_{2}=[0,2]\times [0,1]$ and $Q_{3}=[3,5]\times [0,1]$ is a transversal collection of cubes in $\mathbb{R}^{2}$, but not weakly transversal with pivot $Q_{1}$ since $\pi_{1}(Q_{1})$ intersects both $\pi_{1}(Q_{2})$ and $\pi_{1}(Q_{3})$.
\end{remark}

\begin{remark} We can of course obtain a similar statement if $\mathcal{Q}$ is weakly transversal with any other pivot $Q_{j}$, $j\neq 1$. The linear maps $L_{\ell}$ and $L_{d+\ell}$ would have to be changed accordingly.
\end{remark}

\begin{proof}[Proof of Theorem \ref{thm1-010123}] Suppose that $\mathcal{Q}$ is weakly transversal with pivot $Q_{1}$. We can then assume without loss of generality that 
\begin{equation}\label{wlog1-feb423}
    \begin{dcases}
        \pi_{1}(Q_{1}) \cap \pi_{1}(Q_{2}) =\emptyset, \\
       \qquad \vdots \\
        \pi_{d}(Q_{1}) \cap \pi_{d}(Q_{d+1}) =\emptyset. \\
    \end{dcases}
\end{equation}

The strategy is to apply Theorem \ref{BCCT-030123}. Condition \eqref{cond2BL-010123} is trivially satisfied, so we just have to check \eqref{cond1BL-010123}. Fix the points $x^{j}=(x_{1}^{j},\ldots,x_{d}^{j})\in Q_{j}$, $1\leq j\leq d$. To avoid heavy notation, we will omit the superscripts $x^{1}_{\ell}$ and $x^{\ell+1}$ when referring to $L_{\ell}^{x^{1}_{\ell}}$ and $L_{d+\ell}^{x^{\ell+1}}$, respectively, but these points will be referenced whenever they play an important role. We emphasize that the maps $L_{\ell}$, $1\leq \ell\leq d$, are being identified with the row vector

$$\begin{bmatrix}
    \delta_{1\ell} & \delta_{2\ell} &  \dots  &\delta_{d\ell} & 2x_{\ell}^{1} \\
\end{bmatrix}, $$
whereas the maps $L_{d+\ell}$, $1\leq \ell\leq d$, are identified with the $d\times (d+1)$ matrix
$$\begin{bmatrix}
    1 & 0 &  \dots  &0 & 2x_{1}^{\ell+1} \\
    0 & 1  & \dots  & 0& 2x_{2}^{\ell+1} \\
    \vdots  & \vdots & \ddots & \vdots & \vdots \\
    0 & 0  & \dots  & 1 & 2x_{d}^{\ell+1}
\end{bmatrix}.$$

If $V\subset\mathbb{R}^{d+1}$ is a subspace of dimension $k$, we have to verify that
\begin{equation}\label{ineq1-010123}
dk\leq \sum_{j=1}^{d}\textnormal{dim}(L_{j}V)+\sum_{\ell=1}^{d}\textnormal{dim}(L_{d+\ell}V).
\end{equation}

Suppose that there are exactly $m\geq 0$ indices $j\in\{1,\ldots,d\}$ such that $\textnormal{dim}(L_{j}V)=0$. If $m=0$, we must have $L_{j}V=\mathbb{R}$ for all $1\leq j\leq d$, hence
\begin{equation}\label{eq1-feb423}
\sum_{j=1}^{d}\textnormal{dim}(L_{j}V)=d.
\end{equation}

Surjectivity of $L_{d+\ell}$, $1\leq\ell\leq d$, implies $\textnormal{dim}(\textnormal{ker}(L_{d+\ell}))=1$, which gives the lower bound $\textnormal{dim}(L_{d+\ell}V)\geq k-1$. We then obtain
\begin{equation}\label{eq2-feb423}
\sum_{\ell=1}^{d}\textnormal{dim}(L_{d+\ell}V)\geq d(k-1).
\end{equation}

It is clear that \eqref{eq1-feb423} and \eqref{eq2-feb423} together verify \eqref{ineq1-010123} in the $m=0$ case. If $m\geq 1$, assume without loss of generality that
\begin{equation}\label{hyp1-feb423}
L_{1}V = \ldots L_{m}V=0,
\end{equation}
\begin{equation}\label{hyp2-feb423}
L_{m+1}V=\ldots = L_{d}V=\mathbb{R}.
\end{equation}

This gives us
\begin{equation}\label{ineq2-010123}
\sum_{j=1}^{d}\textnormal{dim}(L_{j}V)=d-m.
\end{equation}

We will show that
\begin{equation}\label{ineq3-010123}
\sum_{\ell=1}^{d}\textnormal{dim}(L_{d+\ell}V)\geq (d-m)(k-1)+mk.
\end{equation}

Observe that \eqref{ineq2-010123} and \eqref{ineq3-010123} together verify \eqref{ineq1-010123} in the $m\geq 1$ case.

We claim that there are at least $m$ maps $L_{\ell_{j}}$ among $L_{\ell+1},\ldots,L_{2d}$ such that $\textnormal{dim}(L_{\ell_{j}}V)=k$. If not, there are $d-m+1$ maps $L_{\ell_{1}},\ldots, L_{\ell_{d-m+1}}$ with $\textnormal{dim}(L_{\ell_{j}}V)\leq k-1$. Since $\textnormal{dim}V=k$, the rank-nullity theorem implies the existence of
\begin{equation}\label{hyp3-feb423}
 0\neq v^{\ell_{j}}\in\textnormal{ker}(L_{\ell_{j}})\cap V,\quad 1\leq j\leq d-m+1.
 \end{equation}

By \eqref{hyp1-feb423},
\begin{equation}\label{condition1-feb423}
L_{r}v^{\ell_{j}}=v_{r}^{\ell_{j}}+2x_{r}^{1}v_{d+1}^{\ell_{j}}=0,\quad 1\leq r\leq m,
\end{equation}
and by \eqref{hyp3-feb423} we have
\begin{equation}\label{condition2-feb423}
L_{\ell_{j}}v^{\ell_{j}}= \begin{bmatrix}
    1 & 0 &  \dots  &0 & 2x_{1}^{\ell_{j}-d+1} \\
    0 & 1  & \dots  & 0& 2x_{2}^{\ell_{j}-d+1} \\
    \vdots  & \vdots & \ddots & \vdots & \vdots \\
    0 & 0  & \dots  & 1 & 2x_{d}^{\ell_{j}-d+1}
\end{bmatrix}\cdot \begin{bmatrix}
    v_{1}^{\ell_{j}}  \\
    v_{2}^{\ell_{j}}  \\
    \vdots  \\
    v_{d+1}^{\ell_{j}}
\end{bmatrix}=\begin{bmatrix}
    v_{1}^{\ell_{j}}+2x_{1}^{\ell_{j}-d+1}v_{d+1}^{\ell_{j}}  \\
     v_{2}^{\ell_{j}}+2x_{2}^{\ell_{j}-d+1}v_{d+1}^{\ell_{j}}  \\
    \vdots \\
     v_{d}^{\ell_{j}}+2x_{d}^{\ell_{j}-d+1}v_{d+1}^{\ell_{j}}
\end{bmatrix}=0
\end{equation}
for $1\leq j\leq d-m+1$. For each $1\leq r\leq m$, combining the information from \eqref{condition1-feb423} and \eqref{condition2-feb423} gives us
$$v_{d+1}^{\ell_{j}}\cdot (x_{r}^{1}-x_{r}^{\ell_{j}-d+1})=0. $$

If $v_{d+1}^{\ell_{j}}=0$, then \eqref{condition2-feb423} also implies $v_{n}^{\ell_{j}}=0$ for all $n\in\{1,\ldots,d\}$, thus $v^{\ell_{j}}=0$, which contradicts \eqref{hyp3-feb423}. Then we must have
$$x_{r}^{1}=x_{r}^{\ell_{j}-d+1},\quad 1\leq r \leq m.$$

Let us now see why this can not happen. We have just shown that there are $d-m+1$ values of $\alpha$ for which

\begin{equation}\label{cont1-010123}
    \begin{dcases}
        \pi_{1}(Q_{1}) \cap \pi_{1}(Q_{\alpha}) \neq\emptyset, \\
       \qquad \vdots \\
        \pi_{m}(Q_{1}) \cap \pi_{m}(Q_{\alpha}) \neq\emptyset. \\
    \end{dcases}
\end{equation}

On the other hand, \eqref{wlog1-feb423} tells us that $\alpha\notin\{2,3,\ldots,m+1\}$, hence there are at most $d-m$ possible values for $\alpha$ (we can not have $\alpha=1$ either), which is a contradiction.

Hence there are at least $m$ maps $L_{\ell_{j}}$ among $L_{\ell+1},\ldots,L_{2d}$ such that $\textnormal{dim}(L_{\ell_{j}}V)=k$. The remaining $d-m$ maps have kernels of dimension $1$, so the image of $V$ through them has dimension at least $k-1$ (again by surjectivity of $L_{\ell_{j}}$ and the rank-nullity theorem). This verifies \eqref{ineq3-010123}.

For the converse implication, suppose that \eqref{cond1-020123} is satisfied by the linear maps in \eqref{maps1-jan1323} for any choice of points $(x_{1}^{j},\ldots,x_{d}^{j})\in Q_{j}$. As a consequence of the proof of Claim \ref{claima2-081221}, each $Q_{l}\in\mathcal{Q}$ can be partitioned into $O(1)$ sub-cubes
$$Q_{l}=\bigcup_{i}Q_{l,i} $$
so that all collections $\widetilde{\mathcal{Q}}$ made of picking one sub-cube $Q_{l,i}$ per $Q_{l}$
$$\widetilde{\mathcal{Q}}=\{\widetilde{Q}_{1},\ldots,\widetilde{Q}_{d+1}\},\quad \widetilde{Q}_{l}\in\{Q_{l,i}\}_{i},$$
satisfy the following:

\begin{enumerate}[(a)]

\item For any two $\widetilde{Q}_{r},\widetilde{Q}_{s}\in \widetilde{\mathcal{Q}}$, either $\pi_{j}(\widetilde{Q}_{r})\cap \pi_{j}(\widetilde{Q}_{s})=\emptyset$, or $\pi_{j}(\widetilde{Q}_{r})= \pi_{j}(\widetilde{Q}_{s})$, or $\pi_{j}(\widetilde{Q}_{r})\cap\pi_{j}(\widetilde{Q}_{s})=\{p_{r,s}\}$, where $p_{r,s}$ is an endpoint of both $\pi_{j}(\widetilde{Q}_{r})$ and $\pi_{j}(\widetilde{Q}_{s})$.

\item All $\pi_{j}(\widetilde{Q}_{s})$ that intersect a given $\pi_{j}(\widetilde{Q}_{r})$ (but distinct from it) do so at the same endpoint.\footnote{In other words, all $\pi_{j}(\widetilde{Q}_{s})$ that intersect a given $\pi_{j}(\widetilde{Q}_{r})$ (but distinct from it) do so on the same side. In short notation, let $\mathcal{S}_{j,r}$ be the set of $s$ for which $\pi_{j}(\widetilde{Q}_{r})\cap\pi_{j}(\widetilde{Q}_{s})\neq\emptyset$. The conclusion is that there is some real number $\gamma_{j}$ such that
$$\gamma_{j}\in\pi_{j}(Q_{r})\cap\bigcap_{s\in\mathcal{S}_{j,r}}\pi_{j}(Q_{s}). $$}
\end{enumerate}

By a slight abuse of notation, let $\mathcal{Q}$ denote one such sub-collection that has the two properties above. Suppose, by contradiction, that $\mathcal{Q}$ is not weakly transversal with pivot $Q_{1}$ (recall that this is a cube obtained from the original $Q_{1}$). The strategy now is to construct a subspace $V\subset\mathbb{R}^{d+1}$ that contradicts \eqref{cond1BL-010123} for a certain choice of one point per cube in $\mathcal{Q}$. This construction will exploit a certain feature of a special subset of $\mathcal{Q}$, which is the content of Claim \ref{claim1-020123}.

For simplicity of future references, let us say that a subset $\mathcal{A}\subset\mathcal{Q}$ has the \textit{property $(P)$} if
\begin{enumerate}
\item $Q_{1}\in\mathcal{A}.$
\item $\mathcal{A}$ is not weakly transversal with pivot $Q_{1}$.
\end{enumerate}

We say that a subset $\mathcal{A}\subset\mathcal{Q}$ is \textit{minimal} if $\mathcal{A}^{\prime}\subset\mathcal{A}$ has the property $(P)$ if and only if $\mathcal{A}^{\prime}=\mathcal{A}$. It is clear that, since $\mathcal{Q}$ has the property $(P)$ itself, it must contain a minimal subset of cardinality at least $2$.

\begin{claim}\label{claim1-020123} Let $\mathcal{A}=\{Q_{1},K_{2},\ldots,K_{n}\}$ be a minimal set of $n$ cubes\footnote{Observe that $Q_{1}$ is the only ``$Q$" cube in this collection. The others are labeled by $K_{j}$.}. There is a set $D$ of $(d-n+2)$ canonical directions $v$ for which
\begin{equation}\label{condv-020123}
 \pi_{v}(Q_{1}) \cap \pi_{v}(K_{j}) \neq\emptyset,\quad\forall \quad 2\leq j\leq n.
\end{equation}
\end{claim}

\begin{proof}[Proof of Claim \ref{claim1-020123}] See Claim \ref{claim1-appendix150123} in the appendix.
\end{proof}

We know that $\mathcal{Q}$ has a minimal subset of cardinality $2\leq n\leq d+1$. By the previous claim and by conditions (a) and (b) of our initial reductions, if $\mathcal{A}^{\prime}=\{Q_{1},K_{2},\ldots,K_{n}\}$ is a minimal subset of $\mathcal{Q}$, for every $v\in D$ there is a number $\gamma_{v}$ such that
 \begin{equation*}
 \gamma_{v}\in \pi_{v}(Q_{1})\cap\bigcap_{j=2}^{n}\pi_{v}(K_{j}).
 \end{equation*}
 
Indeed, $\pi_{v}(Q_{1})$ intersects each $\pi_{v}(Q_{j})$ ``on the same side", so the intersection above must be nonempty (the existence of these $\gamma_{v}$ is the only reason why we may need to decompose the initial collection $\mathcal{Q}$ into sub-collections that satisfy (a) and (b)). 

For simplicity and without loss of generality, assume that $\mathcal{A}=\{Q_{1},Q_{2},\ldots,Q_{n}\}$ is minimal\footnote{Here we are assuming $K_{j}=Q_{j}$, $2\leq j\leq n$.} and $D=\{e_{1},\ldots,e_{d-n+2}\}$. Consider the points
 \begin{equation*}
 \eqalign{
 (\gamma_{1},\ldots,\gamma_{d-n+2},x_{d-n+3}^{j},\ldots,x_{d}^{j})&\displaystyle\in Q_{j},\quad 1\leq j\leq n,\cr
  (x_{1}^{l},\ldots,x_{d}^{l})&\displaystyle\in Q_{l},\quad n+1\leq l\leq d+1,\cr
 }
 \end{equation*}
 By hypothesis, $\textnormal{BL(\textbf{L}$(x)$,\textbf{p})}<\infty$ for the following collection of linear maps and exponents:
 \begin{equation*}
 \eqalign{
 L_{r}^{\gamma_{r}}(v_{1},\ldots,v_{d+1})&\displaystyle = v_{r}+2\gamma_{r}v_{d+1},\qquad 1\leq r\leq d-n+2, \cr
 L_{s}^{x_{s}^{1}}(v_{1},\ldots,v_{d+1})&\displaystyle = v_{s}+2x_{s}^{1}v_{d+1},\qquad d-n+3\leq s\leq d,\cr
 }
 \end{equation*}
 \begin{equation*}
 L_{d+r}^{(\gamma_{1},\ldots,\gamma_{d-n+2},x_{d-n+3}^{r+1},\ldots,x_{d}^{r+1})}(v_{1},\ldots,v_{d+1})= \begin{bmatrix}
    v_{1}+2\gamma_{1}v_{d+1}  \\
       \vdots \\
     v_{d-n+2}+2\gamma_{d-n+2}v_{d+1}  \\
     v_{d-n+3}+2x^{r+1}_{d-n+3}v_{d+1}  \\
    \vdots \\
     v_{d}+2x_{d}^{r+1}v_{d+1}
\end{bmatrix},\quad 1\leq r\leq n-1,
 \end{equation*}
\begin{equation*}
 L_{d+l}^{x^{l+1}}= \begin{bmatrix}
    v_{1}+2x_{1}^{l+1}v_{d+1}  \\
     \vdots \\
     v_{d}+2x_{d}^{l+1}v_{d+1}
\end{bmatrix},\quad n\leq l\leq d,
 \end{equation*}
 \begin{equation*}
 \textnormal{\textbf{p}}=\left(\frac{1}{d},\ldots,\frac{1}{d}\right).
 \end{equation*}
 
Define 

$$V:=\bigcap_{r=1}^{d-n+2}\textnormal{ker}(L_{r}^{\gamma_{r}}).$$ 

Observe that $\textnormal{dim}(V)=n-1$. Indeed, if we start with a vector $v=(v_{1},\ldots,v_{d+1})$ of $d+1$ ``free coordinates", we lose one degree of freedom for each kernel in the intersection above, since $L_{r}^{\gamma_{r}}(v)=0$ gives a relation between $v_{r}$ and $v_{d+1}$. We have $d-n+2$ many of them, hence the total degree of freedom is $(d+1)-(d-n+2)=n-1$, which is the dimension of $V$. On the other hand, for every $v\in V$ we have by definition
$$L_{r}^{\gamma_{r}}(v)=0,\qquad 1\leq r\leq d-n+2,$$
hence
 $$\sum_{j=1}^{d}\textnormal{dim}(L_{j}V)\leq n-2. $$
 
Also,
 \begin{equation*}
 L_{d+r}^{(\gamma_{1},\ldots,\gamma_{d-n+2},x_{d-n+3}^{r+1},\ldots,x_{d}^{r+1})}(v)= \begin{bmatrix}
    0  \\
       \vdots \\
    0  \\
     v_{d-n+3}+2x^{r+1}_{d-n+3}v_{d+1}  \\
    \vdots \\
     v_{d}+2x_{d}^{r+1}v_{d+1}
\end{bmatrix},\quad 1\leq r\leq n-1,
 \end{equation*}
 thus
 $$\textnormal{dim}(L_{d+r}V)\leq n-2, \quad 1\leq r\leq n-1.$$
 
 Since $\textnormal{dim}(V)=n-1$, we have the trivial bound
 $$\textnormal{dim}(L_{d+l}V)\leq n-1, \quad n\leq l\leq d.$$
 
 Altogether, these bounds imply
 \begin{equation*}
 \eqalign{
 \displaystyle\frac{1}{d}\left(\sum_{j=1}^{d}\textnormal{dim}(L_{j}V)+\sum_{\ell=1}^{d}\textnormal{dim}(L_{d+\ell}V)\right)&\displaystyle \leq \frac{1}{d}\left[(n-2)+(n-1)(n-2)+(d-n+1)(n-1)\right] \cr
 &\displaystyle= \frac{1}{d}[(n-1)d-1] \cr
 &\displaystyle< n-1 \cr
 &\displaystyle = \textnormal{dim}(V).\cr
 }
 \end{equation*}
 
Our initial hypothesis, however, is that $\textnormal{BL(\textbf{L}$(x)$,\textbf{p})}<\infty$, therefore by Theorem \ref{BCCT-030123} we must have
 $$\textnormal{dim}(V)\leq \displaystyle\frac{1}{d}\left(\sum_{j=1}^{d}\textnormal{dim}(L_{j}V)+\sum_{\ell=1}^{d}\textnormal{dim}(L_{d+\ell}V)\right),$$
 which gives a contradiction. We conclude that $\mathcal{Q}$ is weakly transversal with pivot $Q_{1}$.
\end{proof}

\subsection{An application to Restriction-Brascamp-Lieb inequalities}

The following conjecture was proposed in \cite{BBFL1} by Bennett, Bez, Flock and Lee:

\begin{conjecture}\label{conj1-010123} Suppose that for each $1\leq j\leq m$, $\Sigma_{j}:U_{j}\mapsto\mathbb{R}^{n}$ is a smooth parametrization of a $n_{j}$-dimensional submanifold $S_{j}$ of $\mathbb{R}^{n}$ by a neighborhood $U_{j}$ of the origin in $\mathbb{R}^{n_{j}}$. Let
$$\mathcal{E}_{j}g_{j}(\xi) := \int_{U_{j}}e^{-2\pi i\xi\cdot\Sigma_{j}(x)}g_{j}(x)\mathrm{d}x $$
be the associated (parametrized) extension operator. If the Brascamp-Lieb constant $\textnormal{BL}(\textnormal{\textbf{L}},\textnormal{\textbf{p}})$ is finite for the linear maps $L_{j}:=(\mathrm{d}\Sigma_{j}(0))^{\ast}:\mathbb{R}^{n}\mapsto\mathbb{R}^{n_{j}}$, then provided the neighborhoods $U_{j}$ of $0$ are chosen to be small enough, the inequality
\begin{equation}\label{ineq1-040123}
\int_{\mathbb{R}^{n}}\prod_{j=1}^{m}\left|\mathcal{E}_{j}g_{j}\right|^{2p_{j}} \lesssim \prod_{j=1}^{m}\|g_{j}\|^{2p_{j}}_{L^{2}(U_{j})}
\end{equation}
holds for all $g_{j}\in L^{2}(U_{j})$, $1\leq j\leq m$.
\end{conjecture}

\begin{remark} The weaker inequality
\begin{equation}
\int_{B(0,R)}\prod_{j=1}^{m}\left|\mathcal{E}_{j}g_{j}\right|^{2p_{j}} \lesssim_{\varepsilon} R^{\varepsilon} \prod_{j=1}^{m}\|g_{j}\|^{2p_{j}}_{L^{2}(U_{j})}
\end{equation}
involving an arbitrary $\varepsilon>0$ loss was established in \cite{BBFL1}.
\end{remark}

\begin{remark} Very few cases of Conjecture \ref{conj1-010123} are fully understood\footnote{Most of them being very elementary situations, as mentioned in \cite{BBFL1}.}. Recently, Bennett, Nakamura and Shiraki settled the \textit{rank-1 case} $n_{1}=\ldots=n_{m}=1$ as an application of their results on \textit{Tomographic Fourier Analysis}\footnote{See \cite{BN} for a more detailed exposition of this approach.}.
\end{remark}

Given their hybrid nature, estimates such as \eqref{ineq1-040123} are called \textit{Restriction-Brascamp-Lieb inequalities}. 

Our goal here is to verify Conjecture \ref{conj1-010123} in a special case. We chose to state the main result of this subsection in a way that does not emphasize the origin in the domains of $\Sigma_{j}$. The reason for this choice is that it brings to light key geometric features of the problem.

We will need a result from \cite{BBFL1} on the stability of Brascamp-Lieb constants\footnote{Theorem \ref{stability-jan1623} says that the map $\textnormal{\textbf{L}}\mapsto \textnormal{BL(\textbf{L},\textbf{p})}$ is \textit{locally bounded} for a fixed $\textnormal{\textbf{p}}$, and this is enough for our purposes. On the other hand, it was shown in \cite{BBCF} that the Brascamp-Lieb constant is \textit{continuous} in $\textnormal{\textbf{L}}$. It was later shown in \cite{BBBCF} that $\textnormal{BL(\textbf{L},\textbf{p})}$ is in fact \textit{locally H\"older continuous} in $\textnormal{\textbf{L}}$.}:

\begin{theorem}[\cite{BBFL1}]\label{stability-jan1623} Suppose that $(\textnormal{\textbf{L}}^{0},\textnormal{\textbf{p}})$ is a Brascamp-Lieb datum for which $\textnormal{BL}(\textnormal{\textbf{L}}^{0},\textnormal{\textbf{p}})<\infty$. Then there exists $\delta>0$ and a constant $C<\infty$ such that
$$\textnormal{BL(\textbf{L},\textbf{p})}\leq C $$
whenever $\|\textnormal{\textbf{L}}-\textnormal{\textbf{L}}^{0}\|<\delta$.
\end{theorem}

Now we are ready to state and prove our result:

\begin{theorem}\label{thm1-jan1623} Let $\Gamma$ and $\gamma_{j}$ be the parametrizations from \eqref{param1-jan1623} and \eqref{param2-jan1623}, respectively. If, for $x^{j}=(x_{1}^{j},\ldots,x_{d}^{j})\in\mathbb{R}^{d}$, the linear maps in \eqref{maps1-jan1323} satisfy 
\begin{equation}\label{cond1-160123}
\textnormal{BL(\textbf{L}$(x)$,\textbf{p})}<\infty \textnormal{ for } \textnormal{
\textbf{L}$(x)$}=(L_{1}^{x_{1}^{1}},\ldots,L_{2d}^{x^{d+1}})\textnormal{ and } \textnormal{\textbf{p}}=\left(\frac{1}{d},\ldots,\frac{1}{d}\right),
\end{equation} 
then there are small enough cube-neighborhoods $U_{i}\subset\mathbb{R}$ ($1\leq i\leq d$) of $x_{i}^{1}$ and $V_{\ell}\subset\mathbb{R}^{d}$ of $x^{\ell}$ ($2\leq\ell\leq d+1$) for which \eqref{ineq1-040123} holds.
\end{theorem}

\begin{remark} Rephrasing Theorem \ref{thm1-jan1623} in terms of the original statement, it says that Conjecture \ref{conj1-010123} holds for\footnote{Observe that we are just translating the domain of the $\Sigma$'s back to the origin.}
\begin{equation*}
\eqalign{
\displaystyle \Sigma_{i} &=  \gamma_{i}-(\delta_{1i}\cdot x_{i}^{1},\ldots,\delta_{di}\cdot x_{i}^{1},0),\qquad 1\leq i\leq d. \cr
\displaystyle \Sigma_{\ell} &=  \Gamma -(x^{\ell-d+1},0),\qquad d+1\leq \ell\leq 2d. \cr
\displaystyle m&=2d, \cr
\displaystyle \textnormal{\textbf{p}}&\displaystyle =\left(\frac{1}{d},\ldots,\frac{1}{d}\right). \cr
}
\end{equation*}
\end{remark}

\begin{proof}[Proof of Theorem \ref{thm1-jan1623}] The argument is just a matter of putting the pieces together. By \eqref{cond1-160123} and Theorem \ref{stability-jan1623}, there are small enough cube-neighborhoods $U_{i}\subset\mathbb{R}$ ($1\leq i\leq d$) of $x_{i}^{1}$ and $V_{\ell}\subset\mathbb{R}^{d}$ of $x^{\ell}$ ($2\leq\ell\leq d+1$) for which \eqref{cond1-160123} still holds\footnote{Our maps $L_{j}$ are sufficiently smooth for the stability theorem to be applied. The entries of the matrices that represent them are polynomials.}. Define
\begin{equation*}
\eqalign{
\displaystyle Q_{1} &\displaystyle := \overline{U_{1}\times\ldots\times U_{d}}, \cr
\displaystyle Q_{\ell} &\displaystyle := \overline{V_{\ell}},\quad 2\leq \ell\leq d+1. \cr
}
\end{equation*}

Now we apply Theorem \ref{thm1-010123} to conclude that the collection $\mathcal{Q}=\{Q_{1},\ldots,Q_{d+1}\}$ can be decomposed into $O(1)$ weakly transversal collections $\mathcal{Q}^{\prime}$ of $d+1$ cubes, each one having a cube $Q_{1}^{\prime}\subset Q_{1}$ as pivot.

\begin{figure}[H]
  \centering
\captionsetup{font=normalsize,skip=1pt,singlelinecheck=on}
  \includegraphics[scale=.65]{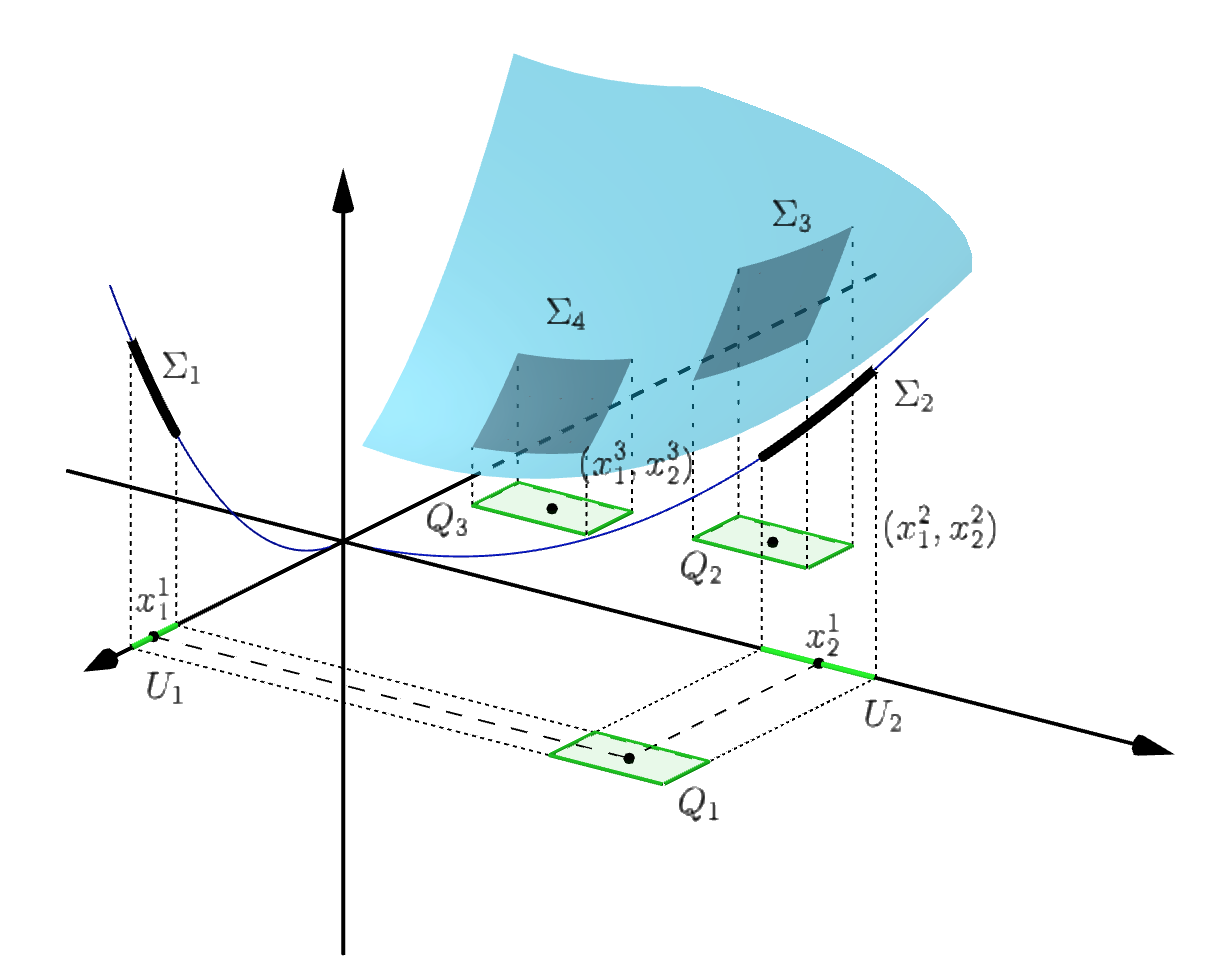}
 \caption{Unveiling the geometric features of the problem when $d=2$. The cubes we find from Theorem \ref{stability-jan1623} are weakly transversal, which gives us access to our earlier results.}
\end{figure}

To each such sub-collection we apply the endpoint estimate from Section \ref{dlineartheory} (all we need to apply it is weak transversality), which finishes the proof.

\end{proof}

\section{Further remarks}\label{finalremarks}

\begin{remark} It was pointed out to us by Jonathan Bennett that the $d$-dimensional estimates \eqref{rest1} for tensors are equivalent to certain $1$-dimensional mixed norm bounds. We present this remark in the following proposition:

\begin{proposition}[Bennett] For all $p,q\geq 1$, the estimate
\begin{equation}\label{ineq1apr3022}
\|\mathcal{E}_{d}g\|_{L^{q}_{\xi_{1},\ldots,\xi_{d+1}}}\lesssim \|g\|_{p}
\end{equation}
holds for tensors $g(x)=g_{1}(x_{1})\cdot\ldots\cdot g_{d}(x_{d})$ if and only if
\begin{equation}\label{ineq2apr3022}
\|\mathcal{E}_{1}f\|_{L^{dq}_{\xi_{2}}L^{q}_{\xi_{1}}} \lesssim \|f\|_{p}.
\end{equation}
holds.
\end{proposition}

\begin{proof} Assume first that \eqref{ineq1apr3022} holds for tensors. Then

\begin{equation*}
\eqalign{
\displaystyle \|\mathcal{E}_{1}f\|_{L^{dq}_{\xi_{2}}L^{q}_{\xi_{1}}} &=\displaystyle \left[\int\left[\int |\mathcal{E}_{1}f(\xi_{1},\xi_{2})|^{q}\mathrm{d}\xi_{1}\right]^{d}\mathrm{d}\xi_{2}\right]^{\frac{1}{dq}} \cr
&=\displaystyle \left[\int \prod_{j=1}^{d}\left[\int |\mathcal{E}_{1}f(\eta_{j},\xi_{2})|^{q}\mathrm{d}\eta_{j}\right]\mathrm{d}\xi_{2}\right]^{\frac{1}{dq}} \cr
&=\displaystyle \left[\int \prod_{j=1}^{d}\int |\mathcal{E}_{d}(f\otimes\ldots\otimes f)(\eta_{1},\ldots,\eta_{d})|^{q}\mathrm{d}\overrightarrow{\eta}\mathrm{d}\xi_{2}\right]^{\frac{1}{dq}} \cr
&=\displaystyle \|\mathcal{E}_{d}(f\otimes\ldots\otimes f)\|_{q}^{\frac{1}{d}} \cr
&\lesssim\displaystyle \|f\otimes\ldots\otimes f\|_{p}^{\frac{1}{d}} \cr \cr
&\lesssim\displaystyle \|f\|_{p},
}
\end{equation*}
which proves \eqref{ineq2apr3022}. Conversely, assuming that \eqref{ineq2apr3022} holds for all $f\in L^{p}([0,1])$ yields

\begin{equation*}
\eqalign{
\|\mathcal{E}_{d}(g_{1}\otimes\ldots\otimes g_{d})\|_{q}^{q} &=\displaystyle \int |\mathcal{E}_{1}g_{1}(\xi_{1},\xi_{d+1})|^{q}\cdot\ldots\cdot |\mathcal{E}_{1}g_{d}(\xi_{d},\xi_{d+1})|^{q}\mathrm{d}\xi_{1}\cdot\ldots\cdot \mathrm{d}\xi_{d+1} \cr
&=\displaystyle \int\prod_{j=1}^{d}\left[\int |\mathcal{E}_{1}g_{j}(\xi_{j},\xi_{d+1})|^{q}\mathrm{d}\xi_{j}\right]\mathrm{d}\xi_{d+1} \cr
&\leq\displaystyle \prod_{j=1}^{d}\left[\int\left[\int |\mathcal{E}_{1}g_{j}(\xi_{j},\xi_{d+1})|^{q}\mathrm{d}\xi_{j}\right]^{d}\mathrm{d}\xi_{d+1}\right]^{\frac{1}{d}} \cr
&=\displaystyle \prod_{j}\|\mathcal{E}_{1}g_{j}\|^{q}_{L^{dq}_{\xi_{d+1}}L^{q}_{\xi_{j}}} \cr
&\lesssim\displaystyle \prod_{j=1}^{d}\|g_{j}\|_{p}^{q} \cr
&=\displaystyle \|g\|_{p}^{q}.
}
\end{equation*}

\end{proof}

Estimates such as \eqref{ineq2apr3022} can be verified directly by interpolation. Taking sup in $\xi_{2}$ gives
\begin{equation}\label{Bennetttrick1}
\|\mathcal{E}_{1}f\|_{L^{\infty}_{\xi_{2}}L^{2}_{\xi_{1}}}\lesssim_{\varepsilon} \|f\|_{L^{2}([0,1])}.
\end{equation}
Conjecture \ref{restriction} for $d=1$ follows from
\begin{equation}\label{Bennetttrick2}
\|\mathcal{E}_{1}f\|_{L^{4+\varepsilon}_{\xi_{2},\xi_{1}}}\lesssim_{\varepsilon} \|f\|_{L^{4}{([0,1])}}
\end{equation}
for all $\varepsilon>0$. Using mixed-norm Riesz-Thorin interpolation with weights $\approx\frac{d-1}{d+1}$ for  \eqref{Bennetttrick1} and $\approx\frac{2}{d+1}$ for  \eqref{Bennetttrick2}, one obtains \eqref{ineq2apr3022} for $p=\frac{2(d+1)}{d}$ and $q=\frac{2(d+1)}{d}+\varepsilon^{\prime}$, which shows \eqref{ineq1apr3022} by the previous claim.

The reader will notice that our proof for the case $k=1$ of Theorem \ref{mainthmpaper} has a similar idea in its core: we interpolate (at the level of the sets $\mathbb{X}^{l_{1},\ldots,l_{d}}$) between two estimates similar to \eqref{Bennetttrick1} and \eqref{Bennetttrick2}. On the other hand, we have not found an extension of Bennett's remark to the case $2\leq k\leq d+1$, in which we still need to interpolate locally instead of globally and assume that only one function has a tensor structure.
\end{remark}

\begin{remark}
In \cite{TVV1} the authors obtain the following off-diagonal type bounds:

\begin{theorem*}[\cite{TVV1}] $\mathcal{ME}_{2,d}$ satisfies
$$\|\mathcal{ME}_{2,d}(g_{1},g_{2})\|_{2}\lesssim \|g_{1}\|_{2}\cdot\|g_{2}\|_{\frac{d+1}{d}}, $$
$$\|\mathcal{ME}_{2,d}(g_{1},g_{2})\|_{2}\lesssim \|f\|_{\frac{d+1}{d}}\cdot\|g\|_{2}. $$
\end{theorem*}

In general, under the extra hypothesis that either $g_{1}$ or $g_{2}$ is a full tensor, one can obtain all $k$-linear off-diagonal type bounds like $L^{p_{1}}\times\ldots\times L^{p_{k}}\mapsto L^{2}$ by a straightforward adaptation of the argument presented in Section \ref{klineartheory}. We chose not to include them in this manuscript.
\end{remark}

\begin{remark} Under the assumption that $g_{j}$ are \textit{full tensors}
$$g_{j}(x_{1},\ldots,x_{d})=g_{j,1}(x_{1})\cdot\ldots\cdot g_{j,d}(x_{d}), \quad 1\leq j\leq k,$$
the methods of this work allow to prove Conjecture \ref{generalklinearapr422}. We will not cover the details of this result here, but the idea is simply to interpolate between the $p=2$ result and the case $k=1$ for tensors.
\end{remark}

\appendix
\addcontentsline{toc}{section}{Appendices}

\section{Sharp examples}\label{appendixa}

The goal of this first appendix is to discuss the sharpness of Theorems \ref{mainthmpaper} and \ref{improvedklinearthm2}. We remark that sharp examples already exist in the literature, notably in the context of the bilinear problem for the sphere in Foschi and Klainerman's work \cite{FK1}, and in the multilinear case for surfaces of any signature in Hickman and Iliopoulou's paper \cite{Hick-Ili}. Our examples, however, exploit different ideas than those present in \cite{FK1} and \cite{Hick-Ili} in the sense that they are robust enough to address weakly transversal configurations of caps and give sharp results in such cases as well.

The first part of this appendix is about Theorem \ref{improvedklinearthm2}, whereas in the second one we prove that, to attain the sharp range of Conjecture \ref{klinear} in general, transversality can not be replaced by the concept of weak transversality that we introduce.

\subsection{Range optimality} The main result of this subsection is the following:

\begin{proposition}\label{prop1-feb1123} The condition $p\geq\frac{2(d+|\tau|+2)}{k(d+|\tau|)}$ is necessary for
Theorem \ref{improvedklinearthm2} to hold.
\end{proposition}

Our examples are constructed based on one-dimensional considerations. For the benefit of simplifying the notation, smoothing the exposition to the reader and to establish a clear link with Conjecture \ref{klinear}, we present them in the $|\tau|=k-1$ case, which is the smallest possible value for the corresponding $|\tau|$ of a given collection of transversal cubes (up to decomposing it into weakly transversal collections, see Claim \ref{claima2-081221}). It will be clear, however, how to work out the general case of arbitrary $|\tau|$, and we will point that out along the proof of Claim \ref{claim2apr1022}.

Consider the caps that project onto the following transversal domains via $x\mapsto |x|^{2}$:

\begin{equation*}
\eqalign{
U_{1}&=[0,1]^{d}, \cr
U_{j}&= [2,3]^{j-2}\times [4,5]\times [0,1]^{d-j+1},\quad 2\leq j\leq k. \cr
}
\end{equation*}

Observe that these caps are transversal as well\footnote{For general $|\tau|$ we would have to start with a different collection of cubes with the appropriate total degree of transversality.}, therefore the following argument for the case $|\tau|=k-1$ of Proposition \ref{prop1-feb1123} also shows that the range of Conjecture \ref{klinear} is necessary.

We present the examples separately to distinguish their features. For $k=d+1$ we will take appropriately placed cubes, whereas for $2\leq k\leq d$ we will take slabs (boxes with edges of two different scales).

\begin{claim}\label{claim1apr1022}Let $k=d+1$, $\delta>0$ small and let $A^{\delta}_{j}$ be given by
 \begin{equation*}
\eqalign{
A^{\delta}_{1}&=[0,\delta]^{d}, \cr
A^{\delta}_{j}&= [2,2+\delta]^{j-2}\times [4,4+\delta]\times [0,\delta]^{d-j+1},\quad 2\leq j\leq d+1. \cr
}
\end{equation*}

Define $f^{\delta}_{j}:=\mathbbm{1}_{A^{\delta}_{j}}$. Then
$$\frac{\left\|\prod_{j=1}^{d+1}\mathcal{E}_{U_{j}}f^{\delta}_{j}\right\|_{p}}{\prod_{j=1}^{d+1} \|f^{\delta}_{j}\|_{2}}\quad \gtrsim\quad\delta^{\frac{d(d+1)}{2}-\frac{1}{p}(d+1)}.$$

Therefore, letting $\delta\rightarrow 0$ implies $p\geq \frac{2}{d}$ is a necessary condition for the $(d+1)$-linear extension conjecture to hold for this choice of $U_{j}$'s and for all $f_{j}$ that are full tensors.
\end{claim}

\begin{claim}\label{claim2apr1022}Let $2\leq k<d+1$, $\delta>0$ small and let $B^{\delta}_{j}$ be given by
 \begin{equation*}
\eqalign{
B^{\delta}_{1}&=[0,\delta^{2}]^{k-1}\times [0,\delta]^{d-k+1}, \cr
B^{\delta}_{j}&= [2,2+\delta^{2}]^{j-2}\times [4,4+\delta^{2}]\times [0,\delta^{2}]^{k-j}\times [0,\delta]^{d-k+1},\quad 2\leq j\leq k. \cr
}
\end{equation*}

Define $g^{\delta}_{j}:=\mathbbm{1}_{B^{\delta}_{j}}$. Then
$$\frac{\left\|\prod_{j=1}^{k}\mathcal{E}_{U_{j}}g^{\delta}_{j}\right\|_{p}}{\prod_{j=1}^{k} \|g^{\delta}_{j}\|_{2}}\quad \gtrsim\quad\delta^{\frac{k}{2}(d+k-1)-\frac{1}{p}(d+k+1)}.$$

Therefore, letting $\delta\rightarrow 0$ implies $p\geq \frac{2(d+k+1)}{k(d+k-1)}$ is a necessary condition for the $k$-linear extension conjecture to hold for this choice of $U_{j}$'s and for all $g_{j}$ that are full tensors.
\end{claim}

\begin{figure}[h]
  \centering
\captionsetup{font=normalsize,skip=1pt,singlelinecheck=on}
  \includegraphics[scale=.6]{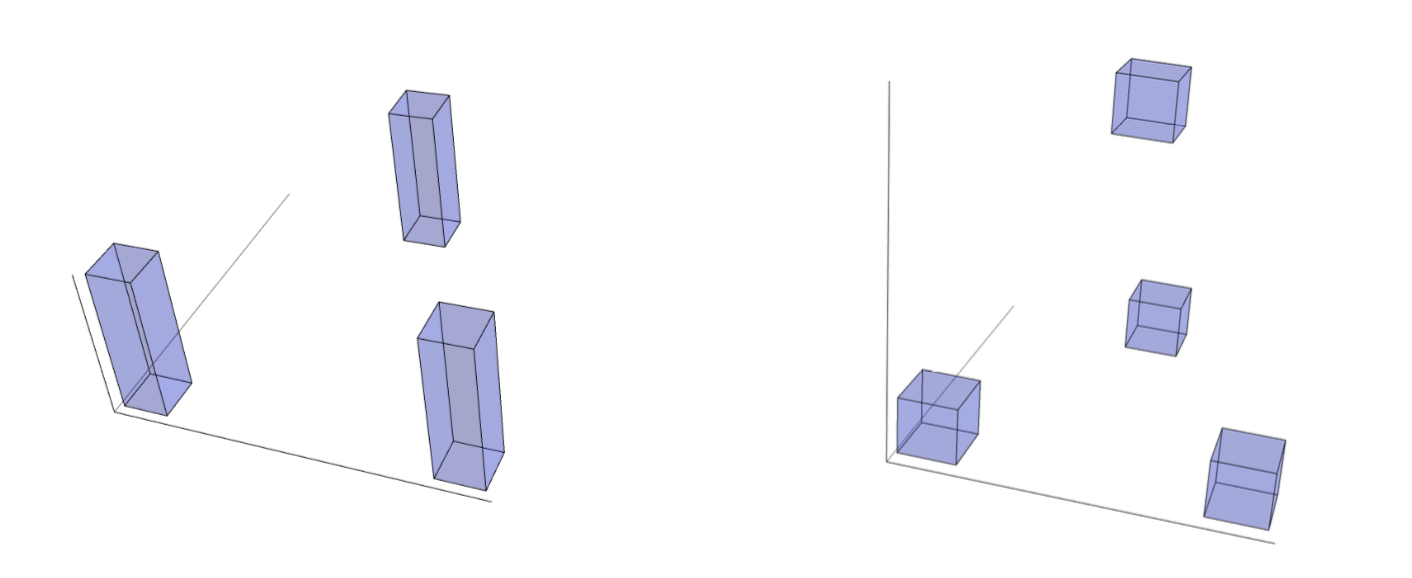}
 \caption{Cases $k=3$ and $k=4$ when $d=3$}
\end{figure}

Before proving the claims, we need the following lemma:

\begin{lemma}[Scale-$1$ phase-space portrait of $e^{2\pi ix^{2}}$] There exists a sequence of smooth bumps $(\varphi_{n})_{n\in\mathbb{Z}}$ such that:
\begin{enumerate}[(i)]
\item $\textnormal{supp}(\varphi_{n})\subset [n-1,n+1]$, $n\in\mathbb{Z}$,
\item $|\varphi^{(\ell)}_{n}(x)|\leq C_{\ell}$ uniformly in $n\in\mathbb{Z}$
and such that
$$e^{2\pi ix^{2}}=\sum_{n\in\mathbb{Z}}e^{4\pi inx}\varphi_{n}(x). $$
\end{enumerate}
\end{lemma}

\begin{proof} See \cite{MS2}, Proposition $1.10$ on page 23.
\end{proof}

Rescaling with $t>0$, the corresponding phase space portrait of $e^{2\pi itx^{2}}$ is

$$e^{2\pi itx^{2}}=e^{2\pi i(\sqrt{t}x)^{2}}=\sum_{n\in\mathbb{Z}}e^{4\pi in\sqrt{t}x}\varphi_{n}(\sqrt{t}x). $$

Observe that $\widetilde{\varphi}_{t}(x)=\varphi_{n}(\sqrt{t}x)$ is adapted to the Heisenberg box $[\frac{n}{\sqrt{t}},\frac{n+1}{\sqrt{t}}]\times [0,\sqrt{t}]$, but strictly supported on $[\frac{n-1}{\sqrt{t}},\frac{n+1}{\sqrt{t}}]$. This way, we can write

\begin{equation}\label{decompositionapr1022}
e^{2\pi itx^{2}}=\sum_{n\in\mathbb{Z}}\Phi_{n,t}(x),
\end{equation}
where $\Phi_{n,t}$ is adapted to the Heisenberg box $[\frac{n}{\sqrt{t}},\frac{n+1}{\sqrt{t}}]\times [2n\sqrt{t},(2n+1)\sqrt{t}]$.

\begin{proof}[Proof of Claim \ref{claim1apr1022}] Motivated by the uncertainty principle, the first step is to analyze the behavior of the extension operator $\mathcal{E}_{U_{j}}$ applied to $f_{j}^{\delta}$ on a box whose sizes are reciprocal to the ones of $\textnormal{supp}(f_{j}^{\delta})$. More precisely, we will show that $|E_{U_{j}}(f_{j}^{\delta})|\gtrsim\delta^{d}$ on such boxes.

If $\delta<\frac{1}{\sqrt{t}}$,
\begin{equation*}
\eqalign{
\mathcal{E}_{U_{1}}(f_{1}^{\delta})(\xi_{1},\ldots,\xi_{d},t)&\displaystyle=\prod_{j=1}^{d}\left[\int_{0}^{\delta}e^{-2\pi i\xi_{j}x_{j}}e^{-2\pi itx_{j}^{2}}\mathrm{d}x_{j}\right] \cr
&\displaystyle= \prod_{j=1}^{d}\left[\int_{0}^{\delta}e^{-2\pi i\xi_{j}x_{j}}\cdot [\Phi_{0,t}(x_{j})+\Phi_{1,t}(x_{j})]\mathrm{d}x_{j}\right], \cr
}
\end{equation*}
since $\textnormal{supp}(\Phi_{n,t})\cap [0,\delta]=\emptyset$ if $n\in \mathbb{Z}\backslash \{0,1\}$. If $|\xi_{j}x_{j}|<\frac{1}{N}$ ($N$ is a big number to be chosen later), we then have:
\begin{equation}\label{eq1may722}
\eqalign{
|\mathcal{E}_{U_{1}}(f_{1}^{\delta})(\xi_{1},\ldots,\xi_{d},t)|&\displaystyle = \prod_{j=1}^{d}\left|\int_{0}^{\delta}e^{-2\pi \xi_{j}x_{j}}\cdot [\Phi_{0,t}(x_{j})+\Phi_{1,t}(x_{j})]\mathrm{d}x_{j}\right|, \cr
&\displaystyle\geq \prod_{j=1}^{d}\left(\left|\int_{0}^{\delta}[\Phi_{0,t}(x_{j})+\Phi_{1,t}(x_{j})]\mathrm{d}x_{j}\right|-\left|\int_{0}^{\delta} [e^{-2\pi \xi_{j}x_{j}}-1]\cdot [\Phi_{0,t}(x_{j})+\Phi_{1,t}(x_{j})]\right|\right) \cr,
}
\end{equation}
where $N$ is picked so that $[e^{-2\pi \xi_{j}x_{j}}-1]$ is close enough to zero to make 
$$A_{j}:=\left|\int_{0}^{\delta}[\Phi_{0,t}(x_{j})+\Phi_{1,t}(x_{j})]\mathrm{d}x_{j}\right|$$
dominate each factor above. Since $A_{j}\gtrsim\delta$ (recall that $\Phi_{0,t}$ and $\Phi_{1,t}$ are adapted to Heisenberg boxes of size $\frac{1}{\sqrt{t}}\times\sqrt{t}$ and $\delta<\frac{1}{\sqrt{t}}$),
we conclude that if $|\xi_{j}|\lesssim \frac{1}{\delta}$ for $1\leq j\leq d$ and $|t|<\frac{1}{\delta^{2}}$, then
$$|\mathcal{E}_{U_{1}}(f_{1}^{\delta})(\xi_{1},\ldots,\xi_{d},t)| \geq \delta^{d}.$$

If $\phi$ is a bump supported on $[-1,1]$, we have just proved that
\begin{equation}\label{eq1apr1122}
|\mathcal{E}_{U_{1}}(f_{1}^{\delta})(\xi_{1},\ldots,\xi_{d},t)|\gtrsim \delta^{d}\phi_{\delta}(\xi_{1})\cdot\ldots\cdot\phi_{\delta}(\xi_{d})\phi_{\delta^{2}}(t),
\end{equation}
where $\phi_{\delta}(\xi):=\phi(\delta x)$. Analogously, if $\delta<\frac{1}{\sqrt{t}}$,
\begin{equation*}
\eqalign{
\mathcal{E}_{U_{2}}(f_{2}^{\delta})&\displaystyle (\xi_{1},\ldots,\xi_{d},t) \cr
&\displaystyle=\left[\int_{4}^{4+\delta}e^{-2\pi i\xi_{1}x_{1}}e^{-2\pi itx_{1}^{2}}\mathrm{d}x_{1}\right]\cdot\prod_{j=2}^{d}\left[\int_{0}^{\delta}e^{-2\pi i\xi_{j}x_{j}}e^{-2\pi itx_{j}^{2}}\mathrm{d}x_{j}\right] \cr
&\displaystyle=\underbrace{\left[\int_{4}^{4+\delta}e^{-2\pi i\xi_{1}x_{1}}\left(\sum_{n\in\mathbb{Z}}\Phi_{n,t}(x_{1})\right)\mathrm{d}x_{1}\right]}_{I_{1}}\cdot \prod_{j=2}^{d}\underbrace{\left[\int_{0}^{\delta}e^{-2\pi i\xi_{j}x_{j}}\cdot [\Phi_{0,t}(x_{j})+\Phi_{1,t}(x_{j})]\mathrm{d}x_{j}\right]}_{I_{j}}, \cr
}
\end{equation*}

There are at most $O(1)$ integers $n$ such that $\textnormal{supp}(\Phi_{n,t})\cap [4,4+\delta]\neq\emptyset$, and they cluster around $\lfloor 4\sqrt{t} \rfloor$. Without loss of generality, one can assume that $n=4\sqrt{t}$ so that the main contribution for $I_{1}$ comes from $\Phi_{4\sqrt{t},t}$ whose Heisenberg box is $[4,4+\frac{1}{\sqrt{t}}]\times[8t,8t+\sqrt{t}]$. The modulation $e^{-2\pi i\xi_{i}x_{i}}$ shifts this box vertically by $-\xi_{1}$, and $I_{1}$ is negligible if the boxes $[4,4+\frac{1}{\sqrt{t}}]\times[8t-\xi_{1},8t+\sqrt{t}-\xi_{1}]$ and $[0,\delta]\times [0,\frac{1}{\delta}]$ are disjoint in frequency, so we need $|\xi_{1}-8t|\lesssim \frac{1}{\delta}$ to have a significant contribution to $I_{1}$. In that case, 

$$|I_{1}|\gtrsim \left|\int_{4}^{4+\delta}e^{-2\pi i\xi_{1}x_{1}}\Phi_{4\sqrt{t},t}(x_{1})\mathrm{d}x_{1}\right|\gtrsim \delta. $$

The analysis of $I_{j}$ for $j\geq 2$ is the same as the one for the factors of $E_{U_{1}}(f_{1}^{\delta})$. We conclude that if $|\xi_{1}-8t|\lesssim \frac{1}{\delta}$, $|\xi_{j}|\lesssim\frac{1}{\delta}$ for $2\leq j\leq d$ and $|t|\leq\frac{1}{\delta^{2}}$, then
$$|\mathcal{E}_{U_{2}}(f_{2}^{\delta})(\xi_{1},\ldots,\xi_{d},t)| \geq \delta^{d}.$$

As before,
\begin{equation*}
|\mathcal{E}_{U_{2}}(f_{2}^{\delta})(\xi_{1},\ldots,\xi_{d},t)|\gtrsim \delta^{d}\phi_{\delta}(\xi_{1}-8t)\cdot \phi_{\delta}(\xi_{2})\ldots\cdot\phi_{\delta}(\xi_{d})\phi_{\delta^{2}}(t).
\end{equation*}

The extensions $\mathcal{E}_{U_{j}}(f^{\delta}_{j})$ for $3\leq j\leq d+1$ are treated in the same way we treated $\mathcal{E}_{U_{2}}(f^{\delta}_{2})$. The conclusion is that
\begin{equation}\label{eq2apr1122}
|\mathcal{E}_{U_{j}}(f_{j}^{\delta})(\xi_{1},\ldots,\xi_{d},t)|\gtrsim \delta^{d}\phi_{\delta}(\xi_{1}-4t)\cdot\ldots\cdot \phi_{\delta}(\xi_{j-2}-4t)\cdot   \phi_{\delta}(\xi_{j-1}-8t)\cdot   \phi_{\delta}(\xi_{j})\cdot\ldots\cdot\phi_{\delta}(\xi_{d})\phi_{\delta^{2}}(t),
\end{equation}
for all $2\leq j\leq d+1$.

Let $\xi=(\xi_{1},\ldots,\xi_{d})$. From \eqref{eq1apr1122} and \eqref{eq2apr1122} we obtain

\begin{equation}\label{eq3apr1122}
\eqalign{
\displaystyle\prod_{j=1}^{d+1}|\mathcal{E}_{U_{j}}(f_{j}^{\delta})(\xi,t)|&\displaystyle\gtrsim \delta^{d(d+1)}\left[\phi_{\delta^{2}}(t)\prod_{l=1}^{d}\phi_{\delta}(\xi_{l})\right] \cr
&\displaystyle\times \left[\prod_{j=2}^{d}\phi_{\delta}(\xi_{1}-4t)\cdot\ldots\cdot \phi_{\delta}(\xi_{j-2}-4t)\cdot   \phi_{\delta}(\xi_{j-1}-8t)\cdot   \phi_{\delta}(\xi_{j})\cdot\ldots\cdot\phi_{\delta}(\xi_{d})\phi_{\delta^{2}}(t) \right] \cr
}
\end{equation}

Now we analyze the support of the product of the right-hand of \eqref{eq3apr1122}. Notice that we have at least one bump like $\phi_{\delta}(\xi_{j})$ for every $1\leq j\leq d+1$, so $|\xi_{j}|\lesssim\frac{1}{\delta}$ is a necessary condition for the product not to be zero. On the other hand, the conditions
\begin{equation*}
\eqalign{
|\xi_{j}|&\displaystyle\lesssim \frac{1}{\delta} \cr
|\xi_{j}-8t|&\displaystyle\lesssim \frac{1}{\delta} \cr
}
\end{equation*}
together imply $|t|\lesssim\frac{1}{\delta}$, which is much more restrictive than the $|t|\lesssim\frac{1}{\delta^{2}}$ that comes from the support of the bump $\phi_{\delta^{2}}(t)$. We conclude that the right-hand side of \eqref{eq3apr1122} is supported on the box
$$R^{\ast}_{\delta}=\left\{(\xi_{1},\ldots,\xi_{d},t)\in\mathbb{R}^{d+1};\quad |t|\lesssim\frac{1}{\delta}, \quad |\xi_{j}|\lesssim\frac{1}{\delta},\quad 1\leq j\leq d\right\}. $$

Finally,
\begin{equation}
\eqalign{
\displaystyle\frac{\left\|\prod_{j=1}^{d+1}\mathcal{E}_{U_{j}}f^{\delta}_{j}\right\|_{p}}{\prod_{j=1}^{d+1} \|f^{\delta}_{j}\|_{2}}\quad  &\displaystyle\gtrsim \quad \frac{\delta^{d(d+1)}\cdot |R^{\ast}_{\delta}|^{\frac{1}{p}}}{\delta^{\frac{d(d+1)}{2}}} \cr
&\displaystyle\gtrsim \quad \delta^{\frac{d(d+1)}{2}-\frac{1}{p}(d+1)} \cr
}
\end{equation}
and the claim follows.
\end{proof}

\begin{proof}[Proof of Claim \ref{claim2apr1022}] The outline of the following argument is the same as the one used in previous proof. Let $\xi = (\xi_{1},\ldots,\xi_{d})$. If $\delta^{2}<\frac{1}{\sqrt{t}}$,
\begin{equation*}
\eqalign{
\mathcal{E}_{U_{1}}(g_{1}^{\delta})(\xi,t)&\displaystyle=\prod_{j=1}^{k-1}\left[\int_{0}^{\delta^{2}}e^{-2\pi i\xi_{j}x_{j}}e^{-2\pi itx_{j}^{2}}\mathrm{d}x_{j}\right]\cdot\prod_{l=k}^{d}\left[\int_{0}^{\delta}e^{-2\pi i\xi_{l}x_{l}}e^{-2\pi itx_{l}^{2}}\mathrm{d}x_{l}\right] \cr
&\displaystyle= \prod_{j=1}^{k-1}\left[\int_{0}^{\delta^{2}}e^{-2\pi i\xi_{j}x_{j}}[\Phi_{0,t}(x_{j})+\Phi_{1,t}(x_{j})]\mathrm{d}x_{j}\right]\cdot\prod_{l=k}^{d}\underbrace{\left[\int_{0}^{\delta}e^{-2\pi i\xi_{l}x_{l}}\left(\sum_{n\in\mathbb{Z}}\Phi_{n,t}(x_{l})\right)\mathrm{d}x_{l}\right]}_{(\ast)}, \cr
}
\end{equation*}
since $\textnormal{supp}(\Phi_{n,t})\cap [0,\delta^{2}]=\emptyset$ if $n\in \mathbb{Z}\backslash \{0,1\}$. If $\delta<\frac{1}{\sqrt{t}}$ (which is stronger than the previous condition $\delta^{2}<\frac{1}{\sqrt{t}}$), we can eliminate most $\Phi_{n,t}$ in $(\ast)$ as well: 

\begin{equation*}
\mathcal{E}_{U_{1}}(g_{1}^{\delta})(\xi,t)\displaystyle= \prod_{j=1}^{k-1}\left[\int_{0}^{\delta^{2}}e^{-2\pi i\xi_{j}x_{j}}[\Phi_{0,t}(x_{j})+\Phi_{1,t}(x_{j})]\mathrm{d}x_{j}\right]\cdot\prod_{l=k}^{d}\left[\int_{0}^{\delta}e^{-2\pi i\xi_{l}x_{l}}\cdot [\Phi_{0,t}(x_{l})+\Phi_{1,t}(x_{l})]\mathrm{d}x_{l}\right],
\end{equation*}

If $|\xi_{j}x_{j}|<\frac{1}{N}$ (for $N$ big enough), we then have:
\begin{equation*}
\eqalign{
|\mathcal{E}_{U_{1}}&\displaystyle (g_{1}^{\delta})(\xi,t)| \cr
&\displaystyle = \prod_{j=1}^{k-1}\left|\int_{0}^{\delta^{2}}e^{-2\pi\xi_{j}x_{j}}[\Phi_{0,t}(x_{j})+\Phi_{1,t}(x_{j})]\mathrm{d}x_{j}\right|\cdot\prod_{l=k}^{d}\left|\int_{0}^{\delta}e^{-2\pi\xi_{l}x_{l}}\cdot [\Phi_{0,t}(x_{l})+\Phi_{1,t}(x_{l})]\mathrm{d}x_{l}\right|, \cr
&\displaystyle\gtrsim \delta^{2(k-1)+(d-k+1)} \cr
&\displaystyle=\delta^{d+k-1},
}
\end{equation*}
by the same argument presented when we analyzed \eqref{eq1may722}. We conclude that if $|\xi_{j}|\lesssim \frac{1}{\delta^{2}}$ for $1\leq j\leq k-1$, $|\xi_{l}|\lesssim\frac{1}{\delta}$ for $k\leq l\leq d$ and $|t|<\frac{1}{\delta^{2}}$, then\footnote{For general $|\tau|$, we would have $|\tau|$ conditions of type $|\xi_{j}|\lesssim\frac{1}{\delta^{2}}$ and $(d-|\tau|)$ like $|\xi_{l}|\lesssim\frac{1}{\delta}$.}
$$|\mathcal{E}_{U_{1}}(g_{1}^{\delta})(\xi,t)| \gtrsim \delta^{d+k-1}.$$

Using the same notation from the proof of Claim \ref{claim1apr1022}, we have just proved that
\begin{equation}\label{eq4apr1122}
|\mathcal{E}_{U_{1}}(g_{1}^{\delta})(\xi,t)|\gtrsim \delta^{d}\phi_{\delta^{2}}(\xi_{1})\cdot\ldots\cdot\phi_{\delta^{2}}(\xi_{k-1})\phi_{\delta}(x_{k})\cdot\ldots\cdot\phi_{\delta}(x_{d})\cdot\phi_{\delta^{2}}(t),
\end{equation}
where $\phi_{\delta}(\xi):=\phi(\delta x)$ and $\phi$ is a bump supported on $[-1,1]$. Analogously, if $\delta<\frac{1}{\sqrt{t}}$,
\begin{equation*}
\eqalign{
\displaystyle\mathcal{E}_{U_{2}}&\displaystyle (g_{2}^{\delta})(\xi,t) \cr
&\displaystyle=\left[\int_{4}^{4+\delta^{2}}e^{-2\pi i\xi_{1}x_{1}}e^{-2\pi itx_{1}^{2}}\mathrm{d}x_{1}\right]\cdot\prod_{j=2}^{k-1}\left[\int_{0}^{\delta^{2}}e^{-2\pi i\xi_{j}x_{j}}e^{-2\pi itx_{j}^{2}}\mathrm{d}x_{j}\right]\cdot \prod_{l=k}^{d}\left[\int_{0}^{\delta}e^{-2\pi i\xi_{l}x_{l}}e^{-2\pi itx_{l}^{2}}\mathrm{d}x_{l}\right]  \cr
&\displaystyle=\underbrace{\left[\int_{4}^{4+\delta^{2}}e^{-2\pi i\xi_{1}x_{1}}\left(\sum_{n\in\mathbb{Z}}\Phi_{n,t}(x_{1})\right)\mathrm{d}x_{1}\right]}_{M_{1}}\cdot \prod_{j=2}^{k-1}\underbrace{\left[\int_{0}^{\delta^{2}}e^{-2\pi i\xi_{j}x_{j}}\cdot [\Phi_{0,t}(x_{j})+\Phi_{1,t}(x_{j})]\mathrm{d}x_{j}\right]}_{M_{j}} \cr
&\displaystyle\times \prod_{l=k}^{d}\underbrace{\left[\int_{0}^{\delta}e^{-2\pi i\xi_{l}x_{l}}\cdot[\Phi_{0,t}(x_{l})+\Phi_{1,t}(x_{l})]\mathrm{d}x_{l}\right]}_{M_{l}}.
}
\end{equation*}

As in the proof of Claim \ref{claim1apr1022}, the main contribution for $M_{1}$ comes from $\Phi_{4\sqrt{t},t}$, whose Heisenberg box is $[4,4+\frac{1}{\sqrt{t}}]\times[8t,8t+\sqrt{t}]$. The modulation $e^{-2\pi i\xi_{i}x_{i}}$ shifts this box vertically by $-\xi_{1}$, and $M_{1}$ is negligible if the boxes $[4,4+\frac{1}{\sqrt{t}}]\times[8t-\xi_{1},8t+\sqrt{t}-\xi_{1}]$ and $[0,\delta^{2}]\times [0,\frac{1}{\delta^{2}}]$ are disjoint in frequency, so we need $|\xi_{1}-8t|\lesssim \frac{1}{\delta^{2}}$ to have a significant contribution to $M_{1}$. In that case, 

$$|M_{1}|\gtrsim \left|\int_{4}^{4+\delta^{2}}e^{-2\pi i\xi_{1}x_{1}}\Phi_{2\sqrt{t},t}(x_{1})\mathrm{d}x_{1}\right|\gtrsim \delta^{2}. $$

The analysis of $M_{j}$ for $2\leq j\leq k-1$ and of $M_{l}$ for $k\leq l\leq d-k+1$ are the same as the one for the factors of $E_{U_{1}}(g_{1}^{\delta})$. We conclude that if $|\xi_{1}-8t|\lesssim \frac{1}{\delta^{2}}$, $|\xi_{j}|\lesssim\frac{1}{\delta^{2}}$ for $2\leq j\leq k-1$, $|\xi_{l}|\lesssim\frac{1}{\delta}$ for $k\leq l\leq d$ and $|t|\leq\frac{1}{\delta^{2}}$, then
$$|\mathcal{E}_{U_{2}}(g_{2}^{\delta})(\xi,t)| \geq \delta^{d+k-1}.$$

As before,
\begin{equation*}
|\mathcal{E}_{U_{2}}(g_{2}^{\delta})(\xi,t)|\gtrsim \delta^{d}\phi_{\delta}(\xi_{1}-8t)\cdot \phi_{\delta^{2}}(\xi_{2})\ldots\cdot\phi_{\delta^{2}}(\xi_{k-1})\cdot\phi_{\delta}(\xi_{k})\cdot\ldots\cdot\phi_{\delta}(\xi_{d})\phi_{\delta^{2}}(t).
\end{equation*}

The extensions $\mathcal{E}_{U_{j}}(g^{\delta}_{j})$ for $3\leq j\leq k$ are treated in the same way. The conclusion is that
\begin{equation}\label{eq5apr1122}
|\mathcal{E}_{U_{j}}(g_{j}^{\delta})(\xi,t)|\gtrsim \delta^{d}\phi_{\delta}(\xi_{1}-4t)\cdot\ldots\cdot \phi_{\delta}(\xi_{j-2}-4t)\cdot   \phi_{\delta}(\xi_{j-1}-8t)\cdot   \phi_{\delta}(\xi_{j})\cdot\ldots\cdot\phi_{\delta}(\xi_{d})\phi_{\delta^{2}}(t)
\end{equation}
for all $2\leq j\leq k$. From \eqref{eq4apr1122} and \eqref{eq5apr1122} we obtain
\begin{equation}\label{eq6apr1122}
\eqalign{
\displaystyle\prod_{j=1}^{k}&\displaystyle|\mathcal{E}_{U_{j}}(g_{j}^{\delta})(\xi,t)| \cr
&\displaystyle\gtrsim \delta^{k(d+k-1)}\left[\phi_{\delta^{2}}(t)\prod_{l=1}^{k-1}\phi_{\delta^{2}}(\xi_{l})\cdot\prod_{n=k}^{d}\phi_{\delta}(\xi_{n})\right] \cr
&\displaystyle\times \left[\prod_{j=2}^{d}\left(\prod_{n=1}^{j-2}\phi_{\delta^{2}}(\xi_{n}-4t)\right)\cdot   \phi_{\delta^{2}}(\xi_{j-1}-8t)\cdot  \left( \prod_{m=j}^{k-1}\phi_{\delta^{2}}(\xi_{m})\right)\cdot \left(\prod_{r=k}^{d}\phi_{\delta}(\xi_{r})\right)\cdot\phi_{\delta^{2}}(t) \right]. \cr
}
\end{equation}

Notice that we have at least one bump like $\phi_{\delta^{2}}(\xi_{j})$ for every $1\leq j\leq k-1$ and at least one $\phi_{\delta}(\xi_{l})$ for $k\leq l\leq d$, so $|\xi_{j}|\lesssim\frac{1}{\delta^{2}}$ and $|\xi_{l}|\lesssim\frac{1}{\delta}$ are necessary conditions for the product not to be zero. On the other hand, the conditions
\begin{equation*}
\eqalign{
|\xi_{j}|&\displaystyle\lesssim \frac{1}{\delta^{2}} \cr
|\xi_{j}-8t|&\displaystyle\lesssim \frac{1}{\delta^{2}} \cr
}
\end{equation*}
together imply $|t|\lesssim\frac{1}{\delta^{2}}$, which does not add any new information compared to the one coming from the bump $\phi_{\delta^{2}}(t)$ (this is the main difference between the analysis in Claims \ref{claim1apr1022} and \ref{claim2apr1022}). We conclude that the right-hand side of \eqref{eq6apr1122} is supported on the box
$$S^{\ast}_{\delta}=\left\{(\xi_{1},\ldots,\xi_{d},t)\in\mathbb{R}^{d+1};\quad |t|\lesssim\frac{1}{\delta^{2}}; \quad |\xi_{j}|\lesssim\frac{1}{\delta^{2}},\quad 1\leq j\leq k-1; \quad |\xi_{l}|\lesssim\frac{1}{\delta},\quad k\leq l\leq d \right\}. $$

Finally,
\begin{equation}
\eqalign{
\displaystyle\frac{\left\|\prod_{j=1}^{k}\mathcal{E}_{U_{j}}g^{\delta}_{j}\right\|_{p}}{\prod_{j=1}^{d+1} \|g^{\delta}_{j}\|_{2}}\quad  &\displaystyle\gtrsim \quad \frac{\delta^{(d+k-1)k}\cdot |S^{\ast}_{\delta}|^{\frac{1}{p}}}{\delta^{\frac{(d+k-1)k}{2}}} \cr
&\displaystyle\gtrsim \quad \delta^{\frac{(d+k-1)k}{2}-\frac{(d+k+1)}{p}} \cr
}
\end{equation}
and the claim follows.

\end{proof}

\subsection{Transversality as a necessary condition in general}\label{appendixa2}

A natural question is: given $k$ cubes $U_{j}$, $1\leq j\leq k$, is it possible to prove

$$\left\| \prod_{j=1}^{k}\mathcal{E}_{U_{j}}g_{j}\right\|_{p}\lesssim \prod_{j=1}^{k}\|g_{j}\|_{2}$$
for $p\geq \frac{2(d+k+1)}{k(d+k-1)}$ and all $g_{j}\in L^{2}(U_{j})$ if the $U_{j}$'s are assumed to be weakly transversal?

The answer is no and we will address it in this second part of the first appendix. As a consequence, we conclude that Theorem \ref{mainthmpaper} is sharp under weak transversality, as observed in Remark \ref{rem1may722}.

We will treat the case $k=3$ and $d=2$ for simplicity, but a similar construction holds in general. If three boxes $U_{1},U_{2},U_{3}\subset\mathbb{R}^{2}$ are not transversal, there is a line that crosses them. Assume without loss of generality that $U_{1}=[0,1]^{2}$, $U_{2}=[2,3]^{2}$ and $U_{3}=[4,5]^{2}$. We will show that

$$\|E_{U_{1}}(h_{1})\cdot E_{U_{2}}(h_{2})\cdot E_{U_{3}}(h_{3})\|_{p} \lesssim \|h_{1}\|_{2}\cdot \|h_{2}\|_{2}\cdot \|h_{3}\|_{2} $$
only if $p\geq\frac{10}{9}$. The trilinear extension conjecture for $d=2$ states that $p\geq 1$ is the sharp range under the transversality hypothesis. 

\begin{claim}\label{claim1apr1222} Define the sets $D_{j}^{\delta}$ by
\begin{align*}
D_{1}^{\delta} &= \left[\frac{\sqrt{2}-\delta^{2}}{2},\frac{\sqrt{2}+\delta^{2}}{2}\right]\times \left[-\frac{\delta}{2},\frac{\delta}{2}\right],\\
D_{2}^{\delta} &= \left[\frac{5\sqrt{2}-\delta^{2}}{2},\frac{5\sqrt{2}+\delta^{2}}{2}\right]\times \left[-\frac{\delta}{2},\frac{\delta}{2}\right],\\
D_{3}^{\delta} &= \left[\frac{9\sqrt{2}-\delta^{2}}{2},\frac{9\sqrt{2}+\delta^{2}}{2}\right]\times \left[-\frac{\delta}{2},\frac{\delta}{2}\right].
\end{align*}

Define $h_{j}^{\delta}:=\mathbbm{1}_{D_{j}^{\delta}}$. Then
$$\frac{\left\|\prod_{j=1}^{3}\mathcal{E}_{D_{j}}h^{\delta}_{j}\right\|_{p}}{\prod_{j=1}^{3} \|h^{\delta}_{j}\|_{2}}\quad \gtrsim\quad\delta^{\frac{9}{2}-\frac{5}{p}}.$$
\end{claim}
\begin{proof} The proof is analogous to the ones of Claims \ref{claim1apr1022} and \ref{claim2apr1022}.
\end{proof}

Let the rhombuses $\widetilde{D}_{j}$ be given as follows:

\begin{align*}
\widetilde{D}_{1} &= \textnormal{Conv}\left((0,0);(\frac{\sqrt{2}}{2},\frac{\sqrt{2}}{2});(\frac{\sqrt{2}}{2},-\frac{\sqrt{2}}{2});(\sqrt{2},0)\right),\\
\widetilde{D}_{2} &= \textnormal{Conv}\left((2\sqrt{2},0);(\frac{5\sqrt{2}}{2},\frac{\sqrt{2}}{2});(\frac{5\sqrt{2}}{2},-\frac{\sqrt{2}}{2});(3\sqrt{2},0)\right),\\
\widetilde{D}_{3} &= \textnormal{Conv}\left((4\sqrt{2},0);(\frac{9\sqrt{2}}{2},\frac{\sqrt{2}}{2});(\frac{9\sqrt{2}}{2},-\frac{\sqrt{2}}{2});(5\sqrt{2},0)\right).
\end{align*}

Observe that $D_{j}^{\delta}\subset \widetilde{D}_{j}$ for $\delta>0$ small enough. Extend the domain of $h_{j}^{\delta}$ to $\widetilde{D}_{j}$ so that it is 0 on $\widetilde{D}_{j}\backslash D^{\delta}_{j}$.

\begin{figure}[h]
  \centering
\captionsetup{font=normalsize,skip=1pt,singlelinecheck=on}
  \includegraphics[scale=.4]{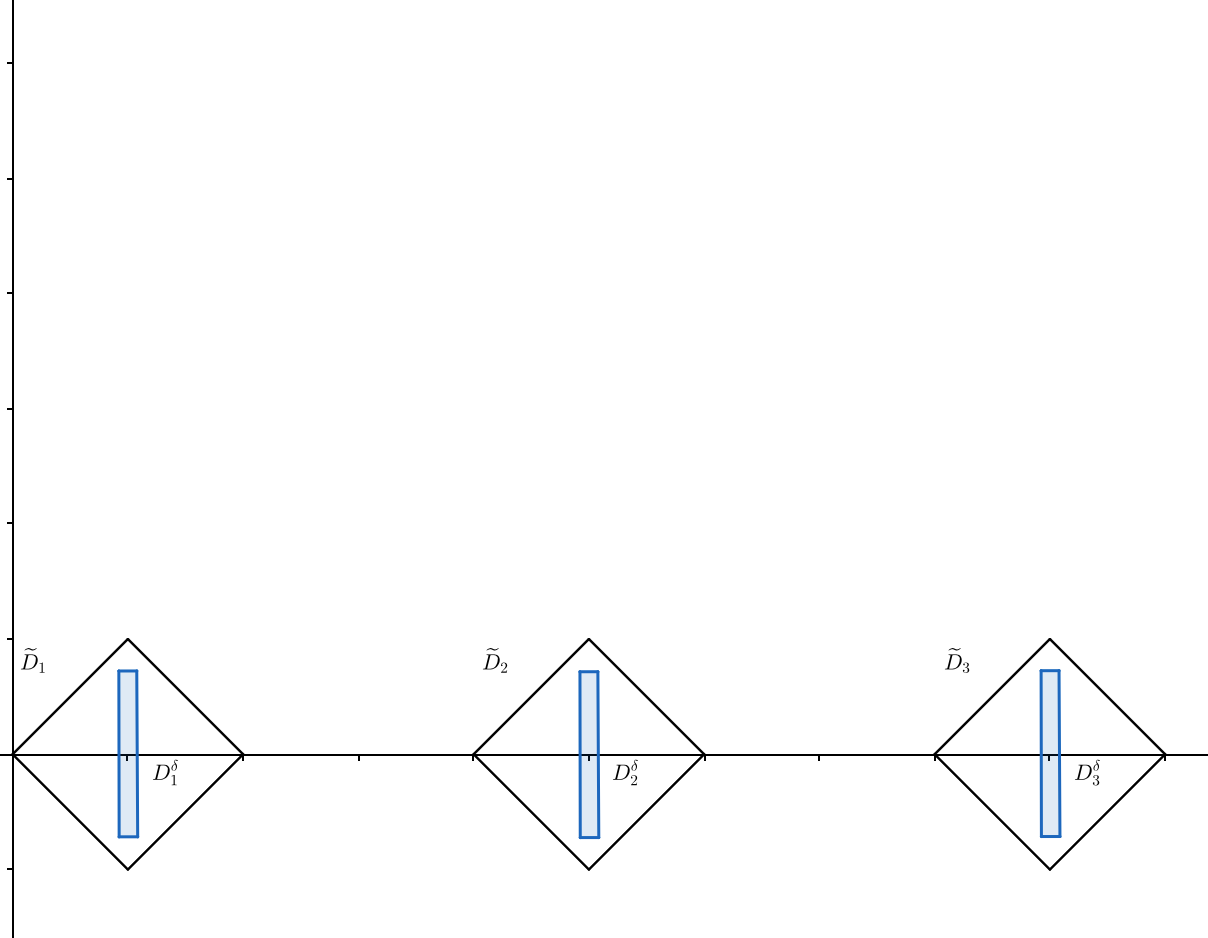}
 \caption{The function $h^{\delta}_{j}$.}
\end{figure}

Let $T$ be a $\frac{\pi}{4}$ counterclockwise rotation and let
$$H^{\delta}_{j}(x):= h^{\delta}_{j}\circ T^{-1}(x).$$
Notice that $T$ takes $\widetilde{D}_{j}$ to $U_{j}$, as shown in the picture below.

\begin{figure}[h]
  \centering
\captionsetup{font=normalsize,skip=1pt,singlelinecheck=on}
  \includegraphics[scale=.4]{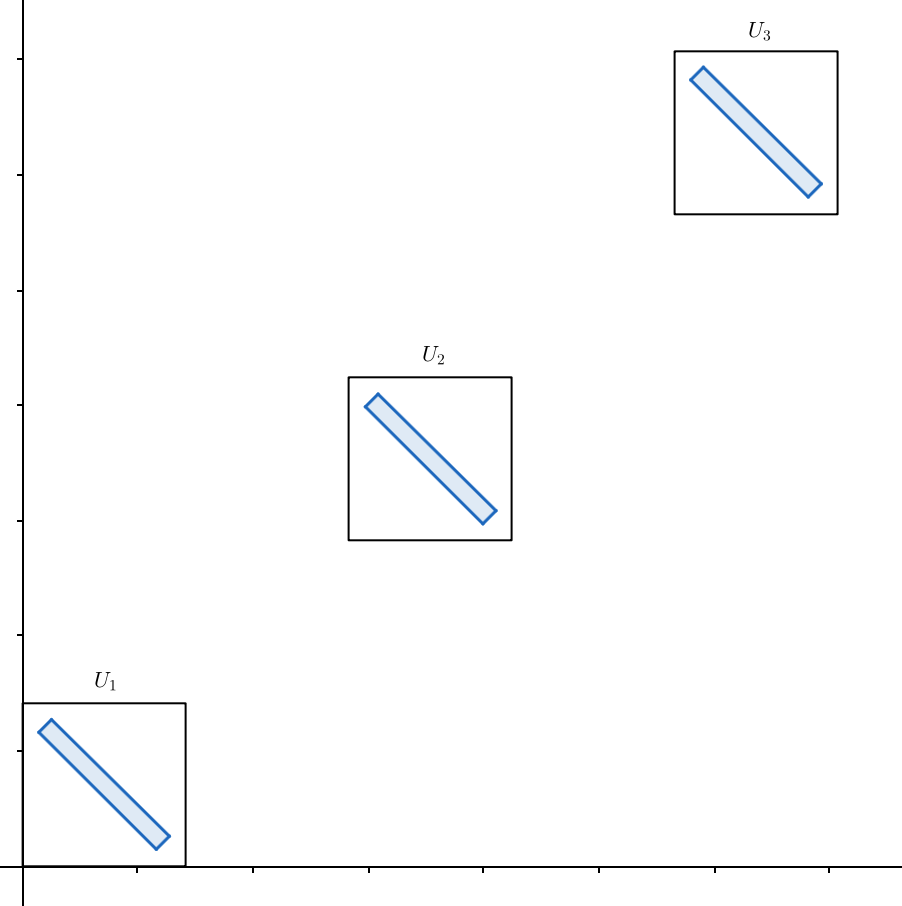}
 \caption{$H^{\delta}_{j}$ is supported on $U_{j}$.}
\end{figure}

Since $L^{p}$ norms are invariant under rotations, we have

$$\frac{\left\|\prod_{j=1}^{3}\mathcal{E}_{U_{j}}H^{\delta}_{j}\right\|_{p}}{\prod_{j=1}^{3} \|H^{\delta}_{j}\|_{2}}\quad \gtrsim\quad\delta^{\frac{9}{2}-\frac{5}{p}}$$
from Claim \ref{claim1apr1222}. Letting $\delta\rightarrow 0$ shows that we need $p\geq\frac{10}{9}$, so the sharp range $p\geq 1$ can not be obtained if the boxes $U_{1}, U_{2}, U_{3}$ are not transversal.

\begin{remark} As expected, the functions $H^{\delta}_{j}$ do not have a tensor structure with respect to the canonical basis. If this was the case, our methods would have allowed us to prove that the corresponding trilinear extension operator maps $L^{2}\times L^{2}\times L^{2}$ to $L^{1}$.
\end{remark}

\section{Technical results}\label{appendixb}
Here we collect a few technical results used throughout the paper.

\begin{theorem}\label{fracint} For $0<\gamma<d$, $1<p<q<\infty$, and $\frac{1}{q}=\frac{1}{p}-\frac{d-\gamma}{d}$, we have
\begin{equation}
\|f\ast(|y|^{-\gamma})\|_{L^{q}(\mathbb{R}^{d})}\leq A_{p,q}\cdot\|f\|_{L^{p}(\mathbb{R}^{d})}.
\end{equation}
\end{theorem}
\begin{proof} Proposition $7.8$ in \cite{MS}
\end{proof}

\begin{theorem}[Nonstationary phase]\label{nonstat-22aug2021} Let $a\in C^{\infty}_{0}$ and
$$I(\lambda)=\int_{\mathbb{R}^{d}}e^{2\pi i\lambda\phi(\xi)}a(\xi)\mathrm{d}\xi. $$
If $\nabla\phi\neq 0$ on supp($a$), then
$$|I(\lambda)|\leq C(N,a,\phi)\lambda^{-N} $$
as $\lambda\rightarrow\infty$, for arbitrary $N\geq 1$.
\end{theorem}
\begin{proof} Lemma $4.14$ in \cite{MS}.
\end{proof}

\begin{theorem}[Stationary phase]\label{stat-22aug2021} If $\nabla\phi(\xi_{0}) = 0$ for some $\xi_{0}\in \textnormal{supp($a$)}$, $\nabla\phi\neq 0$ away from $\xi_{0}$ and the Hessian of $\phi$ at the stationary point $\xi_{0}$ is nondegenerate, i.e., $\det{D^{2}\phi(\xi_{0})}\neq 0$, then for all $\lambda\geq 1$
$$|I(\lambda)|\leq C(N,a,\phi)\lambda^{-\frac{d}{2}}. $$
\end{theorem}
\begin{proof} Lemma $4.15$ in \cite{MS}.
\end{proof}

We now restate and prove the main claim from Section \ref{twt}:

\begin{claim}\label{claima2-081221} Given a collection $\mathcal{Q}=\{Q_{1},\ldots,Q_{k}\}$ of transversal cubes, each $Q_{l}\in\mathcal{Q}$ can be partitioned into $O(1)$ many sub-cubes
$$Q_{l}=\bigcup_{i}Q_{l,i} $$
so that all collections $\widetilde{\mathcal{Q}}$ made of picking one sub-cube $Q_{l,i}$ per $Q_{l}$
$$\widetilde{\mathcal{Q}}=\{\widetilde{Q}_{1},\ldots,\widetilde{Q}_{k}\},\quad \widetilde{Q}_{l}\in\{Q_{l,i}\}_{i},$$
are weakly transversal.
\end{claim}

\begin{proof} For each $1\leq j\leq d$, consider the set $A_{j}$ of endpoints of the intervals $\pi_{j}(Q_{1}),\ldots,\pi_{j}(Q_{k})$. Using these endpoints to partition this collection of intervals, one can assume that there are three cases for two cubes $Q_{r}$ and $Q_{s}$:

\begin{enumerate}
\item $\pi_{j}(Q_{r})\cap \pi_{j}(Q_{s})=\emptyset$.
\item $\pi_{j}(Q_{r})= \pi_{j}(Q_{s})$.
\item $\pi_{j}(Q_{r})\cap\pi_{j}(Q_{s})=\{p_{r,s}\}$, where $p_{r,s}$ is an endpoint of both $\pi_{j}(Q_{r})$ and $\pi_{j}(Q_{s})$.
\end{enumerate}

We can go one step further and assume that all $\pi_{j}(Q_{s})$ that intersect a given $\pi_{j}(Q_{r})$ (but distinct from it) do so at the same endpoint. Indeed, if $\pi_{j}(Q_{s_{1}})\cap\pi_{j}(Q_{r})=\{p\}$, $\pi_{j}(Q_{s_{2}})\cap\pi_{j}(Q_{r})=\{q\}$ and $\pi_{j}(Q_{r})=[p,q]$, we can simply split $\pi_{j}(Q_{r})$ in half and obtain intervals that satisfy this property.

Now we choose a point $x_{j,r}$ in every interval $\pi_{j}(Q_{r})$:

\begin{enumerate}
\item If $\pi_{j}(Q_{r})\cap\pi_{j}(Q_{s})=\emptyset$ for all $s\neq r$, let $x_{j,r}$ be $c_{j,r}$, the center of $\pi_{j}(Q_{r})$.
\item If $\pi_{j}(Q_{r})$ intersects some $\pi_{j}(Q_{s_{1}})$ at $p$, any other $\pi_{j}(Q_{s_{2}})$ that intersects $\pi_{j}(Q_{r})$ also does it at $p$. In this case choose $x_{j,r}=x_{j,s}=p$ for all $s$ such that $\pi_{j}(Q_{r})\cap \pi_{j}(Q_{s})\neq\emptyset$.
\end{enumerate}

Let us now show that, after the reductions above, the transversal set of cubes $\mathcal{Q}$ is weakly transversal. More precisely, for a fixed $1\leq l\leq k$, we will show that there is a set of $(k-1)$ canonical directions that together with $Q_{l}$ satisfy \eqref{eqwtmay722}. Let $\overrightarrow{x_{i}}\in Q_{i}$ for $1\leq i\leq k$ be given in coordinates by
$$\overrightarrow{x_{i}}=(x_{1,i},x_{2,i},\ldots,x_{d,i}). $$

The normal vector to $\mathbb{P}^{d}$ at $\overrightarrow{x_{i}}$ is

$$\overrightarrow{v_{i}} = (-2x_{1,i},-2x_{2,i},\ldots,-2x_{d,i},1). $$

Then the cubes in $\mathcal{Q}$ are transversal if and only if the matrix

\begin{equation*}
\begin{pmatrix}
-2x_{1,1} & -2x_{1,2} & \cdots & -2x_{1,k} \\
-2x_{2,1} & -2x_{2,2} & \cdots & -2x_{2,k} \\
\vdots  & \vdots  & \ddots & \vdots  \\
-2x_{d,1} & -2x_{d,2} & \cdots & -2x_{d,k} \\
1 & 1 & \cdots & 1
\end{pmatrix}
\end{equation*}
has rank $k$ for all $x_{j,i}\in \pi_{j}(Q_{i})$, $1\leq j\leq d$, $1\leq i\leq k$.

By Lemma \ref{matrixlemma-081221} (proven at the end of this appendix), there are $(k-1)$ rows $R_{i_{n}}=(-2x_{i_{n},1},\ldots,-2x_{i_{n},k})$ of the above matrix, $1\leq n\leq k-1$, such that

\[
    \begin{dcases}
        x_{i_{1},l} \neq x_{i_{1},1} \\
       \qquad \vdots \\
        x_{i_{l-1},l} \neq x_{i_{l-1},l-1} \\
         x_{i_{l},l} \neq x_{i_{l},l+1} \\
          \qquad \vdots \\
        x_{i_{k-1},l}\neq x_{i_{k-1},k}. \\
    \end{dcases}
\]

Because of the choices we made,  $x_{i_{n},l}\neq x_{i_{n},r}$ implies $\pi_{i_{n}}(Q_{l})\cap\pi_{i_{n}}(Q_{r})=\emptyset$, which finishes the proof.

\end{proof}

Finally, we state and prove the auxiliary linear algebra lemma used in the proof of Claim \ref{claima2-081221}.

\begin{lemma}\label{matrixlemma-081221} Let $M$ be the $(d+1)\times k$ matrix

\begin{equation*}
M = 
\begin{pmatrix}
a_{1,1} & a_{1,2} & \cdots & a_{1,k} \\
a_{2,1} & a_{2,2} & \cdots & a_{2,k} \\
\vdots  & \vdots  & \ddots & \vdots  \\
a_{d,1} & a_{d,2} & \cdots & a_{d,k} \\
1 & 1 & \cdots & 1
\end{pmatrix}
\end{equation*}
and assume that it has rank $k$. For each column $C_{j}=(a_{1,j},\ldots,a_{d,j},1)$ there are $(k-1)$ rows $R_{i_{l}}=(a_{i_{l},1},\ldots,a_{i_{l},k})$, $1\leq l\leq k-1$, such that

\[
    \begin{dcases}
        a_{i_{1},j} \neq a_{i_{1},l_{1}} \\
        a_{i_{2},j} \neq a_{i_{2},l_{2}} \\
       \qquad \vdots \\
        a_{i_{k-1},j}\neq a_{i_{k-1},l_{k-1}}, \\
    \end{dcases}
\]
where $(l_{1},l_{2},\ldots,l_{k-1})$ is some permutation of $(1,2,\ldots ,j-1,j+1,\ldots ,k)$.
\end{lemma}

\begin{proof} Let us first consider the case $k=d+1$. We have to show that for all columns $C_{j}$ the first $k-1$ rows satisfy the property of the lemma. Observe that the product

\begin{equation*}
MA=
\begin{pmatrix}
a_{1,1} & a_{1,2} & \cdots & a_{1,k} \\
a_{2,1} & a_{2,2} & \cdots & a_{2,k} \\
\vdots  & \vdots  & \ddots & \vdots  \\
a_{k-1,1} & a_{k-1,2} & \cdots & a_{k-1,k} \\
1 & 1 & \cdots & 1
\end{pmatrix}
\cdot
\underbrace{\begin{pmatrix}
1 & 1 & \cdots & 1 &1 & 1 \\
-1 & 0 & \cdots &0 & 0 &0 \\
0 & -1 & \cdots & 0 & 0 &0 \\
\vdots  & \vdots  & \ddots & \vdots & \vdots& \vdots  \\
0 & 0 & \cdots & -1 &0 & 0 \\
0 & 0 & \cdots & 0 & -1 & 0
\end{pmatrix}}_{k\times k\textnormal{ matrix }A}
\end{equation*}

is a rank $k$ matrix equal to

\begin{equation*}
\begin{pmatrix}
(a_{1,1}-a_{1,2}) & (a_{1,1}-a_{1,3}) & \cdots & (a_{1,1}-a_{1,k-1}) & (a_{1,1}-a_{1,k}) & a_{1,1} \\
(a_{2,1}-a_{2,2}) & (a_{2,1}-a_{2,3}) & \cdots & (a_{2,1}-a_{2,k-1}) & (a_{2,1}-a_{2,k}) &a_{2,1} \\
(a_{3,1}-a_{3,2}) & (a_{3,1}-a_{3,3}) & \cdots & (a_{3,1}-a_{3,k-1}) & (a_{3,1}-a_{3,k}) & a_{3,1} \\
\vdots  & \vdots  & \ddots & \vdots & \vdots& \vdots  \\
(a_{k-1,1}-a_{k-1,2}) & (a_{k-1,1}-a_{k-1,3}) & \cdots & (a_{k-1,1}-a_{k-1,k-1}) & (a_{k-1,1}-a_{k-1,k}) & a_{k-1,1} \\
0 & 0 & \cdots & 0 & 0 & 1
\end{pmatrix}.
\end{equation*}

By computing the Laplace expansion with respect to the last row, we conclude that $\det{(MA)}$ is equal to

\begin{equation*}
\det
\begin{pmatrix}
(a_{1,1}-a_{1,2}) & (a_{1,1}-a_{1,3}) & \cdots & (a_{1,1}-a_{1,k-1}) & (a_{1,1}-a_{1,k}) \\
(a_{2,1}-a_{2,2}) & (a_{2,1}-a_{2,3}) & \cdots & (a_{2,1}-a_{2,k-1}) & (a_{2,1}-a_{2,k}) \\
(a_{3,1}-a_{3,2}) & (a_{3,1}-a_{3,3}) & \cdots & (a_{3,1}-a_{3,k-1}) & (a_{3,1}-a_{3,k}) \\
\vdots  & \vdots  & \ddots & \vdots & \vdots  \\
(a_{k-1,1}-a_{k-1,2}) & (a_{k-1,1}-a_{k-1,3}) & \cdots & (a_{k-1,1}-a_{k-1,k-1}) & (a_{k-1,1}-a_{k-1,k})
\end{pmatrix}.
\end{equation*}

The entries of this matrix are

$$x_{i,j}:=a_{i,1}-a_{i,j+1},\quad 1\leq i,j\leq k-1. $$

The column $C_{1}$ has the property of the lemma if and only if there is some permutation $\pi$ of $(1,2,\ldots,k-1)$ such that

\[
    \begin{dcases}
       x_{1,\pi(1)} = a_{1,1} - a_{1,\pi(1)+1} \neq 0 \\
       x_{2,\pi(2)} =  a_{2,1} - a_{2,\pi(2)+1} \neq 0\\
       \quad\qquad \vdots \\
       x_{k-1,\pi(k-1)} =  a_{k-1,1} -a_{k-1,\pi(k-1)+1}\neq 0. \\
    \end{dcases}
\]

If this was not the case, for all such permutations $\pi$ of $(1,2,\ldots,k-1)$ at least one among $x_{1,\pi(1)}$, $x_{2,\pi(2)}$, $\ldots$, $x_{k-1,\pi(k-1)}$ would be zero. Hence
$$det{(MA)}=\sum_{\pi\in S_{k-1}}\textnormal{sgn}(\pi)\cdot x_{1,\pi(1)}\cdot\ldots\cdot x_{k-1,\pi(k-1)}=0,$$
a contradiction. A similar argument shows that any other column also has this property.

The case $k<d+1$ can be reduced to the previous one. Indeed, the rank $k$ condition guarantees that there is a $k\times k$ minor of $M$ that has rank $k$. There are two possibilities:

\begin{enumerate}
\item \underline{There is a $k\times k$ minor of rank $k$ that has a row of $1$'s.}

This is identical to the case $k=d+1$ and we conclude that the rows that generate this minor are the ones that satisfy the property of the lemma.

\item  \underline{No $k\times k$ minor of rank $k$ has a row of $1$'s.}

Here the rows of all non-singular minors are among the first $d$ ones of $M$. Let $R_{i_{l}}$, $1\leq l\leq k$, be $k$ rows of $M$ that generate such a minor $\widetilde{M}$:

\begin{equation*}
\widetilde{M} = 
\begin{pmatrix}
a_{i_{1},1} & a_{i_{1},2} & \cdots & a_{i_{1},k} \\
a_{i_{2},1} & a_{i_{2},2} & \cdots & a_{i_{2},k} \\
\vdots  & \vdots  & \ddots & \vdots  \\
a_{i_{k-1},1} & a_{i_{k-1},2} & \cdots & a_{i_{k-1},k} \\
a_{i_{k},1} & a_{i_{k},2} & \cdots & a_{i_{k},k}
\end{pmatrix}.
\end{equation*}

Proceed as in the case $k=d+1$ and multiply $\widetilde{M}$ by the matrix $A$ to obtain

\begin{equation*}
\widetilde{M}A=
\begin{pmatrix}
(a_{i_{1},1}-a_{i_{1},2}) & (a_{i_{1},1}-a_{i_{1},3}) & \cdots &  (a_{i_{1},1}-a_{i_{1},k}) & a_{i_{1},1} \\
(a_{i_{2},1}-a_{i_{2},2}) & (a_{i_{2},1}-a_{i_{2},3}) & \cdots &  (a_{i_{2},1}-a_{i_{2},k}) & a_{i_{2},1} \\
(a_{i_{3},1}-a_{i_{3},2}) & (a_{i_{3},1}-a_{i_{3},3}) & \cdots &  (a_{i_{3},1}-a_{i_{3},k}) & a_{i_{3},1} \\
\vdots  & \vdots  & \ddots & \vdots & \vdots  \\
(a_{i_{k-1},1}-a_{i_{k-1},2}) & (a_{i_{k-1},1}-a_{i_{k-1},3}) & \cdots  & (a_{i_{k-1},1}-a_{i_{k-1},k}) & a_{i_{k-1},1} \\
(a_{i_{k},1}-a_{i_{k},2}) & (a_{i_{k},1}-a_{i_{k},3}) & \cdots  & (a_{i_{k},1}-a_{i_{k},k}) & a_{i_{k},1}
\end{pmatrix}.
\end{equation*}

By computing the Laplace expansion along the last column of $\widetilde{M}A$, we conclude that at least one $(k-1)\times (k-1)$ minor obtained from the first $(k-1)$ columns of $\widetilde{M}A$ is non-singular. We argue again as in the $k=d+1$ case to find the $k-1$ rows that satisfy the property of the lemma for the column $C_{1}$. An analogous argument works for any other column of $M$, but these $k-1$ special rows may vary from column to column.
\end{enumerate}
\end{proof}

Let us recall some of the terminology from the proof of Theorem \ref{thm1-010123} in Section \ref{WTBLA}. A subset $\mathcal{A}\subset\mathcal{Q}$ has the \textit{property $(P)$} if
\begin{enumerate}
\item $Q_{1}\in\mathcal{A}.$
\item $\mathcal{A}$ is not weakly transversal with pivot $Q_{1}$.
\end{enumerate}

We say that $\mathcal{A}\subset\mathcal{Q}$ is \textit{minimal} if $\mathcal{A}^{\prime}\subset\mathcal{A}$ has the property $(P)$ if and only if $\mathcal{A}^{\prime}=\mathcal{A}$. Since $\mathcal{Q}$ itself has the property $(P)$, it must contain a minimal subset of cardinality at least $2$.

\begin{claim}\label{claim1-appendix150123} Let $\mathcal{A}=\{Q_{1},K_{2},\ldots,K_{n}\}$ be a minimal set of $n$ cubes\footnote{Observe that $Q_{1}$ is the only ``$Q$" cube in this collection. The others are labeled by $K_{j}$.}. There is a set $D$ of $(d-n+2)$ canonical directions $v$ for which

\begin{equation}\label{condv-appendix-020123}
 \pi_{v}(Q_{1}) \cap \pi_{v}(K_{j}) \neq\emptyset,\quad\forall \quad 2\leq j\leq n.
\end{equation}
\end{claim}

\begin{proof}[Proof of Claim \ref{claim1-appendix150123}] If $n=2$, then $Q_{1}\cap K_{2}\neq\emptyset$ and the claim follows directly. If $n>2$, observe that $\mathcal{A}^{\prime}=\{Q_{1},K_{2},\ldots,K_{n-1}\}$ is weakly transversal with pivot $Q_{1}$, otherwise $\mathcal{A}$ would not be minimal. Hence there are $1\leq j_{1},\ldots,j_{n-2}\leq d$ \textit{distinct} such that
\begin{equation}\label{condv2-020123}
    \begin{dcases}
        \pi_{j_{1}}(Q_{1}) \cap \pi_{j_{1}}(K_{2}) =\emptyset, \\
       \qquad \vdots \\
        \pi_{j_{n-2}}(Q_{1}) \cap \pi_{j_{n-2}}(K_{n-1}) =\emptyset. \\
    \end{dcases}
\end{equation}

Let $D:=\{e_{1},\ldots,e_{d}\}\backslash\{e_{j_{1}},\ldots,e_{j_{n-2}}\}$. In what follows, we will show that \eqref{condv-appendix-020123} holds for this set of directions. Notice that if
\begin{equation}\label{condicao1ap-150123}
 \pi_{l}(Q_{1}) \cap \pi_{l}(K_{n}) =\emptyset
\end{equation}
for some $l\in D$, then $\mathcal{A}$ would be weakly transversal with pivot $Q_{1}$ (because \eqref{condv2-020123} together with \eqref{condicao1ap-150123} verify the definition of weak transversality), which is false by hypothesis. Hence \eqref{condv-appendix-020123} holds for $j=n$.

Let us argue by induction that, if \eqref{condv-appendix-020123} holds for $1\leq m<n-1$ cubes $K_{n},K_{\alpha_{1}},\ldots,K_{\alpha_{m-1}}$, then it's possible to find a new one $K_{\alpha_{m}}$ for which \eqref{condv-appendix-020123} also holds\footnote{We are done if there are $m=n-1$ for which \eqref{condv-appendix-020123} holds, therefore we assume the strict inequality $m<n-1$.}. This will be achieved by the following algorithm: consider the set
$$\mathcal{A}^{\prime\prime}:=\{Q_{1},K_{n},K_{\alpha_{1}},\ldots,K_{\alpha_{m-1}}\}. $$

By the minimality of $\mathcal{A}$, $\mathcal{A}^{\prime\prime}$ is weakly transversal with pivot $Q_{1}$, hence there are $1\leq r_{1},\ldots,r_{m}\leq d$ distinct such that

\begin{equation}
    \begin{dcases}
        \pi_{r_{1}}(Q_{1}) \cap \pi_{r_{1}}(K_{n}) =\emptyset, \\
        \pi_{r_{2}}(Q_{1}) \cap \pi_{r_{2}}(K_{\alpha_{1}}) =\emptyset, \\
       \qquad \vdots \\
        \pi_{r_{m}}(Q_{1}) \cap \pi_{r_{m}}(K_{\alpha_{m-1}}) =\emptyset. \\
    \end{dcases}
\end{equation}

Property $(P)$ for $\mathcal{A}$ implies $r_{1}\in \{j_{1},\ldots,j_{n-2}\}$\footnote{Otherwise we face the same problem that appeared in \eqref{condicao1ap-150123}.}. Then there is $j_{\beta_{1}}$ such that $r_{1}=j_{\beta_{1}}$, therefore

\begin{equation}\label{condicao2ap-150123}
    \begin{dcases}
        \pi_{j_{\beta_{1}}}(Q_{1}) \cap \pi_{j_{\beta_{1}}}(K_{\beta_{1}+1}) =\emptyset, \\
        \pi_{j_{\beta_{1}}}(Q_{1}) \cap \pi_{j_{\beta_{1}}}(K_{n}) =\emptyset. \\
    \end{dcases}
\end{equation}

Since $K_{\beta_{1}+1}$ appears in \eqref{condv2-020123}, it is one among $K_{2},\ldots,K_{n-1}$, hence $K_{\beta_{1}+1}\neq K_{n}$. We are done if $K_{\beta_{1}+1}\notin\mathcal{A}^{\prime\prime}$: indeed, if
\begin{equation}
 \pi_{l}(Q_{1}) \cap \pi_{l}(K_{\beta_{1}+1}) =\emptyset
\end{equation}
for some $l\in D$, then
\begin{equation}
    \begin{dcases}
        \pi_{j_{1}}(Q_{1}) \cap \pi_{j_{1}}(K_{2}) =\emptyset, \\
       \qquad \vdots \\
       \pi_{j_{\beta_{1}-1}}(Q_{1}) \cap \pi_{j_{\beta_{1}-1}}(K_{\beta_{1}}) =\emptyset, \\
       \pi_{l}(Q_{1}) \cap  \pi_{l}(K_{\beta_{1}+1}) =\emptyset, \\
        \pi_{j_{\beta_{1}+1}}(Q_{1}) \cap \pi_{j_{\beta_{1}+1}}(K_{\beta_{1}+2}) =\emptyset, \\
        \qquad \vdots \\
        \pi_{j_{n-2}}(Q_{1}) \cap \pi_{j_{n-2}}(K_{n-1}) =\emptyset. \\
        \pi_{j_{\beta_{1}}}(Q_{1}) \cap  \pi_{j_{\beta_{1}}}(K_{n}) =\emptyset, \\
    \end{dcases}
\end{equation}
and $\mathcal{A}$ would be weakly transversal with pivot $Q_{1}$ (by definition again), which contradicts property $(P)$. This way, we would find a new (\textit{not} in $\mathcal{A}^{\prime\prime}$) cube $K_{\beta_{1}+1}$ for which \eqref{condv-appendix-020123} also holds.

On the other hand, if $K_{\beta_{1}+1}=K_{\alpha_{q_{1}}}$ for some $K_{\alpha_{q_{1}}}\in\mathcal{A}^{\prime\prime}\backslash\{K_{n}\}$, then we simply switch the projections $\pi_{j_{\beta_{1}}}$ and $\pi_{r_{q_{1}+1}}$ in \eqref{condv2-020123} (they are distinct because $j_{\beta_{1}}=r_{1}\neq r_{q_{1}+1}$) and consider the conditions
\begin{equation}\label{condswitch1-030223}
    \begin{dcases}
        \pi_{j_{1}}(Q_{1}) \cap \pi_{j_{1}}(K_{2}) =\emptyset, \\
       \qquad \vdots \\
       \pi_{j_{\beta_{1}-1}}(Q_{1}) \cap \pi_{j_{\beta_{1}-1}}(K_{\beta_{1}}) =\emptyset, \\
       \pi_{r_{q_{1}+1}}(Q_{1}) \cap \pi_{r_{q_{1}+1}}(K_{\alpha_{q_{1}}}) =\emptyset, \\
        \pi_{j_{\beta_{1}+1}}(Q_{1}) \cap \pi_{j_{\beta_{1}+1}}(K_{\beta_{1}+2}) =\emptyset, \\
        \qquad \vdots \\
        \pi_{j_{n-2}}(Q_{1}) \cap \pi_{j_{n-2}}(K_{n-1}) =\emptyset \\
        \pi_{j_{\beta_{1}}}(Q_{1}) \cap \pi_{j_{\beta_{1}}}(K_{n}) =\emptyset, \\
    \end{dcases}
\end{equation}
where the last condition is taken from \eqref{condicao2ap-150123}. Since $j_{\beta_{1}}\neq r_{q_{1}+1}$, property $(P)$ for $\mathcal{A}$ again implies that $r_{q_{1}+1}=j_{\beta_{2}}$. Notice that $\beta_{2}\neq\beta_{1}$ because $r_{1}=j_{\beta_{1}}$ and $r_{1}\neq r_{q_{1}+1}$. This way, from \eqref{condv2-020123},

\begin{equation}
    \begin{dcases}
        \pi_{j_{\beta_{2}}}(Q_{1}) \cap \pi_{j_{\beta_{2}}}(K_{\beta_{2}+1}) =\emptyset, \\
        \pi_{j_{\beta_{2}}}(Q_{1}) \cap \pi_{j_{\beta_{2}}}(K_{\alpha_{q_{1}}}) =\emptyset. \\
    \end{dcases}
\end{equation}

The index $j_{\beta_{2}}$ is one of the elements in the set $\{j_{1},\ldots,j_{\beta_{1}-1},j_{\beta_{1}+1},\ldots,j_{n-2}\}$, hence $K_{\beta_{2}+1}$ is in the set $\{K_{2},\ldots,K_{\beta_{1}},K_{\beta_{1}+2},\ldots,K_{n-1}\}$. As before, we are done if $K_{\beta_{2}+1}\notin\mathcal{A}^{\prime\prime}$. If not, $K_{\beta_{2}+1}=K_{\alpha_{q_{2}}}$ for some $K_{\alpha_{q_{2}}}\in\mathcal{A}^{\prime\prime}\backslash\{K_{n},K_{\alpha_{q_{1}}}\}$ and we switch the projections $\pi_{j_{\beta_{2}}}$ and $\pi_{r_{q_{2}+1}}$ in \eqref{condswitch1-030223} to find some $\beta_{3}\notin\{\beta_{1},\beta_{2}\}$ such that
\begin{equation}
    \begin{dcases}
        \pi_{j_{\beta_{3}}}(Q_{1}) \cap \pi_{j_{\beta_{3}}}(K_{\beta_{3}+1}) =\emptyset, \\
        \pi_{j_{\beta_{3}}}(Q_{1}) \cap \pi_{j_{\beta_{3}}}(K_{\alpha_{q_{2}}}) =\emptyset. \\
    \end{dcases}
\end{equation}

We keep doing that until we find some $K_{\beta_{\ell}+1}\notin\mathcal{A}^{\prime\prime}$. This is guaranteed to happen since there are $n-1$ cubes $K_{j}$, but only $m<n-1$ of them in $\mathcal{A}^{\prime\prime}$. The conclusion is that 
\begin{equation*}
m<n-1\textnormal{ cubes $K_{j}$ satisfy \eqref{condv-appendix-020123}}\quad\Longrightarrow \quad m+1\textnormal{ cubes $K_{j}$ satisfy \eqref{condv-appendix-020123}},
\end{equation*}
therefore \eqref{condv-appendix-020123} holds for $2\leq j\leq n$.
\end{proof}

\Addresses

\end{document}